\documentclass[a4paper,twoside]{article}

\usepackage{amssymb}
\usepackage{amsmath}
\usepackage{amsthm}

\usepackage{mathrsfs}

\usepackage{natbib}
\usepackage{url}

\usepackage{esint}                

\newtheoremstyle{citing}
  {3pt}
  {3pt}
  {\itshape}
  {}
  {\bfseries}
  {.}
  {.5em}
  {\thmnote{#3}}
\theoremstyle{citing}
\newtheorem*{citing}{}

\newtheoremstyle{myparagraph}
  {3pt}
  {3pt}
  {\normalfont}
  {}
  {\bfseries}
  {.}
  {0.5em}
  {\thmnote{#3}}
\theoremstyle{myparagraph}

\theoremstyle{definition}

\swapnumbers
\theoremstyle{plain}
\newtheorem{theorem}{Theorem}[section]
\newtheorem{lemma}[theorem]{Lemma}
\newtheorem{corollary}[theorem]{Corollary}

\theoremstyle{remark}
\newtheorem{remark}[theorem]{Remark}
\newtheorem{example}[theorem]{Example}
\newtheorem*{hypotheses}{Hypotheses}

\theoremstyle{definition}
\newtheorem{definition}[theorem]{Definition}
\newtheorem{miniremark}[theorem]{}

\newcounter{counter1}
\newcounter{counter2}

\newcommand{\Var}{\mathbf{V}}     
\newcommand{\RVar}{\mathbf{RV}}   
\newcommand{\IVar}{\mathbf{IV}}   

\newcommand{\Lp}[1]{\mathbf{L}_{#1}}

\newcommand{\Sob}[3]{\mathbf{W}_{#1}^{{#2},{#3}}}

\newcommand{\ccspace}[1]{\mathscr{K}(#1)}

\newcommand{\qspace}{\mathbf{Q}}

\newcommand{\nat}{\mathscr{P}}

\newcommand{\integers}{\integers}

\newcommand{\rel}{\mathbf{R}}

\newcommand{\complex}{\mathbf{C}}

\newcommand{\grass}[2]{\mathbf{G}(#1,#2)}

\newcommand{\pp}{\mathbf{p}}
\newcommand{\qq}{\mathbf{q}}

\newcommand{\oball}[2]{\mathbf{U}(#1,#2)}

\newcommand{\cball}[2]{\mathbf{B}(#1,#2)}

\newcommand{\density}{\boldsymbol{\Theta}}

\newcommand{\unitmeasure}[1]{\boldsymbol{\alpha}(#1)}

\newcommand{\besicovitch}[1]{\boldsymbol{\beta}(#1)}

\newcommand{\isoperimetric}[1]{\boldsymbol{\gamma}(#1)}

\newcommand{\cylind}[3]{\mathbf{C} ( #1, #2, #3 )}

\newcommand{\cylinder}[4]{\mathbf{C} ( #1, #2, #3, #4 )}

\newcommand{\id}[1]{\mathbf{1}_{#1}}

\newcommand{\trans}[1]{\boldsymbol{\tau}_{#1}}

\newcommand{\weakD}{\mathbf{D}}

\newcommand{\ud}{\ensuremath{\,\mathrm{d}}}

\DeclareMathOperator{\with}{:}
\newcommand{\classification}[3]{{#1} \cap \{ {#2} \with {#3} \}}
\newcommand{\eqclassification}[3]{{(#1)} \cap \{ {#2} \with {#3} \}}
\newcommand{\bigclassification}[3]{{#1} \cap \big \{ {#2} \with {#3} \big \}}

\newcommand{\project}[1]{#1_\natural}

\newcommand{\eqproject}[1]{(#1)_\natural}

\newcommand{\perpproject}[1]{#1_\natural^\perp}

\newcommand{\pluslim}[1]{\to {#1}+}

\newcommand{\lIm}{[}
\newcommand{\rIm}{]}
\newcommand{\biglIm}{\big [}
\newcommand{\bigrIm}{\big ]}

\newcommand{\Lbrack}{[\![}
\newcommand{\Rbrack}{]\!]}

\newcommand{\vdim}{{m}}
\newcommand{\codim}{{n-m}}
\newcommand{\adim}{{n}}

\newcommand{\norm}[3]{\boldsymbol{|} #1 \boldsymbol{|}_{#2;#3}}

\newcommand{\dnorm}[3]{\boldsymbol{|} #1 \boldsymbol{|}_{-1,#2;#3}}

\newcommand{\hoelder}[2]{\mathbf{h}_{#1}(#2)}

\newcommand{\pairing}[3]{\boldsymbol{(}{#1},{#2}\boldsymbol{)}_{#3}}

\newcommand{\intertextenum}[1]{\setcounter{counter2}{\value{enumi}}\end{enumerate}#1\begin{enumerate}\setcounter{enumi}{\value{counter2}}}

\newcommand{\printRoman}[1]{\setcounter{counter1}{#1}\Roman{counter1}}

\newcommand{\leftB}{{\left < \rule{0pt}{2ex} \right .}}
\newcommand{\rightB}{{\left . \rule{0pt}{2ex} \right >}}

\DeclareMathOperator{\without}{\sim}

\newcommand{\restrict}{\mathop{\llcorner}}

\DeclareMathOperator{\aplus}{(+)}

\newcommand{\class}[1]{#1}

\newcommand{\tint}[2]{{\textstyle\int_{#1}^{#2}}}
\newcommand{\tfint}[2]{{\textstyle\fint_{#1}^{#2}}}

\newcommand{\tsum}[2]{{\textstyle\sum_{#1}^{#2}}}

\renewcommand{\min}{\inf}

\DeclareMathOperator{\card}{card}

\newcommand{\Clos}[1]{\mathop{\mathrm{Clos}}#1}

\newcommand{\measureball}[2]{{#1}\,{#2}}

\DeclareMathOperator{\Tan}{Tan}     
\DeclareMathOperator{\spt}{spt}     
\DeclareMathOperator{\im}{im}       
\DeclareMathOperator{\Int}{Int}     
\DeclareMathOperator{\diam}{diam}   
\DeclareMathOperator{\Lip}{Lip}     
\DeclareMathOperator{\dmn}{dmn}     
\DeclareMathOperator{\dist}{dist}   
\DeclareMathOperator{\Hom}{Hom}     
\DeclareMathOperator{\graph}{graph} 
\DeclareMathOperator{\ap}{ap}       
\DeclareMathOperator*{\aplim}{\mathrm{ap}\, \lim}   

\newcommand{\Lpnorm}[3]{{#1}_{({#2})}({#3})}
\newcommand{\eqLpnorm}[3]{{(#1)}_{({#2})}({#3})}
\DeclareMathOperator{\Lap}{Lap}

\title{Decay estimates for the quadratic tilt-excess \\ of integral varifolds}
\author{Ulrich Menne\thanks{The research was carried out while the author
was at the ETH Z\"urich and put in its final form
while the author was at the AEI Golm.
\textit{AEI publication number:} AEI-2009-093.}}
\begin{document}
\maketitle
\begin{abstract}
	This paper concerns integral varifolds of arbitrary dimension in an
	open subset of Euclidean space with its first variation given by
	either a Radon measure or a function in some Lebesgue space. Pointwise
	decay results for the quadratic tilt-excess are established for those
	varifolds. The results are optimal in terms of the dimension of the
	varifold and the exponent of the Lebesgue space in most cases, for
	example if the varifold is not two-dimensional.
\end{abstract}
\begin{center}
\end{center}
\tableofcontents
\section*{Introduction}
\addcontentsline{toc}{section}{\numberline{}Introduction}
\paragraph{Overview} This paper investigates pointwise regularity properties
of integral varifolds satisfying integrability conditions on its generalised
mean curvature where pointwise regularity is measured by the decay of the
quadratic tilt-excess. As classical regularity may fail on a set of positive
measure, see Allard \cite[8.1\,(2)]{MR0307015} and Brakke
\cite[6.1]{MR485012}, the notion of tilt-excess decay serves as a weak measure
of regularity suitable for studying regularity near almost every point of a
varifold. In fact, aside from being used as an intermediate step to classical
regularity, see Allard \cite{MR0307015}, decay estimates have been employed as
a tool for both perpendicularity of mean curvature in Brakke \cite{MR485012}
and locality of mean curvature in Sch\"atzle
\cite{rsch:willmore,MR2064971,MR1906780}.

In the present paper it is established that there is a qualitative change in
the nature of the results obtainable when the Sobolev exponent corresponding
to the integrability exponent of the mean curvature drops below $2$. The core
of the proof of the pointwise results relies on the harmonic approximation
procedure introduced by de~Giorgi in \cite{MR0179651} (see also
\cite[p.~231--263]{EnnioDeGiorgi_SelectedPapers}) and Almgren in
\cite{MR0225243} and used in the present setting by Allard in \cite{MR0307015}
and Brakke in \cite{MR485012}. Additionally, to obtain the present pointwise
results, a new coercive estimate is proven, the Sobolev Poincar\'e type
estimates of \cite{snulmenn.poincare} are adapted and a new iteration
procedure is introduced. The latter may also be used in studying partial
regularity for systems of elliptic partial differential equations.
\paragraph{Known results} The notation follows Federer \cite{MR41:1976} and,
concerning varifolds, Allard \cite{MR0307015}, see Section \ref{sec:not}.
\begin{hypotheses}
	Suppose $\vdim$ and $\adim$ are positive integers, $\vdim < \adim$, $1
	\leq p \leq \infty$, $U$ is an open subset of $\rel^\adim$, $V \in
	\IVar_\vdim ( U)$, $\| \delta V \|$ is a Radon measure and, if $p >
	1$,
	\begin{gather} \label{hp}
		\begin{aligned}
			& ( \delta V ) ( g ) = - {\textstyle\int} g (z)
			\bullet \mathbf{h} (V;z) \ud \| V \| z \quad
			\text{whenever $g \in \mathscr{D} ( U, \rel^\adim )$},
			\\
			& \mathbf{h} (V;\cdot) \in \Lp{p} ( \| V \| \restrict
			K, \rel^\adim ) \quad \text{whenever $K$ is a compact
			subset of $U$}.
		\end{aligned} \tag{$H_p$}
	\end{gather}
\end{hypotheses}
The present research is motivated by the question for which $0 < \alpha \leq
1$ the given hypotheses imply
\begin{gather*}
	\limsup_{r \to 0+} r^{-\alpha-\vdim/2} \big ( \tint{\oball{a}{r}
	\times \grass{\adim}{\vdim}}{} | \project{S} - \project{T} |^2 \ud V
	(z,S) \big )^{1/2} < \infty
\end{gather*}
for $V$ almost all $(a,T) \in U \times \grass{\adim}{\vdim}$. Brakke has shown
that one can take any $0 < \alpha < 1$ in case $p = 2$ and $\alpha = 1/2$ with
``$< \infty$'' replaced by ``$=0$'' in case $p = 1$ in
\cite[5.5,\,7]{MR485012}. Sch\"atzle \cite{MR2064971} has used results on
viscosity solutions from Caffarelli \cite{MR1005611} and Trudinger
\cite{MR995142} to establish several regularity results, in particular that if
$p > \vdim$, $p \geq 2$ and $\codim = 1$ then one can take $\alpha = 1$, see
also Sch\"atzle \cite{MR1906780} for a special case. Moreover, Sch\"atzle
showed in \cite[Theorem 3.1]{rsch:willmore} that if $p = 2$ then the key to
the general case is to prove existence of an approximate second order
structure of the varifold. Namely, if $p=2$ and there exists a countable
collection $C$ of $\vdim$ dimensional submanifolds of $\rel^\adim$ of class
$\class{2}$ with $\| V \| ( U \without \bigcup C ) = 0$ then one can take
$\alpha = 1$.

Whereas consideration of varifolds associated to submanifolds of class
$\class{2}$ clearly shows that $\alpha = 1$ is the largest $\alpha$ possibly
having this property, in case $\sup \{ 2, p \} < \vdim$ and $\frac{\vdim
p}{\vdim-p} < 2$ it can be seen from the examples in
\cite[1.2]{snulmenn.isoperimetric} that one cannot take $\alpha > \frac{\vdim
p}{2(\vdim-p)}$. Comparing this to Brakke's results, little is known for the
case $1 < p < 2$ and also in case $p = 1$ and $\vdim > 2$ there is a gap
between known positive results for $\alpha \leq 1/2$ and known counterexamples
for $\alpha > \frac{\vdim}{2(\vdim-1)}$.
\paragraph{Results of the present paper}
In case $\sup \{ 2,p \} < \vdim$ and $\frac{\vdim p}{\vdim-p} < 2$ these gaps
are closed by the following corollary.
\begin{citing} [\ref{cor:some_decay_rates} Corollary]
	Suppose $\vdim$, $\adim$, $p$, $U$, and $V$ are as in the preceding
	hypotheses \eqref{hp}, and either $\vdim \in \{ 1, 2 \}$ and $0 < \tau
	< 1$ or $\sup \{ 2, p \} < \vdim$ and $\tau = \frac{\vdim p}{2(
	\vdim-p)} < 1$.

	Then
	\begin{gather*}
		\limsup_{r \to 0+} r^{-\tau-\vdim/2} \big ( \tint{\oball{a}{r}
		\times \grass{\adim}{\vdim}}{} | \project{S} - \project{T} |^2
		\ud V (z,S) \big )^{1/2} < \infty
	\end{gather*}
	for $V$ almost all $(a,T)$.
\end{citing}
From the afore-mentioned examples it follows that $\tau$ cannot be replaced by
any larger number if $\vdim > 2$, see \ref{remark:discussion1}. However, using
the present result, it will be shown in \cite{snulmenn:c2.v3} that ``$<
\infty$'' can be replaced by ``$=0$'', see \ref{remark:discussion1}. The
corollary is a direct consequence of the following pointwise result.
\begin{citing} [\ref{thm:pointwise_decay} Theorem]
	Suppose $\vdim$, $\adim$, and $p$ are as in the preceding hypotheses
	\eqref{hp}, $Q$ is a positive integer,
	$0 < \delta \leq 1$, $0 < \alpha \leq 1$, $0 < \tau \leq 1$, and
	\begin{enumerate}
		\item if $\vdim = 1$ then $p=1$ and $\tau = 1$,
		\item if $\vdim = 2$ then $1 \leq p < \vdim$ and $p/2 \leq
		\tau < \frac{\vdim p}{2(\vdim-p)}$,
		\item if $\vdim > 2$ then $1 \leq p < \vdim$ and $\tau =
		\frac{\vdim p}{2 ( \vdim-p)}$.
	\end{enumerate}

	Then there exist positive, finite numbers $\varepsilon$ and $\Gamma$
	with the following property.

	If $a \in \rel^\adim$, $0 < r < \infty$, $V \in \IVar_\vdim (
	\oball{a}{r} )$, $V$ is related to $p$ as in the preceding hypotheses
	\eqref{hp}, $\psi$ is the measure defined by
	\begin{gather*}
		\psi = \| \delta V \| \quad \text{if $p = 1$} \qquad
		\text{and} \qquad \psi = | \mathbf{h} ( V ; \cdot ) |^p \| V
		\| \quad \text{if $p > 1$},
	\end{gather*}
	$T \in \grass{\adim}{\vdim}$, $\omega : \classification{\rel}{t}{0 < t
	\leq 1} \to \rel$ with
	\begin{gather*}
		\omega (t) = t^{\alpha \tau} \quad \text{if $\alpha \tau < 1$}
		\qquad \text{and} \qquad \omega (t) = t ( 1 + \log (1/t))
		\quad \text{if $\alpha \tau = 1$}
	\end{gather*}
	whenever $0 < t \leq 1$, and $0 < \gamma \leq \varepsilon$,
	\begin{gather*}
		\density^{\ast \vdim} ( \| V \|, a ) \geq Q-1+\delta, \quad
		\measureball{\| V \|}{\oball{a}{r}} \leq ( Q + 1 - \delta )
		\unitmeasure{\vdim} r^\vdim, \\
		\big ( r^{-\vdim} \tint{}{}| \project{S} - \project{T} |^2 \ud
		V (z,S) \big )^{1/2} \leq \gamma, \\
		\| V \| (
		\classification{\cball{a}{\varrho}}{z}{\density^\vdim ( \| V
		\|, z ) \leq Q-1} ) \leq \varepsilon \unitmeasure{\vdim}
		\varrho^\vdim \quad \text{for $0 < \varrho < r$}, \\
		\varrho^{1-\vdim/p} \psi ( \cball{a}{\varrho} )^{1/p} \leq
		\gamma^{1/\tau} ( \varrho/r )^\alpha \quad \text{for $0 <
		\varrho < r$},
	\end{gather*}
	then $\density^\vdim ( \| V \|, a ) = Q$, $R = \Tan^\vdim ( \| V \|, a
	) \in \grass{\adim}{\vdim}$ and
	\begin{gather*}
		\big ( \varrho^{-\vdim} \tint{\oball{a}{\varrho} \times
		\grass{\adim}{\vdim}}{} | \project{S} - \project{R} |^2 \ud V
		(z,S) \big )^{1/2} \leq \Gamma \gamma \omega ( \varrho/r )
		\quad \text{whenever $0 < \varrho \leq r$}.
	\end{gather*}
\end{citing}
In order to comment on this theorem, assume $\vdim > 2$.

In case $\frac{\vdim p}{\vdim-p} = 2$, the theorem states that if the first
variation, i.e.~the mean curvature if $p >1$, expressed in terms of $\psi$
decays with power $\alpha < 1$ so does the tilt-excess of the varifold
provided essentially that the tilt-excess is initially small and the density,
restricted to the complement of a set with small density at $a$, is lower
semicontinuous at $a$.  If $\alpha = 1$, the modulus of continuity $\omega$
obtained is optimal as demonstrated by an example in
\ref{remark:pointwise_decay}, in particular one cannot take $\omega (t) = t$.
Moreover, this sharp result seems not to be obtainable using classical excess
decay methods as will be explained below.

In the case $\frac{\vdim p}{\vdim-p} < 2$, the situation is different. Decay
of the mean curvature with power $\alpha$ implies, under the same assumptions
as before, decay of the tilt-excess with some smaller power $\alpha \tau$ with
$\tau = \frac{\vdim p}{2(\vdim-p)}$. This number $\tau$ cannot be replaced by
any larger number, see \ref{remark:pointwise_decay2}.

For comparison one may consider the analogous question replacing integral
varifolds by weakly differentiable functions and variation of mass by
variation of the Dirichlet integral. Therefore suppose $u : \rel^\vdim \to
\rel^\codim$ is weakly differentiable, $T$ is the distributional Laplacian of
$u$, i.e. $T \in \mathscr{D}' ( \rel^\vdim, \rel^\codim )$ is given by
\begin{gather*}
	T ( \theta) = - \tint{}{} D \theta (x) \bullet \weakD u (x) \ud
	\mathscr{L}^\vdim x \quad \text{for $\theta \in \mathscr{D} (
	\rel^\vdim, \rel^\codim )$},
\end{gather*}
$T$ is representable by integration and, if $p > 1$, $T$ corresponds to a
locally $p$-th power summable function. Then one may investigate which decay
properties of
\begin{gather*}
	\big ( \tfint{\oball{c}{\varrho}}{} | \weakD u (x) - \tau |^2 \ud
	\mathscr{L}^\vdim x \big )^{1/2}
\end{gather*}
as $\varrho \to 0+$, where $(c,\tau) \in \rel^\vdim \times \Hom ( \rel^\vdim,
\rel^\codim )$, are implied by decay hypotheses on
\begin{gather*}
	\varrho^{1-\vdim} \measureball{\| T \|}{\oball{c}{\varrho}} \quad
	\text{if $p=1$}, \qquad \varrho^{1-\vdim/p} \norm{f}{p}{c,\varrho}
	\quad \text{if $p > 1$}.
\end{gather*}
Clearly, the varifold problem behaves less regular than the problem for weakly
differentiable functions as known examples show that a decay hypothesis on
$\psi$ alone is not sufficient to infer decay of the tilt-excess, see
\ref{remark:classical_examples}. However, apart from this the varifold problem
behaves equally regular if $\frac{\vdim p}{\vdim-p} = 2$ as the same decay
implications hold and it even behaves more regular if $\frac{\vdim p}{\vdim-p}
< 2$ since in this case decay results are only valid in the varifold case (as
$\weakD u$ may not be locally square summable). In case $p=1$ this latter
phenomenon was already apparent from the results of Brakke.

Summarising, the pointwise implications of Theorem \ref{thm:pointwise_decay}
are essentially optimal and determine the optimal $\alpha$ for which the
answer to the initial question is in the affirmative if $\vdim > 2$ and $p <
2\vdim/(\vdim+2)$. Using the estimate \ref{lemma:aux_c2rekt} of the present
paper, the optimal $\alpha$ is determined in case $\vdim=1$ or $\vdim = 2$ and
$p > 1$ or $\vdim > 2$ and $p \geq 2 \vdim / ( \vdim+2 )$ in
\cite{snulmenn:c2.v3}, see \ref{remark:discussion2}. This then covers all
cases except $(\vdim,p) = (2,1)$ where Corollary \ref{cor:some_decay_rates}
solves the subcase $\alpha < 1$.
\paragraph{Overview of proof}
As indicated above the main tool in the pointwise regularity proof is the
harmonic approximation procedure introduced by de~Giorgi and Almgren, see
\cite{MR0179651,EnnioDeGiorgi_SelectedPapers,MR0225243}. It requires the
varifold to be weakly close to a plane with density $Q$ and strongly close to
a varifold with density at least $Q$. Initially, the latter condition was
phrased as $\density^\vdim ( \| V \|, z ) \geq Q$ for $\| V \|$ almost all $z
\in \oball{a}{r}$ in Allard \cite[\S 8]{MR0307015}, however the set of points
$a$ not satisfying this condition for suitable $Q$ and $r$ may have positive
$\| V \|$ measure even if the hypotheses are satisfied with $p = \infty$, see
Allard \cite[8.1\,(2)]{MR0307015} and Brakke \cite[6.1]{MR485012}. Replacing
the condition by the requirement on $\density^\vdim ( \| V \|, \cdot )$ to be
$\| V \|$ approximately (lower semi) continuous, Brakke was able to treat
almost all points with $p = 2$ using an approximation by Almgren's
``$Q$-valued'' functions, i.e.~functions with values in $\qspace_Q (
\rel^\codim )$, see below. Additionally, Brakke established a coercive
estimate which allowed him to obtain partial results also for the case $p=1$.

Taking this as a starting point, it will be described, firstly, the new
ingredient needed to obtain the optimal modulus of continuity for the case
$p=2$, secondly, the new ingredient needed to obtain optimal results in case
$p < 2$ and, thirdly, how these new ingredients can be implemented within the
known framework of a (partial or pointwise) regularity proof.
\subparagraph{Obtaining the optimal modulus of continuity for $p=2$} For this
purpose a new iteration procedure is introduced which is now presented in the
simple case of the Laplace operator. Additionally, in Section
\ref{sec:model_case}, it is shown how to implement this method in a model case
from partial regularity theory for second order elliptic systems in divergence
form.  Suppose $c \in \rel^\vdim$, $u \in \Sob{}{1}{2} ( \oball{c}{1},
\rel^\codim)$, $T \in \mathscr{D}' ( \oball{c}{1}, \rel^\codim )$ is the
distributional Laplacian of $u$, and assume for some $0 \leq \gamma < \infty$
and $0 < \alpha \leq 1$ that
\begin{gather*}
	\varrho^{-\vdim/2} | T ( \theta ) | \leq \gamma \varrho^\alpha
	\norm{D\theta}{2}{c,\varrho}
\end{gather*}
whenever $\theta \in \mathscr{D} ( \oball{c}{1}, \rel^\codim )$ with $\spt
\theta \subset \oball{c}{\varrho}$ and $0 < \varrho \leq 1$, where
$\norm{f}{p}{c,\varrho}$ denotes the seminorm of $|f| \in \Lp{p} (
\mathscr{L}^\vdim \restrict \oball{c}{\varrho} )$. Define $J =
\classification{\rel}{r}{0 < \varrho \leq 1}$, for each $\varrho \in J$ choose
$u_\varrho : \oball{c}{\varrho} \to \rel^\codim$ harmonic with boundary values
given by $u$, i.e.
\begin{gather*}
	u_\varrho \in \mathscr{E} ( \oball{c}{\varrho}, \rel^\codim ) \quad
	\text{with} \quad \Lap u_\varrho = 0, \\
	u-u_\varrho \in \Sob{0}{1}{2} ( \oball{c}{\varrho}, \rel^\codim),
\end{gather*}
define $\phi_1 : J \to \rel$ and $\phi_2 : J \times \Hom (\rel^\vdim,
\rel^\codim) \to \rel$ by
\begin{gather*}
	\phi_1 (\varrho) = \norm{D^2 u_\varrho}{\infty}{c,\varrho/2}, \quad
	\phi_2 (\varrho,\sigma) = \varrho^{-\vdim/2} \norm{\weakD
	(u-\sigma)}{2}{c,\varrho}
\end{gather*}
for $(\varrho,\sigma) \in J \times \Hom ( \rel^\vdim, \rel^\codim )$ and
choose $\sigma_\varrho \in \Hom ( \rel^\vdim, \rel^\codim )$ such that
\begin{gather*}
	\phi_2 ( \varrho, \sigma_\varrho ) \leq \phi_2 (\varrho,\sigma) \quad
	\text{whenever $\sigma \in \Hom ( \rel^\vdim, \rel^\codim )$, $\varrho
	\in J$}.
\end{gather*}
Using a priori estimates, see \cite[Theorems 7.26\,(ii), 8.10,
9.11]{MR1814364}, one estimates
\begin{gather*}
	\begin{aligned}
		& \phi_1 (\varrho/4) - \phi_1 (\varrho) \leq \norm{D^2 (
		u_\varrho - u_{\varrho/4})}{\infty}{c,\varrho/8} \leq \Delta
		\varrho^{-1-\vdim/2} \norm{D ( u_\varrho -
		u_{\varrho/4})}{2}{c,\varrho/4}  \\
		& \qquad \leq \Delta \varrho^{-1-\vdim/2} \big ( \norm{\weakD
		(u-u_{\varrho/4})}{2}{c,\varrho/4} + \norm{\weakD
		(u-u_\varrho)}{2}{c,\varrho} \big ) \leq 2 \Delta \gamma
		\varrho^{\alpha-1}
	\end{aligned}
\end{gather*}
for some positive, finite number $\Delta$ depending only on $\adim$ and
\begin{gather*}
	\begin{aligned}
		\phi_2 ( \varrho, \sigma_\varrho ) & \leq \varrho^{-\vdim/2}
		\big ( \norm{\weakD(u-u_\varrho)}{2}{c,\varrho} + \norm{D
		(u_\varrho-Du_\varrho(c))}{2}{c,\varrho} \big ) \\
		& \leq \gamma \varrho^\alpha + \unitmeasure{\vdim}^{1/2}
		\varrho \phi_1 (\varrho),
	\end{aligned}
\end{gather*}
hence obtains the two \emph{iteration inequalities}
\begin{gather*}
	\phi_1 (\varrho/4) \leq \phi_1 (\varrho) + \Gamma \gamma
	\varrho^{\alpha-1}, \quad \phi_2 (\varrho,\sigma_\varrho) \leq \Gamma
	\big ( \varrho \phi_1 (\varrho) + \gamma \varrho^\alpha \big )
\end{gather*}
for $\varrho \in J$ where $\Gamma = \sup \{ 2\Delta, 1,
\unitmeasure{\vdim}^{1/2} \}$.

Now, if $0 \leq \gamma_1 < \infty$, $\phi_1 (\varrho) \leq \gamma_1
\varrho^{\alpha-1}$ and $\alpha < 1$ then
\begin{gather*}
	\phi_1 (\varrho/4) \leq (\varrho/4)^{\alpha-1} \big ( 4^{\alpha-1}
	\gamma_1 + \Gamma \gamma \big ) \leq \gamma_1 (\varrho/4)^{\alpha-1}
\end{gather*}
provided $\gamma_1 \geq (1-4^{\alpha-1})^{-1} \Gamma \gamma$, noting
$4^{\alpha-1} < 1$. Similarly, if $0 \leq \gamma_1 < \infty$, $\phi_1 (
\varrho ) \leq \gamma_1 ( 1 + \log (1/\varrho))$ and $\alpha = 1$ then
\begin{gather*}
	\phi_1 (\varrho/4) \leq \gamma_1 ( 1 + \log (4/\varrho) ) - ( \log 4 )
	\gamma_1 + \Gamma \gamma \leq \gamma_1 ( 1 + \log (4/\varrho) )
\end{gather*}
provided $\gamma_1 \geq \Gamma \gamma ( \log 4 )^{-1}$. In both cases it has
been used crucially that the factor in front of $\phi_1(\varrho)$ in the first
iteration inequality is $1$. This is the reason for choosing $\phi_1$ rather
than $\phi_2$ as leading iteration quantity. The decay of $\phi_2
(\varrho,\sigma_\varrho)$ in terms of $\varrho$ then follows.

Classically, an excess decay inequality of type
\begin{gather*}
	\phi_2 ( \lambda \varrho , \sigma_{\lambda \varrho} ) \leq \Gamma_1
	\lambda \phi_2 (\varrho,\sigma_\varrho) + \Gamma_2 \gamma
	\varrho^\alpha \quad \text{for $0 < \lambda \leq 1/2$, $0 < \varrho
	\leq 1$}
\end{gather*}
where $1 \leq \Gamma_i < \infty$ for $i \in \{1,2\}$ is used, see
e.g.~\cite[5.3.13]{MR41:1976} or Duzaar and Steffen \cite[(5.14)]{MR1900994}.
Sometimes, $\Gamma_2$ additionally depends on $\lambda$. However, concerning
the case $\alpha=1$, the optimal modulus of continuity cannot be deduced from
such an inequality since if $1 < \Gamma_1 < \infty$ and $1/e< \Gamma_2 <
\infty$ then it does not exclude that $\phi_2 (\varrho,\sigma_\varrho)$ may
equal $\gamma \varrho ( 1 + \log (1/\varrho) )^s$ for some $s > 1$ with
$2^{s-1} \leq \Gamma_1$ and $(2s/e)^s \leq 2 \Gamma_2$.
\subparagraph{Treating the case $p < 2$} The second new ingredient in the
regularity proof will be described focusing on the case $\vdim > 2$. In doing
so, a quantity of type
\begin{gather*}
	\varrho^{-1-\vdim/q} \big ( \tint{\cball{a}{\varrho}}{} \dist
	(z-a,T)^q \ud \| V \| z \big )^{1/q}
\end{gather*}
for $U$ and $V$ as in the hypotheses with $a \in \rel^\adim$, $0 <\varrho <
\infty$, $\cball{a}{\varrho} \subset U$, $T \in \grass{\adim}{\vdim}$ and $1
\leq q < \infty$ will be referred to as $q$-height. To derive sharp results
with respect to the integrability of the mean curvature two observations will
be essential.  Firstly, the dependence on the mean curvature in Brakke's
coercive estimate, see \cite[5.5]{MR485012}, can be improved at the price of
using the $q$-height with $q=\frac{2\vdim}{\vdim-2}$ instead of the
$2$-height, see \ref{lemma:coercive_estimate}. Secondly, in order to control
the $q$-height, the Sobolev Poincar\'e type estimates of
\cite{snulmenn.poincare} are adapted. However, a subtlety arises. The
mentioned estimates are in full strength only available for the $q$-height on
the set $H$ of points satisfying a smallness condition on the mean curvature,
see also the discussion in \cite[4.6]{snulmenn.poincare}. As estimating the
$q$-height on the complement of $H$ by mean curvature would be insufficient
for the present purpose, the coercive estimate of Brakke has to be improved a
second time by showing the $q$-height on $H$, mean curvature and $2$-height
are actually sufficient to control the tilt-excess, see
\ref{lemma:coercive_estimate_rect}.  This is accomplished by constructing a
possibly noncontinuous cut-off function with properties reminiscent of a
weakly differentiable function, including a partial integration formula,
Sobolev embedding and approximate differentiability, see \ref{lemma:capacity}
and \ref{remark:capacity}.  These properties are deduced directly from the
construction rather than from a general theory.
\subparagraph{Implementation of proof} Finally, it will be indicated briefly
how the previously described pieces fit into the well known pattern of a
partial regularity proof. As usual, one assumes the varifold to be close to
$Q$ parallel planes with respect to mass, tilt-excess and first variation.
Fixing a suitable orthogonal coordinate system, one approximates the varifold
by a Lipschitzian $\qspace_Q ( \rel^\codim )$ valued function $f$. Recall that
$\qspace_Q ( \rel^\codim )$ may be described as the $Q$ fold product of
$\rel^\codim$ divided by the action of the group of permutations of $\{ 1,
\ldots, Q \}$. The accuracy of this approximation is controlled by tilt-excess
and mean curvature. To obtain the comparison functions $u_\varrho$, one
considers the Dirichlet problem with the linear elliptic system with constant
coefficients given by a suitable linearisation of the nonparametric area
integrand and boundary values given by the ``average'' $g$ of $f$. This is
somewhat different from the usual procedure where the comparison functions are
often constructed either within contradiction arguments (see e.g.~Allard
\cite[8.16]{MR0307015} or Brakke \cite[5.6]{MR485012}) or by an ``$A$-harmonic
approximation lemma'' which confines the contradiction argument to the
situation of linear systems with constant coefficients (see e.g.~Simon
\cite[21.1]{MR87a:49001} or Duzaar and Steffen \cite[3.3]{MR1900994}); however
see also Schoen and Simon \cite{MR652826} for a different approach.  The
distributional right hand side for $g-u_\varrho$ can be estimated by mean
curvature and a small multiple of the tilt-excess provided a suitable weak
norm is employed, namely a norm dual to the norm mapping a smooth function
with compact support to the $\Lp{\infty} ( \mathscr{L}^\vdim , \Hom (
\rel^\vdim, \rel^\codim ) )$ norm of its derivatives. This only yields
smallness of $g-u_\varrho$ in Lebesgue spaces with exponent below
$\frac{\vdim}{\vdim-1}$ if $\vdim > 1$, e.g.~in $\Lp{1} ( \mathscr{L}^\vdim
\restrict \oball{c}{\varrho} , \rel^\codim )$, here $c \in \rel^\vdim$
corresponds to $a \in \rel^\adim$, see
\ref{lemma:iteration}\,\eqref{item:iteration:l1_estimate}.  However, assuming
that the set of points with density strictly below $Q$ is small with respect
to $\| V \|$, the graph of $g$ coincides with the varifold on a large set,
hence using interpolation (Section \ref{sec:interpolation}) and estimates for
the approximation by $f$ (see Section \ref{sec:approx}), one can ultimately
convert $\Lp{1} ( \mathscr{L}^\vdim \restrict \oball{c}{\varrho} ,
\rel^\codim)$ closeness of $g$ to an affine function via the coercive estimate
to control of the tilt-excess of the varifold with respect to the
corresponding plane. From these estimates one readily obtains modified
versions of the iteration inequalities which -- upon simultaneous iteration --
yield the result.
\paragraph{Acknowledgement} The author offers his thanks to Prof. Dr. Reiner
Sch\"atz\-le from whose education the author benefited greatly in writing the
present paper. Additionally, the author would like to thank Prof. Dr. Tom
Ilmanen for several related discussions.
\section{Notation} \label{sec:not}
\paragraph{General} The notation follows \cite{MR41:1976}, see the list of
symbols on pp.~669--671 therein. In  particular, recall the following maybe
less common symbols: $\nat$ denoting the positive integers, $\oball{a}{r}$ and
$\cball{a}{r}$ denoting respectively the open and closed ball with centre $a$
and radius $r$, $\bigodot^i ( V,W )$ and $\bigodot^i V$ denoting the vector
space of all $i$ linear symmetric functions (forms) mapping $V^i$ into $W$ and
$\rel$ respectively, and the seminorms $\phi_{(p)}$ for $1 \leq p \leq \infty$
corresponding to the Lebesgue spaces
\begin{align*}
	\Lpnorm{\phi}{p}{f} & = \big( \tint{}{} |f|^p \ud \phi \big)^{1/p}
	\quad \text{in case $1 \leq p < \infty$}, \\
	\Lpnorm{\phi}{\infty}{f} & = \inf ( \classification{\rel}{t}{\phi (
	\classification{X}{x}{|f(x)|>t} ) = 0} )
\end{align*}
whenever $\phi$ measures $X$, $Y$ is a Banach space, and $f : X \to Y$ is
$\phi$ measurable, see \cite[2.2.6, 2.8.1, 1.10.1, 2.4.12]{MR41:1976}. The
notation for the Lebesgue seminorms is particularly convenient when longer
expressions replace the measure $\phi$ as will repeatedly be the case in
\ref{lemma:lipschitz_approximation}\,\eqref{item:lipschitz_approximation:poincare}.

Moreover, the following slight modifications and additions apply. (For the
convenience of the reader in this section for nearly every symbol the
appropriate reference to its definition in \cite{MR41:1976} is given at its
first occurrence.)

One defines $f \lIm A \rIm = \{ y \with \text{$(x,y) \in f$ for some $x \in
A$} \}$ whenever $f$ is a relation and $A$ is a set, see
\cite[p.\,8]{MR16:1136c}.

If $\vdim, \adim \in \nat$, $\vdim \leq \adim$, $T \in \grass{\adim}{\vdim}$
then $\project{T}$ is characterised by, see \cite[2.2.6, 1.6.2,
1.7.4]{MR41:1976},
\begin{gather*}
	\project{T} \in \Hom ( \rel^\adim, \rel^\adim ), \quad \project{T} =
	\project{T}^\ast, \quad \project{T} \circ \project{T} = \project{T},
	\quad \im \project{T} = T
\end{gather*}
and $T^\perp = \ker \project{T}$, see Almgren \cite[T.1\,(9)]{MR1777737} and
Allard \cite[2.3]{MR0307015}.

Similar to Allard's definition in \cite[8.10]{MR0307015}, the \emph{closed
cuboid} $\cylinder{T}{a}{r}{h}$ is defined by
\begin{gather*}
	\cylinder{T}{a}{r}{h} =
	\classification{\rel^\adim}{z}{\text{$|\project{T}(z-a)|\leq r$ and
	$|\project{T}^\perp(z-a)| \leq h$}}
\end{gather*}
whenever $\vdim, \adim \in \nat$, $\vdim < \adim$, $T \in
\grass{\adim}{\vdim}$, $a \in \rel^\adim$, $0 < r < \infty$, and $0 < h \leq
\infty$. One abbreviates $\cylinder{T}{a}{r}{\infty} = \cylind{T}{a}{r}$. (The
symbol $\cylind{T}{a}{r}$ is used by Allard in \cite[8.10]{MR0307015} to
denote $\classification{\rel^\adim}{z}{\text{$| \project{T} (z-a) | < r$}}$.)

Whenever $\phi$ measures $X$, $0 < \phi ( A ) < \infty$, $Y$ is a Banach
space, and $f \in \Lp{1} ( \phi \restrict A, Y )$ the symbol $\fint_A f \ud
\phi$ denotes $\phi(A)^{-1} \int_A f \ud \phi$, see \cite[2.4.12]{MR41:1976}.

Following Almgren \cite[p.~464]{MR855173}, whenever $\adim \in \nat$ the
number $\besicovitch{\adim}$ denotes the least positive integer with the
following property, see \cite[2.8.14]{MR41:1976}: If $F$ is a family of closed
balls in $\rel^\adim$ with $\sup \{ \diam S \with S \in F \} < \infty$ then
there exist disjointed subfamilies $F_1, \ldots, F_{\besicovitch{\adim}}$ of
$F$ such that, see \cite[2.8.8, 2.8.1]{MR41:1976},
\begin{gather*}
	\{ z \with \text{$\cball{z}{r} \in F$ for some $0 < r < \infty$} \}
	\subset {\textstyle\bigcup\bigcup} \{ F_i \with i = 1, \ldots,
	\besicovitch{\adim} \}.
\end{gather*}

\paragraph{Varifolds} The meaning of the symbols $\Var_\vdim$, $\RVar_\vdim$,
$\IVar_\vdim$, $\| V \|$, $\delta V$, and $\| \delta V \|$ will be introduced
in accordance with Allard \cite[3.1, 3.5, 4.2]{MR0307015}.

Suppose $U$ is an open subset of $\rel^\adim$ and the Grassmann manifold
$\grass{\adim}{\vdim}$ of all $\vdim$ dimensional subspaces is equipped with
the usual topology, see \cite[3.2.28\,(4)]{MR41:1976}. An \emph{$\vdim$
dimensional varifold} $V$ in $U$ is a Radon measure on $U \times
\grass{\adim}{\vdim}$. The \emph{weight} $\| V \|$ of $V$ is given by $\| V \|
(A) = V ( A \times \grass{\adim}{\vdim} )$ for $A \subset U$. The
distributional \emph{first variation} with respect to area of a varifold $V$
is given by
\begin{gather*}
	\delta V ( \theta ) = \tint{}{} D\theta(z) \bullet \project{S} \ud V
	(z,S) \quad \text{whenever $\theta \in \mathscr{D} ( U, \rel^\adim)$}
\end{gather*}
with associated Borel regular measure $\| \delta V \|$ characterised by
\begin{gather*}
	\| \delta V \| ( Z ) = \sup \{ \delta V ( \theta ) \with \text{$\theta
	\in \mathscr{D} ( U, \rel^\adim )$ with $\spt \theta \subset Z$ and
	$|g(z)| \leq 1$ for $z \in U$} \}
\end{gather*}
whenever $Z$ is an open subset of $U$, see \cite[4.1.1, 2.2.3]{MR41:1976}. If
$V$ is an $\vdim$ dimensional varifold in $U$ and $\| \delta V \|$ is a Radon
measure, the \emph{generalised mean curvature vector of $V$ at $z$} is the
unique $\mathbf{h} (V;z) \in \rel^\adim$ such that
\begin{gather*}
	\mathbf{h} (V;z) \bullet v = - \lim_{r \pluslim{0}} \frac{(\delta V) (
	b_{z,r} \cdot v )}{\measureball{\| V \|}{\cball{z}{r}}} \quad
	\text{for $v \in \rel^\adim$}
\end{gather*}
where $b_{z,r}$ is the characteristic function of $\cball{z}{r}$; hence $z \in
\dmn \mathbf{h} (V;\cdot)$ if and only if the above limit exists for every $v
\in \rel^\adim$. This modifies Allard's definition \cite[4.3]{MR0307015} in
the spirit of \cite[4.1.7]{MR41:1976}.

An $\vdim$ dimensional varifold $V$ in $U$ is \emph{rectifiable} if and only
if there exist sequences $c_i$, $A_i$ and $M_i$ such that $0 < c_i < \infty$,
$M_i$ are $\vdim$ dimensional submanifolds of class $\class{1}$, $A_i$ are
$\mathscr{H}^\vdim$ measurable subsets of $M_i$ and
\begin{gather*}
	V(f) = \tsum{i=1}{\infty} c_i \tint{A_i}{} f ( z, \Tan ( M_i, z ) )
	\ud \mathscr{H}^\vdim z \quad \text{for $f \in \mathscr{K} ( U \times
	\grass{\adim}{\vdim} )$},
\end{gather*}
see \cite[3.1.21, 2.5.14, 2.10.2]{MR41:1976}. In this case $0 < \density^\vdim
( \| V \|, z ) < \infty$ and $\Tan^\vdim ( \| V \|, z ) \in
\grass{\adim}{\vdim}$ for $\| V \|$ almost all $z$ and
\begin{gather*}
	V(f) = \tint{}{} f (z,\Tan^\vdim ( \| V \|, z ) ) \density^\vdim ( \|
	V \|, z ) \ud \mathscr{H}^\vdim z \quad \text{for $f \in \mathscr{K} (
	U \times \grass{\adim}{\vdim})$},
\end{gather*}
see \cite[2.10.19, 3.2.16]{MR41:1976}. A rectifiable varifold is called
\emph{integral} if and only if $\density^\vdim ( \| V \|, z)$ is a positive
integer for $\| V \|$ almost all $z$. The set of all rectifiable [integral]
$\vdim$ dimensional varifolds in $U$ is denoted by $\RVar_\vdim (U )$
[$\IVar_\vdim ( U)$].

As in \cite[2.2--2.4]{snulmenn.isoperimetric} whenever $\vdim \in \nat$ the
smallest number with the following property will be denoted by
$\isoperimetric{\vdim}$: If $\adim \in \nat$, $\vdim \leq \adim$, $V \in
\RVar_\vdim ( \rel^\adim )$, $\| V \| ( \rel^\adim ) < \infty$, and $\| \delta
V \| ( \rel^\adim ) < \infty$, then
\begin{gather*}
	\| V \| ( \classification{\rel^\adim}{z}{\density^\vdim ( \| V \|, z )
	\geq 1 )} ) \leq \isoperimetric{\vdim} \| V \| ( \rel^\adim
	)^{1/\vdim} \| \delta V \| ( \rel^\adim ).
\end{gather*}
Note $\vdim^{-1} \unitmeasure{\vdim}^{-1/\vdim} \leq \isoperimetric{\vdim} <
\infty$.
\paragraph{Weakly differentiable functions and distributions} Suppose $\vdim
\in \nat$, $U$ is an open subset of $\rel^\vdim$, $e_1, \ldots, e_\vdim$
denote the standard base of $\rel^\vdim$, $Y$ is a finite dimensional Hilbert
space, $k$ is a nonnegative integer, and $u$ is an $\mathscr{L}^\vdim
\restrict U$ measurable function with values in $Y$.  Then $u$ is called
\emph{$k$ times weakly differentiable in $U$} if and only if
\begin{enumerate}
	\item $u \in \Lp{1} ( \mathscr{L}^\vdim \restrict K, Y )$ for every
	compact subset $K$ of $U$,
	\item defining $T \in \mathscr{D}' ( U, Y )$ by $T(\theta) = \int_U
	\theta \bullet u \ud \mathscr{L}^\vdim$ for $\theta \in \mathscr{D} (
	U, Y )$, the distributions $D^\alpha T$ corresponding to all $\alpha
	\in \Xi ( \vdim, i )$ and $i = 0, \ldots, k$ are representable by
	integration and the measures $\| D^\alpha T \|$ are absolutely
	continuous with respect to $\mathscr{L}^\vdim \restrict U$, see
	\cite[1.9.2, 1.10.1, 2.9.2, 4.1.1, 4.1.5]{MR41:1976}, ($\alpha$ is
	sometimes called ``multi-index of length $i$'').  \intertextenum{In
	this case for $i = 0, \ldots, k$ the $\mathscr{L}^\vdim \restrict U$
	measurable functions $\weakD^i u$ with values in $\bigodot^i (
	\rel^\vdim, Y )$ are characterised by the following two conditions
	(here and in the following $\bigodot^i ( \rel^\vdim, Y )$ is equipped
	with an inner product as in \cite[1.10.6]{MR41:1976}):}
	\item $D^\alpha T ( \theta ) = \int_U \theta (x) \bullet \left <
	e^\alpha, \weakD^i u (x) \right > \ud \mathscr{L}^\vdim x$ whenever
	$\theta \in \mathscr{D} (U,Y )$ and $\alpha \in \Xi ( \vdim, i )$
	where $e^\alpha = (e_1)^{\alpha_1} \odot \cdots \odot
	(e_\vdim)^{\alpha_\vdim}$ is constructed from the standard base $e_1,
	\ldots, e_\vdim$ of $\rel^\vdim$, see \cite[1.9.2, 1.10.1]{MR41:1976};
	in particular $\weakD^i u$ is $0$ times weakly differentiable in $U$.
	\item \label{item:representative} $\weakD^i u (a) = \lim_{r
	\pluslim{0}} \fint_{\cball{a}{r}} \weakD^i u \ud \mathscr{L}^\vdim$
	whenever $a \in U$; hence $a \in \dmn \weakD^i u$ if and only if the
	preceding limit exists.
\end{enumerate}
Also, $1$ times weakly differentiable in $U$ is abbreviated to \emph{weakly
differentiable in $U$} and $\weakD^1 u$ to $\weakD u$. In particular, the
symbols $\weakD^i$, $\weakD$ will not be used in the sense of \cite[1.5.2,
2.9.1, 4.1.6]{MR41:1976}. $\Sob{}{k}{p} ( U, Y)$ denotes the \emph{Sobolev
space} of all $k$ times weakly differentiable functions in $U$ with values in
$Y$ such that $\weakD^i u \in \Lp{p} \big ( \mathscr{L}^\vdim \restrict U,
\bigodot^i ( \rel^\vdim, Y) \big )$ whenever $i = 0, \ldots, k$; the
corresponding seminorm of $u$ is given by $\sum_{i=0}^k ( \mathscr{L}^\vdim
\restrict U)_{(p)} ( \weakD^i u )$, see \cite[2.4.12]{MR41:1976}.
$\Sob{0}{k}{p} ( U, Y)$ denotes the closure of $\mathscr{D} ( U, Y )$ in
$\Sob{}{k}{p} (U,Y)$. Note that in these definitions neither in the Sobolev
spaces nor in the Lebesgue spaces functions agreeing $\mathscr{L}^\vdim
\restrict U$ almost everywhere are treated as single elements; instead
condition \eqref{item:representative} is employed.

If $\vdim \in \nat$, $U$ is an open subset $\rel^\vdim$, $Y$ is a separable
Hilbert space, $1 \leq p \leq \infty$, $A$ is an $\mathscr{L}^\vdim \restrict
U$ measurable set, and $u$ and $v$ are $\mathscr{L}^\vdim \restrict U$
measurable functions with values in $Y$ then $\norm{u}{p}{A} = (
\mathscr{L}^\vdim \restrict A )_{(p)} ( u)$ and, provided $\int_A |u(x)
\bullet v(x)| \ud \mathscr{L}^\vdim x < \infty$, $\pairing{u}{v}{A} = \int_A
u(x) \bullet v(x) \ud \mathscr{L}^\vdim x$. Moreover, $\norm{u}{p}{a,r} =
\norm{u}{p}{\oball{a}{r}}$ and $\pairing{u}{v}{a,r} =
\pairing{u}{v}{\oball{a}{r}}$ whenever $a \in \rel^\vdim$, $0 < r < \infty$
with $\oball{a}{r} \subset U$, see \cite[2.8.1]{MR41:1976}.  These notions
extend \cite[5.2.1]{MR41:1976}.  If additionally, $i$ is a nonpositive
integer, $1 \leq p \leq \infty$, $1 \leq q \leq \infty$, $1/p+1/q=1$, $T$ is a
real valued linear functional on $\mathscr{D} ( U, Y )$, and $V$ is an open
subset of $U$, then
\begin{gather*}
	\boldsymbol{|} T \boldsymbol{|}_{i,p;V} = \sup T \biglIm
	\classification{\mathscr{D}( U, Y)}{\theta}{\text{$\norm{D^{-i}
	\theta}{q}{U} \leq 1$ and $\spt \theta \subset V$}} \bigrIm
\end{gather*}
and $\boldsymbol{|} T \boldsymbol{|}_{i,p;a,r} = \boldsymbol{|} T
\boldsymbol{|}_{i,p;\oball{a}{r}}$ whenever $a \in \rel^\vdim$, $0 < r <
\infty$ with $\oball{a}{r} \subset U$.
\paragraph{Almgren's multiple valued functions} The notation for functions
with values in $\qspace_Q ( \rel^\codim )$ for $\vdim, \adim, Q \in \nat$ with
$\vdim < \adim$ which originate from Almgren's work in \cite{MR1777737} will
be introduced in Section \ref{sec:almgren} together with basic properties.
\paragraph{A convention} Finally, each statement asserting the existence of a
positive, finite number, small ($\varepsilon$) or large ($\Gamma$), will give
rise to a function depending on the listed parameters whose ``name'' is
$\varepsilon_{\mathrm{x.y}}$ or $\Gamma_{\mathrm{x.y}}$ where $\mathrm{x.y}$
denotes the number of the statement. Occasionally, also
$\lambda_{\mathrm{x.y}}$ is used similarly.
\section{Basic facts for $\qspace_Q (V)$ valued functions} \label{sec:almgren}
This section provides some basic definitions for $\qspace_Q (V)$ valued
functions mainly taken from Almgren \cite{MR1777737} in
\ref{miniremark:almgren1}, \ref{miniremark:almgren2} and
\ref{definition:norms} and a proposition from \cite{snulmenn.poincare} in
\ref{miniremark:f_i}. Finally, the first variation for the varifold associated
to the ``graph'' of a $\qspace_Q ( \rel^\codim )$ valued functions is given in
\ref{miniremark:extend_distrib} and \ref{miniremark:first_variation}.
\begin{miniremark} [cf. \protect{\cite[1.1\,(1)\,(3), 2.3\,(2)]{MR1777737}}] \label{miniremark:almgren1}
	Suppose $Q \in \nat$ and $V$ is a finite dimensional Euclidean
	vector space.

	$\qspace_Q ( V )$ is defined to be the set of all $0$ dimensional
	integral currents $R$ such that $R = \sum_{i=1}^Q \Lbrack x_i \Rbrack$
	for some $x_1, \ldots, x_Q \in V$. A metric $\mathscr{G}$ on
	$\qspace_Q (V)$ is defined such that
	\begin{gather*}
		\mathscr{G} \big ( {\textstyle\sum_{i=1}^Q} \Lbrack
		x_i \Rbrack , {\textstyle\sum_{i=1}^Q}
		\Lbrack y_i \Rbrack \big ) = \min \Big \{
		\big ( {\textstyle\sum_{i=1}^Q} | x_i - y_{\pi(i)} |^2 \big
		)^{1/2} \with \pi \in P(Q) \Big \}
	\end{gather*}
	whenever $x_1,\ldots,x_Q, y_1,\ldots,y_Q \in V$ where $P(Q)$ denotes
	the set of permutations of $\{ 1, \ldots, Q \}$. The function
	$\boldsymbol{\eta}_Q : \qspace_Q ( V ) \to
	V$\index{$\boldsymbol{\eta}_Q$} is defined by
	\begin{gather*}
		\boldsymbol{\eta}_Q ( R ) = Q^{-1} {\textstyle\int} x \ud \|
		R \| (x) \quad \text{whenever $R \in \qspace_Q (
		V )$}.
	\end{gather*}
	If $R = \sum_{i=1}^Q \Lbrack x_i \Rbrack$ for some
	$x_1, \ldots, x_Q \in V$, then $\boldsymbol{\eta}_Q ( R ) =
	\frac{1}{Q} \sum_{i=1}^Q x_i$. $\Lip \boldsymbol{\eta}_Q = Q^{-1/2}$.
	
	Whenever $f : X \to \qspace_Q (V)$ one defines
	\begin{align*}
		\graph_Q f & = \eqclassification{X \times V}{(x,v)}{\text{$v
		\in \spt f(x)$}}
	\end{align*}
	and with $g : X \to V$ also $f \aplus g : X \to \qspace_Q (V)$ by
	\begin{gather*}
		( f \aplus g ) ( x ) = (\boldsymbol{\tau}_{g (x)})_\# ( f (x)
		) \quad \text{whenever $x \in X$}.
	\end{gather*}
\end{miniremark}
\begin{miniremark} [cf. \protect{\cite[1.1\,(9)\,(10)]{MR1777737}}]
\label{miniremark:almgren2}
	Suppose $\vdim,\adim,Q \in \nat$ and $\vdim < \adim$.

	A function $f : \rel^\vdim \to \qspace_Q(\rel^\codim)$ is
	called \emph{affine} if and only if there exist affine functions $f_i
	: \rel^\vdim \to \rel^\codim$, $i = 1,\ldots,Q$ such that
	\begin{gather*}
		f (x) = {\textstyle\sum_{i=1}^Q} \Lbrack f_i (x)
		\Rbrack \quad \text{whenever $x \in \rel^\vdim$}.
	\end{gather*}
	$f_1,\ldots,f_Q$ are uniquely determined up to order. Moreover, one
	defines
	\begin{gather*}
		| f | = \big ( {\textstyle\sum_{i=1}^Q} | Df_i(0) |^2 \big
		)^{1/2}.
	\end{gather*}

	Let $a \in A \subset \rel^\vdim$ and $f : A \to \qspace_Q (
	\rel^\codim )$. $f$ is called \emph{affinely approximable at $a$} if
	and only if $a \in \Int A$ and there exists an affine function $g :
	\rel^\vdim \to \qspace_Q ( \rel^\codim )$ such that
	\begin{gather*}
		\lim_{x \to a} \mathscr{G} ( f(x),g(x) ) /|x-a| = 0.
	\end{gather*}
	The function $g$ is unique and denoted by $Af(a)$. $f$ is called
	\emph{strongly affinely approximable at $a$} if and only
	if $Af(a)$ has the following property: If $Af(a)(x) =
	\sum_{i=1}^Q \Lbrack g_i (x) \Rbrack$ for some
	affine functions $g_i : \rel^\vdim \to \rel^\codim$ and $g_i
	(a) = g_j (a)$ for some $i$ and $j$, then $Dg_i(a)=Dg_j(a)$. The
	concepts of \emph{approximate affine approximability} and
	\emph{approximate strong affine approximability} are obtained through
	omission of the condition $a \in \Int A$ and replacement of $\lim$ by
	$\aplim$. The corresponding affine function is denoted by $\ap Af(a)$.
\end{miniremark}
\begin{miniremark} \label{miniremark:f_i}
	The following proposition, see \cite[2.5,\,8]{snulmenn.poincare},
	will be used for calculations involving Lipschitzian $\qspace_Q (
	\rel^\codim )$ valued functions.

	\emph{If $\vdim,\adim,Q \in \nat$, $\vdim < \adim$, $A$ is
	$\mathscr{L}^\vdim$ measurable, $f : A \to \qspace_Q ( \rel^\codim)$
	is Lipschitzian, $I$ is countable, and to each $i \in I$ there
	corresponds a function $f_i \subset \graph_Q f$ with
	$\mathscr{L}^\vdim$ measurable domain and $\Lip f_i \leq \Lip f$ such
	that $$\card \{ i \with f_i(x)=y \} = \density^0 ( \| f (x) \|, y )
	\quad \text{whenever $(x,y) \in A \times \rel^\codim$},$$ then $f$ is
	approximately strongly affinely approximable with $$\ap Af(a)(v) =
	{\textstyle\sum_{i \in I(a)}} \Lbrack f_i (x) + \left < v, \ap Df_i
	(x) \right > \Rbrack \quad \text{whenever $v \in \rel^\vdim$}$$ at
	$\mathscr{L}^\vdim$ almost all $a \in A$ where $I(a) =
	\classification{I}{i}{a \in \dmn \ap D f_i }$. Moreover, such
	functions $f_i$ do exist whenever $\vdim$, $\adim$, $Q$, $A$, and $f$
	are as above, in particular $\graph_Q f$ is countably $\vdim$
	rectifiable. If $A$ is open, then $\ap Af$ may be replaced by $Af$.}
\end{miniremark}
\begin{definition} \label{definition:norms}
	Suppose $\vdim, \adim, Q \in \nat$, $\vdim < \adim$, $A
	\subset B \subset \rel^\vdim$, $A$ is $\mathscr{L}^\vdim$ measurable
	and $f : B \to \qspace_Q ( \rel^\codim )$ is Lipschitzian, $C_1 = \dmn
	\ap Af$, $C_2 = \dmn A f$, and $g : B \to \rel$ and $h_i : C_i \to
	\rel$ for $i \in \{1,2\}$ are defined by
	\begin{gather*}
		g ( x ) = \mathscr{G} ( f(x), Q \Lbrack 0 \Rbrack ) \quad
		\text{for $x \in B$}, \\
		h_1 (x) = | \ap A f (x) | \quad \text{for $x \in C_1$}, \quad
		h_2 (x) = | A f(x) | \quad \text{for $x \in C_2$}.
	\end{gather*}
	
	Then one defines for $1 \leq p \leq \infty$, noting
	\ref{miniremark:f_i},
	\begin{gather*}
		\norm{f}{p}{A} = \norm{g}{p}{A}, \quad
		\norm{\ap Af}{p}{A} = \norm{h_1}{p}{A}, \\
		\norm{Af}{p}{A}
		= \norm{h_2}{p}{A} \quad \text{if $A$ is open}.
	\end{gather*}
	Moreover, if $\oball{a}{r} \subset B$ for some $a \in \rel^\vdim$, $0
	< r < \infty$, then
	\begin{gather*}
		\norm{f}{p}{a,r} = \norm{f}{p}{\oball{a}{r}}, \quad
		\norm{\ap Af}{p}{a,r} = \norm{\ap Af}{p}{\oball{a}{r}}, \\
		\norm{Af}{p}{a,r} = \norm{Af}{p}{\oball{a}{r}}.
	\end{gather*}
\end{definition}
\begin{miniremark} \label{miniremark:extend_distrib}
	Suppose $U$ is an open subset of $\rel^\vdim$, $Y$ is a Banach space
	and $T \in \mathscr{D}' ( U, Y )$. Then $T$ has a unique extension $S$
	to $\classification{\mathscr{E} ( U, Y )}{\theta}{\text{$\spt \theta
	\cap \spt T$ is compact}}$ characterised by the requirement
	\begin{gather*}
		S ( \theta ) = S ( \eta ) \quad \text{whenever $\spt T \subset
		\Int \{ x \with \theta (x) = \eta (x) \}$}.
	\end{gather*}
	The extension will usually be denoted by the same symbol $T$.
\end{miniremark}
\begin{miniremark} \label{miniremark:first_variation}
	Suppose $\vdim, \adim, Q \in \nat$ with $\vdim < \adim$.

	Following \cite[5.1.9]{MR41:1976}, the projections $\pp \in
	\mathbf{O}^\ast (\adim, \vdim)$, $\qq \in \mathbf{O}^\ast ( \adim,
	\codim )$ are defined by
	\begin{gather*}
		\pp ( z ) = ( z_1, \ldots, z_\vdim ), \quad \qq ( z ) = (
		z_{\vdim + 1}, \ldots, z_\adim )
	\end{gather*}
	whenever $z = (z_1, \ldots, z_\adim ) \in \rel^\adim$. In case
	\begin{gather*}
		z = \pp^\ast (x) + \qq^\ast (y) = (x_1, \ldots, x_\vdim, y_1,
		\ldots, y_\codim ) \quad \text{for $x \in \rel^\vdim$, $y \in
		\rel^\codim$}
	\end{gather*}
	sometimes $(x,y)$ will be written instead of $z$, $f(x,y)$ instead of
	$f(z)$ for functions $f$ with $\dmn f \subset \rel^\adim$ and
	$\grass{\adim}{\vdim}$ instead of $\grass{\rel^\vdim \times
	\rel^\codim}{\vdim}$.

	If $U$ is an open subset of $\rel^\vdim$, $A$ is an
	$\mathscr{L}^\vdim$ measurable subset of $U$, $f : A \to \qspace_Q (
	\rel^\codim)$ is Lipschitzian, and $f_i$ for $i \in I$ are as in
	\ref{miniremark:f_i}, then defining $V \in \IVar_\vdim ( \pp^{-1} \lIm
	U \rIm )$ by the requirement
	\begin{gather*}
		\| V \| (Z) = {\textstyle\int_{Z\cap\pp^{-1}\lIm A \rIm}}
		\density^0 ( \| f ( \pp (z) ) \|, \qq (z) ) \ud
		\mathscr{H}^\vdim z
	\end{gather*}
	for every Borel subset $Z$ of $\pp^{-1}\lIm U \rIm$, a simple
	calculation shows
	\begin{gather*}
		( \delta V ) ( \qq^\ast \circ \theta \circ \pp ) =
		{\textstyle\sum_{i \in I} \int_{\dmn f_i}} \big < D \theta
		(x), D\Psi_0^\S ( \ap Df_i (x) ) \big > \ud \mathscr{L}^\vdim
		x
	\end{gather*}
	whenever $\theta \in \mathscr{D} ( U, \rel^\codim )$; here $\Psi_0^\S$
	denotes the nonparametric integrand at $0$ associated with the area
	integrand $\Psi$, i.e.~$\Psi_0^\S : \Hom ( \rel^\vdim, \rel^\codim )
	\to \rel$ with
	\begin{gather*}
		\Psi_0^\S ( \sigma ) = {\textstyle \left ( \sum_{i=0}^{\vdim}
		| \bigwedge_i \sigma |^2 \right )^{1/2}} \quad \text{for
		$\sigma \in \Hom ( \rel^\vdim, \rel^\codim )$},
	\end{gather*}
	see \cite[5.1.9]{MR41:1976}, and the convention
	\ref{miniremark:extend_distrib} is used.
\end{miniremark}
\section{Some preliminaries} \label{sec:prerequisites}
The purpose of this section is to list several known statements for convenient
reference. This includes, in \ref{app:thm:bilip_embedding}, some of Almgren's
results on $\qspace_Q ( \rel^l )$ valued functions obtained in \cite[\S
1]{MR1777737}, and, in
\ref{app:lemma:mydiss}--\ref{app:lemma:lipschitz_approximation}, adaptions of
the approximation techniques of integral varifolds by such functions
originating from Almgren \cite[\S 3]{MR1777737} and Brakke \cite[\S
5]{MR485012} carried out by the author in
\cite{mydiss,snulmenn.isoperimetric,snulmenn.poincare}.
\begin{theorem} [cf. Almgren \protect{\cite[1.1\,(6), 1.2\,(3), 1.3\,(1)\,(2),
1.4\,(3)]{MR1777737}}] \label{app:thm:bilip_embedding}
	Suppose $Q, l \in \nat$.

	Then there exist $P \in \nat$ and maps $\xi :
	\qspace_Q ( \rel^l ) \to \rel^{PQ}$ and $\varrho : \rel^{PQ} \to
	\rel^{PQ}$ such that
	\begin{gather*}
 		\xi ( Q \Lbrack 0 \Rbrack ) = 0, \quad \Lip \xi < \infty,
		\quad \text{$\xi$ is univalent}, \quad \Lip \xi^{-1} < \infty,
		\\
		\Lip \varrho < \infty, \quad
		\varrho \circ \varrho = \varrho, \quad \im \varrho = \im \xi,
		\\
		| D ( \xi \circ f ) (x) | \leq ( \Lip \xi ) | Af (x) | \quad
		\text{for $x \in \dmn D ( \xi \circ f )$}
	\end{gather*}
	whenever $f$ maps an open subset of $\rel^\vdim$ into $\qspace_Q (
	\rel^l )$. In particular, a function $f$ mapping a subset of
	$\rel^\vdim$ into $\qspace_Q ( \rel^l )$ admits an extension $F :
	\rel^\vdim \to \qspace_Q ( \rel^l)$ such that $\Lip F \leq \Gamma \Lip
	f$ with $\Gamma = \Lip \xi \Lip \varrho \Lip \xi^{-1}$.
\end{theorem}
\begin{lemma} [cf. \protect{\cite[A.7]{mydiss}}] \label{app:lemma:mydiss}
	Suppose $\vdim, \adim \in \nat$, $\vdim < \adim$, $a \in \rel^\adim$,
	$0 < r < \infty$, $V \in \RVar_\vdim (\oball{a}{r})$, $\| \delta V\|$
	is a Radon measure, $\density^\vdim ( \| V \| , z ) \geq 1$ for $\| V
	\|$ almost all $z$, $a \in \spt \| V \|$, and $\alpha : \{ s \with 0 <
	s < r \} \to \rel$ satisfies
	\begin{gather*}
		\alpha(s) = \measureball{\| V \|}{\cball{a}{s}} \quad
		\text{whenever $0 < s < r$}.
	\end{gather*}

	Then
	\begin{gather*}
		\isoperimetric{\vdim}^{-1} \leq \alpha(s)^{1/\vdim-1} (
		\measureball{\| \delta V \|}{\cball{a}{s}} + \alpha'(s))
	\end{gather*}
	for $\mathscr{L}^1$ almost all $0 < s < r$.
\end{lemma}
\begin{remark}
	A similar statement can be found in Leonardi and Masnou
	\cite[Proposition 3.1]{leomas08}.
\end{remark}
\begin{lemma} [cf. \protect{\cite[2.5]{snulmenn.isoperimetric}}]
\label{app:lemma:good_point} 
	Suppose $\vdim, \adim \in \nat$, $\vdim < \adim$, $a \in \rel^\adim$,
	$0 < r < \infty$, $V \in \RVar_\vdim (\oball{a}{r})$, $\| \delta V\|$
	is a Radon measure, $\density^\vdim ( \| V \| , z ) \geq 1$ for $\| V
	\|$ almost all $z$, $a \in \spt \| V \|$, and
	\begin{gather*}
		\measureball{\| \delta V \|}{\cball{a}{s}} \leq (2
		\isoperimetric{\vdim} )^{-1} \| V \| (
		\cball{a}{s})^{1-1/\vdim} \quad \text{whenever $0 < s < r$}.
	\end{gather*}

	Then
	\begin{gather*}
		\measureball{\| V \|}{\cball{a}{s}} \geq ( 2 \vdim
		\isoperimetric{\vdim} )^{-\vdim} s^\vdim \quad \text{whenever
		$0 < s < r$}.
	\end{gather*}
\end{lemma}
\begin{remark}
	Both \ref{app:lemma:mydiss} and \ref{app:lemma:good_point} are
	variants of Allard \cite[8.3]{MR0307015}. Moreover, in view of Allard
	\cite[5.5]{MR0307015} one could replace $\RVar_\vdim$ by $\Var_\vdim$
	in \ref{app:lemma:mydiss} and \ref{app:lemma:good_point}.
\end{remark}
\begin{lemma} [cf. \protect{\cite[3.1]{snulmenn.poincare}}] \label{app:lemma:planes}
	Suppose $\vdim, \adim \in \nat$, $\vdim < \adim$, $a \in \rel^\adim$,
	$0 < r < \infty$, $T \in \grass{\adim }{\vdim }$, $V \in \IVar_\vdim (
	\oball{a}{r} )$, $\delta V = 0$, $S = T$ for $V$ almost all $(z,S)$,
	and $R(z) = \classification{\oball{a}{r}}{\xi}{\xi-z \in T}$ for $z
	\in \rel^\adim$.
	
	Then $\perpproject{T} \lIm \spt \| V \| \rIm$ is discrete and closed
	relative to $\perpproject{T} \lIm \oball{a}{r} \rIm$ and
	\begin{gather*}
		\density^\vdim ( \| V \|, z ) \in \nat \quad \text{and} \quad
		\| V \| \restrict R(z) = \density^\vdim ( \| V \|, z )
		\mathcal{H}^\vdim \restrict R(z)
	\end{gather*}
	whenever $z \in \spt \| V \|$.
\end{lemma}
\begin{remark}
	This is a variant of Almgren \cite[3.6]{MR1777737}.
\end{remark}
\begin{lemma} [cf. \protect{\cite[3.2]{snulmenn.poincare}}] \label{app:lemma:little_helper}
	Suppose $1 < \adim \in \nat$, $0 < \delta \leq 1$, $0 \leq \lambda <
	1$, and $0 \leq M < \infty$.

	Then there exists a positive, finite number $\varepsilon$ with the
	following property.

	If $\adim > \vdim \in \nat$, $a \in \rel^\adim$, $0 < r < \infty$, $T
	\in \grass{\adim}{\vdim}$, $V \in \IVar_\vdim  ( \oball{a}{r} )$ and
	\begin{gather*}
		\measureball{\| V \|}{\oball{a}{r}} \leq M \unitmeasure{\vdim}
		r^\vdim, \quad \measureball{\| \delta V \|}{\oball{a}{r}} \leq
		\varepsilon \| V \| ( \oball{a}{r} )^{1-1/\vdim}, \\
		\tint{}{} | \project{S} - \project{T} | \ud V (z,S) \leq
		\varepsilon \measureball{\| V \|}{\oball{a}{r}}, \\
		\measureball{\| V \|}{\cball{a}{\varrho}} \geq \delta
		\unitmeasure{\vdim} \varrho^\vdim \quad \text{for $0 < \varrho
		< r$},
	\end{gather*}
	then
	\begin{gather*}
		\| V \| ( \classification{\oball{a}{r}}{z}{ | \project{T}
		(z-a) | > \lambda | z-a |} ) \geq ( 1- \delta )
		\unitmeasure{\vdim} r^\vdim.
	\end{gather*}
\end{lemma}
\begin{proof}
	Assume $M \geq 1$ and take $s=\lambda$, $d=0$, $t=r$, and $\zeta=0$ in
	\cite[3.2]{snulmenn.poincare}.
\end{proof}
\begin{remark}
	This is a simple consequence of Allard's compactness theorem for
	integral varifolds, see e.g.~\cite[6.4]{MR0307015} or
	\cite[42.8]{MR87a:49001}.
\end{remark}
\begin{lemma} [Multilayer monotonicity with variable offset, cf.
\protect{\cite[3.11]{snulmenn.poincare}}] \label{app:lemma:multilayer_monotonicity_offset} 
	Suppose $\adim,Q \in \nat$, $0 \leq M < \infty$, $\delta > 0$, and $0
	\leq s < 1$.
	
	Then there exists a positive, finite number $\varepsilon$ with the
	following property.
	
	If $\adim > \vdim \in \nat$, $Z \subset \rel^\adim$, $T \in
	\grass{\adim}{\vdim}$, $0 \leq d < \infty$, $0 < r < \infty$, $0 < t <
	\infty$, $f : Z \to \rel^\adim$,
	\begin{gather*}
		| \project{T} ( z_1-z_2 ) | \leq s | z_1-z_2 |, \quad |
		\project{T} ( f(z_1) - f(z_2) ) | \leq s | f(z_1)-f(z_2) |, \\
		f(z) - z \in T \cap \cball{0}{d}, \quad d \leq M t,
		\quad d + t \leq r
	\end{gather*}
	for $z,z_1,z_2 \in Z$, $V \in \IVar_\vdim \big (\bigcup \{
	\oball{z}{r} \with z \in Z \} \big )$, $\| \delta V \|$ is a Radon
	measure,
	\begin{gather*}
		{\textstyle\sum_{z\in Z}} \density^\vdim_\ast ( \|V\|, z)
		\geq Q-1+\delta, \quad \measureball{\| V \|}{\oball{z}{r}}
		\leq M \unitmeasure{\vdim} r^\vdim
	\end{gather*}
	whenever $z \in Z \cap \spt \|V\|$, and
	\begin{gather*}
		\measureball{\| \delta V \|}{\cball{z}{\varrho}} \leq
		\varepsilon \, \|V\| ( \cball{z}{\varrho} )^{1-1/\vdim},
		\\
		{\textstyle\int_{\cball{z}{\varrho} \times
		\grass{\adim}{\vdim}}} | \project{S} - \project{T} | \ud
		V(\xi,S) \leq \varepsilon \, \measureball{\| V
		\|}{\cball{z}{\varrho}},
	\end{gather*}
	whenever $0 < \varrho < r$, $z \in Z \cap \spt \|V\|$, then
	\begin{gather*}
		\|V\| \big ( {\textstyle\bigcup} \big \{ \oball{f(z)}{t}
		\cap \{ \xi \with | \project{T} (\xi-z) | > s |\xi-z| \} \with
		z \in Z \big \} \big ) \geq (Q-\delta) \unitmeasure{\vdim}
		t^\vdim.
	\end{gather*}
\end{lemma}
\begin{remark}
	This is an extension of Brakke \cite[5.3]{MR485012}.
\end{remark}
\begin{lemma} [cf. \protect{\cite[3.12]{snulmenn.poincare}}] \label{app:lemma:inverse_multilayer_monotonicity}
	Suppose $\vdim, \adim, Q \in \nat$, $\vdim < \adim$, $0 < \delta_1
	\leq 1$, $0 < \delta_2 \leq 1$, $0 \leq s < 1$, $0 \leq s_0 < 1$, $0
	\leq M < \infty$, and $0 < \lambda < 1$ is uniquely defined by the
	requirement
	\begin{gather*}
		( 1 - \lambda^2 )^{\vdim /2} = ( 1 - \delta_2 ) + \Big (
		\frac{(s_0)^2}{1-(s_0)^2} \Big)^{\vdim /2} \lambda^\vdim .
	\end{gather*}
	
	Then there exists a positive, finite number $\varepsilon$ with the
	following property.
	
	If $Z \subset \rel^\adim $, $T \in \grass{\adim }{\vdim }$, $0 \leq d
	< \infty$, $0 < r < \infty$, $0 < t < \infty$, $\zeta \in \rel^\adim
	$,
	\begin{gather*}
		\card \project{T} \lIm Z \rIm = 1, \quad \zeta \in T \cap
		\cball{0}{d}, \quad d \leq M t, \quad d + t \leq r,
	\end{gather*}
	$V \in \IVar_\vdim ( \bigcup \{ \oball{z}{r} \with z \in Z\} ) $, $\|
	\delta V\|$ is a Radon measure,
	\begin{gather*}
		\density^\vdim  ( \| V \|, z ) \in \nat \quad \text{for $z \in
		Z$}, \\
		{\textstyle\sum_{z \in Z}} \density^\vdim  ( \| V \|, z) = Q,
		\qquad \measureball{\| V \|}{\oball{z}{r}} \leq M
		\unitmeasure{\vdim } r^\vdim  \quad \text{for $z \in Z$},
	\end{gather*}
	and whenever $0 < \varrho < r$, $z \in Z$
	\begin{gather*}
		\measureball{\| \delta V \|}{\cball{z}{\varrho}} \leq
		\varepsilon \, \| V \| ( \cball{z}{\varrho} )^{1-1/\vdim },
		\\
		{\textstyle\int_{\cball{z}{\varrho} \times
		\grass{\adim}{\vdim}}} | \project{S} - \project{T} | \ud V
		(\xi,S) \leq \varepsilon \, \measureball{\| V
		\|}{\cball{z}{\varrho}}
	\end{gather*}
	satisfying
	\begin{multline*}
		\| V \| \big ( {\textstyle\bigcup} \{ \xi \in \oball{z +
		\zeta}{t} \with | \project{T} ( \xi - z ) | > s_0 | \xi - z | \}
		\with z \in Z \} \big ) \\
		\leq ( Q + 1 - \delta_2 ) \unitmeasure{\vdim } t^\vdim ,
	\end{multline*}
	then the following two statements hold:
	\begin{enumerate}
		\item \label{app:item:inverse_multilayer_monotonicity:upper_bound}
		If $0 < \tau \leq \lambda t$, then
		\begin{gather*}
			\| V \| \big ( {\textstyle\bigcup} \{ \cball{z}{\tau}
			\with z \in Z \} \big ) \leq ( Q + \delta_1 )
			\unitmeasure{\vdim } \tau^\vdim .
		\end{gather*}
		\item \label{app:item:inverse_multilayer_monotonicity:lip_related}
		If $\xi \in \rel^\adim$ with $\dist (\xi, Z) \leq \lambda t
		/2$ and
		\begin{gather*}
			\measureball{\| V \|}{\cball{\xi}{\varrho}} \geq
			\delta_1 \unitmeasure{\vdim} \varrho^\vdim \quad
			\text{for $0 < \varrho < \delta_1 \dist ( \xi, Z )$},
		\end{gather*}
		then for some $z \in Z$
		\begin{gather*}
			| \project{T} ( \xi - z ) | \geq s | \xi - z |.
		\end{gather*}
	\end{enumerate}
\end{lemma}
\begin{miniremark} [cf. \protect{\cite[3.13]{snulmenn.poincare}}] \label{app:miniremark:tilt}
	\emph{If $\vdim, \adim \in \nat$, $\vdim < \adim$, and $S, T \in
	\grass{\adim}{\vdim}$, then
	\begin{gather*}
		{\textstyle 1 - \left \| \bigwedge_\vdim ( \project{T} | S )
		\right \|^2 \leq \vdim \| \project{T} - \project{S} \|^2}.
	\end{gather*}}
\end{miniremark}
\begin{lemma} [Approximation by $\qspace_Q ( \rel^\codim )$ valued functions,
cf. \protect{\cite[3.15]{snulmenn.poincare}}]
\label{app:lemma:lipschitz_approximation}
	Suppose $\vdim, \adim, Q \in \nat$, $\vdim < \adim$, $0 < L < \infty$,
	$1 \leq M < \infty$, and $0 < \delta_i \leq 1$ for $i \in
	\{1,2,3,4,5\}$ with $\delta_5 \leq ( 2 \isoperimetric{\vdim} \vdim
	)^{-\vdim} / \unitmeasure{\vdim}$.
	
	Then there exists a positive, finite number $\varepsilon$ with the
	following property.
	
	If $0 < r < \infty$, $0 < h \leq \infty$, $h > 2 \delta_4 r$, $T = \im
	\pp^\ast$,
	\begin{gather*}
		U = \eqclassification{\rel^\vdim \times
		\rel^\codim}{(x,y)}{\dist ((x,y), \cylinder{0}{r}{h}{T})
		< 2r },
	\end{gather*}
	$V \in \IVar_\vdim ( U )$, $\| \delta V \|$ is a Radon measure,
	\begin{gather*}
		( Q - 1 + \delta_1 ) \unitmeasure{\vdim } r^\vdim  \leq \| V
		\| ( \cylinder{0}{r}{h}{T} ) \leq ( Q + 1 - \delta_2 )
		\unitmeasure{\vdim } r^\vdim , \\
		\| V \| ( \cylinder{0}{r}{h+\delta_4 r}{T} \without \cylinder
		{0}{r}{h-2\delta_4 r}{T}) \leq ( 1 - \delta_3 )
		\unitmeasure{\vdim } r^\vdim , \\
		\| V \| ( U ) \leq M \unitmeasure{\vdim } r^\vdim ,
	\end{gather*}
	$0 < \delta \leq \varepsilon$, $B$ denotes the set of all $z \in
	\cylinder{0}{r}{h}{T}$ with $\density^{\ast \vdim } ( \| V \|, z)
	> 0$ such that
	\begin{gather*}
		\text{either} \quad \measureball{\| \delta V
		\|}{\cball{z}{\varrho}} > \delta \, \| V \| (
		\cball{z}{\varrho} )^{1-1/\vdim } \quad \text{for some $0 <
		\varrho < 2 r$}, \\
		\text{or} \quad {\textstyle\int_{\cball{z}{\varrho} \times
		\grass{\adim}{\vdim}}} |
		\project{S} - \project{T} | \ud V (\xi,S) >
		\delta \, \measureball{\| V \|}{\cball{z}{\varrho}} \quad
		\text{for some $0 < \varrho < 2 r$},
	\end{gather*}
	$A = \cylinder{T}{0}{r}{h} \without B$, $A(x) =
	\classification{A}{z}{\pp (z) = x}$ for $x \in \rel^\vdim$, $X_1$ is
	the set of all $x \in \rel^\vdim \cap \cball{0}{r}$ such that
	\begin{gather*}
		{\textstyle\sum_{z \in A(x)}} \density^\vdim ( \| V \|, z )
		= Q \quad \text{and} \quad \text{$\density^\vdim ( \| V \|,
		z ) \in \nat \cup \{0\}$ for $z \in A(x)$},
	\end{gather*}
	$X_2$ is the set of all $x \in \rel^\vdim \cap \cball{0}{r}$ such
	that
	\begin{gather*}
		{\textstyle\sum_{z \in A(x)}} \density^\vdim ( \| V \|, z )
		\leq Q - 1 \quad \text{and} \quad \text{$\density^\vdim
		( \| V \|, z ) \in \nat \cup \{ 0 \}$ for $z \in A(x)$},
	\end{gather*}
	$N = \rel^\vdim \cap \cball{0}{r} \without ( X_1 \cup X_2 )$, $f :
	X_1 \to \qspace_Q ( \rel^{\codim} )$ is characterised by the
	requirement
	\begin{gather*}
		\density^\vdim ( \| V \|, z) = \density^0 ( \| f (x) \|, \qq
		(z) ) \quad \text{whenever $x \in X_1$ and $z \in A(x)$},
	\end{gather*}
	and $H$ denotes the set of all $z \in \cylinder{0}{r}{h}{T}$ such that
	\begin{gather*}
		\measureball{\| \delta V \|}{\oball{z}{2r}} \leq
		\varepsilon \, \| V \| ( \oball{z}{2r} )^{1-1/\vdim }, \\
		{\textstyle\int_{\oball{z}{2r} \times \grass{\adim}{\vdim}}} |
		\project{S} - \project{T} | \ud V ( \xi,S ) \leq \varepsilon
		\, \measureball{\| V \|}{\oball{z}{2r}}, \\
		\measureball{\| V \|}{\cball{z}{\varrho}} \geq \delta_5
		\unitmeasure{\vdim} \varrho^\vdim \quad \text{for $0 < \varrho
		< 2r $},
	\end{gather*}
	then the following six statements hold:
	\begin{enumerate}
		\item \label{app:item:lipschitz_approximation:N}
		$\mathscr{L}^\vdim ( N ) = 0$.
		\item \label{app:item:lipschitz_approximation_2:def} $A$
		and $B$ are Borel sets and
		\begin{gather*}
			\qq \lIm A \cap \spt \| V \| \rIm \subset
			\cball{0}{h-\delta_4r}.
		\end{gather*}
		\item \label{app:item:lipschitz_approximation_2:lip} The
		function $f$ is Lipschitzian with $\Lip f \leq L$.
		\item \label{app:item:lipschitz_approximation_2:misc} For
		$\mathcal{L}^\vdim $ almost all $x \in X_1$ the following is
		true:
		\begin{enumerate}
			\item
			\label{app:item:item:lipschitz_approximation:misc:a} The
			function $f$
			is approximately strongly affinely approximable at
			$x$.
			\item
			\label{app:item:item:lipschitz_approximation:misc:apf0}
			If $(x,y) \in \graph_Q f$ then
			\begin{gather*}
				\Tan^\vdim ( \| V \|, (x,y) ) = \Tan \big (
				\graph_Q \ap Af(x),(x,y) \big ) \in
				\grass{\adim}{\vdim}.
			\end{gather*}
		\end{enumerate}
		\item \label{app:item:lipschitz_approximation_2:lip_related} If
		$z \in H$, then $| \qq ( z ) | \leq h - \delta_4 r$ and for $x
		\in X_1 \cap
		\cball{\pp(z)}{\lambda_{\eqref{app:item:lipschitz_approximation_2:lip_related}}r}$
		there exists $\xi \in A (x)$ satisfying
		\begin{gather*}
			\density^\vdim ( \| V \|, \xi) \in \nat \quad
			\text{and} \quad \big | \perpproject{T} ( \xi - z )
			\big | \leq L \, | \project{T} ( \xi - z ) |,
		\end{gather*}
		where $0 <
		\lambda_{\eqref{app:item:lipschitz_approximation_2:lip_related}} <
		1$ depends only on $\vdim $, $\delta_2$, and $\delta_4$.
		Moreover,
		\begin{gather*}
			A \cap \spt \| V \| \subset H \quad \text{and} \quad H
			\cap \pp^{-1} \lIm X_1 \rIm = \graph_Q f.
		\end{gather*}
		\item \label{app:item:lipschitz_approximation_2:boundary}
		$\left ( \mathcal{L}^\vdim + \pp_\# ( \| V \| \restrict H)
		\right ) ((\Clos X_1) \without X_1)=0$.
	\end{enumerate}
\end{lemma}
\begin{proof}
	Assume $r=1$. First, note that the sets $Y$ and $Z$ defined in the
	last paragraph of the proof of
	\cite[3.15\,(1)\,(2)]{snulmenn.poincare} equal $X_1$ and $X_2$ and are
	shown there to satisfy $\mathscr{L}^\vdim ( \cball{0}{1} \without ( X
	\cup Y )) = 0$. Hence part \eqref{app:item:lipschitz_approximation:N}
	is evident and the parts
	\eqref{app:item:lipschitz_approximation_2:def},
	\eqref{app:item:lipschitz_approximation_2:lip},
	\eqref{app:item:item:lipschitz_approximation:misc:a},
	\eqref{app:item:lipschitz_approximation_2:lip_related}, and
	\eqref{app:item:lipschitz_approximation_2:boundary} correspond to
	parts (2), (1), (7a), (4), and (5) of \cite[3.15]{snulmenn.poincare}
	respectively. Finally, part
	\eqref{app:item:item:lipschitz_approximation:misc:apf0} is implied by
	\cite[3.15\,(7b)]{snulmenn.poincare} in conjunction with the last
	statement of \cite[3.15\,(4)]{snulmenn.poincare}.
\end{proof}
\begin{lemma} \label{lemma:simple_interpolation}
	Suppose $k, \vdim, \adim \in \nat$, $\vdim < \adim$, $a \in
	\rel^\vdim$, $0 < r < \infty$, and $u : \oball{a}{r} \to \rel^\codim$
	is of class $k$.

	Then
	\begin{gather*}
		\tsum{i=0}{k} r^i \norm{D^iu}{\infty}{a,r} \leq \Gamma \big (
		r^k \norm{D^ku}{\infty}{a,r} + r^{-\vdim} \norm{u}{1}{a,r}
		\big )
	\end{gather*}
	where $\Gamma$ is a positive, finite number depending only on $k$ and
	$\adim$.
\end{lemma}
\begin{proof}
	Assuming $r=1$, this is a consequence of Ehring's lemma, see
	e.g.~\cite[Theorem\,\printRoman{1}.7.3]{MR895589}, and Arzel\`a's and
	Ascoli's theorem.
\end{proof}
\begin{lemma} \label{lemma:poincare}
	Suppose $\vdim, \adim \in \nat$, $\vdim < \adim$, $a \in \rel^\vdim$,
	$0 < r < \infty$, and $u \in \Sob{}{1}{2} ( \oball{a}{r}, \rel^\codim
	)$.

	Then there exists $h \in \rel^\codim$ with
	\begin{gather*}
		\norm{u-h}{2}{a,r} \leq \Gamma r \norm{\weakD u}{2}{a,r}
	\end{gather*}
	where $\Gamma$ is a positive, finite number depending only on $\adim$.
\end{lemma}
\begin{proof}
	This is Poincar\'e's inequality, see e.g.~\cite[(7.45)]{MR1814364}.
\end{proof}
\section{A coercive estimate} \label{sec:coercive}
In the present section two improved versions of Brakke's coercive estimate in
\cite[5.5]{MR485012} are derived in \ref{lemma:coercive_estimate_rect} and
\ref{lemma:coercive_estimate}. First, some computations for the catenoid are
carried out in \ref{example:catenoid} which are used in \ref{remark:catenoid}
to rule out a certain generalisation of the coercive estimate. Then, some
basic facts about approximate differentiability with respect to the weight
measure of a varifold are given in \ref{lemma:approx_diff} which are needed to
construct a cut-off function in \ref{lemma:capacity}. Finally, the coercive
estimate for rectifiable varifolds satisfying a lower bound on the density is
proven in \ref{lemma:coercive_estimate_rect} and a simpler version for general
varifolds is indicated in \ref{lemma:coercive_estimate}.
\begin{miniremark} \label{miniremark:projections}
	Frequently, the following estimates from Allard
	\cite[8.9\,(5)]{MR0307015} will be used:

	\emph{Suppose $\vdim, \adim \in \nat$, $\vdim < \adim$, $T \in
	\grass{\adim}{\vdim}$ and $\eta_1, \eta_2 \in \Hom ( S, S^\perp )$. If
	\begin{gather*}
		S_i = \classification{\rel^\adim}{z}{z+\eta_i (z) \with z \in
		S} \quad \text{for $i = 1,2$},
	\end{gather*}
	then
	\begin{gather*}
		\| \eqproject{S_1} - \eqproject{S_2} \| \leq \| \eta_1 -
		\eta_2 \|, \\
		\big ( 1 - \| \eqproject{S_1} - \project{S} \|^2 \big ) \|
		\eta_1 - \eta_2 \|^2 \leq \big ( 1 + \| \eta_2 \|^2 \big ) \|
		\eqproject{S_1} - \eqproject{S_2} \|^2.
	\end{gather*}}
\end{miniremark}
\begin{example} \label{example:catenoid}
	Suppose $\vdim = 2$, $\adim = 3$, and $f :
	\classification{\rel}{t}{1 \leq t < \infty} \to \rel$ as well as $N$,
	$T$, and $P_R$ are defined by
	\begin{gather*}
		f (t) = \log \big ( t + ( t^2-1)^{1/2} \big  ) \quad \text{for
		$1 \leq t < \infty$}, \\
		N = \classification{\rel^3}{z}{ | \qq (z) | = f ( | \pp (z) |
		)}, \quad T = \im \pp^\ast, \\
		P_R = \classification{\rel^3}{z}{| \qq (z) | = \log (2R)}
		\quad \text{for $2 \leq R < \infty$}.
	\end{gather*}

	Then there exists a universal, positive, finite number $\Gamma$ with
	the following two properties:
	\begin{enumerate}
		\item \label{item:catenoid:height} $\int_{\rel^3 \cap
		\cball{0}{R}} | \dist ( z, P_R ) |^2 \ud ( \mathscr{H}^2
		\restrict N ) z \leq \Gamma R^2$ for $2 \leq R < \infty$.
		\item \label{item:catenoid:tilt} $\int_{\rel^3 \cap
		\cball{0}{R}} | \project{\Tan ( N, z )} - \project{T} |^2 \ud
		( \mathscr{H}^2 \restrict N ) z \geq \Gamma^{-1} \log R$ for $2
		\leq R < \infty$.
	\end{enumerate}
\end{example}
\begin{proof} [Construction of example]
	First, note
	\begin{gather*}
		f'(t) = \frac{1}{t+(t^2-1)^{1/2}} \cdot \left (
		1+\frac{t}{(t^2-1)^{1/2}} \right ) \quad \text{for $1 < t <
		\infty$},
	\end{gather*}
	hence $( \Gamma_1 )^{-1} t^{-1} \leq f'(t) \leq \Gamma_1 t^{-1}$ for
	$2 \leq t < \infty$ and some universal, positive, finite number
	$\Gamma_1$, in particular $\Lip f | \classification{\rel}{s}{s \geq 2}
	< \infty$.
	
	To prove \eqref{item:catenoid:height}, one estimates
	\begin{gather*}
		\tint{\cylind{T}{0}{R} \without
		\cylind{T}{0}{2}}{} \dist (z, P_R)^2 \ud (
		\mathscr{H}^2 \restrict N ) z \leq \Gamma_2 ( a_1 + a_2 )
	\end{gather*}
	where $\Gamma_2$ is a universal, positive, finite number and
	\begin{gather*}
		\begin{aligned}
			a_1 & = \tint{\cball{0}{R} \without \cball{0}{2}}{} |
			\log (2R) - \log ( 2 |x| ) |^2 \ud \mathscr{L}^2 x, \\
			a_2 & = \tint{\cball{0}{R} \without \cball{0}{2}}{} |
			\log ( 2 |x| ) - f ( |x| ) |^2 \ud \mathscr{L}^2 x.
		\end{aligned}
	\end{gather*}
	Concerning $a_1$, note
	\begin{gather*}
		a_1 = 2 \pi \tint{2}{R} | \log ( t/R ) |^2 t \ud \mathscr{L}^1
		t \leq 2 \pi R^2 \tint{0}{1} | \log ( t ) |^2 t \ud
		\mathscr{L}^1 t < \infty.
	\end{gather*}
	To estimate $a_2$, define $h : \classification{\rel}{t}{t > 0} \to
	\rel$ by $h(t) = t^{1/2}$ and note for $2 \leq t < \infty$
	\begin{gather*}
		| \log ( 2t) - \log ( t + (t^2-1)^{1/2} ) | \leq \Lip ( \log |
		\classification{\rel}{s}{s \geq t} ) | t - ( t^2 -1 )^{1/2} |
		\\
		\leq t^{-1} \Lip (h| \classification{\rel}{s}{s \geq (t^2-1)})
		\leq t^{-1} 2^{-1} ( t^2-1)^{-1/2} \leq 2^{-1/2} t^{-2},
	\end{gather*}
	hence $a_2 \leq \pi \tint{2}{R} t^{-3} \ud \mathscr{L}^1 t \leq \pi
	/8$. Together, the estimates for $a_1$ and $a_2$ yield
	\eqref{item:catenoid:height}. By \ref{miniremark:projections}, it
	follows
	\begin{gather*}
		\| \project{\Tan ( N,z )} - \project{T} \| \leq f' ( |\pp(z)|
		) \leq \Gamma_1 | \pp (z) |^{-1}
	\end{gather*}
	for $z \in N \without \cylind{T}{0}{2}$, hence by
	\ref{miniremark:projections} with $S$, $S_1$, $S_2$ replaced by $T$,
	$\Tan (N,z)$, $T$,
	\begin{gather*}
		| \project{\Tan (N,z)} - \project{T} | \geq \| \project{\Tan
		(N,z)} - \project{T} \| \geq f'(|\pp (z)|)/2 \geq  ( 2
		\Gamma_1)^{-1} | \pp (z) |^{-1}
	\end{gather*}
	for $z \in N \without \cylind{T}{0}{2 \Gamma_1}$. Noting for
	$2 \leq R < \infty$
	\begin{gather*}
		f (t) \leq f (R) \leq 2R \quad \text{for $1 \leq t \leq R$},
		\qquad N \cap \cylind{T}{0}{R} \subset \rel^3 \cap
		\cball{0}{3R},
	\end{gather*}
	this implies for $2 \sup \{ \Gamma_1,1 \} \leq R < \infty$ that
	\begin{gather*}
		\begin{aligned}
			& \tint{\rel^3 \cap \cball{0}{3R}}{} | \project{\Tan
			(N,z)} - \project{T} |^2 \ud ( \mathscr{H}^2 \restrict
			N ) z \\
			& \qquad \geq \tint{\cylind{T}{0}{R}
			\without \cylind{T}{0}{2\Gamma_1}}{} |
			\project{\Tan (N,z)} - \project{T} |^2 \ud (
			\mathscr{H}^2 \restrict N ) z \\
			& \qquad \geq ( 2\Gamma_1)^{-2} \tint{2\Gamma_1}{R}
			t^{-1} \ud \mathscr{L}^1 t = ( 2\Gamma_1 )^{-2} \log (
			R/(2\Gamma_1) ).
		\end{aligned}
	\end{gather*}
	Since $\int_{\rel^3 \cap \cball{0}{2}} | \project{\Tan (N,z)} -
	\project{T} |^2 \ud ( \mathscr{H}^2 \restrict N ) z  > 0$, one infers
	\eqref{item:catenoid:tilt}.
\end{proof}
\begin{miniremark} \label{miniremark:situation}
	The following situation will be studied: $\vdim, \adim \in \nat$,
	$\vdim < \adim$, $1 \leq p \leq \infty$, $U$ is an open subset of
	$\rel^\adim$, $V \in \Var_\vdim ( U)$, $\| \delta V \|$ is a Radon
	measure and, if $p > 1$,
	\begin{gather*}
		( \delta V ) ( g ) = - {\textstyle\int} g (z) \bullet \mathbf{h}
		(V;z) \ud \| V \| (z) \quad \text{whenever $g \in \mathscr{D}
		( U, \rel^\adim )$}, \\
		\mathbf{h} (V;\cdot) \in \Lp{p} ( \| V \| \restrict K,
		\rel^\adim ) \quad \text{whenever $K$ is a compact subset of
		$U$}.
	\end{gather*}

	If $p < \infty$ then the measure $\psi$ is defined by
	\begin{gather*}
		\psi = \| \delta V \| \quad \text{if $p = 1$}, \qquad \psi =
		| \mathbf{h} ( V ; \cdot ) |^p \| V \| \quad \text{if $p >
		1$}.
	\end{gather*}
\end{miniremark}
\begin{miniremark} \label{miniremark:extension}
	Suppose $\vdim$, $\adim$, $p=1$, $U$ and $V$ are as in
	\ref{miniremark:situation}. Then $\delta V \in \mathscr{D}' ( U,
	\rel^\adim )$ will be extended to $\Lp{1} ( \| \delta V \|, \rel^\adim
	)$ by continuity with respect to $\| \delta V \|_{(1)}$ and $(\delta V
	) (g)$ will be used to denote this extension for $g \in \Lp{1} (
	\| \delta V \|, \rel^\adim )$ as in \cite[4.1.5]{MR41:1976}.
\end{miniremark}
\begin{lemma} \label{lemma:approx_diff}
	Suppose $\vdim, \adim \in \nat$, $\vdim \leq \adim$, $U$ is an open
	subset of $\rel^\adim$, and $V \in \RVar_\vdim (U)$.

	Then the following four statements hold:
	\begin{enumerate}
		\item \label{item:approx_diff:measurability} If $f : U \to
		\rel$ is $\| V \|$ measurable and $A$ denotes the set of all
		$z \in U$ such that $f$ is $(\| V \|, \vdim )$ approximately
		differentiable at $z$, then $A$ is $\| V \|$ measurable and
		$( \| V \|, \vdim ) \ap Df(z) \circ \project{\Tan^\vdim ( \| V
		\|, z )}$ depends $\| V \| \restrict A$ measurably on $z$.
		\item \label{item:approx_diff:diff} If $f : U \to \rel$ is
		Lipschitzian, then $f$ is $(\|V\|,\vdim)$ approximately
		differentiable at $\|V\|$ almost all $z$.
		\item \label{item:approx_diff:conv} If $f_i : U \to \rel$ is a
		sequence of functions converging locally uniformly to $f : U
		\to \rel$ and $\sup \{ \Lip f_i \with i \in \nat \} < \infty$,
		then
		\begin{gather*}
			{\textstyle\int} \left < g(z), ( \| V \|, \vdim ) \ap
			Df_i (z) \right > \ud \| V \| z \to {\textstyle\int}
			\left < g (z), ( \| V \|, \vdim ) \ap D f (z) \right >
			\ud \| V \| z
		\end{gather*}
		as $i \to \infty$ whenever $g \in \Lp{1} ( \| V \|,
		\rel^\adim)$ with $g(z) \in \Tan^\vdim ( \| V \|, z )$ for $\|
		V \|$ almost all $z$.
		\item \label{item:approx_diff:variation} If $f : U \to
		\rel^\adim$ is a Lipschitzian function with compact support in
		$U$ and $\| \delta V \|$ is a Radon measure, then (see
		\ref{miniremark:extension})
		\begin{gather*}
			\delta V ( f ) = {\textstyle\int} \project{S} \bullet
			(( \|V \|, \vdim ) \ap Df (z) \circ \project{S} ) \ud
			V (z,S).
		\end{gather*}
	\end{enumerate}
\end{lemma}
\begin{proof} [Proof of \eqref{item:approx_diff:measurability} and \eqref{item:approx_diff:diff}]
	Since $\| V \| ( \classification{U}{z}{\density^{\ast \vdim} ( \| V
	\|, z ) = \infty} ) = 0$, a set $B$ is $\| V \|$ measurable if and
	only if $\classification{B}{z}{\density^{\ast \vdim} ( \| V \|, z ) >
	0}$ is $\mathscr{H}^\vdim$ measurable by
	\cite[2.10.19\,(1)\,(3)]{MR41:1976}. Hence
	\eqref{item:approx_diff:measurability} and
	\eqref{item:approx_diff:diff} follow from \cite[3.2.17--19, 3.1.4,
	2.10.19\,(4), 2.9.9]{MR41:1976}.
\end{proof}
\begin{proof} [Proof of \eqref{item:approx_diff:conv}]
	Clearly, the assertion needs only to be verified for elements $g$ of
	some subset $X$ of $\Lp{1} ( \| V \|, \rel^\adim )$ whose span is
	$\| V \|_{(1)}$ dense in $\classification{\Lp{1} ( \| V \|,
	\rel^\adim)}{g}{g(z) \in \text{$\Tan^\vdim ( \| V \|, z )$ for $z \in
	U$}}$.  Therefore one may first assume $\| V \| = \mathscr{H}^\vdim
	\restrict W$ for some $( \mathscr{H}^\vdim, \vdim )$ rectifiable and
	$\mathscr{H}^\vdim$ measurable subset of $U$ by \cite[3.2.19,
	2.10.19\,(4), 2.9.9]{MR41:1976} and then $\vdim = \adim$, $\| V \| =
	\mathscr{L}^\vdim$ by \cite[3.2.17--20, 3.1.5, 2.9.11]{MR41:1976}.
	This case can be treated with $X = \mathscr{D} ( \rel^\vdim,
	\rel^\vdim )$ using partial integration.
\end{proof}
\begin{proof} [Proof of \eqref{item:approx_diff:variation}]
	\eqref{item:approx_diff:conv} readily implies
	\eqref{item:approx_diff:variation} by means of convolution.
\end{proof}
\begin{remark}
	Concerning the possible use of $(\| V\|, \vdim)$ approximate
	differentials for a similar purpose, see Federer \cite[\S 2,
	p.~415]{MR833403}.  Also, an argument similar to the proof of
	\eqref{item:approx_diff:conv} and \eqref{item:approx_diff:variation}
	is indicated in Hutchinson \cite[p.~60]{MR1066398}.
\end{remark}
\begin{lemma} \label{lemma:capacity}
	Suppose $\vdim$, $\adim$, $p$, $U$, $V$, and $\psi$ are as in
	\ref{miniremark:situation}, $p < \vdim$, $V \in \RVar_\vdim (U)$,
	$\density^\vdim ( \| V \| , z ) \geq 1$ for $\| V \|$ almost all $z$,
	$K$ is a compact subset of $U$, $0 < \delta \leq \frac{1}{40}$, and
	$H$ is the set of all $z \in \spt \| V \|$ such that
	\begin{gather*}
		\measureball{\| V \|}{\cball{z}{r}} \geq \delta^\vdim (
		\isoperimetric{\vdim} \vdim )^{-\vdim} r^\vdim \quad
		\text{whenever $0 < r < \infty$, $\cball{z}{r} \subset K$}.
	\end{gather*}

	Then there exists a Baire function $f : U \to
	\classification{\rel}{t}{0 \leq t \leq 1 }$ satisfying for $g \in
	\mathscr{D} ( U, \rel^\adim )$
	\begin{gather*}
		\classification{\rel^\adim}{z}{f (z) \neq 0} \subset K, \quad
		\| V \| ( \classification{U}{z}{f(z) \neq 1} \without H ) = 0,
		\\
		\text{$f$ is $(\|V\|,\vdim)$ approximately differentiable
		at $\| V \|$ almost all $z$}, \\
		{\textstyle\int} \project{S} \bullet Dg(z) f (z) \ud V
		(z,S) = \delta V ( f g ) - {\textstyle\int} \left <
		\project{S} ( g(z)), \ap D f (z) \right > \ud V ( z,S ), \\
		\Lpnorm{\| V \|}{p}{| \ap D f |} \leq \delta (400)^\vdim \,
		\psi (K)^{1/p}, \\
		\| V \| ( \classification{U}{z}{f(z)\neq 0} ) \leq \Gamma
		\, \psi (K)^{\vdim/(\vdim-p)}
	\end{gather*}
	(see \ref{miniremark:extension}) where $\Gamma = ((400)^\vdim
	\isoperimetric{\vdim} \vdim)^{\vdim p/(\vdim-p)}$.
\end{lemma}
\begin{proof}
	Let $B = \eqclassification{U \without H}{z}{\density_\ast^\vdim ( \| V
	\|, z ) \geq 1}$ and assume $B \neq \emptyset$. First, the following
	assertion will be shown: \emph{Whenever $z \in B$ there exists $0 < t
	< \infty$ such that $\cball{z}{10t} \subset K$ and
	\begin{gather*}
		t^{-1} \| V \|\cball{z}{10t})^{1/p} \leq \delta (400)^\vdim \,
		\psi ( \cball{z}{t} )^{1/p}, \\
		\measureball{\| V \|}{\cball{z}{10t}} \leq \Gamma \,
		\psi (
		\cball{z}{t})^{\vdim/(\vdim-p)}.
	\end{gather*}}
	For this purpose choose $0 < r < \infty$ with $\cball{z}{r} \subset K$
	and
	\begin{gather*}
		\measureball{\| V\|}{\cball{z}{r}} \leq \delta^\vdim ( 
		\isoperimetric{\vdim} \vdim )^{-\vdim} r^\vdim,
	\end{gather*}
	let $P$ denote the set of all $0 < t \leq r$ such that
	\begin{gather*}
		\measureball{\| V \|}{\cball{z}{t}} \leq (20\delta)^\vdim (
		\isoperimetric{\vdim} \vdim )^{-\vdim} t^\vdim
	\end{gather*}
	and $Q$ the set of all $0 < t \leq \frac{r}{20}$ such that $\{ s \with
	t \leq s \leq 20t \} \subset P$. One notes for $\frac{r}{20} \leq s
	\leq r$
	\begin{gather*}
		s^{-\vdim} \measureball{\| V \|}{\cball{z}{s}} \leq ( 20
		)^\vdim r^{-\vdim} \measureball{\| V
		\|}{\cball{z}{r}} \leq (20 \delta )^\vdim (
		\isoperimetric{\vdim} \vdim )^{-\vdim},
	\end{gather*}
	hence $\frac{r}{20} \in Q$. Let $\varrho = \inf Q$ and note $\varrho >
	0$ since $20 \delta < 1$ and $( \isoperimetric{\vdim} \vdim
	)^{-\vdim} \leq \unitmeasure{\vdim}$. Clearly, $\{ s \with \varrho
	\leq s \leq 20 \varrho \} \subset P$. Also, whenever $\varrho \leq s
	\leq 20 \varrho$
	\begin{gather*}
		s^{-\vdim} \measureball{\| V \|}{\cball{z}{s}} \geq
		(20)^{-\vdim} \varrho^{-\vdim} \measureball{\| V
		\|}{\cball{z}{\varrho}} = \delta^\vdim (
		\isoperimetric{\vdim} \vdim )^{-\vdim}
	\end{gather*}
	because $\varrho \in \Clos{(\{ s \with s < \varrho \} \without P
	)}$.

	Define $\alpha : \{ s \with 0 < s < r \} \to \rel$ and $\beta : \{ s
	\with 0 < s < r \} \to \rel$ by
	\begin{gather*}
		\alpha(s) = \measureball{\| V \|}{\cball{z}{s}}, \quad
		\beta(s) = \psi ( \cball{z}{s} )^{1/p}
	\end{gather*}
	whenever $0 < s < r$. Then by \ref{app:lemma:mydiss}
	\begin{gather*}
		\isoperimetric{\vdim}^{-1} \leq \alpha(s)^{1/\vdim-1} (
		\measureball{\| \delta V \|}{\cball{z}{s}} + \alpha'(s))
	\end{gather*}
	for $\mathscr{L}^1$ almost all $0 < s < r$, hence by H\"older's
	inequality
	\begin{gather*}
		( \vdim \isoperimetric{\vdim} )^{-1} \leq
		\alpha(s)^{1/\vdim-1/p} \beta(s) + (\alpha^{1/\vdim} )' (s)
	\end{gather*}
	for $\mathscr{L}^1$ almost all $0 < s < r$. This inequality implies
	the existence of $\varrho < t < 2\varrho$ satisfying
	\begin{gather*}
		t^{-1} \alpha ( 10t )^{1/p}\leq \delta (400)^\vdim \beta(t);
	\end{gather*}
	in fact if this were not the case, then for $\mathscr{L}^1$ almost all
	$\varrho < s < 2\varrho$, recalling $\{ s, 10s \} \subset P$,
	\begin{gather*}
		\begin{aligned}
			( \isoperimetric{\vdim} \vdim )^{-1} -
			(\alpha^{1/\vdim})'(s) & < \alpha (s)^{1/\vdim-1/p}
			(400)^{-\vdim} \delta^{-1} s^{-1} \alpha ( 10s)^{1/p} \\
			& \leq (1/2) ( \isoperimetric{\vdim} \vdim )^{-1},
		\end{aligned} \\
		(20\delta) ( \isoperimetric{\vdim} \vdim )^{-1} \leq (1/2) (
		\isoperimetric{\vdim} \vdim )^{-1} < (\alpha^{1/\vdim})'(s),
	\end{gather*}
	hence, using $\alpha^{1/\vdim} (\varrho) = (20 \delta ) (
	\isoperimetric{\vdim} \vdim )^{-1} \varrho$ and
	\cite[2.9.19]{MR41:1976} or \cite[3.29]{MR2003a:49002}, one would
	obtain for $\varrho < s < 2 \varrho$
	\begin{gather*}
		\alpha^{1/\vdim} (s) \geq \alpha^{1/\vdim} ( \varrho) +
		{\textstyle\int_\varrho^s} (\alpha^{1/\vdim})'(t) \ud
		\mathscr{L}^1 t > (20\delta) ( \isoperimetric{\vdim} \vdim
		)^{-1} s, \quad s \notin P.
	\end{gather*}
	The second part of the assertion now follows, noting $10t \leq
	20\varrho$, from
	\begin{align*}
		\| V \| ( \cball{z}{10t} )^{1/p-1/\vdim} & \leq t^{-1}
		\delta^{-1} \isoperimetric{\vdim}\vdim \, \| V \| (
		\cball{z}{10t} )^{1/p} \\
		& \leq (400)^\vdim \isoperimetric{\vdim} \vdim \, \psi (
		\cball{z}{t} )^{1/p}.
	\end{align*}

	By the preceding assertion and Vitali's covering theorem, see
	e.g.~\cite[2.8.5]{MR41:1976} or \cite[3.3]{MR87a:49001}, there exist a
	nonempty, countable set $I$ and $z_i \in B$, $0 < t_i < \infty$ and
	$u_i : U \to \rel$ for $i \in I$ such that
	\begin{gather*}
		u_i (z) = \sup \{ 0, 1 - \dist (z, \cball{z_i}{5t_i})/t_i \}
		\quad \text{for $z \in U$, $i \in I$}, \\
		\spt u_i \subset \cball{z_i}{10t_i} \subset K \quad \text{for
		$i \in I$}, \\
		\cball{z_i}{t_i} \cap \cball{z_j}{t_j} = \emptyset \quad
		\text{whenever $i,j \in I$, $i \neq j$}, \\
		\begin{aligned}
			\Lpnorm{\| V \|}{p}{|\ap Du_i|} & \leq \delta ( 400
			)^\vdim \, \psi ( \cball{z_i}{t_i} )^{1/p}, \\
			\measureball{\| V \|}{\cball{z_i}{10t_i}} & \leq
			\Gamma \, \psi ( \cball{z_i}{t_i} )^{\vdim/(\vdim-p)},
		\end{aligned} \\
		B \subset {\textstyle\bigcup} \{ \cball{z_i}{5t_i} \with i \in
		I \}.
	\end{gather*}
	Define $v_J : U \to \rel$ by
	\begin{gather*}
		v_J (z) = \sup ( \{ 0 \} \cup \{ u_j (z) \with j \in J \} )
		\quad \text{for $z \in U$}
	\end{gather*}
	whenever $J \subset I$, and $f = v_I$. Note $0 \leq f \leq 1$ and
	\begin{gather*}
		u_i (z) = 1 \quad \text{whenever $z \in \cball{z_i}{5t_i}$, $i
		\in I$}, \qquad f(z) = 1 \quad \text{for $z \in B$}.
	\end{gather*}
	Noting \ref{lemma:approx_diff}\,\eqref{item:approx_diff:diff} and
	defining $g = \sup \{ | \ap Du_i | \with i \in I \}$, one estimates
	for $J \subset I$
	\begin{gather*}
		\begin{aligned}
			\Lpnorm{\| V \|}{p}{g}^p & \leq {\textstyle\sum_{i \in
			I}} \Lpnorm{\| V \|}{p}{|\ap Du_i|}^p \\
			& \leq \delta^p ( 400 )^{\vdim p} {\textstyle\sum_{i
			\in I}} \psi ( \cball{z_i}{t_i} ) \leq \delta^p
			(400)^{\vdim p} \psi (K),
		\end{aligned} \\
		\begin{aligned}
			& \| V \| ( \classification{U}{z}{f(z)> v_J(z)} ) \\
			& \qquad \leq {\textstyle\sum_{i\in I \without J}}
			\measureball{\| V\|}{\cball{z_i}{10t_i}} \leq \Gamma
			\, {\textstyle\sum_{i\in I \without J}} \psi (
			\cball{z_i}{t_i})^{\vdim/(\vdim-p)} \\
			& \qquad \leq \Gamma \left ( {\textstyle\sum_{i \in I
			\without J}} \measureball{\psi}{\cball{z_i}{t_i}}
			\right)^{\vdim/(\vdim-p)} \leq \Gamma \, \psi
			(K)^{\vdim/(\vdim-p)}.
		\end{aligned}
	\end{gather*}
	Choose a sequence $J(k)$ with $J(k) \subset J(k+1) \subset I$,
	$\card J(k) < \infty$ for $k \in \nat$ and $\bigcup \{ J(k) \with k
	\in \nat \} = I$. Then
	\begin{gather*}
		\| V \| \left ( U \cap {\textstyle\bigcap} \left \{ \{ z \with
		f (z) > v_{J(k)} (z) \} \with k \in \nat \right \} \right )
		= 0,
	\end{gather*}
	hence $f$ is $(\|V\|, \vdim)$ approximately differentiable at $\| V
	\|$ almost all $z$ and
	\begin{gather*}
		\sup \{ | \ap Dv_{J(k)} (z) |, | \ap Df(z) | \} \leq g (z)
		\quad \text{for $\| V \|$ almost all $z$}, \\
		\Lpnorm{\| V \|}{p}{ | \ap D v_{J(k)} - \ap D f | } \to 0
		\quad \text{as $k \to \infty$}
	\end{gather*}
	by \cite[2.10.19\,(4)]{MR41:1976} or \cite[3.5]{MR87a:49001} and
	\ref{lemma:approx_diff}\,\eqref{item:approx_diff:measurability}.  The
	integral formula holds with $f$ replaced by $v_{J(k)}$ for $k \in
	\nat$ by \ref{lemma:approx_diff}\,\eqref{item:approx_diff:variation},
	hence, taking the limit $k \to \infty$, also for $f$.
\end{proof}
\begin{remark} \label{remark:capacity}
	The function $f$ cannot be required to be continuous at $\| V \|$
	almost all $z$. To prove this let $\vdim p / ( \vdim - p ) < \eta <
	\infty$, $\adim = \vdim + 1$, $U = \rel^\adim$, apply
	\cite[1.2]{snulmenn.isoperimetric} with $\alpha_1 q_1 = \alpha_2 q_2 =
	\eta$ to obtain $\mu$ and $T$ and define $V$ by the requirement $\| V
	\| = \mu$. Take $\xi \in T$ with $\density^\vdim ( \psi , \xi ) = 0$;
	the existence of such $\xi$ follows from
	\cite[2.10.19\,(4)]{MR41:1976} or \cite[3.5]{MR87a:49001} as $\psi (T)
	= 0$. (Alternately, it follows from the estimates in
	\cite[1.2]{snulmenn.isoperimetric} that one can take any $\xi \in T$.)
	Let $0 < r \leq 1$ and $K = \cball{\xi}{2r}$. One verifies the
	existence of $\varepsilon > 0$ depending only on $V$, $\delta$,
	$\eta$, and $\vdim$ such that
	\begin{gather*}
		\classification{\cball{\xi}{r}}{z}{0 < \dist (z,T) \leq
		\varepsilon } \cap H = \emptyset.
	\end{gather*}
	Therefore any such function $f$ would have to satisfy $f ( z ) = 1$
	for $\| V \|$ almost all $z \in T \cap \oball{\xi}{r}$, hence
	\begin{gather*}
		\| V \| ( \classification{U}{z}{f(z) \neq 0} ) \geq
		\unitmeasure{\vdim} r^\vdim
	\end{gather*}
	which would be incompatible with the last inequality of
	\ref{lemma:capacity} for small $r$ even if $\Gamma$ would be allowed
	to depend additionally on $V$ and $\delta$.
\end{remark}
\begin{miniremark} \label{miniremark:schaukel}
	If $a \geq 0$, $b \geq 0$, $c > 0$ and $d > 0$ then
	\begin{gather*}
		\inf \{ a t^c + b t^{-d} \with 0 < t < \infty \} = \big (
		(d/c)^{c/(c+d)} + ( d/c )^{-d/(c+d)} \big ) a^{d/(c+d)}
		b^{c/(c+d)}.
	\end{gather*}
\end{miniremark}
\begin{lemma} \label{lemma:coercive_estimate_rect}
	Suppose $\vdim$, $\adim$, $p$, $U$, $V$, and $\psi$ are as in
	\ref{miniremark:situation}, $p < \vdim$, $V \in \RVar_\vdim (U)$,
	$\density^\vdim ( \| V \|, z ) \geq 1$ for $\| V \|$ almost all $z$,
	$K$ is a compact subset of $U$, $H$ is the set of all $z \in \spt \| V
	\|$ such that
	\begin{gather*}
		\measureball{\| V \|}{\cball{z}{r}} \geq (40)^{-\vdim} (
		\isoperimetric{\vdim} \vdim )^{-\vdim} r^\vdim \quad
		\text{whenever $0 < r < \infty$, $\cball{z}{r} \subset K$},
	\end{gather*}
	$\phi \in \mathscr{D}^0 ( U )$, $0 \leq \phi \leq 1$, $\spt
	\phi \subset K$, $1 \leq q \leq \infty$, $1/p + 1/q \geq 1$, $a \in
	\rel^\adim$, $T \in \grass{\adim}{\vdim}$, $h : U \to \rel$ with $h(z)
	= \dist (z-a,T)$ for $z \in U$, and
	\begin{gather*}
		\alpha = \psi (K)^{1/p}, \quad \beta = \left (
		{\textstyle\int} \phi (z)^2 | \project{S} - \project{T} |^2
		\ud V(z,S) \right )^{1/2}, \\
		\begin{aligned}
			\gamma & = \eqLpnorm{ \phi^2 \| V \| \restrict H
			}{q}{h} && \quad \text{if $q < \infty$},
			\\
			\gamma & = \sup \{ h(z) \with z \in
			\spt \| V \|, \phi (z) > 0 \} && \quad \text{if $q =
			\infty$},
		\end{aligned} \\
		\xi = \eqLpnorm{ \| V \| \restrict H }{2}{| D \phi | h}.
	\end{gather*}

	Then
	\begin{gather*}
		\beta^2 \leq \Gamma \big ( \alpha^{\vdim p/(\vdim-p)} + (
		\alpha \gamma )^{1/(1/p+1/q)} \big ) + ( 16 + 4 \vdim )
		\xi^2
	\end{gather*}
	where $\Gamma$ is a positive, finite number depending only on $\vdim$,
	$p$, and $q$.
\end{lemma}
\begin{proof}
	Assume $a = 0$, hence $h (z) = |\perpproject{T} (z) |$ for $z \in U$.
	Use \ref{lemma:capacity} with $\delta = \frac{1}{40}$ to obtain $f$
	and define $V_1,V_2 \in \RVar_\vdim (U)$ by
	\begin{gather*}
		V_1 (A) = {\textstyle\int_A^\ast} f (z) \ud V(z,S) \quad
		\text{for $A \subset U \times \grass{\adim}{\vdim}$}
	\end{gather*}
	and $V_2 = V-V_1$. Using \cite[2.10.19\,(4)]{MR41:1976} or
	\cite[3.5]{MR87a:49001}, one remarks
	\begin{gather*}
		\text{$f(z) = 1$ and $\ap Df(z) = 0$ for $\| V \|$ almost all
		$z \in U \without H$}, \\
		{\textstyle\int} \phi (z)^2 | \project{S} - \project{T} |^2
		\ud V_1(z,S) \leq 4 \vdim \Gamma_{\ref{lemma:capacity}} (
		\vdim, p ) \, \alpha^{\vdim p/(\vdim - p)}, \\
		\| \delta V_2 \| \leq (1-f) \| \delta V \| + | \ap Df | \| V
		\|, \quad \Lpnorm{\| V \|}{p}{| \ap Df |} \leq ( 400 )^\vdim
		\alpha.
	\end{gather*}
	Defining $g = \phi^2 ( \perpproject{T} | U )$, one obtains
	\begin{gather*}
		{\textstyle\int} \phi (z)^2 | \project{S} - \project{T} |^2
		\ud V_2(z,S) \leq 4 | ( \delta V_2 ) ( g ) | + 16 \xi^2
	\end{gather*}
	as in \cite[5.5]{MR485012}. If $1/p+1/q = 1$ then the conclusion is a
	consequence of the preceding remarks and H\"older's inequality.
	Therefore suppose $1/p+1/q > 1$, hence $p < \infty$ and $q < \infty$.

	Letting $0 < t < \infty$, $r = 1 - q (1-1/p)$, and defining $\eta : \{
	s \with 0 \leq s < \infty \} \to \{ s \with 0 \leq s \leq 1 \}$ by
	$\eta(s) = \inf \{ 1 , t s^{-r} \}$ for $0 \leq s < \infty$, one
	observes $0 < r \leq 1$ and
	\begin{gather*}
		0 \leq s \eta (s) \leq t s^{1-r} \quad \text{whenever $0 < s <
		\infty$}, \\
		| s \eta'(s) | + |1-\eta(s)| \leq 1 \quad \text{whenever
		$t^{1/r} < s < \infty$}.
	\end{gather*}
	Moreover, defining $\eta_1 : U \to \rel^\adim$, $\eta_2 : U \to
	\rel^\adim$ by
	\begin{gather*}
		\eta_1 (z) = \eta ( | \perpproject{T} ( z ) | )
		\perpproject{T} (z), \quad \eta_2 (z) = ( 1 - \eta ( |
		\perpproject{T} (z) | ) ) \perpproject{T} ( z )
	\end{gather*}
	whenever $z \in U$,
	\begin{gather*}
		Z_1 = \bigclassification{U}{z}{0 < h(z) < t^{1/r}}, \quad Z_2
		= \bigclassification{U}{z}{t^{1/r} < h(z)},
	\end{gather*}
	one notes $\eta_1+\eta_2 = \perpproject{T} | U$ and computes
	\begin{gather*}
		\left < v, D\eta_2 (z) \right > = - \eta' ( | \perpproject{T}
		(z) |) \frac{\perpproject{T} (z) \bullet v}{| \perpproject{T}
		( z ) |} \perpproject{T} (z) + (1-\eta ( | \perpproject{T} ( z
		) | )) \perpproject{T} ( v )
	\end{gather*}
	for $z \in Z_2$, $v \in \rel^\adim$, hence
	\begin{gather*}
		\| D \eta_2 (z) \| \leq 1 \quad \text{for $z \in Z_2$}
	\end{gather*}
	and for $z \in U$
	\begin{gather*}
		|\eta_1 (z) | \leq t h(z)^{1-r} \quad \text{if $r < 1$},
		\qquad |\eta_1 (z) | \leq t \quad \text{if $r = 1$}.
	\end{gather*}
	Letting $g_1 = \phi^2 \eta_1$, $g_2 = \phi^2 \eta_2$,
	one notes $g_1 + g_2 = g$ and infers $|g_1| = \phi^2 |\eta_1|$,
	\begin{align*}
		& \| D g_2 (z) \| \leq 2\phi(z)|D\phi(z)| h (z) + \phi^2 (z)
		\| D\eta_2 (z) \| \\
		& \qquad \leq 2 \phi^2 (z) + | D\phi(z)|^2 h (z)^2 \leq 2
		\phi^2 (z) t^{-q/r} h (z)^q + | D \phi (z) |^2 h (z)^2
	\end{align*}
	for $z \in Z_2$. Since $Dg_2 (z) = 0$ for $z \in Z_1$ and $\phi$,
	$D\phi$, and $h$ are continuous, approximating $g_1$ and $g_2$ by
	smooth functions yields that $| ( \delta V_2 ) ( g ) |$ does not
	exceed
	\begin{gather*}
		t \| \delta V_2 \| ( \phi^2 h^{1-r} ) + \vdim \| V_2 \| \big (
		2t^{-q/r} \phi^2 h^q + |D\phi|^2 h^2 \big ) \quad \text{if $r
		< 1$}, \\
		t \| \delta V_2 \| ( \phi^2) + \vdim \| V_2 \| \left ( 2t^{-q}
		\phi^2 h^q + |D\phi|^2 h^2 \right ) \quad \text{if $r = 1$},
	\end{gather*}
	hence, using H\"older's inequality and recalling the remarks of the
	first paragraph, one obtains
	\begin{gather*}
		| ( \delta V_2 ) ( g ) | \leq t (800)^\vdim \alpha
		\gamma^{1-r} + 2 \vdim t^{-q/r} \gamma^q + \vdim \xi^2 \quad
		\text{if $r < 1$}, \\
		| ( \delta V_2 ) ( g ) | \leq t (800)^\vdim \alpha + 2 \vdim
		t^{-q} \gamma^q + \vdim \xi^2 \quad \text{if $r = 1$}.
	\end{gather*}
	The conclusion is now a consequence of \ref{miniremark:schaukel}.
\end{proof}
\begin{remark} \label{remark:coercive_estimate_rect}
	Using the inequality relating arithmetic and geometric means (cf.
	\cite[2.4.13]{MR41:1976}), one obtains for $0 < \lambda < \infty$
	\begin{gather*}
		( \alpha \gamma )^{1/(1/p+1/q)} \leq
		{\textstyle\frac{2(1/p+1/q)-1}{2 (1/p+1/q)}}
		(\alpha/\lambda)^{\frac{2}{2(1/p+1/q)-1}} +
		{\textstyle\frac{1}{2 (1/p+1/q)}} ( \lambda \gamma )^2.
	\end{gather*}
	Note, concerning the exponent of $\alpha$, if $1/q = 1/2 - 1/\vdim$,
	then $\frac{2}{2(1/p+1/q) -1} = \frac{\vdim p}{\vdim-p}$.
\end{remark}
\begin{remark}
	The estimate for $| ( \delta V_2 ) (g) |$ is adapted from Brakke
	\cite[5.5]{MR485012} where $p \in \{ 1, 2 \}$ and $q=2$.
\end{remark}
\begin{remark} \label{remark:catenoid}
	One cannot replace $h$ by the distance from two planes parallel to
	$T$, as may be seen from the estimates for the catenoid in
	\ref{example:catenoid} considering $R \to \infty$. This behaviour is
	in contrast to the Sobolev Poincar\'e type inequality in
	\cite[4.4]{snulmenn.poincare}.
\end{remark}
\begin{lemma} \label{lemma:coercive_estimate}
	Suppose $\vdim$, $\adim$, $p$, $U$, and $V$ are as in
	\ref{miniremark:situation}, $\phi \in \mathscr{D}^0 ( U )$, $\phi \geq
	0$, $1 \leq q \leq \infty$, $1/p + 1/q \geq 1$, $a \in \rel^\adim$, $T
	\in \grass{\adim}{\vdim}$, $h : U \to \rel$ with $h(z) = \dist
	(z-a,T)$ for $z \in U$, and
	\begin{gather*}
		\alpha = \| \delta V \| ( \phi^2 ) \quad \text{if $p = 1$},
		\qquad \alpha = ( \phi^2 \| V \| )_{(p)} ( \mathbf{h}
		(V;\cdot) ) \quad \text{if $p > 1$}, \\
		\beta = \big ( {\textstyle\int} \phi (z)^2 | \project{S} -
		\project{T} |^2 \ud V(z,S) \big )^{1/2}, \quad \xi =
		\eqLpnorm{\| V \|}{2}{| D\phi| h }, \\
		\gamma = ( \phi^2 \| V \| )_{(q)} ( h ) \quad
		\text{if $q < \infty$}, \qquad \gamma = ( \phi^2 \| \delta V
		\| )_{(\infty)} ( h ) \quad \text{if $q =
		\infty$}.
	\end{gather*}

	Then
	\begin{align*}
		\beta^2 \leq \Gamma ( \alpha \gamma )^{1/(1/p+1/q)} +
		(16+4\vdim) \xi^2
	\end{align*}
	where $\Gamma$ is a positive, finite number depending only on $\vdim$,
	$p$, and $q$.
\end{lemma}
\begin{proof}
	The proof of \ref{lemma:coercive_estimate_rect} has been designed such
	that a proof of the present assertion results when the arguments
	involving the function $f$ are omitted.
\end{proof}
\section{Approximation by $\qspace_Q( \rel^\codim )$ valued functions}
\label{sec:approx}
The purpose of this section is to establish the necessary adaptions and
extensions of the approximation by $\qspace_Q ( \rel^\codim )$ valued
functions carried out in \cite[3.15]{snulmenn.poincare}. This is done in
\ref{lemma:lipschitz_approximation}\,\eqref{item:lipschitz_approximation:yz}--\eqref{item:lipschitz_approximation:poincare}
and supplemented by a basic estimate concerning the partial differential
equation satisfied by the ``average'' of the approximating function in
\ref{lemma:lipschitz_approximation}\,\eqref{item:lipschitz_approximation:pde}
leaving the estimates more directly related to the purposes of the present
paper to Section \ref{sec:iteration}. The results are based on those in
\cite[\S 3]{snulmenn.poincare}. To effectively treat measurability questions
the concept of universal measurability is recalled in
\ref{def:univ}--\ref{lemma:univ_meas:coarea}.
\begin{definition} \label{def:univ}
	A subset of a topological space $X$ is called \emph{universally
	measurable} if and only if it is measurable with respect to every
	measure $\phi$ on $X$ which has the property that all closed sets are
	$\phi$ measurable.
	
	A function between topological spaces is \emph{universally
	measurable} if and only if every preimage of an open set is
	universally measurable.
\end{definition}
\begin{remark} \label{remark:univ_meas}
	Among the basic properties of the concept of universal measurability
	are the following:
	\begin{enumerate}
		\item \label{item:univ_meas:borel_family} The universally
		measurable sets form a Borel family containing the Borel sets.
		(Note that ``Borel family'' is termed ``$\sigma$-algebra'' in
		\cite[1.1]{MR87a:49001} and ``tribe'' in
		\cite[\printRoman{3},\,\S 0]{MR57:7169}.)
		\item \label{item:univ_meas:preimage1} The preimage of a Borel
		set under a universally measurable function is
		universally measurable.
		\item \label{item:univ_meas:preimage2} The preimage of a
		universally measurable set under a Borel function is
		universally measurable.
		\item \label{item:univ_meas:projection} If $X$ is a complete
		separable metric space, $A$ is a Borel subset of $X$, $Y$ is a
		Hausdorff space and $f : X \to Y$ is continuous then $f \lIm A
		\rIm$ is universally measurable.
	\end{enumerate}
	\eqref{item:univ_meas:borel_family} is evident and implies
	\eqref{item:univ_meas:preimage1}, \eqref{item:univ_meas:preimage2} is
	readily verified by means of \cite[2.1.2]{MR41:1976} and
	\eqref{item:univ_meas:projection} is a consequence of
	\cite[2.2.13]{MR41:1976}.
\end{remark}
\begin{example}
	The following classical example illustrates the use of
	\ref{remark:univ_meas}\,\eqref{item:univ_meas:projection} in the proof
	of
	\ref{lemma:lipschitz_approximation}\,\eqref{item:lipschitz_approximation:measurability}.
	There exists a Borel subset $A$ of $\rel^2$ and an orthogonal
	projection $f : \rel^2 \to \rel$ such that $f \lIm A \rIm$ is not a
	Borel subset of $\rel$. A proof may be obtained by appropriately
	combining the results in \cite[2.2.9,\,11]{MR41:1976}.
\end{example}
\begin{remark}
	The present definition can be shown to be a special case of the
	concept introduced in \cite[\printRoman{3}.21]{MR57:7169}.
\end{remark}
\begin{lemma} \label{lemma:univ_meas:coarea}
	Suppose $X$ is a complete, separable metric space, $Y$ is a Hausdorff
	topological space, $f : X \to Y$ is continuous, $B$ is a Borel subset
	of $X$, and $g : B \to \{ t \with 0 \leq t \leq \infty \}$ is a Borel
	function.
	
	Then $h : Y \to \{ t \with 0 \leq t \leq \infty \}$ defined by
	\begin{gather*}
		h (y) = \sum_{B \cap f^{-1} \lIm\{y\}\rIm} g \quad
		\text{whenever $y \in Y$}
	\end{gather*}
	is universally measurable.
\end{lemma}
\begin{proof}
	One may adapt \cite[2.10.10, 2.3.2\,(4)--(6), 2.3.3]{MR41:1976} by use
	of
	\ref{remark:univ_meas}\,\eqref{item:univ_meas:borel_family}\,\eqref{item:univ_meas:projection}
	to obtain the conclusion.
\end{proof}
\begin{lemma} \label{lemma:quasi_linear}
	Suppose $X$, $Y$ are normed vector spaces, $f : X \to Y$ is of class
	$\class{1}$, $a \in X$, $0 < r < \infty$, $Q \in \nat$, $x_i \in
	\cball{a}{r}$ for $i = 1, \ldots, Q$, and $\gamma = \Lip ( Df |
	\cball{a}{r} )$.

	Then
	\begin{gather*}
		\left | \frac{1}{Q} \sum_{i=1}^Q f (x_i) - f \left (
		\frac{1}{Q} \sum_{i=1}^Q x_i \right ) \right | \leq \gamma
		r^2.
	\end{gather*}
\end{lemma}
\begin{proof}
	Let $P : X \to Y$ by defined by $P(x) = f (a) + \left < x-a, Df(a)
	\right >$ for $x \in X$. Then for $x \in \cball{a}{r}$
	\begin{gather*}
		| f(x)-P(x) | = \big | \big < x-a, {\textstyle\int_0^1} Df ( a
		+ t(x-a)) - Df (a) \ud \mathscr{L}^1 t \big > \big | \leq
		( \gamma/2 ) r^2.
	\end{gather*}
	Since $\frac{1}{Q} \sum_{i=1}^Q P (x_i) = P ( Q^{-1} \sum_{i=1}^Q x_i
	)$, this implies the conclusion.
\end{proof}
\begin{lemma} \label{lemma:lipschitz_approximation}
	Suppose $\adim, Q \in \nat$, $0 < L < \infty$, $1 \leq M < \infty$,
	and $0 < \delta_i \leq 1$ for $i \in \{1,2,3,4,5\}$.
	
	Then there exists a positive, finite number $\varepsilon$ with the
	following property.
	
	If $\vdim \in \nat$, $\vdim < \adim$, $0 < r < \infty$, $0 < h \leq
	\infty$, $h > 2 \delta_4 r$, $T = \im \pp^\ast$,
	\begin{gather*}
		U = \eqclassification{\rel^\vdim \times
		\rel^\codim}{(x,y)}{\dist ((x,y), \cylinder{T}{0}{r}{h}) <
		2r},
	\end{gather*}
	$V \in \IVar_\vdim ( U )$, $\| \delta V \|$ is a Radon measure,
	\begin{gather*}
		( Q - 1 + \delta_1 ) \unitmeasure{\vdim} r^\vdim \leq \| V \| (
		\cylinder{T}{0}{r}{h} ) \leq ( Q + 1 - \delta_2 )
		\unitmeasure{\vdim} r^\vdim, \\
		\| V \| ( \cylinder{T}{0}{r}{h+\delta_4 r} \without
		\cylinder{T}{0}{r}{h-2\delta_4 r}) \leq ( 1 - \delta_3 )
		\unitmeasure{\vdim} r^\vdim, \\
		\| V \| ( U ) \leq M \unitmeasure{\vdim} r^\vdim,
	\end{gather*}
	$0 < \delta \leq \varepsilon$, $B$ denotes the set of all $z
	\in \cylinder{T}{0}{r}{h}$ with $\density^{\ast \vdim} ( \| V \|, z) >
	0$ such that
	\begin{gather*}
		\text{either} \quad
		\measureball{\| \delta V \|}{\cball{z}{\varrho}} >
		\delta \, \| V \| ( \cball{z}{\varrho} )^{1-1/\vdim}
		\quad \text{for some $0 < \varrho < 2 r$}, \\
		\text{or} \quad {\textstyle\int_{\cball{z}{\varrho} \times
		\grass{\adim}{\vdim}}} | \project{S} - \project{T} | \ud V
		(\xi,S) > \delta \, \measureball{\| V
		\|}{\cball{z}{\varrho}} \quad \text{for some $0 < \varrho <
		2 r$},
	\end{gather*}
	$A = \cylinder{T}{0}{r}{h} \without B$, $A (x) =
	\classification{A}{z}{\pp (z) = x}$ for $x \in \rel^\vdim$, $X_1$ is
	the set of all $x \in \rel^\vdim \cap \cball{0}{r}$ such that
	\begin{gather*}
		{\textstyle\sum_{z \in A(x)}} \density^\vdim ( \| V \|, z )
		= Q \quad \text{and} \quad \text{$\density^\vdim ( \| V \|,
		z ) \in \nat \cup \{0\}$ for $z \in A(x)$},
	\end{gather*}
	$X_2$ is the set of all $x \in \rel^\vdim \cap \cball{0}{r}$ such
	that
	\begin{gather*}
		{\textstyle\sum_{z \in A(x)}} \density^\vdim ( \| V \|, z )
		\leq Q - 1 \quad \text{and} \quad \text{$\density^\vdim
		( \| V \|, z ) \in \nat \cup \{ 0 \}$ for $z \in A(x)$},
	\end{gather*}
	$N = \rel^\vdim \cap \cball{0}{r} \without ( X_1 \cup X_2 )$, and $f :
	X_1 \to \qspace_Q ( \rel^{\codim} )$ is characterised by the
	requirement
	\begin{gather*}
		\density^\vdim ( \| V \|, z) = \density^0 ( \| f (x) \|, \qq
		(z) ) \quad \text{whenever $x \in X_1$ and $z \in A(x)$},
	\end{gather*}
	then the following nine statements hold:
	\begin{enumerate}
		\item \label{item:lipschitz_approximation:yz} $X_1$ and $X_2$
		are universally measurable, and $\mathscr{L}^\vdim (N) =0$.
		\item \label{item:lipschitz_approximation:ab} $A$ and $B$
		are Borel sets and
		\begin{gather*}
			\qq \lIm A \cap \spt \| V \| \rIm \subset
			\cball{0}{h-\delta_4r}.
		\end{gather*}
		\item \label{item:lipschitz_approximation:para} $\pp \lIm
		\classification{A}{z}{\density^\vdim ( \| V \|, z ) = Q} \rIm
		\subset X_1$.
		\item \label{item:lipschitz_approximation:lip} The function
		$f$ is Lipschitzian with $\Lip f \leq L$.
		\item \label{item:lipschitz_approximation:misc} For
		$\mathscr{L}^\vdim$ almost all $x \in X_1$ the following is
		true:
		\begin{enumerate}
			\item
			\label{item:item:lipschitz_approximation:misc:a} The
			function $f$ is approximately strongly affinely
			approximable at $x$.
			\item
			\label{item:item:lipschitz_approximation:misc:b}
			If $(x,y) \in \graph_Q f$ then
			\begin{gather*}
				\Tan^\vdim ( \| V \|, (x,y) ) = \Tan \big (
				\graph_Q \ap A f ( x ) , (x,y) \big) \in
				\grass{\adim}{\vdim}.
			\end{gather*}
		\end{enumerate}
		\item \label{item:lipschitz_approximation:measurability} If $a
		\in \cylinder{T}{0}{r}{h}$, $0 < \varrho \leq r - | \pp(a) |$,
		$| \qq(a) | + \delta_4 \varrho \leq h$, and
		\begin{align*}
			B_{a,\varrho} & = \cylinder{T}{a}{\varrho}{\delta_4
			\varrho} \cap B, \\
			C_{a,\varrho} & = \cball{\pp (a)}{\varrho}
			\without ( X_1 \without \pp \lIm B_{a,\varrho} \rIm ),
			\\
			D_{a,\varrho} & = \cylinder{T}{a}{\varrho}{\delta_4
			\varrho} \cap \pp^{-1} \lIm C_{a,\varrho} \rIm,
		\end{align*}
		then $B_{a,\varrho}$ is a Borel set and $C_{a,\varrho}$ and
		$D_{a,\varrho}$ are universally measurable.
		\item \label{item:lipschitz_approximation:estimate_b} If $a$,
		$\varrho$, $B_{a,\varrho}$, $C_{a,\varrho}$, and
		$D_{a,\varrho}$ are as in
		\eqref{item:lipschitz_approximation:measurability} and
		\begin{gather*}
			\graph_Q f | \cball{\pp(a)}{\varrho} \subset
			\cylinder{T}{a}{\varrho}{\delta_4\varrho/2}, \\
			\| V \| ( \cylinder{T}{a}{\varrho}{\delta_4\varrho} )
			\geq ( Q - 1/4 ) \unitmeasure{\vdim} \varrho^\vdim,
		\end{gather*}
		then
		\begin{gather*}
			\mathscr{L}^\vdim ( C_{a,\varrho} ) + \| V \| (
			D_{a,\varrho}) \leq
			\Gamma_{\eqref{item:lipschitz_approximation:estimate_b}}
			\, \| V \| ( B_{a,\varrho} )
		\end{gather*}
		with
		$\Gamma_{\eqref{item:lipschitz_approximation:estimate_b}} =
		3 + 2Q + ( 12Q + 6 ) 5^\vdim$.
		\item \label{item:lipschitz_approximation:poincare} Suppose
		$H$ denotes the set of all $z \in \cylinder{T}{0}{r}{h}$ such
		that
		\begin{gather*}
			\measureball{\| \delta V \|}{\oball{z}{2r}} \leq
			\varepsilon \, \| V \| ( \oball{z}{2r} )^{1-1/\vdim},
			\\
			{\textstyle\int_{\oball{z}{2r} \times
			\grass{\adim}{\vdim}}} | \project{S} - \project{T} |
			\ud V (z,S) \leq \varepsilon \, \measureball{\| V
			\|}{\oball{z}{2r}}, \\
			\measureball{\| V \|}{\cball{z}{\varrho}} \geq
			\delta_5 \unitmeasure{\vdim} \varrho^\vdim \quad
			\text{for $0 < \varrho < 2r$}.
		\end{gather*}

		Then there exists a positive, finite number
		$\varepsilon_{\eqref{item:lipschitz_approximation:poincare}}$
		depending only on $\vdim$, $\delta_2$, and $\delta_4$ with the
		following property:

		If $c \in \rel^\vdim \cap \oball{0}{r}$, $0 < \varrho \leq r -
		|c|$, $\mathscr{L}^\vdim ( \cball{c}{\varrho} \without X_1 )
		\leq
		\varepsilon_{\eqref{item:lipschitz_approximation:poincare}}
		\unitmeasure{\vdim} \varrho^\vdim$, $\emptyset \neq P \subset
		\cylind{T}{\pp^\ast ( c)}{ \varrho}$, for every $z \in P$ and
		$x \in \cball{c}{\varrho}$ there exists $y$ with $(x,y) \in P$
		and $|y-\qq(z)| \leq |x-\pp(z)|$, and $d :
		\cylinder{T}{\pp^\ast (c)}{\varrho}{h} \to \rel$ and $g : X_1
		\cap \cball{c}{\varrho} \to \rel$ are defined by
		\begin{align*}
			d(z) & = \inf \{ | \qq ( \xi - z ) | \with \xi \in P,
			\pp ( \xi ) = \pp ( z ) \} && \quad \text{for $z \in
			\cylinder{T}{\pp^\ast (c)}{\varrho}{h}$}, \\
			g(x) & = \sup \{ d(x,y) \with y \in \spt f (x) \} &&
			\quad \text{for $x \in X_1 \cap \cball{c}{\varrho}$},
		\end{align*}
		then $\Lip d \leq 2^{1/2}$, $\Lip g \leq 2^{1/2} ( 1 + L )$,
		and
		\begin{gather*}
			\eqLpnorm{\| V \| \restrict H \cap \cylinder{T}{\pp^\ast
			(c)}{\varrho}{h}}{q}{d} \\
			\leq
			\Gamma_{\eqref{item:lipschitz_approximation:poincare}}
			Q \big ( \eqLpnorm{\mathscr{L}^\vdim \restrict
			\cball{c}{\varrho} \cap X_1}{q}{g} + \mathscr{L}^\vdim
			( \cball{c}{\varrho} \without X_1)^{1/q+1/\vdim} \big)
		\end{gather*}
		whenever $1 \leq q \leq \infty$ where
		$\Gamma_{\eqref{item:lipschitz_approximation:poincare}}$ is a
		positive, finite number depending only on $\vdim$.
		\item \label{item:lipschitz_approximation:pde} If $a$,
		$\varrho$, $C_{a,\varrho}$, $D_{a,\varrho}$ are as in
		\eqref{item:lipschitz_approximation:measurability},
		\begin{gather*}
			\graph_Q f | \cball{\pp(a)}{\varrho} \subset
			\cylinder{T}{a}{\varrho}{\delta_4\varrho/2},
		\end{gather*}
		$g : \rel^\vdim \to \rel^\codim$, $\Lip g < \infty$, $g | X_1
		= \boldsymbol{\eta}_Q \circ f$, $\tau \in \Hom ( \rel^\vdim,
		\rel^\codim )$, $\theta \in \mathscr{D} ( \rel^\vdim,
		\rel^\codim)$, $\eta \in \mathscr{D}^0 ( \rel^\codim )$,
		\begin{gather*}
			\spt \theta \subset \oball{\pp(a)}{\varrho}, \qquad
			0 \leq \eta (y) \leq 1 \quad \text{for $y \in
			\rel^\codim$}, \\
			\spt \eta \subset \oball{\qq (a)}{\delta_4 \varrho},
			\quad \cball{\qq (a)}{\delta_4 \varrho/2 } \subset
			\Int ( \classification{\rel^\codim}{y}{\eta (y) = 1 }
			),
		\end{gather*}
		and $\Psi^\S$ denotes the nonparametric integrand associated
		to the area integrand $\Psi$, then
		\begin{gather*}
			\big | Q \tint{}{} \big < D \theta (x), D\Psi^\S_0 ( D
			g (x)) \big > \ud \mathscr{L}^\vdim x - ( \delta V ) (
			( \eta \circ \qq) \cdot ( \qq^\ast \circ \theta \circ
			\pp ) ) \big | \\
			\begin{aligned}
				& \leq \gamma_1 Q \vdim^{1/2} \Lip g
				{\textstyle\int_{C_{a,\varrho}}} | D \theta |
				\ud \mathscr{L}^\vdim \\
				& \phantom{\leq} \ + \gamma_2
				{\textstyle\int_{E_{a,\varrho} \without
				C_{a,\varrho}}} | D \theta (x) | | \ap A f (x)
				\aplus ( - \tau ) |^2 \ud \mathscr{L}^\vdim x
				\\
				& \phantom{\leq} \ + \vdim^{1/2}
				{\textstyle\int_{D_{a,\varrho}}} | D ( ( \eta
				\circ \qq ) \cdot ( \qq^\ast \circ \theta
				\circ \pp ) ) | \ud \| V \|
			\end{aligned}
		\end{gather*}
		where
		\begin{gather*}
			\gamma_1 = \sup \| D^2 \Psi_0^\S \| \lIm
			\cball{0}{\vdim^{1/2}\Lip g} \rIm, \\
			\gamma_2 = \Lip \big ( D^2 \Psi_0^\S|
			\cball{0}{\vdim^{1/2}(L+2\| \tau \|) } \big ), \\
			E_{a,\varrho} = \cball{\pp (a)}{\varrho } \cap
			\classification{X_1}{x}{\density^0 ( \| f (x) \|, g
			(x)) \neq Q }.
		\end{gather*}
	\end{enumerate}
\end{lemma}
\begin{proof} [Choice of constants]
	One can assume $2L \leq \delta_4$ and $\delta_5 \leq ( 2
	\isoperimetric{\vdim} \vdim)^{-\vdim}/\unitmeasure{\vdim}$ whenever
	$\vdim \in \nat$ with $\vdim < \adim$.
	
	Choose $0 < s_0 < 1$, $0 < s < 1$ close to $1$ satisfying
	\begin{gather*}
		( s_0^{-2} - 1 )^{1/2} \leq \delta_4 / 2, \quad ( s^{-2} - 1
		)^{1/2} \leq \min \{ \delta_4/4, L \}
	\end{gather*}
	and define $\varepsilon > 0$ so small that
	\begin{gather*}
		1 - \adim \varepsilon^2 \geq 1/2, \quad ( 1 - \adim
		\varepsilon^2) ( Q - 1/4 ) \geq Q - 1/2
	\end{gather*}
	and not larger than the infimum of the following numbers corresponding
	to $\vdim \in \nat$ with $\vdim < \adim$
	\begin{gather*}
		\varepsilon_{\ref{app:lemma:lipschitz_approximation}} ( \vdim,
		\adim, Q, L, M, \delta_1, \delta_2, \delta_3, \delta_4 ,
		\delta_5), \quad ( 2 \isoperimetric{\vdim} )^{-1}, \\
		\varepsilon_{\ref{app:lemma:multilayer_monotonicity_offset}} (
		\adim, Q+1, M, \inf \{ \delta_2 / 2, ( 2
		\isoperimetric{\vdim}\vdim )^{-\vdim}/\unitmeasure{\vdim} \},
		s ) \quad
		\varepsilon_{\ref{app:lemma:multilayer_monotonicity_offset}} (
		\adim, Q, M, 1/4, s ), \\
		\varepsilon_{\ref{app:lemma:inverse_multilayer_monotonicity}}
		( \vdim, \adim, 1, \delta_2, 0, s_0, M ).
	\end{gather*}
	
	Clearly, $\delta$ satisfies the same inequalities as
	$\varepsilon$ and one can assume $r = 1$.
\end{proof}
\begin{proof}
[Proof of
\eqref{item:lipschitz_approximation:yz}\,\eqref{item:lipschitz_approximation:ab}\,\eqref{item:lipschitz_approximation:lip}\,\eqref{item:lipschitz_approximation:misc}]
	By
	\ref{app:lemma:lipschitz_approximation}\,\eqref{app:item:lipschitz_approximation_2:def},
	\ref{remark:univ_meas}\,\eqref{item:univ_meas:preimage1} and
	\ref{lemma:univ_meas:coarea} the sets $X_1$ and $X_2$ are universally
	measurable. Hence the assertion follows from
	\ref{app:lemma:lipschitz_approximation}\,\eqref{app:item:lipschitz_approximation:N}\,\eqref{app:item:lipschitz_approximation_2:def}\,\eqref{app:item:lipschitz_approximation_2:lip}\,\eqref{app:item:lipschitz_approximation_2:misc}.
\end{proof}
\begin{proof} [Proof of \eqref{item:lipschitz_approximation:para}]
	Let $\eta = \inf \{ \delta_2/2, ( 2 \isoperimetric{\vdim} \vdim
	)^{-\vdim} / \unitmeasure{\vdim} \}$, consider $z \in A$ with
	$\density^\vdim ( \| V \|, z ) = Q$, $Z = A ( \pp (z) )$, note, using
	\eqref{item:lipschitz_approximation:ab}, that
	\begin{gather*}
		\classification{\oball{\xi-\pp^\ast ( \pp
		(z))}{1}}{\kappa}{|\project{T} ( \kappa - \xi ) > s | \kappa -
		\xi |} \subset \cylinder{T}{0}{1}{h}
	\end{gather*}
	for $\xi \in A ( \pp (z) )$ and apply
	\ref{app:lemma:multilayer_monotonicity_offset} with
	\begin{gather*}
		\text{$Q$, $\delta$, $d$, $r$, $t$, and $f$} \\
		\text{replaced by $Q+1$, $\eta$, $1$, $2$, $1$, and
		$\trans{-\pp^\ast (\pp(z))}|Z$}
	\end{gather*}
	to obtain $\sum_{\xi \in A ( \pp (z))} \density_\ast^\vdim ( \| V \|,
	\xi ) < Q + \eta$, hence \ref{app:lemma:good_point} implies
	\eqref{item:lipschitz_approximation:para}.
\end{proof}
\begin{proof} [Proof of \eqref{item:lipschitz_approximation:measurability}]
	Recalling \eqref{item:lipschitz_approximation:ab}, the set $\pp \lIm
	B_{a,\varrho} \rIm$ is universally measurable by
	\ref{remark:univ_meas}\,\eqref{item:univ_meas:projection}, hence
	$C_{a,\varrho}$, $D_{a,\varrho}$ are universally measurable sets by
	\eqref{item:lipschitz_approximation:yz} and
	\ref{remark:univ_meas}\,\eqref{item:univ_meas:borel_family}\,\eqref{item:univ_meas:preimage2}.
\end{proof}
\begin{proof} [Proof of \eqref{item:lipschitz_approximation:estimate_b}]
	Let $\nu$ denote the Radon measure characterised by
	\begin{gather*}
		\nu (Z) = {\textstyle\int_Z \lVert \bigwedge_\vdim ( \pp | S )
		\rVert \ud V (z,S)}
	\end{gather*}
	whenever $Z$ is a Borel subset of $U$, and note
	\begin{gather*}
		| \project{S} - \project{T} | \leq
		\varepsilon \quad \text{for $V$ almost all $(z,S) \in A \times
		\grass{\adim}{\vdim}$},
	\end{gather*}
	hence $1 - \lVert \bigwedge_\vdim ( \pp | S ) \rVert \leq 1 -
	\lVert \bigwedge_\vdim ( \project{T} | S ) \rVert^2 \leq \vdim
	\varepsilon^2$ for those $(z,S)$ by \ref{app:miniremark:tilt}.
	Therefore
	\begin{gather*}
		( 1 - \vdim \varepsilon^2 ) \, \| V \| \restrict A \leq \nu
		\restrict A.
	\end{gather*}
	This implies the \emph{coarea estimate}
	\begin{gather*}
		( 1 - \vdim \varepsilon^2 ) \, \| V \| \bigl (
		\cylinder{T}{a}{\varrho}{\delta_4 \varrho} \cap \pp^{-1} \lIm
		W \rIm \bigr ) \\
		\leq \| V \| \bigl ( B_{a,\varrho} \cap \pp^{-1} \lIm W\rIm \bigr
		) + Q \mathscr{L}^\vdim ( X_1 \cap W ) + ( Q-1 )
		\mathscr{L}^\vdim ( X_2 \cap W)
	\end{gather*}
	for every subset $W$ of $\rel^\vdim$; in fact the estimate holds for
	every Borel set by the coarea formula, see
	e.g.~\cite[3.2.22\,(3)]{MR41:1976} or \cite[12.7]{MR87a:49001}, and
	$\pp_\# ( \| V \| \restrict B_{a,\varrho} )$ is a Radon measure by
	\cite[2.2.17]{MR41:1976}. In particular, taking $W = \cball{\pp (a)}{
	\varrho }$ yields
	\begin{gather*}
		( 1 - \vdim \varepsilon^2 ) \| V \| ( \cylinder{T}{a}{ \varrho}{
		\delta_4 \varrho} ) \leq \| V \| ( B_{a,\varrho} ) + Q
		\unitmeasure{\vdim} \varrho^\vdim,
	\end{gather*}
	thus one can assume, since $8Q + 6 \leq
	\Gamma_{\eqref{item:lipschitz_approximation:estimate_b}}$, that
	\begin{gather*}
		\| V \| ( B_{a,\varrho} ) \leq {\textstyle\frac{1}{4}}
		\unitmeasure{\vdim} \varrho^\vdim.
	\end{gather*}
	
	Next, it will be shown that this assumption implies
	\begin{gather*}
		\mathscr{L}^\vdim ( X_1 \cap \cball{ \pp (a)}{ \varrho } ) >
		0;
	\end{gather*}
	in fact, using the coarea estimate with $W = \cball{ \pp (a)}{
	\varrho}$, one obtains
	\begin{gather*}
		\begin{aligned}
			& ( Q - 1/2 ) \unitmeasure{\vdim}
			\varrho^\vdim \\
			& \qquad \leq (1-\vdim\varepsilon^2) \| V \| ( \cylinder{T}{
			a}{ \varrho}{ \delta_4 \varrho} ) \\
			& \qquad \leq \| V \| ( B_{a,\varrho} ) + Q \mathscr{L}^\vdim
			( X_1 \cap \cball{ \pp (a)}{ \varrho } ) + ( Q - 1 )
			\mathscr{L}^\vdim ( X_2 \cap \cball{ \pp (a) }{
			\varrho}) \\
			& \qquad \leq ( Q - 1/2 ) \unitmeasure{\vdim} \varrho^\vdim +
			\mathscr{L}^\vdim ( X_1 \cap \cball{ \pp (a)}{ \varrho
			}) - {\textstyle\frac{1}{4}} \mathscr{L}^\vdim ( X_2
			\cap \cball{ \pp (a)}{ \varrho}),
		\end{aligned} \\
		\mathscr{L}^\vdim ( X_2 \cap \cball{ \pp (a)}{ \varrho } )
		\leq 4 \, \mathscr{L}^\vdim ( X_1 \cap \cball{ \pp (a)}{
		\varrho } ), \quad \mathscr{L}^\vdim ( X_1 \cap \cball{ \pp
		(a)}{ \varrho } ) > 0.
	\end{gather*}
	In order to estimate $\mathscr{L}^\vdim (X_2 \cap
	\cball{\pp(a)}{\varrho})$, the following assertion will be proven.
	\emph{If $x \in X_2 \cap \cball{ \pp (a)}{ \varrho }$ and
	$\density^\vdim ( \mathscr{L}^\vdim \restrict \rel^\vdim \without X_2
	, x ) = 0$, then there exist $\zeta \in \rel^\vdim$ and $0 < t <
	\infty$ with
	\begin{gather*}
		x \in \cball{\zeta}{t} \subset \cball{ \pp (a)}{ \varrho },
		\quad \measureball{\mathscr{L}^\vdim}{\cball{ \zeta}{5t }}
		\leq 6 \cdot 5^\vdim \, \| V \| \big ( B_{a,\varrho} \cap
		\pp^{-1} \lIm \cball{\zeta}{t} \rIm \big ).
	\end{gather*}}
	Since $\mathscr{L}^\vdim ( X_1 \cap \cball{ \pp (a)}{ \varrho } ) > 0$,
	some element $\cball{\zeta}{t}$ of the family of balls
	\begin{gather*}
		\{ \cball{ ( 1-\theta) x + \theta \pp (a) }{ \theta
		\varrho } \with 0 < \theta \leq 1 \}
	\end{gather*}
	will satisfy
	\begin{gather*}
		x \in \cball{\zeta}{t} \subset \cball{ \pp (a)}{ \varrho
		}, \quad 0 < \mathscr{L}^\vdim ( X_1 \cap \cball{\zeta}{t} )
		\leq {\textstyle\frac{1}{2}} \mathscr{L}^\vdim ( X_2 \cap
		\cball{\zeta}{t} ).
	\end{gather*}
	Hence there exists $\eta \in X_1 \cap \oball{ \zeta}{ t }$. Noting for
	$\xi \in A(\eta)$ with $\density^\vdim ( \| V \|, \xi ) > 0$
	\begin{gather*}
		\oball{\boldsymbol{\tau}_{\pp^\ast ( \zeta - \eta )}
		(\xi)}{t} \subset \pp^{-1} \lIm \cball{\zeta}{t} \rIm, \quad
		\xi \in \spt f ( \eta ) \subset
		\cball{\qq(a)}{\delta_4\varrho/2}, \\
		( s^{-2} - 1)^{1/2} | \pp ( \kappa - \xi ) | \leq \delta_4 t /
		2 \leq \delta_4 \varrho / 2 \quad \text{for $\kappa \in
		\pp^{-1} \lIm \cball{\zeta}{t} \rIm$},
	\end{gather*}
	the inclusion
	\begin{gather*}
		\classification{\oball{\boldsymbol{\tau}_{\pp^\ast
		(\zeta-\eta)} ( \xi)}{t}}{\kappa}{| \pp ( \kappa - \xi) | > s
		| \kappa - \xi | } \subset \cylinder{T}{ a}{ \varrho}{
		\delta_4 \varrho} \cap \pp^{-1} \lIm \cball{\zeta}{t} \rIm
	\end{gather*}
	is valid for such $\xi$ and
	\ref{app:lemma:multilayer_monotonicity_offset} can be applied with
	\begin{gather*}
		\text{$\delta$, $Z$, $d$, $r$, and $f$ replaced by} \\
		\text{$1/4$, $\classification{A (\eta)}{\xi}{\density^\vdim (
		\| V \| , \xi ) > 0 }$, $t$, $2$}, \\
		\text{and $\boldsymbol{\tau}_{\pp^\ast ( \zeta - \eta)} |
		\classification{A (\eta)}{\xi}{\density^\vdim ( \| V \|, \xi )
		> 0 }$}
	\end{gather*}
	to obtain
	\begin{gather*}
		( Q - 1/4 ) \unitmeasure{\vdim} t^\vdim \leq \| V \| \big (
		\cylinder{T}{ a}{ \varrho}{ \delta_4 \varrho} \cap \pp^{-1}
		\lIm \cball{\zeta}{t} \rIm \big ).
	\end{gather*}
	The coarea estimate with $W = \cball{\zeta}{t}$ now implies
	\begin{align*}
		& ( Q - 1/2 ) \unitmeasure{\vdim} t^\vdim - \| V \| (
		B_{a,\varrho} \cap \pp^{-1} \lIm \cball{\zeta}{t} \rIm ) \\
		& \qquad \leq Q \mathscr{L}^\vdim ( X_1 \cap \cball{\zeta}{t}
		) + ( Q-1 ) \mathscr{L}^\vdim ( X_2 \cap \cball{\zeta}{t} ) \\
		& \qquad = ( Q - 1/2 ) \unitmeasure{\vdim} t^\vdim +
		{\textstyle\frac{1}{2}} \mathscr{L}^\vdim ( X_1 \cap
		\cball{\zeta}{t} ) - \textstyle{\frac{1}{2}} \mathscr{L}^\vdim
		( X_2 \cap \cball{\zeta}{t} ),
	\end{align*}
	hence, recalling $\mathscr{L}^\vdim ( X_1 \cap \cball{\zeta}{t} ) \leq
	{\textstyle\frac{1}{2}} \mathscr{L}^\vdim ( X_2 \cap \cball{\zeta}{t}
	)$,
	\begin{gather*}
		{\textstyle\frac{2}{3}} \mathscr{L}^\vdim ( \cball{\zeta}{t}
		) \leq \mathscr{L}^\vdim ( X_2 \cap \cball{\zeta}{t} ) \leq 4
		\, \| V \| \big ( B_{a,\varrho} \cap \pp^{-1} \lIm
		\cball{\zeta}{t} \rIm \big)
	\end{gather*}
	and the assertion follows.
	
	The assumption of the last assertion is satisfied for
	$\mathscr{L}^\vdim$ almost all $x \in X_2 \cap \cball{ \pp (a)}{
	\varrho}$ by \cite[2.9.11]{MR41:1976} or \cite[3.65]{MR2003a:49002}
	and Vitali's covering theorem, see e.g.~\cite[2.8.5]{MR41:1976} or
	\cite[3.3]{MR87a:49001}, implies
	\begin{gather*}
		\mathscr{L}^\vdim ( X_2 \cap \cball{ \pp (a)}{ \varrho } ) \leq
		6 \cdot 5^\vdim \| V \| ( B_{a,\varrho} ).
	\end{gather*}
	Clearly,
	\begin{gather*}
		\mathscr{L}^\vdim ( \pp \lIm B_{a,\varrho} \rIm ) \leq
		\mathscr{H}^\vdim ( B_{a,\varrho} ) \leq \| V \| ( B_{a,\varrho}
		).
	\end{gather*}
	Since $C_{a,\varrho} \without N \subset ( X_2 \cap \cball{ \pp (a)}{
	\varrho } ) \cup \pp \lIm B_{a,\varrho} \rIm$, it follows
	\begin{gather*}
		\mathscr{L}^\vdim ( C_{a,\varrho} ) \leq ( 1 + 6 \cdot 5^\vdim
		) \| V \| ( B_{a,\varrho} ).
	\end{gather*}
	Finally, applying the coarea estimate with $W=C_{a,\varrho}$ yields
	\begin{align*}
		( 1-\vdim\varepsilon^2) \| V \| ( D_{a,\varrho} ) & \leq \| V
		\| ( B_{a,\varrho} ) + Q \mathscr{L}^\vdim ( C_{a,\varrho} )
		\\
		& \leq ( 1 + Q + 6Q \cdot 5^\vdim ) \| V \| ( B_{a,\varrho} )
	\end{align*}
	and the conclusion follows.
\end{proof}
\begin{proof} [Proof of \eqref{item:lipschitz_approximation:poincare}]
	Choose $0 < \lambda \leq 1$ such that
	\begin{gather*}
		\lambda \leq \inf \{
		\lambda_{\ref{app:lemma:lipschitz_approximation}\,\eqref{app:item:lipschitz_approximation_2:lip_related}}
		( \vdim, \delta_2, \delta_4 ) ,
		\lambda_{\ref{app:lemma:inverse_multilayer_monotonicity}} (
		\vdim, \delta_2, s_0 ) / 2 \}
	\end{gather*}
	and define
	$\varepsilon_{\eqref{item:lipschitz_approximation:poincare}} = (1/2) (
	\lambda / 6 )^\vdim \leq 1$.
	
	Suppose $z_1,z_2 \in \cylinder{T}{\pp^\ast (c)}{\varrho}{h}$ and $\xi_1
	\in P$ with $\pp ( \xi_1 ) = \pp ( z_1 )$. Then there exists $\xi_2 \in
	P$ such that $\pp ( \xi_2 ) = z_2$ and $| \qq ( \xi_1 - \xi_2 ) | \leq
	| \pp ( \xi_1 - \xi_2 ) |$, hence
	\begin{align*}
		| \qq ( \xi_2 - z_2 ) | & \leq | \qq ( \xi_2 - \xi_1 ) | + | \qq
		( \xi_1 - z_1 ) | + | \qq ( z_1 - z_2 ) | \\
		& \leq 2^{1/2} | z_1 - z_2 | + | \qq ( \xi_1 - z_1 ) |
	\end{align*}
	and $\Lip d \leq 2^{1/2}$.

	Suppose $x_1, x_2 \in X_1 \cap \cball{c}{\varrho}$, $y_1 \in \spt f
	(x_1 )$. Then there exists $y_2 \in \spt f (x_2)$ with $|y_1-y_2|
	\leq L |x_1-x_2|$, hence
	\begin{gather*}
		d (x_1,y_1) \leq 2^{1/2} | (x_1,y_1) - (x_2,y_2) | + d
		(x_2,y_2) \leq 2^{1/2} ( 1 + L ) | x_1 - x_2 | + g ( x_2 )
	\end{gather*}
	and $\Lip g \leq 2^{1/2} (1+L)$.

	First, \emph{the case $q < \infty$} will be treated. Note $A \cap \spt
	\| V \| \subset H$ and $H \cap \pp^{-1} \lIm X_1 \rIm = \graph_Q f$
	by
	\ref{app:lemma:lipschitz_approximation}\,\eqref{app:item:lipschitz_approximation_2:lip_related},
	let $\psi = \| V \| \restrict H \cap \cylinder{T}{\pp^\ast
	(c)}{\varrho}{h}$ and recall
	\begin{gather*}
		( \pp_\# \psi ) \restrict X_1 \leq 2 ( \pp_\# ( \nu \restrict
		H ) ) \restrict X_1 \leq 2 Q \mathscr{L}^\adim \restrict X_1
	\end{gather*}
	with $\nu$ as in the proof of
	\eqref{item:lipschitz_approximation:estimate_b}. Using
	\begin{gather*}
		\classification{H \cap \cylinder{T}{\pp^\ast (c)}{\varrho}{h}
		\cap \pp^{-1} \lIm X_1 \rIm}{z}{d(z) > \gamma} \\
		\subset H \cap \pp^{-1} \lIm \classification{X_1 \cap
		\cball{c}{\varrho}}{x}{g(x) > \gamma} \rIm
	\end{gather*}
	for $0 < \gamma < \infty$, one infers
	\begin{gather*}
		\eqLpnorm{\psi \restrict \pp^{-1} \lIm X_1 \rIm}{q}{d} \leq 2Q
		\eqLpnorm{\mathscr{L}^\vdim \restrict X_1 \cap
		\cball{c}{\varrho}}{q}{g}.
	\end{gather*}
	Therefore it remains to estimate $\eqLpnorm{\psi \restrict U \without
	\pp^{-1} \lIm X_1 \rIm }{q}{d}$.

	Whenever $x \in \cball{c}{\varrho} \without \Clos{X_1}$ there exist
	$\zeta \in \rel^\vdim$, $0 < t \leq (2
	\varepsilon_{\eqref{item:lipschitz_approximation:poincare}})^{1/\vdim}
	\varrho = \lambda \varrho/ 6$ such that
	\begin{gather*}
		x \in \cball{\zeta}{t} \subset \cball{c}{\varrho}, \quad
		\mathscr{L}^\vdim ( \cball{\zeta}{t} \cap X_1 ) =
		\mathscr{L}^\vdim ( \cball{\zeta}{t} \without X_1 )
	\end{gather*}
	as may be verified by consideration of the family of closed balls
	\begin{gather*}
		\{ \cball{\theta c + (1-\theta) x}{\theta \varrho } \with 0 <
		\theta \leq (2
		\varepsilon_{\eqref{item:lipschitz_approximation:poincare}})^{1/\vdim}
		\}.
	\end{gather*}
	Therefore Vitali's covering theorem, see e.g.~\cite[2.8.5]{MR41:1976}
	or \cite[3.3]{MR87a:49001}, yields a countable set $I$ and $\zeta_i
	\in \rel^\vdim$, $0 < t_i \leq \lambda\varrho/6$ and $x_i \in X_1 \cap
	\cball{\zeta_i}{t_i}$ for each $i \in I$ such that
	\begin{gather*}
		\cball{\zeta_i}{t_i} \subset \cball{c}{\varrho}, \quad
		\mathscr{L}^\vdim ( \cball{\zeta_i}{t_i} \cap X_1 ) =
		\mathscr{L}^\vdim ( \cball{\zeta_i}{t_i} \without X_1 ), \\
		\cball{\zeta_i}{t_i} \cap \cball{\zeta_j}{t_j} = \emptyset
		\quad \text{whenever $i,j \in I$ with $i \neq j$}, \\
		\cball{c}{\varrho} \without \Clos X_1 \subset
		{\textstyle\bigcup} \{ E_i \with i \in I \} \subset
		\cball{c}{\varrho}
	\end{gather*}
	where $E_i = \cball{\zeta_i}{5t_i} \cap \cball{c}{\varrho}$ for $i \in
	I$. Let
	\begin{gather*}
		h_i = g (x_i), \quad Z_i = \classification{A
		(x_i)}{\xi}{\density^\vdim ( \| V \|, \xi ) \in \nat}
	\end{gather*}
	for $i \in I$, $J = \classification{I}{i}{h_i \geq 24t_i}$, and $K = I
	\without J$.

	In view of \ref{app:lemma:lipschitz_approximation}\,\eqref{app:item:lipschitz_approximation_2:boundary} there holds
	\begin{gather*}
		\begin{aligned}
			& \eqLpnorm{\psi \restrict U \without \pp^{-1} \lIm
			X_1 \rIm}{q}{d} \\
			& \qquad \leq \eqLpnorm{\psi \restrict \pp^{-1} \lIm
			{\textstyle\bigcup} \{ E_j \with j \in J \}
			\rIm}{q}{d} + \eqLpnorm{\psi \restrict \pp^{-1} \lIm
			{\textstyle\bigcup} \{ E_k \with k \in K \}
			\rIm}{q}{d}.
		\end{aligned}
	\end{gather*}
	In order to estimate the terms on the right hand side, two
	observations will be useful. Firstly, \emph{if $i \in I$, $z \in H \cap
	\cylinder{T}{\pp^\ast(c)}{\varrho}{h} \cap \pp^{-1} \lIm
	E_i \rIm$, then
	\begin{gather*}
		d (z) \leq 24 t_i + h_i;
	\end{gather*}}
	in fact $| \pp (z) - x_i | \leq 6t_i \leq \lambda \varrho \leq
	\lambda$ and
	\ref{app:lemma:lipschitz_approximation}\,\eqref{app:item:lipschitz_approximation_2:lip_related}
	yields a point $\xi \in Z_i$ with $|\qq (z-\xi) | \leq L | \pp (
	z-\xi) |$, hence
	\begin{gather*}
		| z - \xi | \leq ( 1 + L ) | \pp ( z-\xi ) | = ( 1 + L ) |
		\pp ( z ) - x_i | \leq 12 t_i, \\
		d(z) \leq 2^{1/2} |z-\xi| + d ( \xi ) \leq 24 t_i + h_i.
	\end{gather*}
	Moreover, since
	\begin{gather*}
		H \cap \cylinder{T}{\pp^\ast(c)}{\varrho}{h} \cap \pp^{-1}
		\lIm E_i \rIm \subset {\textstyle\bigcup} \{
		\cball{\xi}{12t_i} \with \xi \in Z_i \},
	\end{gather*}
	one may apply \ref{app:lemma:inverse_multilayer_monotonicity}\,\eqref{app:item:inverse_multilayer_monotonicity:upper_bound}, verifying
	\begin{gather*}
		\classification{\oball{z-\pp^\ast(x_i)}{1}}{\xi}{|\pp(\xi-z)|
		> s_0 | \xi - z |} \subset \cylinder{T}{0}{1}{h}
	\end{gather*}
	whenever $z \in A(x_i)$ with the help of
	\eqref{item:lipschitz_approximation:ab}, with
	\begin{gather*}
		\text{$\delta_1$, $s$, $\lambda$, $X$, $d$, $r$, $t$, $\zeta$,
		$\mu$, and $\tau$ replaced by} \\
		\text{$1$, $0$,
		$\lambda_{\ref{app:lemma:inverse_multilayer_monotonicity}\,\eqref{app:item:inverse_multilayer_monotonicity:upper_bound}}
		(\vdim,\delta_2,s_0)$, $Z_i$, $1$, $2$, $1$, $-\pp^\ast(x_i)$,
		$\| V \|$, and $12t_i$}
	\end{gather*}
	to obtain the second observation, \emph{namely
	\begin{gather*}
		\psi \big ( \pp^{-1} \lIm E_i \rIm \big ) \leq ( Q + 1 )
		\unitmeasure{\vdim} ( 12 t_i )^\vdim \quad \text{whenever $i
		\in I$}.
	\end{gather*}}

	Now, the first term will be estimated. Note, if $j \in J$, then
	\begin{gather*}
		d(z) \leq 2 h_j \quad \text{whenever $z \in H \cap
		\cylinder{T}{\pp^\ast (c)}{\varrho}{h} \cap \pp^{-1}
		\lIm E_j \rIm$}, \\
		2 h_j \leq 3 g(x) \quad \text{whenever $x \in X_1 \cap
		\cball{\zeta_j}{t_j}$},
	\end{gather*}
	because
	\begin{gather*}
		g(x) \geq g(x_j) - 4 |x_j-x| \geq h_j - 8 t_j \geq 2 h_j/3.
	\end{gather*}
	Using this fact and the preceding observations, one estimates with $J
	( \gamma ) = \classification{J}{j}{2h_j > \gamma}$ for $0 < \gamma <
	\infty$
	\begin{gather*}
		\begin{aligned}
			& \psi \big ( \classification{\pp^{-1} \lIm
			{\textstyle\bigcup} \{ E_j \with j \in J \}
			\rIm}{z}{d(z) > \gamma} \big ) \leq {\textstyle\sum_{j
			\in J(\gamma)}} \psi \big ( \pp^{-1} \lIm E_j \rIm
			\big ) \\
			& \qquad \leq {\textstyle\sum_{j \in J(\gamma)}} ( Q +
			1 ) \unitmeasure{\vdim} ( 12t_j )^\vdim \\
			& \qquad \leq ( Q + 1 ) (12)^\vdim \mathscr{L}^\vdim
			\big ( {\textstyle\bigcup} \{ \cball{\zeta_j}{t_j}
			\with j \in J ( \gamma ) \} \big ) \\
			& \qquad \leq 2 ( Q+1 ) ( 12 )^\vdim \mathscr{L}^\vdim
			\big ( {\textstyle\bigcup} \{ X_1 \cap
			\cball{\zeta_j}{t_j} \with j \in J ( \gamma ) \} \big
			) \\
			& \qquad \leq 2 ( Q + 1 ) ( 12 )^\vdim
			\mathscr{L}^\vdim ( \classification{X_1 \cap
			\cball{c}{\varrho}}{x}{g(x) > \gamma / 3 },
		\end{aligned}
	\end{gather*}
	hence
	\begin{gather*}
		\eqLpnorm{\psi \restrict \pp^{-1} \lIm {\textstyle\bigcup} \{
		E_j \with j \in J \} \rIm}{q}{d} \leq Q ( 12 )^{\vdim + 1}
		\eqLpnorm{\mathscr{L}^\vdim \restrict X_1 \cap
		\cball{c}{\varrho} }{q}{g}.
	\end{gather*}

	To estimate the second term, one notes
	\begin{gather*}
		d(z) < 48 t_k \quad \text{whenever $k \in K$, $z \in H \cap
		\cylinder{T}{\pp^\ast (c)}{\varrho}{h} \cap \pp^{-1} \lIm E_k
		\rIm$}.
	\end{gather*}
	Therefore one estimates with $K( \gamma ) =
	\classification{K}{k}{48t_k > \gamma}$ for $0 < \gamma < \infty$ and
	$u : \rel^\vdim \to \rel$ defined by $u = \sum_{i \in I}t_i b_i$ where
	$b_i$ is the characteristic function of $\cball{\zeta_i}{t_i}$
	\begin{gather*}
		\begin{aligned}
			& \psi \big ( \classification{\pp^{-1} \lIm
			{\textstyle\bigcup} \{ E_k \with k \in K \}
			\rIm}{z}{d(z) > \gamma} \big ) \leq {\textstyle\sum_{k
			\in K( \gamma )}} \psi \big ( \pp^{-1} \lIm E_k \rIm
			\big ) \\
			& \qquad \leq {\textstyle\sum_{k \in K ( \gamma)}} ( Q
			+ 1 ) \unitmeasure{\vdim} ( 12 t_k )^\vdim \\
			& \qquad \leq ( Q + 1 ) ( 12)^\vdim \mathscr{L}^\vdim
			( {\textstyle\bigcup} \{ \cball{\zeta_k}{t_k} \with k
			\in K( \gamma ) \} ) \\
			& \qquad \leq ( Q + 1 ) ( 12 )^\vdim \mathscr{L}^\vdim
			( \classification{\rel^\vdim}{x}{u(x) > \gamma/(48)}
			),
		\end{aligned}
	\end{gather*}
	hence
	\begin{gather*}
		\eqLpnorm{\psi \restrict \pp^{-1} \lIm {\textstyle\bigcup} \{
		E_k \with k \in K \} \rIm}{q}{d} \leq Q ( 12 )^{\vdim
		+ 2} \Lpnorm{\mathscr{L}^\vdim}{q}{u}.
	\end{gather*}

	Combining these two estimates and
	\begin{gather*}
		\mathscr{L}^\vdim ( {\textstyle\bigcup} \{
		\cball{\zeta_i}{t_i} \with i \in I \} ) \leq 2 \mathscr{L}^\vdim
		( \cball{c}{\varrho} \without X_1 ), \\
		\begin{aligned}
			{\textstyle\int} |u|^q \ud \mathscr{L}^\vdim & =
			\unitmeasure{\vdim}^{-q/\vdim} {\textstyle\sum_{i \in
			I}} \mathscr{L}^\vdim ( \cball{\zeta_i}{t_i}
			)^{1+q/\vdim} \\
			& \leq \unitmeasure{\vdim}^{-q/\vdim} \big (
			{\textstyle\sum_{i \in I}} \mathscr{L}^\vdim (
			\cball{\zeta_i}{t_i} ) \big)^{1+q/\vdim},
		\end{aligned} \\
		\eqLpnorm{\mathscr{L}^\vdim}{q}{u} \leq 4
		\unitmeasure{\vdim}^{-1/\vdim} \mathscr{L}^\vdim (
		\cball{c}{\varrho} \without X_1 )^{1/q+1/\vdim},
	\end{gather*}
	one obtains the conclusion for $q < \infty$.

	\emph{The case $q = \infty$} follows by taking the limit $q \to
	\infty$ with the help of \cite[2.4.17]{MR41:1976}.
\end{proof}
\begin{proof} [Proof of \eqref{item:lipschitz_approximation:pde}]
	Let $I$, $f_i$ be associated to $f$ as in \ref{miniremark:f_i}, and
	define $C_i = \dmn f_i$ for $i \in I$ and $G = \graph_Q f$. Note
	\begin{gather*}
		G \cap \pp^{-1} \lIm \cball{ \pp (a)}{ \varrho } \without
		C_{a,\varrho} \rIm = G \cap \cylinder{T}{a}{\varrho}{ \delta_4
		\varrho/2} \without \pp^{-1} \lIm C_{a,\varrho} \rIm, \\
		\pp \lIm B_{a,\varrho} \rIm \subset C_{a,\varrho}, \quad \| V
		\| \big ( \cylinder{T} {a}{\varrho}{\delta_4 \varrho} \without
		(  G \cup \pp^{-1} \lIm C_{a,\varrho} \rIm ) \big ) = 0.
	\end{gather*}
	Therefore one computes using \ref{miniremark:first_variation} and
	recalling that $C_{a,\varrho}$, $D_{a,\varrho}$, and, by
	\ref{remark:univ_meas}\,\eqref{item:univ_meas:preimage2}, also
	$\pp^{-1} \lIm C_{a,\varrho} \rIm$ are universally measurable
	\begin{align*}
		& \phantom{=} \ \sum_{i \in I} {\textstyle\int_{C_i \cap
		\cball{ \pp (a)}{ \varrho} \without C_{a,\varrho}}} \leftB D
		\theta (x), D\Psi^\S_0 ( \ap Df_i (x) ) \rightB \ud
		\mathscr{L}^\vdim x \\
		& = \delta \big ( V \restrict (G \cap \pp^{-1} \lIm \cball{\pp
		(a)}{ \varrho } \without C_{a,\varrho} \rIm ) \times
		\grass{\adim}{\vdim} \big ) ( \qq^\ast \circ \theta \circ \pp)
		\\
		& = \delta \big ( V \restrict (G \cap \cylinder{T}{ a}{
		\varrho}{ \delta_4 \varrho/2} \without \pp^{-1} \lIm
		C_{a,\varrho} \rIm ) \times \grass{\adim}{\vdim} \big ) ( (
		\eta \circ \qq ) \cdot ( \qq^\ast \circ \theta \circ \pp )) \\
		& = \delta \big ( V \restrict ( \cylinder{T} { a} {\varrho}{
		\delta_4 \varrho} \without \pp^{-1} \lIm C_{a,\varrho} \rIm )
		\times \grass{\adim}{\vdim} \big) ( ( \eta \circ \qq ) \cdot (
		\qq^\ast \circ \theta \circ \pp ) ) \\
		& = ( \delta V ) ( ( \eta \circ \qq ) \cdot ( \qq^\ast \circ
		\theta \circ \pp ) ) - \delta ( V \restrict (
		D_{a,\varrho} \times \grass{\adim}{\vdim} )) ( ( \eta \circ
		\qq ) \cdot ( \qq^\ast \circ \theta \circ \pp ) ),
	\end{align*}
	hence
	\begin{gather*}
		Q \tint{}{} \leftB D \theta (x), D\Psi^\S_0 ( D g (x) )
		\rightB \ud \mathscr{L}^\vdim x - ( \delta V ) ( ( \eta \circ
		\qq ) \cdot ( \qq^\ast \circ \theta \circ \pp ) ) \\
		\begin{aligned}
			& = Q {\textstyle\int_{C_{a,\varrho}}} \leftB D
			\theta (x), D\Psi_0^\S ( D g (x) ) \rightB \ud
			\mathscr{L}^\vdim x \\
			& \phantom{=}\ + Q \Big ( {\textstyle\int_{\cball{
			\pp (a)}{ \varrho} \without C_{a,\varrho} }} \leftB
			D \theta (x), D\Psi_0^\S ( Dg (x) ) \rightB \ud
			\mathscr{L}^\vdim x \\
			& \phantom{=\ + Q \Big (} \ - \frac{1}{Q} \sum_{i \in
			I} {\textstyle\int_{C_i \cap \cball{ \pp (a)}{
			\varrho } \without C_{a,\varrho}}} \leftB D \theta
			(x), D\Psi_0^\S ( \ap D f_i (x) ) \rightB \ud
			\mathscr{L}^\vdim x \Big ) \\
			& \phantom{=} \ - \delta ( V \restrict (
			D_{a,\varrho} \times \grass{\adim}{\vdim})) ( ( \eta
			\circ \qq) \cdot ( \qq^\ast \circ \theta \circ \pp )
			).
		\end{aligned}
	\end{gather*}
	The first summand may be estimated using
	\begin{gather*}
		D\Psi_0^\S(0)=0, \quad \| D\Psi_0^\S ( \alpha ) \| \leq
		\gamma_1 | \alpha | \leq \gamma_1 \vdim^{1/2} \Lip g
	\end{gather*}
	for $\alpha \in \Hom ( \rel^\vdim, \rel^\codim )$ with $\|
	\alpha \| \leq \Lip g$. The second summand can be treated noting
	\begin{gather*}
		Dg (x) = \frac{1}{Q} \sum_{i \in I(x)} \ap D f_i (x) \quad
		\text{where $I(x) = \classification{I}{i}{x \in \dmn \ap Df_i
		}$}
	\end{gather*}
	for $\mathscr{L}^\vdim$ almost all $x \in \cball{ \pp (a)}{ \varrho }
	\without C_{a,\varrho}$ and applying \ref{lemma:quasi_linear} with
	\begin{gather*}
		\text{$X$, $Y$, $f$, $a$, $r$, and $\{ x_1, \ldots, x_Q \}$}
		\\
		\text{replaced by $\Hom ( \rel^\vdim, \rel^\codim )$,
		$\Hom ( \Hom ( \rel^\vdim, \rel^\codim ), \rel)$,
		$D\Psi_0^\S$, $\tau$,} \\
		\text{$Q^{-1/2} | \ap Af (x) \aplus ( - \tau ) |$, and $\{ \ap
		Df_i (x) \with i \in I(x) \}$}
	\end{gather*}
	for $\mathscr{L}^\vdim$ almost all $x \in E_{a,\varrho} \without
	C_{a,\varrho}$. Finally, the third summand is estimated by use of
	\begin{gather*}
		| \project{S} \bullet \beta | \leq \vdim^{1/2} | \beta | \quad
		\text{for $S \in \grass{\adim}{\vdim}$, $\beta \in \Hom (
		\rel^\adim, \rel^\adim )$}. \qedhere
	\end{gather*}
\end{proof}
\begin{remark}
	If $a$ and $\varrho$ are as in
	\eqref{item:lipschitz_approximation:measurability}, $a \in A$,
	$\density^\vdim ( \| V \| , a ) = Q$, $0 < s < 1$, $(s^{-2}-1)^{1/2}
	\leq \delta_4$, $\delta \leq
	\varepsilon_{\ref{app:lemma:multilayer_monotonicity_offset}} ( \adim, Q,
	M, 1/4, s )$, then
	\begin{gather*}
		\classification{\oball{a}{\varrho}}{\xi}{| \pp (\xi-a) | > s |
		\xi - a|} \subset \cylinder{T}{a}{\varrho}{\delta_4\varrho}
	\end{gather*}
	and \ref{app:lemma:multilayer_monotonicity_offset} applied with
	\begin{gather*}
		\text{$\delta$, $Z$, $d$, $r$, $t$, and $f$ replaced by} \\
		\text{$1/4$, $\{a\}$, $0$, $2$, $\varrho$, and $\id{\{a\}}$}
	\end{gather*}
	yields
	\begin{gather*}
		\| V \| ( \cylinder{T}{a}{\varrho}{\delta_4\varrho} ) \geq ( Q
		- 1/4 ) \unitmeasure{\vdim} \varrho^\vdim.
	\end{gather*}
	Moreover, if additionally $L \leq \delta_4/2$ then
	\eqref{item:lipschitz_approximation:para} implies $a \in \graph_Q f$
	and
	\begin{gather*}
		\graph_Q f | \cball{\pp (a)}{\varrho} \subset
		\cylinder{T}{a}{\varrho}{\delta_4 \varrho/2}.
	\end{gather*}
\end{remark}
\section{An interpolation inequality} \label{sec:interpolation}
In this section an interpolation inequality for weakly differentiable
functions defined in a ball $\oball{a}{r}$ with $a \in \rel^\vdim$, $0 < r <
\infty$ with values in $\rel^\codim$ is proven (see
\ref{lemma:function_with_holes}) which states that the Lebesgue seminorm of a
function can be controlled by a small multiple of a suitable Lebesgue seminorm
of its weak derivative and a large multiple of the $\Lp{1} ( \mathscr{L}^\vdim
\restrict A , \rel^\codim)$ seminorm of the function where $A$ is subset of
$\oball{a}{r}$ which is large in $\mathscr{L}^\vdim$ measure. The possibility
to neglect a set of small $\mathscr{L}^\vdim$ measure will be important in
Section \ref{sec:iteration}.  The proof is accomplished following essentially
the usual lines (see e.g.~\cite[Theorem 7.27]{MR1814364}). The case of
Lipschitzian functions with values in $\qspace_Q ( \rel^\codim )$ then is a
simple consequence of Almgren's bi-Lipschitzian embedding of $\qspace_Q (
\rel^\codim )$ into $\rel^{PQ}$ for some $P$, see
\ref{lemma:Q_function_with_holes}.
\begin{lemma} \label{lemma:conv_embedding}
	Suppose $\vdim, \adim \in \nat$, $1 \leq \zeta \leq \vdim < \adim$,
	either $\zeta = \vdim = 1$ or $\zeta < \vdim$, $q = \infty$ if $\vdim
	= 1$, $q = \vdim \zeta / (\vdim-\zeta)$ if $\vdim > 1$, $U$ is an
	open, bounded, convex subset of $\rel^\vdim$, $A$ is an
	$\mathscr{L}^\vdim$ measurable subset of $U$ with $\mathscr{L}^\vdim
	(A) > 0$, $u \in \Sob{}{1}{1} (U,\rel^\codim)$ and $h = \fint_A u \ud
	\mathscr{L}^\vdim$.

	Then
	\begin{gather*}
		\norm{u-h}{q}{U} \leq \Gamma \frac{( \diam
		U)^\vdim}{\mathscr{L}^\vdim (A)} \norm{\weakD u}{\zeta}{U}
	\end{gather*}
	where $\Gamma$ is a positive, finite number depending only on $\vdim$
	and $\zeta$.
\end{lemma}
\begin{proof}
	If $\zeta = \vdim = 1$ then $u$ is $\mathscr{L}^1 \restrict
	\oball{a}{r}$ almost equal to an absolutely continuous function by
	\cite[4.5.9\,(30), 4.5.16]{MR41:1976} and the assertion follows from
	\cite[2.9.20]{MR41:1976}; alternately one may use
	\cite[p.~139]{MR2003a:49002}.

	If $\zeta < \vdim$ this fact can be obtained by combining the method
	of \cite[Lemma 7.16]{MR1814364} with estimates for convolutions, see
	e.g.~O'Neil \cite{MR0146673}.
\end{proof}
\begin{miniremark} \label{miniremark:ball_contained}
	Suppose $a, x \in \rel^\vdim$, $0 < \varrho \leq 2r < \infty$, $x \in
	\oball{a}{r}$ and $b = a$ if $|x-a| < \varrho/2$ and $b = x + (
	\varrho / 2 ) (a-x)/|a-x|$ else. Then one readily verifies
	$\oball{b}{\varrho/2} \subset \oball{a}{r} \cap \oball{x}{\varrho}$.
\end{miniremark}
\begin{lemma} \label{lemma:function_with_holes}
	Suppose $\vdim, \adim \in \nat$, $1 \leq \zeta \leq \vdim < \adim$,
	either $\zeta = \vdim = 1$ or $\zeta < \vdim$, $q = \infty$ if $\vdim
	= 1$, $q = \vdim \zeta / ( \vdim - \zeta )$ if $\vdim > 1$, $1 \leq
	\xi \leq q$, $\zeta \leq s \leq q$, $0 < \lambda < \infty$, $a \in
	\rel^\vdim$, $0 < r < \infty$, $u \in \Sob{}{1}{1} ( \oball{a}{r},
	\rel^\codim )$, $A$ is an $\mathscr{L}^\vdim$ measurable subset of
	$\oball{a}{r}$, and $\mathscr{L}^\vdim ( \oball{a}{r} \without A
	) \leq \lambda \leq (1/2) \unitmeasure{\vdim} r^\vdim$.

	Then
	\begin{gather*}
		\norm{u}{q}{a,r} \leq \Gamma \lambda^{1/\zeta-1/s}
		\norm{\weakD u}{s}{a,r} + 2^{5\vdim+2} \lambda^{1/q-1/\xi}
		\norm{u}{\xi}{A}
	\end{gather*}
	where $\Gamma$ is a positive, finite number depending only on $\vdim$
	and $\zeta$.
\end{lemma}
\begin{proof}
	Define $\Delta_1 = \Gamma_{\ref{lemma:conv_embedding}} ( \vdim, \zeta
	) \unitmeasure{\vdim}^{-1} 2^{3\vdim+2}$, $\Delta_2 = 2^{\vdim+1}$ and
	$\Gamma = 2^{4\vdim+1} \Delta_1$. Let $\varrho = \lambda^{1/\vdim}
	\unitmeasure{\vdim}^{-1/\vdim} 2^{1+1/\vdim}$, note $\varrho \leq 2r$
	and define
	\begin{gather*}
		E (b,t) = \oball{a}{r} \cap \oball{b}{t} \quad
		\text{whenever $b \in \rel^\vdim$, $0 < t < \infty$}.
	\end{gather*}
	One estimates, using \ref{miniremark:ball_contained},
	\begin{gather*}
		\mathscr{L}^\vdim ( E (b,\varrho) \without A ) \leq \lambda =
		2^{-1-\vdim} \unitmeasure{\vdim} \varrho^\vdim \leq
		\mathscr{L}^\vdim ( E(b,\varrho) )/2 \leq
		\mathscr{L}^\vdim ( A \cap E (b,\varrho) ), \\
		\mathscr{L}^\vdim ( E(b,\varrho) ) \leq \unitmeasure{\vdim}
		\varrho^\vdim = 2^{\vdim+1} \lambda,
	\end{gather*}
	whenever $b \in \oball{a}{r}$. Therefore one applies
	\ref{lemma:conv_embedding} with $h_b = \fint_{A \cap E(b,\varrho)} u
	\ud \mathscr{L}^\vdim$ to obtain
	\begin{gather*}
		\norm{u}{q}{E(b,\varrho)} \leq
		\Gamma_{\ref{lemma:conv_embedding}} ( \vdim, \zeta )
		2^{2\vdim+1} \unitmeasure{\vdim}^{-1} \norm{\weakD
		u}{\zeta}{E(b,\varrho)} + 2^{(\vdim+1)/q} \lambda^{1/q} |
		h_b |
	\end{gather*}
	for $b \in \oball{a}{r}$. Using H\"older's inequality, this yields
	\begin{gather*}
		\norm{u}{q}{E(b,\varrho)} \leq \Delta_1 \lambda^{1/\zeta-1/s}
		\norm{\weakD u}{s}{E(b,\varrho)} + \Delta_2
		\lambda^{1/q-1/\xi} \norm{u}{\xi}{A \cap E(b,\varrho)}
	\end{gather*}
	for $b \in \oball{a}{r}$. If $q=\infty$, the conclusion is now evident.

	If $q < \infty$, choosing a maximal set $B$ (with respect to
	inclusion) such that
	\begin{gather*}
		B \subset \oball{a}{r}, \quad \text{$\{ E(b,\varrho/2) \with b
		\in B \}$ is disjointed},
	\end{gather*}
	one notes for $x \in B$ and $S_x = \classification{B}{b}{E(b,\varrho)
	\cap E(x,\varrho) \neq \emptyset}$
	\begin{gather*}
		\oball{a}{r} \subset {\textstyle\bigcup} \{ E(b,\varrho) \with
		b \in B \}, \quad \card{S_x} \leq 2^{4\vdim};
	\end{gather*}
	in fact for the estimate one uses \ref{miniremark:ball_contained} to
	infer
	\begin{gather*}
		E (b,\varrho) \subset E(x,3\varrho) \quad \text{whenever $b
		\in S_x$}, \\
		\begin{aligned}
			( \card{S_x} ) \unitmeasure{\vdim} 2^{-2\vdim}
			\varrho^\vdim & \leq \tsum{b \in S_x}{}
			\mathscr{L}^\vdim ( E (b,\varrho/2) ) \\
			& \leq \mathscr{L}^\vdim ( E (x,3\varrho)) \leq
			\unitmeasure{\vdim} 3^\vdim \varrho^\vdim.
		\end{aligned}
	\end{gather*}
	Therefore, as $q \geq \sup \{ s, \xi \}$,
	\begin{gather*}
		\tsum{b \in B}{} 
		\norm{\weakD u}{s}{E(b,\varrho)}^q \leq \big ( \tsum{b \in
		B}{} \norm{\weakD u}{s}{E(b,\varrho)}^s \big )^{q/s} \leq \big
		( 2^{4\vdim} \norm{\weakD u}{s}{a,r} \big )^q, \\
		\tsum{b \in B}{} \norm{u}{\xi}{A \cap E(b,\varrho)}^q \leq
		\big ( \tsum{b\in B}{} \norm{u}{\xi}{A \cap E(b,\varrho)}^\xi
		\big )^{q/\xi} \leq \big ( 2^{4\vdim} \norm{u}{\xi}{A} \big
		)^q,
	\end{gather*}
	hence one obtains form the estimate of the preceding paragraph
	\begin{gather*}
		\begin{aligned}
			\norm{u}{q}{a,r}^q & \leq 2^{q-1} \tsum{b\in B}{} \Big
			( \big (
			\Delta_1 \lambda^{1/\zeta-1/s} \norm{\weakD
			u}{s}{E(b,\varrho)}
			\big )^q + \big ( \Delta_2 \lambda^{1/q-1/\xi}
			\norm{u}{\xi}{A \cap E(b,\varrho)} \big )^q \Big )\\
			& \leq \big ( 2^{4\vdim+1} \Delta_1
			\lambda^{1/\zeta-1/s} \norm{\weakD u}{s}{a,r} \big )^q
			+ \big ( 2^{4\vdim+1} \Delta_2
			\lambda^{1/q-1/\xi} \norm{u}{\xi}{A} \big )^q.
		\end{aligned} 
	\end{gather*}
	and the conclusion follows.
\end{proof}
\begin{lemma} \label{lemma:Q_function_with_holes}
	Suppose $\vdim, \adim, Q \in \nat$, $\vdim < \adim$, $q = \infty$ if
	$\vdim = 1$, $2 \leq q < \infty$ if $\vdim = 2$, $2 \leq q \leq 2
	\vdim/ (\vdim-2)$ if $\vdim > 2$, $a \in \rel^\vdim$, $0 < r <
	\infty$, $f : \oball{a}{r} \to \qspace_Q ( \rel^\codim )$ is
	Lipschitzian, $0 < \eta \leq 1/2$, and $A$ is an
	$\mathscr{L}^\vdim$ measurable subset of $\oball{a}{r}$ with
	$\mathscr{L}^\vdim ( \oball{a}{r} \without A ) \leq \eta
	\unitmeasure{\vdim} r^\vdim$, then
	\begin{gather*}
		r^{-\vdim/q} \norm{f}{q}{a,r} \leq \Gamma \big (
		\eta^{1/q+1/\vdim-1/2} r^{1-\vdim/2} \norm{Af}{2}{a,r} +
		\eta^{1/q-1} r^{-\vdim} \norm{f}{1}{A} \big )
	\end{gather*}
	where $\Gamma$ is a positive, finite number depending only on $\adim$,
	$Q$, and $q$.
\end{lemma}
\begin{proof}
	Suppose $P$ and $\xi : \qspace_Q ( \rel^\codim ) \to \rel^{PQ}$ are as
	in \ref{app:thm:bilip_embedding}. Define $u = \xi \circ f$, $\mu =
	1/q+1/\vdim-1/2 \geq 0$, $\nu = 1-1/q \geq 1/2$, $\zeta = 1$ if $\vdim
	= 1$ and $\zeta = q\vdim/(\vdim+q)$ if $\vdim > 1$, hence $1 \leq
	\zeta < \vdim$ and $\zeta \vdim / ( \vdim-\zeta ) = q$ if $\vdim > 1$.
	From \ref{lemma:function_with_holes} applied with $\lambda$, $s$ and
	$\xi$ replaced by $\eta \unitmeasure{\vdim} r^\vdim$, $2$, and $1$ one
	obtains
	\begin{gather*}
		r^{-\vdim/q} \norm{u}{q}{a,r} \leq \Delta \big ( \eta^\mu
		r^{1-\vdim/2} \norm{Du}{2}{a,r} + \eta^{-\nu} r^{-\vdim}
		\norm{u}{1}{A} \big )
	\end{gather*}
	where $\Delta = \sup \big \{ \Gamma_{\ref{lemma:function_with_holes}} (
	\vdim, \zeta ) \unitmeasure{\vdim}^{1/\zeta-1/2}, 2^{5\vdim+2}
	\unitmeasure{\vdim}^{1/q-1} \big \}$. Since
	\begin{gather*}
		( \Lip \xi )^{-1} | u(x) | \leq \mathscr{G} ( f(x), Q \Lbrack
		0 \Rbrack ) \leq \Lip \xi^{-1} \, | u(x) | \quad \text{for $x
		\in \oball{a}{r}$}, \\
		| Du(x) | \leq \Lip \xi \, | Af (x) | \quad \text{for $x \in
		\dmn Du$}
	\end{gather*}
	by \ref{app:thm:bilip_embedding}, the conclusion follows.
\end{proof}
\section{Some estimates concerning linear second order elliptic systems}
\label{sec:elliptic}
The purpose of the present section is to gather some standard estimates
precisely in the form needed in Section \ref{sec:iteration}. Proofs are
included for the convenience of the reader.
\begin{miniremark} \label{miniremark:function_C}
	The following situation will occur repeatedly: $\vdim, \adim \in
	\nat$, $\vdim < \adim$, $0 < c \leq M < \infty$, and $\Upsilon \in
	\bigodot^2 \Hom ( \rel^\vdim, \rel^\codim )$ with $\| \Upsilon \| \leq
	M$ is strongly elliptic with ellipticity bound $c$, i.e.~$\Upsilon$ is
	an $\rel$ valued bilinear form on $\Hom ( \rel^\vdim, \rel^\codim )$
	with $\Upsilon ( \sigma, \tau ) \leq M |\sigma| |\tau|$ whenever
	$\sigma, \tau \in \Hom (\rel^\vdim, \rel^\codim )$ and
	\begin{gather*}
		\tint{}{} \Upsilon ( D \theta (x), D \theta (x) ) - c | D
		\theta (x) |^2 \ud \mathscr{L}^\vdim x \geq 0 \quad
		\text{whenever $\theta \in \mathscr{D} ( \rel^\vdim,
		\rel^\codim )$}.
	\end{gather*}

	Following \cite[5.2.11]{MR41:1976}, one associates to any
	$\Upsilon \in \bigodot^2 \Hom ( \rel^\vdim, \rel^\codim )$ a linear
	function $S : \bigodot^2 ( \rel^\vdim, \rel^\codim ) \cong (
	\bigodot^2 \rel^\vdim ) \otimes \rel^\codim \to \rel^\codim$
	characterised by
	\begin{gather*}
		\left < ( \xi \odot \psi ) y, S \right > \bullet \upsilon =
		\left < ( \xi \, y, \psi \, \upsilon ), \Upsilon \right > +
		\left < ( \psi \, y, \xi \, \upsilon ), \Upsilon \right >
	\end{gather*}
	whenever $\xi, \psi \in \bigodot^1 \rel^\vdim$, $y, \upsilon \in
	\rel^\codim$; here $\xi \, y \in \Hom ( \rel^\vdim, \rel^\codim )$ is
	given by $( \xi \, y ) (x) = \xi (x) y$ for $x \in \rel^\vdim$.
	Applying this construction with the area integrand $\Psi$ to $D^2
	\Psi_0^\S ( \sigma )$ for each $\sigma \in \Hom ( \rel^\vdim,
	\rel^\codim )$, one obtains a function $C : \Hom ( \rel^\vdim,
	\rel^\codim ) \to \Hom \big( \bigodot^2 ( \rel^\vdim , \rel^\codim )
	,\rel^\codim \big )$ which satisfies
	\begin{gather*}
		\left < \phi, C ( \sigma ) \right > = \sum_{i=1}^\vdim
		\sum_{j=1}^\codim \sum_{k=1}^\vdim \sum_{l=1}^\codim \leftB (
		X_i \upsilon_j, X_k \upsilon_l ), D^2 \Psi_0^\S ( \sigma )
		\rightB ( \phi ( e_i, e_k ) \bullet \upsilon_j ) \upsilon_l
	\end{gather*}
	for $\phi \in \bigodot^2 ( \rel^\vdim , \rel^\codim )$ where $e_1,
	\ldots, e_\vdim$ and $X_1, \ldots, X_\vdim$ are dual orthonormal bases
	of $\rel^\vdim$ and $\bigodot^1 \rel^\vdim$, and $\upsilon_1, \ldots,
	\upsilon_\codim$ form an orthonormal base of $\rel^\codim$. Hence
	whenever $U$ is an open subset of $\rel^\vdim$, $u \in \Sob{}{2}{1} (
	U, \rel^\codim )$ is Lipschitzian, $v \in \Sob{}{2}{1} ( U,
	\rel^\codim )$, $\sigma \in \Hom ( \rel^\vdim, \rel^\codim )$, and
	$\theta \in \mathscr{D} ( U, \rel^\codim )$ one obtains by partial
	integration the formulae
	\begin{gather*}
		- {\textstyle\int_U} \leftB D \theta (x), D\Psi_0^\S ( D u (x)
		) \rightB \ud \mathscr{L}^\vdim x = {\textstyle\int_U} \theta
		(x) \bullet \left < \weakD^2 u (x), C ( D u(x) ) \right > \ud
		\mathscr{L}^\vdim x, \\
		- {\textstyle\int_U} \leftB D \theta (x) \odot \weakD v (x),
		D^2\Psi_0^\S ( \sigma) \rightB \ud \mathscr{L}^\vdim x =
		{\textstyle\int_U} \theta (x) \bullet \left < \weakD^2 v (x),
		C ( \sigma ) \right > \ud \mathscr{L}^\vdim x,
	\end{gather*}
	here $\odot$ denotes multiplication in $\bigodot_\ast \Hom ( \rel^\vdim,
	\rel^\codim )$, see \cite[1.9.1]{MR41:1976}.
\end{miniremark}
\begin{lemma} \label{lemma:standard_w012}
	Suppose $\vdim$, $\adim$, $c$, $M$, and $\Upsilon$ are as in
	\ref{miniremark:function_C}, $a \in \rel^\vdim$, $0 < r < \infty$,
	$v \in \Sob{}{1}{2} ( \oball{a}{r}, \rel^\codim )$, $T \in
	\mathscr{D}' ( \oball{a}{r}, \rel^\codim )$ with $\dnorm{T}{2}{a,r} <
	\infty$.

	Then there exists an $\mathscr{L}^\vdim \restrict \oball{a}{r}$ almost
	unique $u \in \Sob{}{1}{2} ( \oball{a}{r}, \rel^\codim )$ such that
	\begin{gather*}
		- {\textstyle\int_{\oball{a}{r}}} \left < D \theta (x) \odot
		\weakD u (x) , \Upsilon \right > \ud \mathscr{L}^\vdim x = T
		( \theta ) \quad \text{for $\theta \in \mathscr{D} (
		\oball{a}{r}, \rel^\codim )$}, \\
		u-v \in \Sob{0}{1}{2} ( \oball{a}{r}, \rel^\codim ).
	\end{gather*}
	Moreover, for every affine function $P : \rel^\vdim \to \rel^\codim$
	\begin{gather*}
		\norm{\weakD (u-v)}{2}{a,r} \leq c^{-1} \big ( M
		\norm{\weakD (v-P)}{2}{a,r} + \dnorm{T}{2}{a,r}
		\big ).
	\end{gather*}
\end{lemma}
\begin{proof}
	To prove existence, assume $v = 0$, let $R$ denote the extension of
	$T$ to $\Sob{0}{1}{2} ( \oball{a}{r}, \rel^\codim )$ by continuity and
	observe that one can take $u$ to be a minimiser of
	\begin{gather*}
		{\textstyle\frac{1}{2} \textstyle\int_{\oball{a}{r}}} \left <
		\weakD u (x) \odot \weakD u (x), \Upsilon \right > \ud
		\mathscr{L}^\vdim x + R ( u )
	\end{gather*}
	in $\Sob{0}{1}{2} ( \oball{a}{r}, \rel^\codim )$

	To prove the estimate, assuming $P = 0$ by possibly replacing $u$,
	$v$, $P$ by $u-P$, $v-P$, $0$, one lets $\theta$ approximate $u-v$
	in $\Sob{0}{1}{2} ( \oball{a}{r}, \rel^\codim )$ to obtain
	\begin{gather*}
		c \norm{\weakD (u-v)}{2}{a,r}^2 \leq \big ( M \norm{\weakD
		(v-P)}{2}{a,r} + \dnorm{T}{2}{a,r} \big ) \norm{\weakD
		(u-v)}{2}{a,r}.
	\end{gather*}

	The uniqueness follows from the estimate.
\end{proof}
\begin{remark} \label{remark:standard_w012}
	If $T=0$ then $u$ is $\mathscr{L}^\vdim \restrict \oball{a}{r}$ almost
	equal to an analytic $\Upsilon$ harmonic function by
	\cite[5.2.5,\,6]{MR41:1976}.
\end{remark}
\begin{lemma} \label{lemma:interior_c2a_estimate}
	Suppose $\vdim$, $\adim$, $c$, $M$, $\Upsilon$, and $S$ are as in
	\ref{miniremark:function_C}, $0 < \alpha < 1$, $a \in \rel^\vdim$, $0
	<  r < \infty$, $u : \oball{a}{r} \to \rel^\codim$ is of class
	$\class{2}$, $D^2 u$ locally satisfies a H\"older condition with
	exponent $\alpha$, $f : \oball{a}{r} \to \rel^\codim$, and $S \circ
	D^2 u = f$.

	Then
	\begin{gather*}
		r^{-\alpha} \norm{D^2u}{\infty}{a,r/2} + \hoelder{\alpha}{D^2
		u | \cball{a}{r/2}} \leq \Gamma \big ( r^{-2-\alpha-\vdim}
		\norm{u}{1}{a,r} + \hoelder{\alpha}{f} \big )
	\end{gather*}
	where $\Gamma$ is a positive, finite number depending only on $\adim$,
	$c$, $M$, and $\alpha$.
\end{lemma}
\begin{proof}
	Interpolating by use of Ehring's lemma, see
	e.g.~\cite[Theorem\,\printRoman{1}.7.3]{MR895589}, and Arzel\`a's and
	Ascoli's theorem, it is enough to prove the assertion remaining when
	the term $r^{-\alpha} \norm{D^2u}{\infty}{a,r/2}$ is omitted.

	Considering slightly smaller $r$, one may assume $\hoelder{\alpha}{
	D^2 u } < \infty$.

	Applying \cite[5.2.14]{MR41:1976} to the partial derivatives of $u$
	and using Ehring's lemma as above, one infers the existence of a
	positive, finite number $\Delta$ depending only on $\adim$, $c$, $M$,
	and $\alpha$ such that
	\begin{align*}
		\hoelder{\alpha}{D^2u| \cball{b}{s}} & \leq 2^{-6-\vdim}
		\hoelder{\alpha}{D^2u| \cball{b}{2s}} \\
		& \quad + \Delta \big ( s^{-2-\alpha-\vdim} \norm{u}{1}{b,2s}
		+ \hoelder{\alpha}{f| \cball{b}{2s}} \big )
	\end{align*}
	whenever $b \in \rel^\vdim$, $0 < s < \infty$ and $\cball{b}{2s}
	\subset \oball{a}{r}$.

	Defining $h : \oball{a}{r} \to \rel$ by $h (x) = \frac{1}{4} \dist (x,
	\rel^\vdim \without \oball{a}{r} )$ for $x \in \oball{a}{r}$,
	\begin{gather*}
		\mu = \sup \big \{ h(b)^{2+\alpha+\vdim} \hoelder{\alpha}{D^2u |
		\cball{b}{h(b)}} \with b \in \oball{a}{r} \big \}
	\end{gather*}
	and noting $\mu \leq r^{2+\alpha+\vdim} \hoelder{\alpha}{D^2u} <
	\infty$, one estimates for $b \in \oball{a}{r}$
	\begin{gather*}
		\begin{aligned}
			\hoelder{\alpha}{D^2u | \cball{b}{h(b)}} & \leq
			2^{-6-\vdim} \hoelder{\alpha}{D^2u | \cball{b}{2h(b)}}
			\\
			& \quad + \Delta \big ( h(b)^{-2-\alpha-\vdim}
			\norm{u}{1}{a,r} + \hoelder{\alpha}{f} \big ),
		\end{aligned} \\
		| h(b)-h(c) | \leq ( \Lip h ) |b-c| \leq h(b)/2, \  h(b)
		\leq 2 h(c) \qquad \text{for $c \in \cball{b}{2h(b)}$}, \\
		h(b)^{2+\alpha+\vdim} \hoelder{\alpha}{D^2u |
		\cball{b}{2h(b)}} \leq 2^{4+\alpha+\vdim} \mu, \\
		h(b)^{2+\alpha+\vdim} \hoelder{\alpha}{D^2u| \cball{b}{h(b)}}
		\leq \mu / 2 + \Delta \big ( \norm{u}{1}{a,r} +
		r^{2+\alpha+\vdim} \hoelder{\alpha}{f} \big ),
	\end{gather*}
	hence
	\begin{gather*}
		(r/4)^{2+\alpha+\vdim} \hoelder{\alpha}{D^2u|\cball{a}{r/2}}
		\leq 2^{5 + \vdim} \mu \leq 2^{6 + \vdim} \Delta \big (
		\norm{u}{1}{a,r} + r^{2+\alpha+\vdim} \hoelder{\alpha}{f} \big
		)
	\end{gather*}
	and the remaining assertion is evident.
\end{proof}
\begin{remark}
	Similar absorption procedures can be found for example in
	\cite[5.2.14]{MR41:1976} or \cite[Theorem 9.11]{MR1814364}.
\end{remark}
\begin{lemma} \label{lemma:w2p_estimate_ball}
	Suppose $\vdim$, $\adim$, $c$, $M$, and $\Upsilon$ are as in
	\ref{miniremark:function_C}, $2 \leq p < \infty$, $a \in \rel^\vdim$,
	and $0 < r < \infty$.

	Then for every $f \in \Lp{p} ( \mathscr{L}^\vdim \restrict
	\oball{a}{r}, \rel^\codim )$ there exists an $\mathscr{L}^\vdim
	\restrict \oball{a}{r}$ almost unique $u \in \Sob{0}{1}{p} (
	\oball{a}{r}, \rel^\codim )$ such that
	\begin{gather*}
		- {\textstyle\int_{\oball{a}{r}}} \left < D \theta (x) \odot
		\weakD u (x) , \Upsilon \right > \ud \mathscr{L}^\vdim x =
		\pairing{\theta}{f}{a,r} \quad \text{for $\theta \in
		\mathscr{D} ( \oball{a}{r}, \rel^\codim )$}.
	\end{gather*}
	Moreover, $u \in \Sob{}{2}{p} ( \oball{a}{r}, \rel^\codim )$ and
	\begin{gather*}
		{\textstyle\sum_{i=0}^2} r^{i-2} \norm{\weakD^i u}{p}{a,r}
		\leq \Gamma \norm{f}{p}{a,r}
	\end{gather*}
	where $\Gamma$ is a positive, finite number depending only on $\adim$,
	$c$, $M$, and $p$.
\end{lemma}
\begin{proof}
	See \cite[p.~368-370]{MR1962933}.
\end{proof}
\begin{remark}
	The condition $p \geq 2$ can, of course, be replaced by $p > 1$.
	For example \cite[Theorem 10.15]{MR1962933} extends to this case via
	duality and the estimate of the second order derivatives can be
	carried out by using the method of difference quotients starting from
	a suitably localised version of the theorem cited.
\end{remark}
\begin{lemma} \label{lemma:l1_estimate}
	Suppose $\vdim$, $\adim$, $c$, $M$, and $\Upsilon$ are as in
	\ref{miniremark:function_C}, $a \in \rel^\vdim$, $0 <  r < \infty$, $u
	\in \Sob{0}{1}{1} ( \oball{a}{r}, \rel^\codim )$, $T \in \mathscr{D}'
	( \oball{a}{r}, \rel^\codim )$, and
	\begin{gather*}
		- {\textstyle\int_{\oball{a}{r}}} \left < D \theta (x) \odot
		\weakD u (x), \Upsilon \right > \ud \mathscr{L}^\vdim x = T (
		\theta ) \quad \text{for $\theta \in \mathscr{D} (
		\oball{a}{r}, \rel^\codim )$}.
	\end{gather*}

	Then
	\begin{gather*}
		\norm{u}{1}{a,r} \leq \Gamma r \dnorm{T}{1}{a,r}
	\end{gather*}
	where $\Gamma$ is a positive, finite number depending only on
	$\adim$, $c$, and $M$.
\end{lemma}
\begin{proof}
	Let $p = 2\vdim$ and $q=p/(p-1)$ and assume $r=1$.

	Whenever $\theta \in \mathscr{D} ( \oball{a}{r}, \rel^\codim )$ one
	obtains $\eta \in \Sob{0}{1}{p} ( \oball{a}{r}, \rel^\codim )$ from
	\ref{lemma:w2p_estimate_ball} such that with $\Delta_1 =
	\Gamma_{\ref{lemma:w2p_estimate_ball}} ( \adim, c, M, p )$
	\begin{gather*}
		- {\textstyle\int_{\oball{a}{1}}} \left < D \zeta (x) \odot
		\weakD \eta (x), \Upsilon \right > \ud \mathscr{L}^\vdim x =
		\pairing{\zeta}{\theta}{a,1} \quad \text{for $\zeta \in
		\mathscr{D} ( \oball{a}{1}, \rel^\codim )$}, \\
		{\textstyle\sum_{i=0}^2} \norm{\weakD^i \eta}{p}{a,1} \leq
		\Delta_1 \norm{\theta}{p}{a,1},
	\end{gather*}
	hence by \cite[Theorem 7.26\,(ii)]{MR1814364}
	\begin{gather*}
		\norm{\weakD \eta}{\infty}{a,1} \leq \Delta_2 \big (
		\norm{\weakD \eta}{p}{a,1} +
		\norm{\weakD^2 \eta}{p}{a,1} \big ) \leq \Delta_1 \Delta_2
		\norm{\theta}{p}{a,1}
	\end{gather*}
	where $\Delta_2$ is a positive, finite number depending only on
	$\adim$ and $p$. Approximating and $u$ by $\zeta_i \in \mathscr{D} (
	\oball{a}{1}, \rel^\codim )$ in $\Sob{0}{1}{1} ( \oball{a}{1},
	\rel^\codim )$ and $\eta$ by a sequence $\eta_i \in \mathscr{D}
	(\oball{a}{r}, \rel^\codim )$ such that
	\begin{gather*}
		\eta_i \to \eta \quad \text{in $\Sob{}{1}{p} ( \oball{a}{1},
		\rel^\codim )$ as $i \to \infty$}, \quad \lim_{i \to \infty}
		\norm{D \eta_i}{\infty}{a,1} = \norm{\weakD
		\eta}{\infty}{a,1},
	\end{gather*}
	one obtains
	\begin{gather*}
		\pairing{\theta}{u}{a,1} = - {\textstyle\int_{\oball{a}{1}}}
		\left < \weakD \eta (x) \odot \weakD u (x), \Upsilon \right >
		\ud \mathscr{L}^\vdim x \leq \dnorm{T}{1}{a,1}
		\norm{\weakD \eta}{\infty}{a,1}.
	\end{gather*}
	Therefore (cp. \cite[2.4.16]{MR41:1976})
	\begin{gather*}
		\norm{u}{1}{a,1} \leq \unitmeasure{\vdim}^{1/p}
		\norm{u}{q}{a,1} \leq \unitmeasure{\vdim}^{1/p} \Delta_1
		\Delta_2 \dnorm{T}{1}{a,1}
	\end{gather*}
	and one may take $\Gamma = \sup \{ \unitmeasure{i}^{1/p} \Delta_1
	\Delta_2 \with \adim > i \in \nat \}$.
\end{proof}
\begin{remark}
	If $\vdim > 1$ the estimate may be sharpened to
	\begin{gather*}
		\sup \big \{ t \mathscr{L}^\vdim (
		\classification{\oball{a}{r} }{x}{|u(x)|>t} )^{1-1/\vdim} : 0
		< t < \infty \big \} \leq \Gamma \dnorm{T}{1}{a,r};
	\end{gather*}
	in fact one may follow the same line of arguments with the Lorentz
	space $\Lp{\vdim,1}$ replacing $\Lp{p}$.
\end{remark}
\section{A model case of partial regularity} \label{sec:model_case}
The present section uses the new iteration technique in the setting of
pointwise decay estimates for the Euler Lagrange differential operator
associated to an integrand satisfying a quadratic growth condition. Its
purpose is to indicate applications in the study of partial regularity for
elliptic systems as well as to outline some of the techniques used in Section
\ref{sec:iteration} in a significantly simpler setting. However, the results
of this section are not needed in the remaining sections. They depend only on
Section \ref{sec:elliptic} and \ref{lemma:simple_interpolation},
\ref{lemma:poincare}.
\begin{miniremark} \label{miniremark:situation_pde}
	Suppose $\vdim, \adim \in \nat$, $\vdim < \adim$, $0 < c \leq M <
	\infty$, and $F : \Hom ( \rel^\vdim, \rel^\codim ) \to \rel$ is of
	class $\class{2}$ such that for $\sigma, \tau \in \Hom ( \rel^\vdim,
	\rel^\codim )$
	\begin{gather*}
		\left < \sigma \odot \sigma, D^2 F ( \tau ) \right > \geq c |
		\sigma |^2, \quad \| D^2 F ( \tau ) \| \leq M.
	\end{gather*}
\end{miniremark}
\begin{lemma} \label{lemma:caccioppoli_inequality}
	Suppose $\vdim$, $\adim$, $c$, $M$, and $F$ are as in
	\ref{miniremark:situation_pde}, $a \in \rel^\vdim$, $0 < r < \infty$,
	$u \in \Sob{}{1}{2}(\oball{a}{r},\rel^\codim)$, $T \in \mathscr{D} (
	\oball{a}{r}, \rel^\codim )$, and
	\begin{gather*}
		- \tint{\oball{a}{r}}{} \left < D \theta (x), DF ( \weakD u (x))
		\right > \ud \mathscr{L}^\vdim x = T ( \theta ) \quad
		\text{for $\theta \in \mathscr{D} ( \oball{a}{r}, \rel^\codim
		)$}.
	\end{gather*}

	Then there holds for every affine function $P : \rel^\vdim \to
	\rel^\codim$
	\begin{gather*}
		r^{-\vdim/2} \norm{\weakD (u-P)}{2}{a,r/2} \leq \Gamma \big (
		r^{-1-\vdim} \norm{u-P}{1}{a,r} + r^{-\vdim/2}
		\dnorm{T}{2}{a,r} \big )
	\end{gather*}
	where $\Gamma$ is a positive, finite number depending only on $\vdim$,
	$\adim$, $c$, and $M$.
\end{lemma}
\begin{proof}
	Assume $r=1$ and abbreviate $v=u-P$. Observing
	\begin{gather*}
		- \tint{\oball{a}{r}}{} \left < D \theta (x) \odot \weakD v
		(x), A (x) \right > \ud \mathscr{L}^\vdim x = T ( \theta )
		\quad \text{for $\theta \in \mathscr{D} ( \oball{a}{1},
		\rel^\codim)$} \\
		\text{where $A(x) = \tint{0}{1} D^2F ( t \weakD u(x) + (1-t)
		DP(x) ) \ud \mathscr{L}^1 t$},
	\end{gather*}
	one may infer, e.g.~as in \cite[5.2.3]{MR41:1976}, that
	\begin{gather*}
		\norm{\weakD v}{2}{b,\varrho} \leq c^{-1/2} M^{1/2}
		\varrho^{-1} \norm{v}{2}{b,2\varrho} + c^{-1}
		\dnorm{T}{2}{b,2\varrho}
	\end{gather*}
	whenever $b \in \rel^\vdim$, $0 < \varrho < \infty$ with
	$\oball{b}{2\varrho} \subset \oball{a}{1}$.
	
	From \cite[Theorem 7.26\,(i)]{MR1814364} and Ehring's lemma, see
	e.g.~\cite[Theorem\,\printRoman{1}.7.3]{MR895589}, it follows that for
	every $0 < \kappa < \infty$ there exists a positive, finite number
	$\Delta$ depending only on $\adim$ and $\kappa$ such that
	\begin{gather*}
		\varrho^{-1} \norm{v}{2}{b,2\varrho} \leq \delta \norm{\weakD
		v}{2}{b,2\varrho} + \Delta \varrho^{-1-\vdim/2}
		\norm{v}{1}{b,2\varrho}
	\end{gather*}
	whenever $b \in \rel^\vdim$, $0 < \varrho < \infty$ with
	$\oball{b}{2\varrho} \subset \oball{a}{1}$. Therefore one readily
	verifies the conclusion by use of Simon's absorption lemma
	\cite[p.~398]{MR1459795}.
\end{proof}
\begin{miniremark} \label{miniremark:situation_pde2}
	If $\vdim$, $\adim$, $c$, $M$, and $F$ are as in
	\ref{miniremark:situation_pde} then $D^2F$ is uniformly continuous if
	and only if there exists
	$\Omega : \{ t \with 0 \leq t < \infty \} \to \{ t \with 0 \leq t \leq
	2 M \}$ such that
	\begin{gather*}
		\text{$\Omega$ is continuous at $0$ with $\Omega (0)=0$},
		\quad \text{$\Omega^2$ is concave}, \\
		\| D^2 F ( \sigma ) - D^2 F ( \tau ) \| \leq \Omega ( |
		\sigma-\tau | ) \quad \text{for $\sigma, \tau \in \Hom (
		\rel^\vdim, \rel^\codim )$}.
	\end{gather*}
	Observe that such $\Omega$ is nondecreasing and satisfies $\Omega ( st
	) \leq s^{1/2} \Omega ( t )$ for $1 \leq s < \infty$ and $0 \leq t <
	\infty$.

	Moreover, let $0 < \alpha \leq 1$ and define $\omega : \{ t \with 0 <
	t \leq 1 \} \to \{ t \with 0 \leq t \leq 1 \}$ by
	\begin{gather*}
		\omega (t) = t^\alpha \quad \text{if $\alpha < 1$}, \quad
		\omega (t) = t ( 1 + \log (1/t) ) \quad \text{if $\alpha = 1$}
	\end{gather*}
	whenever $0 < t \leq 1$.
\end{miniremark}
\begin{theorem} \label{thm:example_decay}
	Suppose $\vdim, \adim \in \nat$, $\vdim < \adim$, $0 < c \leq M <
	\infty$, and $0 < \alpha \leq 1$.
	
	Then there exists a positive, finite number $\varepsilon$ with the
	following property.

	If $a \in \rel^\vdim$, $0 < r < \infty$, $F$, $\Omega$, $\omega$ are
	related to $\vdim$, $\adim$, $c$, $M$, $\alpha$ as in
	\ref{miniremark:situation_pde} and \ref{miniremark:situation_pde2}, $u
	\in \Sob{}{1}{2} ( \oball{a}{r} , \rel^\codim )$, $T \in \mathscr{D}'
	( \oball{a}{r}, \rel^\codim )$, $\sigma \in \Hom ( \rel^\vdim,
	\rel^\codim )$, $0 \leq \gamma < \infty$, and
	\begin{gather*}
		\Omega ( \gamma ) \leq \varepsilon \quad \text{if $\alpha <
		1$}, \qquad \Omega ( t ) \leq \varepsilon ( 1 + \log (
		\gamma/t) )^{-1} \quad \text{for $0 < t \leq \gamma$ if
		$\alpha = 1$}, \\
		- \tint{\oball{a}{r}}{} \left < D \theta (x), DF ( \weakD u (x))
		\right > \ud \mathscr{L}^\vdim x = T ( \theta ) \quad
		\text{for $\theta \in \mathscr{D} ( \oball{a}{r} ,
		\rel^\codim )$}, \\
		\big ( \tfint{\oball{a}{r}}{} | \weakD (u-\sigma) |^2 \ud
		\mathscr{L}^\vdim \big )^{1/2} \leq \gamma, \\
		\varrho^{-\vdim/2} \dnorm{T}{2}{a,\varrho} \leq \gamma (
		\varrho/r )^\alpha \quad \text{for $0 < \varrho \leq r$},
	\end{gather*}
	then $a \in \dmn \weakD u$ and
	\begin{gather*}
		\big ( \tfint{\oball{a}{\varrho}}{} | \weakD (u-\weakD u(a))
		|^2 \ud \mathscr{L}^\vdim \big )^{1/2} \leq \Gamma \omega (
		\varrho/r ) \gamma \quad \text{for $0 < \varrho \leq r$}
	\end{gather*}
	where $\Gamma$ is a positive, finite number depending only on $\vdim$,
	$\adim$, $c$, $M$, and $\alpha$.
\end{theorem}
\begin{proof}
	Define
	\begin{gather*}
		\Delta_1 = \sup \{ \unitmeasure{\vdim},
		\unitmeasure{\vdim}^{1/2} \} \Gamma_{\ref{lemma:l1_estimate}}
		( \adim, c, M ), \quad \Delta_2 = 2^{\vdim+5} ( \vdim+1
		)^{\vdim+2} ( M/c)^{\vdim+1}, \\
		\Delta_3 = \sup \{ 2^{4+2\vdim}, \adim ( \codim ) \}
		\Gamma_{\ref{lemma:interior_c2a_estimate}} ( \adim, c, M, 1/2
		), \quad \Delta_4 = 2 \Delta_3 \sup \{ \Delta_1, 2^\vdim
		\Delta_2 \}, \\
		\Delta_5 = \unitmeasure{\vdim}^{-1/2} 2^{1+2\vdim}
		\Gamma_{\ref{lemma:caccioppoli_inequality}} ( \vdim, \adim, c,
		M ), \quad \Delta_6 = \Delta_5 \sup \{ 1 + \Delta_1,
		\unitmeasure{\vdim} \}, \\
		\Delta_7 = \Gamma_{\ref{lemma:interior_c2a_estimate}} ( \adim,
		c, M, 1/2 ) \big ( \Delta_1 ( 2 M + 1 ) +
		\unitmeasure{\vdim} \Gamma_{\ref{lemma:poincare}} (
		\adim ) \big ).
	\end{gather*}
	Moreover, define
	\begin{gather*}
		\Delta_8 = 1- 4^{\alpha-1} \quad \text{if $\alpha < 1$},
		\qquad \Delta_8 = \log 4 \quad \text{if $\alpha = 1$}, \\
		\Delta_9 = \sup \{ 2^{\vdim+3} \Delta_7, 2 \Delta_4
		\Delta_8^{-1} \}, \quad \Delta_{10} = \sup \{ 2^{\vdim+2}, 8
		\Delta_6 \}, \\
		\Delta_{11} = \sup \{ s^{1/2} ( 1 + \log (1/s) ) \with 0 < s
		\leq 1 \}, \quad
		\Delta_{12} = \big ( 8 \Delta_6 \big (1+2\Delta_{11}^{1/2}
		\big ) \big
		)^{-1}, \\
		\Delta_{13} = \inf \big \{ \Delta_{12}, \big ( \Delta_4 ( 2
		\Delta_{11}^{1/2}) ( 1 + \Delta_{12}^{-1} ) \big )^{-1}
		\Delta_8/2 \big \}, \\
		\gamma_1 = \sup \{ \Delta_9, \Delta_{10} \Delta_{12} \}, \quad
		\gamma_2 = \Delta_{12}^{-1} \gamma_1, \quad \varepsilon =
		\Delta_{13} \gamma_2^{-1/2}, \\
		\Delta_{14} = ( 1 + 4^{-\alpha} )^{-1} \quad \text{if $\alpha
		< 1$}, \qquad \Delta_{14} = (4/3) + (4/9) \log 4 \quad
		\text{if $\alpha=1$}, \\
		\Delta_{15} = \gamma_2 \Delta_{14}, \quad \Gamma = \gamma_2
		+ 2^{\vdim+1} \Delta_{15}.
	\end{gather*}

	Suppose $a$, $r$, $F$, $\Omega$, $\omega$, $u$, $T$, $\sigma$, and
	$\gamma$ satisfy the hypotheses in the body of the theorem with
	$\varepsilon$.

	Assume $r=1$.

	Define $\sigma_\varrho = \tfint{\oball{a}{\varrho}}{} \weakD u \ud
	\mathscr{L}^\vdim \in \Hom (\rel^\vdim, \rel^\codim )$ for $0 <
	\varrho \leq 1$ and note
	\begin{gather*}
		\tint{\oball{a}{\varrho}}{} | \weakD (u-\sigma_\varrho ) |^2
		\ud \mathscr{L}^\vdim \leq \tint{\oball{a}{\varrho}}{} |\weakD
		(u-\tau) |^2 \ud \mathscr{L}^\vdim
	\end{gather*}
	whenever $0 < \varrho \leq 1$ and $\tau \in \Hom ( \rel^\vdim,
	\rel^\codim )$. Denote by $u_\varrho$ the unique function such that,
	see \ref{lemma:standard_w012},\,\ref{remark:standard_w012},
	\begin{gather*}
		u_\varrho \in \mathscr{E} ( \oball{a}{\varrho}, \rel^\codim
		), \quad u-u_\varrho \in \Sob{0}{1}{2} ( \oball{a}{\varrho},
		\rel^\codim ), \\
		\tint{\oball{a}{\varrho}}{} \left < D \theta (x) \odot \weakD
		u_\varrho (x), D^2 F ( \sigma_\varrho ) \right > \ud
		\mathscr{L}^\vdim x = 0 \quad \text{for $\theta \in
		\mathscr{D} ( \oball{a}{\varrho}, \rel^\codim )$}
	\end{gather*}
	whenever $0 < \varrho \leq 1$. Define $\phi_i : \{ \varrho \with 0 <
	\varrho \leq 1 \} \to \rel$ for $i \in \{ 1,2,3 \}$ and $S_\varrho,
	R_\varrho \in \mathscr{D}' ( \oball{a}{\varrho}, \rel^\codim )$ by
	\begin{gather*}
		\phi_1 ( \varrho ) = \norm{D^2
		u_\varrho}{\infty}{a,\varrho/2}, \quad \phi_2 ( \varrho ) =
		\unitmeasure{\vdim}^{-1/2} \varrho^{-\vdim/2} \norm{\weakD (u -
		\sigma_\varrho)}{2}{a,\varrho}, \\
		\phi_3 ( \varrho ) = \varrho^{-\vdim/2}
		\dnorm{T}{2}{a,\varrho}, \\
		\begin{aligned}
			R_\varrho ( \theta ) & = - \tint{\oball{a}{\varrho}}{}
			\left < D \theta (x) \odot \weakD ( u-u_\varrho ) (x),
			D^2 F ( \sigma_\varrho) \right >  \ud
			\mathscr{L}^\vdim x, \\
			S_\varrho ( \theta ) & = - \tint{\oball{a}{\varrho}}{}
			\left < D \theta (x), DF ( \weakD u (x) ) \right > \ud
			\mathscr{L}^\vdim x
		\end{aligned}
	\end{gather*}
	whenever $\theta \in \mathscr{D} ( \oball{a}{\varrho}, \rel^\codim )$
	and $0 < \varrho \leq 1$. Moreover, define $P_\varrho : \rel^\vdim \to
	\rel^\codim$ by $P_\varrho ( x) = u_\varrho (a) + Du_\varrho ( x-a )$
	for $x \in \rel^\vdim$.

	Next, the following four inequalities valid for $0 < \varrho \leq 1$
	will be established.
	\begin{align}
		\varrho^{-1-\vdim} \norm{u-u_\varrho}{1}{a,\varrho} & \leq
		\Delta_1 \big ( \Omega ( \phi_2 ( \varrho ) ) \phi_2 (
		\varrho) + \phi_3 ( \varrho ) \big ),
		\label{eqn:diff_estimate} \\
		\phi_1(\varrho) & \leq \Delta_7 \varrho^{-1} ( \phi_2(\varrho)
		+  \phi_3 ( \varrho ) ), \label{eqn:final_estimate} \\
		\phi_1 ( \varrho/4 ) & \leq \phi_1 ( \varrho ) + \Delta_4
		\left ( \Omega ( \phi_2 ( \varrho ) ) ( \phi_1 ( \varrho ) +
		\varrho^{-1} \phi_2 ( \varrho ) ) + \varrho^{-1} \phi_3 (
		\varrho ) \right ), \label{eqn:first_iteration} \\
		\phi_2 ( \varrho/4) & \leq \Delta_6 \big ( \varrho \phi_1 (
		\varrho ) + \Omega ( \phi_2 ( \varrho ) ) \phi_2 ( \varrho ) +
		\phi_3 ( \varrho ) \big ).  \label{eqn:second_iteration}
	\end{align}

	To prove \eqref{eqn:diff_estimate}, compute for $\mathscr{L}^\vdim$
	almost all $x \in \oball{a}{\varrho}$ by means of Taylor's formula
	\begin{multline*}
		DF ( \weakD u (x) ) = DF ( \sigma_\varrho ) + ( \weakD
		u(x)-\sigma_\varrho ) \mathop{\lrcorner} D^2F ( \sigma_\varrho
		) \\
		+ ( \weakD u(x) -\sigma_\varrho ) \mathop{\lrcorner}
		\tint{0}{1} D^2F ( t \weakD u (x) + (1-t) \sigma_\varrho ) -
		D^2 F ( \sigma_\varrho) \ud \mathscr{L}^1 t
	\end{multline*}
	and observe for $\theta \in \mathscr{D} ( \oball{a}{\varrho},
	\rel^\codim)$
	\begin{gather*}
		( S_\varrho - R_\varrho ) ( \theta ) = -
		\tint{\oball{a}{\varrho}}{} \left < D \theta (x) \odot \weakD
		(u-\sigma_\varrho) (x), A(x) \right > \ud \mathscr{L}^\vdim x
		\\
		\text{where $A(x) = \tint{0}{1} D^2 F(t \weakD u (x) + (1-t)
		\sigma_\varrho ) - D^2 F ( \sigma_\varrho ) \ud \mathscr{L}^1
		t$},
	\end{gather*}
	hence, one readily estimates by use H\"older's inequality and Jensen's
	inequality
	\begin{gather*}
		\begin{aligned}
			\unitmeasure{\vdim}^{-1} \varrho^{-\vdim}
			\dnorm{R_\varrho-S_\varrho}{1}{a,\varrho} & \leq
			\tfint{\oball{a}{\varrho}}{} | \weakD
			(u-\sigma_\varrho)| ( \Omega \circ | \weakD (
			u-\sigma_\varrho) | ) \ud \mathscr{L}^\vdim \\
			& \leq \Omega ( \phi_2 ( \varrho ) ) \phi_2 ( \varrho
			),
		\end{aligned} \\
		\varrho^{-\vdim} \dnorm{R_\varrho}{1}{a,\varrho} \leq
		\unitmeasure{\vdim} \Omega ( \phi_2 ( \varrho ) ) \phi_2 (
		\varrho ) + \unitmeasure{\vdim}^{1/2} \phi_3 ( \varrho )
	\end{gather*}
	for $0 < \varrho \leq 1$. Consequently, one infers
	\eqref{eqn:diff_estimate} by \ref{lemma:l1_estimate}.

	To prove \eqref{eqn:final_estimate}, note for every affine function $Q
	: \rel^\vdim \to \rel^\codim$
	\begin{gather*}
		\phi_1 ( \varrho ) \leq
		\Gamma_{\ref{lemma:interior_c2a_estimate}} ( \adim, c, M,
		1/2 ) \varrho^{-2-\vdim} ( \norm{u_\varrho-u}{1}{a,\varrho} +
		\norm{u-Q}{1}{a,\varrho} )
	\end{gather*}
	by \ref{lemma:interior_c2a_estimate}, hence \eqref{eqn:diff_estimate}
	and \ref{lemma:poincare} imply \eqref{eqn:final_estimate}.

	To prove \eqref{eqn:first_iteration}, first compute
	\begin{gather*}
		\begin{aligned}
			& \tint{\oball{a}{\varrho/4}}{} \left < D \theta (x)
			\odot D ( u_\varrho-u_{\varrho/4}) (x), D^2 F (
			\sigma_{\varrho/4} ) \right > \ud \mathscr{L}^\vdim x
			\\
			& \qquad = \tint{\oball{a}{\varrho/4}}{} \left < D
			\theta (x) \odot D u_\varrho (x), D^2 F
			(\sigma_{\varrho/4}) - D^2 F ( \sigma_\varrho ) \right
			>
			\ud \mathscr{L}^\vdim x
		\end{aligned}
	\end{gather*}
	for $\theta \in \mathscr{D} ( \oball{a}{\varrho/4}, \rel^\codim )$.
	Therefore, noting
	\begin{gather*}
		| \sigma_{\varrho/4} - \sigma_\varrho | \leq 2^{\vdim+1}
		\phi_2 ( \varrho ), \quad \phi_2 ( \varrho/4 ) \leq 2^\vdim
		\phi_2 ( \varrho ), \quad \phi_3 ( \varrho/4 ) \leq 2^\vdim
		\phi_3 ( \varrho ), \\
		\varrho^{1/2} \mathbf{h}_{1/2} ( D^2 u_\varrho |
		\cball{a}{\varrho/4} ) \leq \Delta_2 \phi_1 ( \varrho )
	\end{gather*}
	by \cite[5.2.5]{MR41:1976}, one uses \ref{lemma:interior_c2a_estimate}
	and \eqref{eqn:diff_estimate} to infer
	\begin{gather*}
		\begin{aligned}
			& \norm{D^2 ( u_\varrho-u_{\varrho/4}
			)}{\infty}{a,\varrho/8} \\
			& \quad \leq \Delta_3 \big ( \varrho^{-2-\vdim}
			\norm{u_\varrho-u_{\varrho/4}}{1}{a,\varrho/4} +
			\Omega ( | \sigma_{\varrho/4}-\sigma_\varrho| )
			\varrho^{1/2} \mathbf{h}_{1/2} ( D^2 u_\varrho |
			\cball{a}{\varrho/4} ) \big ) \\
			& \quad \leq \Delta_4 \left ( \Omega ( \phi_2 (
			\varrho ) ) ( \phi_1 ( \varrho ) + \varrho^{-1}
			\phi_2 ( \varrho ) ) + \varrho^{-1} \phi_3 (
			\varrho ) \right )
		\end{aligned}
	\end{gather*}
	and \eqref{eqn:first_iteration} follows.

	To prove \eqref{eqn:second_iteration}, apply
	\ref{lemma:caccioppoli_inequality} with $r$, $u$, $T$, and $P$
	replaced by $\varrho/2$, $u | \oball{a}{\varrho/2}$, $S_{\varrho/2}$
	and $P_\varrho$ to infer
	\begin{gather*}
		\phi_2 ( \varrho/4 ) \leq \Delta_5 \big ( \varrho^{-1-\vdim}
		( \norm{u-u_\varrho}{1}{a,\varrho} +
		\norm{u_\varrho-P_\varrho}{1}{a,\varrho/2} ) + \phi_3 (
		\varrho ) \big )
	\end{gather*}
	and use \eqref{eqn:diff_estimate} and Taylor's formula to verify
	\eqref{eqn:second_iteration}.

	Next, \emph{it will be shown
	\begin{gather} \label{eqn:iteration}
		\phi_1 ( \varrho ) \leq \gamma \gamma_1 \varrho^{-1} \omega (
		\varrho ), \quad \phi_2 ( \varrho ) \leq \gamma \gamma_2 \omega
		( \varrho )
	\end{gather}
	for $0 < \varrho \leq 1$}. If $1/4 \leq \varrho \leq 1$ then
	\eqref{eqn:iteration} holds for $\varrho$ since by
	\eqref{eqn:final_estimate}
	\begin{gather*}
		\phi_1 ( \varrho ) \leq 2^{\vdim+2} \Delta_7 ( \phi_2 ( 1 ) +
		\phi_3 ( 1 ) ) \leq \gamma \gamma_1 \leq \gamma \gamma_1
		\varrho^{-1} \omega ( \varrho ), \\
		\phi_2 ( \varrho ) \leq 2^\vdim \phi_2 (1) \leq \gamma
		2^{\vdim+2} \varrho^\alpha \leq \gamma
		\gamma_2 \omega ( \varrho ).
	\end{gather*}
	Suppose now \eqref{eqn:iteration} holds for some $0 < \varrho
	\leq 1$. In \emph{case $\alpha < 1$}, noting $\Omega ( \gamma \gamma_2
	) \leq\gamma_2^{1/2} \Omega (\gamma) \leq \Delta_{13} \leq \Delta_{12}$,
	\eqref{eqn:first_iteration} and \eqref{eqn:second_iteration} imply
	\begin{gather*}
		\phi_1 ( \varrho/4 ) \leq \gamma \gamma_1 ( \varrho/4
		)^{\alpha-1} \big ( 4^{\alpha-1} + \Delta_4 \Omega ( \gamma
		\gamma_2 ) ( 1 + \Delta_{12}^{-1} ) + \Delta_4 \gamma_1^{-1}
		\big ) \leq \gamma \gamma_1 ( \varrho/4 )^{\alpha-1}, \\
		\phi_2 ( \varrho/4 ) \leq \gamma \gamma_2 ( \varrho/4 )^\alpha
		\big ( 4 \Delta_6 ( 2 \Delta_{12} + \gamma_{2}^{-1} ) \big )
		\leq \gamma \gamma_2 ( \varrho/4)^\alpha.
	\end{gather*}
	and \eqref{eqn:iteration} holds for $\varrho/4$. In \emph{case
	$\alpha = 1$}, noting
	\begin{gather*}
		\begin{aligned}
			\Omega \big ( \gamma \gamma_2 \varrho ( 1 + \log
			(1/\varrho) ) \big ) & \leq ( \gamma_2
			\Delta_{11})^{1/2} \Omega \big ( \gamma \varrho^{1/2}
			\big ) \\
			& \leq 2 \Delta_{11}^{1/2} \Delta_{13} ( 1 + \log ( 1
			/ \varrho ) )^{-1} \leq 2 \Delta_{11}^{1/2}
			\Delta_{12},
		\end{aligned}
	\end{gather*}
	\eqref{eqn:first_iteration} and \eqref{eqn:second_iteration} imply
	\begin{gather*}
		\begin{aligned}
			\phi_1 ( \varrho/4 )
			& \leq \gamma \gamma_1
			\Big ( ( 1 + \log ( 1/\varrho ) ) \big ( 1 + \Delta_4
			\Omega ( \gamma \gamma_2 \varrho ( 1 + \log
			(1/\varrho) ) ) ( 1 + \Delta_{12}^{-1} ) \big ) \\
			& \quad \phantom{\gamma \gamma_1 \Big ( } + \Delta_4
			\gamma_1^{-1} \Big ) \\
			& \leq \gamma \gamma_1 \big ( ( 1 + \log
			(1/\varrho) ) + 2 \Delta_4 \Delta_{11}^{1/2} ( 1 +
			\Delta_{12}^{-1} ) \Delta_{13} + \Delta_4
			\Delta_9^{-1} \big ) \\
			& \leq \gamma \gamma_1 ( 1 + \log (4/\varrho)),
		\end{aligned} \\ \displaybreak[0]
		\begin{aligned}
			\phi_2 ( \varrho/4 ) & \leq \gamma \gamma_2 \varrho (
			1 + \log (1/\varrho) ) \Delta_6 \big ( \Delta_{12} +
			\Omega ( \gamma \gamma_2 \varrho ( 1 + \log
			(1/\varrho) ) ) + \gamma_2^{-1} \big ), \\
			& \leq \gamma \gamma_2 \omega ( \varrho/4 ) \big ( 4
			\Delta_6 \Delta_{12} ( 1 + 2 \Delta_{11}^{1/2} ) + 4
			\Delta_6 \Delta_{10}^{-1} \big )
			\\
			& \leq \gamma \gamma_2 \omega ( \varrho/4 )
		\end{aligned}
	\end{gather*}
	and \eqref{eqn:iteration} holds for $\varrho/4$. Hence the
	assertion follows in both cases.

	One readily estimates by use of \eqref{eqn:iteration}
	\begin{gather*}
		\tsum{\nu=0}{\infty} \phi_2 ( 4^{-\nu} \varrho ) \leq
		\Delta_{15} \gamma \omega ( \varrho ) \quad \text{for $0 <
		\varrho \leq 1$}
	\end{gather*}
	hence, noting $| \sigma_\varrho - \sigma_s | \leq 2^{\vdim+1} \phi_2 (
	\varrho )$ if $\varrho/4 \leq s \leq \varrho$, one infers the
	existence of $\tau \in \Hom ( \rel^\vdim, \rel^\codim )$ such that
	\begin{gather*}
		| \tau - \sigma_\varrho | \leq 2^{\vdim+1} \Delta_{15} \gamma
		\omega ( \varrho ) \quad \text{for $0 < \varrho \leq 1$}.
	\end{gather*}
	Therefore, noting \eqref{eqn:iteration},
	\begin{gather*}
		\big ( \tfint{\oball{a}{\varrho}}{} | \weakD ( u-\tau ) |^2 \ud
		\mathscr{L}^\vdim \big )^{1/2} \leq \Gamma \gamma \omega (
		\varrho ) \quad \text{for $0 < \varrho \leq 1$},
	\end{gather*}
	in particular $a \in \dmn \weakD u$ with $\tau = \weakD u (a)$.
\end{proof}
\begin{remark}
	A similar but simpler argument shows the following proposition:
	\emph{If $\adim \in \nat$ and $0 < c \leq M < \infty$ then there exist
	positive, finite numbers $\varepsilon$ and $\Gamma$ such that if
	$\adim > \vdim \in \nat$, $a \in \rel^\vdim$, $0 < r < \infty$, $A :
	\oball{a}{r} \to \bigodot^2 \Hom ( \rel^\vdim, \rel^\codim )$ is
	$\mathscr{L}^\vdim \restrict \oball{a}{r}$ measurable,
	\begin{gather*}
		\| A (a) \| \leq M, \quad \text{$A(a)$ is strongly elliptic
		with ellipticity bound $c$}, \\
		\sup \{ (1 + \log (r/|x-a|) ) \| A(x)-A(a) \| \with x \in
		\oball{a}{r} \without \{ a \} \} \leq \varepsilon,
	\end{gather*}
	$u \in \Sob{}{1}{2} ( \oball{a}{r}, \rel^\codim )$, $T \in
	\mathscr{D}' ( \oball{a}{r}, \rel^\codim )$, $0 \leq \gamma < \infty$,
	\begin{gather*}
		\tint{\oball{a}{r}}{} \left < D \theta (x) \odot \weakD u (x),
		A(x) \right > \ud \mathscr{L}^\vdim x = T ( \theta ) \quad
		\text{for $\theta \in \mathscr{D} ( \oball{a}{r}, \rel^\codim
		)$}, \\
		\varrho^{-\vdim/2} \dnorm{T}{2}{a,\varrho} \leq \gamma \quad
		\text{for $0 < \varrho \leq r$}
	\end{gather*}
	then with $\sigma_\varrho = \tfint{\oball{a}{\varrho}}{} \weakD u \ud
	\mathscr{L}^\vdim$
	\begin{gather*}
		\varrho^{-\vdim/2} \norm{\weakD
		(u-\sigma_\varrho)}{2}{a,\varrho} \leq \Gamma \big (
		r^{-\vdim/2} \norm{ \weakD u }{2}{a,r} + \gamma \big ) \quad
		\text{for $0 < \varrho \leq r$}.
	\end{gather*}}
	One may use the example exhibited by Jin, Maz'ya and Van Schaftingen
	in \cite[Proposition 1.6]{MR2543981} to verify that ``$\leq
	\varepsilon$'' cannot be replaced by ``$\leq M$'' even if $\codim=1$
	and $T=0$. Moreover, if $F : \Hom ( \rel^\vdim, \rel^\codim ) \to
	\rel$ is of class $\class{2}$, $\Omega$ is related to $F$ as in
	\ref{miniremark:situation_pde2}, $0 < \beta < 1$, $0 < \delta <
	\infty$, $1 \leq \Delta < \infty$, $v \in \Sob{}{2}{2} ( \oball{a}{r},
	\rel^\codim )$, $v$ is of class $1$, $\mathbf{h}_\beta ( Dv ) \leq
	\Delta \delta r^{-\beta}$, $\sigma = Dv(a)$,
	\begin{gather*}
		\| D^2 F (\sigma) \| \leq M, \quad \text{$D^2 F(\sigma)$ is
		strongly elliptic with ellipticity bound $c$}, \\
		\Omega ( t ) \leq \Delta^{-1/2} \beta \varepsilon ( 1 + \log (
		\delta/t ) )^{-1} \quad \text{for $0 < t \leq \delta$}, \\
		\tint{\oball{a}{r}}{} \left < D \theta (x), DF ( Dv (x) )
		\right > \ud \mathscr{L}^\vdim x = 0 \quad \text{for $\theta
		\in \mathscr{D} ( \oball{a}{r}, \rel^\codim )$},
	\end{gather*}
	then the preceding proposition applies with $A$, $u$, $T$, and
	$\gamma$ replaced by $D^2 F \circ Dv$, $D_iv$, $0$, and $0$ whenever
	$i \in \{ 1, \ldots, \vdim \}$.
\end{remark}
\begin{remark}
	More information and references on the regularity questions for
	elliptic systems may be found in the surveys of Mingione
	\cite{MR2291779} and Duzaar and Mingione \cite{MR2499905}. The latter
	specifically describes the approximation techniques originating from
	De~Giorgi \cite{MR0179651} which are used also in the present paper in
	modified form.
\end{remark}
\section{Estimates concerning the quadratic tilt-excess} \label{sec:iteration}
The estimates of the present section constitute the core of the proof of the
pointwise regularity theorem, Theorem \ref{thm:pointwise_decay}, in Section
\ref{sec:thm}. All constructions are based on the approximation by a
$\qspace_Q ( \rel^\codim )$ valued function of Section \ref{sec:approx}.
First, in \ref{lemma:aux_monotonicity} and \ref{lemma:lower_mass_bound} some
lower mass bounds are derived by a simple adaption of \cite[17.7]{MR87a:49001}
and a straightforward use of Allard's compactness theorem for integral
varifolds, see \cite[6.4]{MR0307015} or \cite[42.8]{MR87a:49001}.  Then, in
\ref{lemma:iteration_prep} several auxiliary estimates concerning the
approximation by a $\qspace_Q ( \rel^\codim)$ valued function in
\ref{lemma:lipschitz_approximation} are carried out. In \ref{lemma:iteration}
the main elliptic estimates are established, see below for a more detailed
description. Finally, a reformulation of a special case of
\ref{lemma:iteration}\,\eqref{item:iteration:prep_tilt} replacing any
reference to the specific approximating functions used there by quantities
more tightly connected to the varifold is provided in \ref{lemma:aux_c2rekt}
for use in \cite{snulmenn:c2.v3}.

Next, an overview of the constructions in \ref{lemma:iteration} in comparison
to the estimates \eqref{eqn:diff_estimate}--\eqref{eqn:iteration} in the proof
of the model case \ref{thm:example_decay} is given. One considers cylinders
centred at a fixed point $a \in \rel^\adim$ with projection $c \in
\rel^\vdim$. For any radius $\varrho$ functions $u_\varrho$ solving a
Dirichlet problem in $\oball{c}{\varrho}$ for a suitable linear elliptic
system with constant coefficients with the ``average'' $g$ of the
approximating $\qspace_Q ( \rel^\codim )$ valued function $f$ as boundary
values are defined. It is readily seen in
\ref{lemma:iteration}\,\eqref{item:iteration:starting_control} that $\phi_1 (
\varrho ) = \norm{D^2 u_\varrho}{\infty}{c,\varrho/2}$, the leading quantity
in the iteration, is controlled by the tilt-excess of the varifold and mean
curvature, compare \ref{thm:example_decay}\,\eqref{eqn:final_estimate}. More
importantly, an estimate of $\norm{u-g}{1}{c,\varrho}$, compare
\ref{thm:example_decay}\,\eqref{eqn:diff_estimate}, mainly in terms of mean
curvature is established in
\ref{lemma:iteration}\,\eqref{item:iteration:l1_estimate} by use of
\ref{lemma:l1_estimate}. Using this estimate, the iteration inequality for
$\phi_1$, compare \ref{thm:example_decay}\,\eqref{eqn:first_iteration},
follows in \ref{lemma:iteration}\,\eqref{item:iteration:key_estimate}. In
order to derive an iteration inequality for the tilt-excess of the varifold,
i.e.~controlling the tilt-excess basically by $\phi_1$ and mean curvature, the
estimate \ref{lemma:iteration}\,\eqref{item:iteration:prep_tilt} is
established. It asserts that $\norm{f \aplus (-P)}{1}{X}$ with $P : \rel^\vdim
\to \rel^\codim$ an affine function and $X$ a large (with respect to
$\mathscr{L}^\vdim$) subset of $\oball{c}{\varrho/2}$ together with mean
curvature essentially controls the tilt-excess. Here the coercive estimates of
Section \ref{sec:coercive}, the interpolation procedure of Section
\ref{sec:interpolation} and the adaptions of the Sobolev Poincar\'e type
estimates of \cite{snulmenn.poincare} in
\ref{lemma:lipschitz_approximation}\,\eqref{item:lipschitz_approximation:poincare}
are used. Assuming that $f$ agrees with its ``average'' $g$ on a large set,
for example because the density of the varifold is at least $Q$ on a large
set, the iteration inequality for the tilt-excess, compare
\ref{thm:example_decay}\,\eqref{eqn:second_iteration}, is then primarily a
consequence of Taylor's expansion, see
\ref{lemma:iteration}\,\eqref{item:iteration:tilt_estimate}.  Finally, both
iteration inequalities are iterated in
\ref{lemma:iteration}\,\eqref{item:iteration:iteration} as long as the
afore-mentioned density condition is satisfied on the scales involved, compare
\ref{thm:example_decay}\,\eqref{eqn:iteration}. As all the preceding estimates
only hold under various side conditions which have to be checked at each
iteration step and the interdependence of the various constants occurring is
not entirely straightforward, the iteration procedure is presented in some
detail to ease verification.

Finally, it should be mentioned that the current iteration procedure has to be
carried out within a fixed coordinate systems as differences of functions
corresponding to different iteration steps have to be computed, see the
Introduction and \ref{lemma:iteration}\,\eqref{item:iteration:key_estimate}.
Though this fact does not pose a serious difficulty it nevertheless
contributes significantly to the level of technicality, see for example the
definition of $J_4$ and
\ref{lemma:iteration_prep}\,\eqref{item:iteration_prep:estimate_bad_set}.
\begin{lemma} \label{lemma:aux_monotonicity}
	Suppose $\vdim, \adim \in \nat$, $\vdim \leq \adim$, $a \in
	\rel^\adim$, $0 < r < \infty$, $V \in \Var_\vdim ( \oball{a}{r} )$, $a
	\in \spt \| V \|$, $1 \leq p < \infty$, $0 < \alpha \leq 1$, $0 \leq M
	< \infty$, and
	\begin{gather*}
		\measureball{\| \delta V \|}{\cball{a}{\varrho}} \leq M
		\| V \| (\cball{a}{\varrho})^{1-1/p}
		\varrho^{\vdim/p + \alpha - 1} r^{-\alpha} \quad \text{for $0
		< \varrho < r$}.
	\end{gather*}

	Then
	\begin{gather*}
		\big ( \varrho^{-\vdim} \measureball{\| V
		\|}{\oball{a}{\varrho}} \big )^{1/p} + M p^{-1} \alpha^{-1}
		\varrho^\alpha r^{-\alpha}
	\end{gather*}
	is monotone increasing in $\varrho$ for $0 < \varrho < r$. In
	particular, $0 \leq \density^\vdim ( \| V \|, a ) < \infty$.
\end{lemma}
\begin{proof}
	Suppose $0 < \lambda < 1$ and $\phi \in \mathscr{E}^0 ( \rel )$ with
	$\phi' \leq 0$ and $\phi (t) = 1$ for $-\infty < t \leq \lambda$ and
	$\phi (t) = 0$ for $1 \leq t < \infty$ and $f :
	\classification{\rel}{\varrho}{0 < \varrho < r} \to \rel$ is defined
	by $f ( \varrho ) = \varrho^{-\vdim} \int \phi ( \varrho^{-1} |z-a| )
	\ud \| V \| z$ for $0 < \varrho < r$. Then one obtains as in
	\cite[17.7]{MR87a:49001} that
	\begin{gather*}
		\begin{aligned}
			& f'(\varrho) \geq \varrho^{-\vdim-1} ( \delta V )_z
			\big ( \phi ( \varrho^{-1} |z-a| ) (z-a) \big ) \\
			& \quad \geq - M ( \varrho^{-\vdim} \measureball{\| V
			\|}{\oball{z}{\varrho}} )^{1-1/p} \varrho^{\alpha-1}
			r^{-\alpha} \geq - M \big ( \lambda^{-\vdim} f (
			\lambda^{-1} \varrho ) \big )^{1-1/p}
			\varrho^{\alpha-1} r^{-\alpha}
		\end{aligned}
	\end{gather*}
	for $0 < \varrho < \lambda r$, hence multiplying by $p^{-1}
	f(\varrho)^{1/p-1}$ and integrating yields
	\begin{gather*}
		f ( t )^{1/p} - f (s)^{1/p} \geq - Mp^{-1}r^{-\alpha}
		\tint{s}{t} (\lambda^{-\vdim} f (
		\varrho/\lambda)/f(\varrho))^{1-1/p} \varrho^{\alpha-1} \ud
		\mathscr{L}^1 \varrho
	\end{gather*}
	for $0 < s < t < \lambda r$. Thus, approximating the characteristic
	function of $\classification{\rel}{t}{t<1}$ by such $\phi$ and letting
	$\lambda$ tend to $1$ implies the conclusion.
\end{proof}
\begin{lemma} \label{lemma:lower_mass_bound}
	Suppose $\adim, Q \in \nat$, $0 < \alpha \leq 1$, $1 \leq p < \infty$,
	and $0 < \delta \leq 1$.

	Then there exists a positive, finite number $\varepsilon$ with the
	following property.

	If $\adim > \vdim \in \nat$, $a \in \rel^\adim$, $0 < r < \infty$, $U
	= \classification{\oball{a}{r}}{z}{| \perpproject{T} (z-a) | < \delta
	r}$, $V \in \IVar_\vdim ( U )$, $\psi$ is related to $V$ and $p$ as in
	\ref{miniremark:situation}, $T \in \grass{\adim}{\vdim}$,
	\begin{gather*}
		\density^{\ast \vdim} ( \| V \|, a ) \geq Q-1+\delta, \quad
		\tint{}{} | \project{S} - \project{T} | \ud V (z,S) \leq
		\varepsilon r^\vdim, \\
		\varrho^{1-\vdim/p} \psi ( U \cap \cball{a}{\varrho} )^{1/p}
		\leq \varepsilon ( \varrho/r )^\alpha \quad \text{whenever $0
		< \varrho < r$},
	\end{gather*}
	then
	\begin{gather*}
		\| V \| ( U  ) \geq ( Q-\delta ) \unitmeasure{\vdim} r^\vdim.
	\end{gather*}
\end{lemma}
\begin{proof}
	If the lemma were false for some $\adim$, $Q$, $\alpha$, $p$, and
	$\delta$, there would exist a sequence $\varepsilon_i$ with
	$\varepsilon_i \downarrow 0$ as $i \to \infty$ and sequences
	$\vdim_i$, $a_i$, $r_i$, $U_i$, $V_i$, $\psi_i$, and $T_i$ showing
	that $\varepsilon= \varepsilon_i$ does not have the asserted property.

	One could assume for some $\vdim \in \nat$, $a \in \rel^\adim$, $T \in
	\grass{\adim}{\vdim}$
	\begin{gather*}
		\vdim_i = \vdim, \quad a_i = a, \quad r_i = 1, \quad T_i = T
	\end{gather*}
	whenever $i \in \nat$. Abbreviating $U =
	\classification{\oball{a}{1}}{z}{ | \perpproject{T} (z-a) | < \delta
	}$ one would deduce for large $i$
	\begin{gather*}
		\| V_i \| ( U \cap \oball{a}{\varrho} ) \geq ( Q-1+\delta/2 )
		\unitmeasure{\vdim} \varrho^\vdim \quad \text{whenever $0 <
		\varrho < \delta$}
	\end{gather*}
	from \ref{lemma:aux_monotonicity} in conjunction with H\"older's
	inequality. Clearly, also
	\begin{gather*}
		\| V_i \| ( U ) \leq ( Q - \delta ) \unitmeasure{\vdim} \quad
		\text{for $i \in \nat$}.
	\end{gather*}
	By Allard's compactness theorem for integral varifolds, see
	e.g.~\cite[6.4]{MR0307015} or \cite[42.8]{MR87a:49001}, possibly
	passing to a subsequence, there would exist $V \in \IVar_\vdim ( U )$
	such that $\delta V = 0$ and
	\begin{gather*}
		V_i ( f ) \to V (f) \quad \text{as $i \to \infty$ for $f \in
		\ccspace{U \times \grass{\adim}{\vdim}}$}, \\
		S = T \quad \text{for $V$ almost all $(z,S) \in U \times
		\grass{\adim}{\vdim}$},
	\end{gather*}
	hence, noting \ref{app:lemma:planes},
	\begin{gather*}
		\density^\vdim ( \| V \|, a ) \geq Q, \quad
		\unitmeasure{\vdim} Q \leq \| V \| ( U ) \leq
		\unitmeasure{\vdim} ( Q - \delta ),
	\end{gather*}
	a contradiction.
\end{proof}
\begin{lemma} \label{lemma:iteration_prep}
	Suppose the hypotheses of \ref{lemma:lipschitz_approximation} are
	satisfied with $h=3r$, i.e.~suppose $\vdim, \adim, Q \in \nat$,
	$\vdim< \adim$, $0 < L < \infty$, $1 \leq M < \infty$, and $0 <
	\delta_i \leq 1$ for $i \in \{1,2,3,4,5\}$, $\varepsilon =
	\varepsilon_{\ref{lemma:lipschitz_approximation}} ( \adim, Q, L, M,
	\delta_1, \delta_2, \delta_3, \delta_4, \delta_5 )$, $0 < r < \infty$,
	$T = \im \pp^\ast$,
	\begin{gather*}
		U = \eqclassification{\rel^\vdim \times
		\rel^\codim}{(x,y)}{\dist ((x,y), \cylinder{T}{0}{r}{3r}) <
		2r},
	\end{gather*}
	$V \in \IVar_\vdim ( U )$, $\| \delta V \|$ is a Radon measure,
	\begin{gather*}
		( Q - 1 + \delta_1 ) \unitmeasure{\vdim} r^\vdim \leq \| V \| (
		\cylinder{T}{0}{r}{3r} ) \leq ( Q + 1 - \delta_2 )
		\unitmeasure{\vdim} r^\vdim, \\
		\| V \| ( \cylinder{T}{0}{r}{3r+\delta_4 r} \without
		\cylinder{T}{0}{r}{3r-2\delta_4 r}) \leq ( 1 - \delta_3 )
		\unitmeasure{\vdim} r^\vdim, \\
		\| V \| ( U ) \leq M \unitmeasure{\vdim} r^\vdim,
	\end{gather*}
	$0 < \delta \leq \varepsilon$, $B$ denotes the set of all $z
	\in \cylinder{T}{0}{r}{3r}$ with $\density^{\ast \vdim} ( \| V \|, z) >
	0$ such that
	\begin{gather*}
		\text{either} \quad
		\measureball{\| \delta V \|}{\cball{z}{\varrho}} >
		\delta \, \| V \| ( \cball{z}{\varrho} )^{1-1/\vdim}
		\quad \text{for some $0 < \varrho < 2 r$}, \\
		\text{or} \quad {\textstyle\int_{\cball{z}{\varrho} \times
		\grass{\adim}{\vdim}}} | \project{S} - \project{T} | \ud V
		(\xi,S) > \delta \, \measureball{\| V
		\|}{\cball{z}{\varrho}} \quad \text{for some $0 < \varrho <
		2 r$},
	\end{gather*}
	$A = \cylinder{T}{0}{r}{3r} \without B$, $A (x) =
	\classification{A}{z}{\pp (z) = x}$ for $x \in \rel^\vdim$, $X_1$ is
	the set of all $x \in \rel^\vdim \cap \cball{0}{r}$ such that
	\begin{gather*}
		{\textstyle\sum_{z \in A(x)}} \density^\vdim ( \| V \|, z )
		= Q \quad \text{and} \quad \text{$\density^\vdim ( \| V \|,
		z ) \in \nat \cup \{0\}$ for $z \in A(x)$},
	\end{gather*}
	and $f : X_1 \to \qspace_Q ( \rel^{\codim} )$ is characterised by the
	requirement
	\begin{gather*}
		\density^\vdim ( \| V \|, z) = \density^0 ( \| f (x) \|, \qq
		(z) ) \quad \text{whenever $x \in X_1$ and $z \in A(x)$}.
	\end{gather*}
	Suppose additionally:
	\begin{enumerate}
		\item \label{item:iteration_prep:initial} Suppose $L \leq
		\delta_4 / 8$, $\delta \leq \inf \{ 1, ( 2
		\isoperimetric{\vdim} )^{-1} \}$, $a \in \Int
		\cylinder{T}{0}{r}{3r}$, $c = \pp (a)$, and $0 < \kappa <
		\infty$.
		\item \label{item:iteration_prep:extension} Suppose $F :
		\rel^\vdim \to \qspace_Q ( \rel^\codim )$ with $F | X_1 = f$
		and $\Lip F \leq
		\Gamma_{\eqref{item:iteration_prep:extension}} \Lip f$ where
		$\Gamma_{\eqref{item:iteration_prep:extension}}$ is a
		positive, finite number depending only on $\codim$ and $Q$, see
		\ref{app:thm:bilip_embedding}. Moreover, let $g =
		\boldsymbol{\eta}_Q \circ F$.
		\item Suppose either $p = \vdim = 1$ or $1 \leq p < \vdim$ and
		$p$, $\psi$ are related to $V$ as in
		\ref{miniremark:situation}.
		\item Define $J = \{ \varrho \with 0 < \varrho < \infty \}$
		and
		$\phi_2 : J \times \grass{\adim}{\vdim} \to \rel$
		and $\phi_3 : J \to \rel$, $\phi_4 : J \to \rel$ by
		\begin{gather*}
			\begin{aligned}
				\phi_2 ( \varrho, R ) & = \big (
				\varrho^{-\vdim} {\textstyle\int_{(U \cap
				\cylinder{T}{a}{\varrho}{\delta_4\varrho})
				\times \grass{\adim}{\vdim}}} | \project{S} -
				\project{R} |^2 \ud V (z,S) \big )^{1/2}
				\span & \span & 
				\\
				\phi_3 ( \varrho ) & = \varrho^{1-\vdim/p}
				\psi ( U \cap
				\cylinder{T}{a}{\varrho}{\delta_4\varrho}
				)^{1/p} \span & \span &
				\\
				\phi_4 ( \varrho ) & = \delta^{-\vdim
				p/(\vdim-p)} \phi_3 ( \varrho )^{\vdim p /
				(\vdim-p)} & \text{if $\vdim > 1$}, \\
				\phi_4 ( \varrho ) & = 0 &
				\text{if $\vdim = 1$},
			\end{aligned}
		\end{gather*}
		whenever $\varrho \in J$, $R \in
		\grass{\adim}{\vdim}$.\footnote{The symbol $\phi_1$ will
		denote the leading iteration quantity introduced in
		\ref{lemma:iteration}\,\eqref{item:iteration:leading_quantity}.}
		\item For $0 < \varrho < \infty$ suppose $T_\varrho \in
		\grass{\adim}{\vdim}$ is defined such that
		\begin{gather*}
			\phi_2 ( \varrho, T_\varrho ) \leq \phi_2 ( \varrho, R
			) \quad \text{whenever $R \in \grass{\adim}{\vdim}$}.
		\end{gather*}
		\item Define
		\begin{align*}
			J_0 & = \classification{J}{\varrho}{0 < \varrho \leq r -
			| \pp (a) |,|\qq(a)| + \delta_4 \varrho \leq 3r}, \\
			J_1 & = \classification{J }{\varrho}{\pp \lIm
			T_\varrho \rIm = \rel^\vdim} \\
			J_2 & = \classification{J}{\varrho}{\| \delta V \| ( U
			\cap \cylinder{T}{a}{\varrho}{\delta_4\varrho} ) \leq
			\kappa \varrho^{\vdim-1}}, \\
			J_3 & =
			\classification{J}{\varrho}{{\textstyle\int_{(U \cap
			\cylinder{T}{a}{\varrho}{\delta_4\varrho}) \times
			\grass{\adim}{\vdim}}} | \project{S} - \project{T} |
			\ud V (z,S) \leq \kappa \varrho^\vdim}, \\
			J_4 & = \classification{J}{\varrho}{\text{$\varrho +
			t/\delta_4 \in J_2 \cap J_3$ for $0 \leq t < 2r$}}, \\
			J_5 & = \classification{J_0}{\varrho}{\| V \| (
			\cylinder{T}{a}{\varrho}{\delta_4 \varrho/4} ) \geq
			\unitmeasure{\vdim} ( Q-1/4) \varrho^\vdim}.
		\end{align*}
		and $T_\varrho = \sigma_\varrho \in \Hom ( \rel^\vdim ,
		\rel^\codim )$ for $ \varrho \in J_1$.
		\item Define $B_{a,\varrho}$, and $C_{a,\varrho}$ for $\varrho
		\in J_0$ as in
		\ref{lemma:lipschitz_approximation}\,\eqref{item:lipschitz_approximation:measurability},
		i.e.
		\begin{gather*}
			B_{a,\varrho} = \cylinder{T}{a}{\varrho}{\delta_4
			\varrho} \cap B, \quad C_{a,\varrho} = \cball{\pp
			(a)}{\varrho} \without ( X_1 \without \pp \lIm
			B_{a,\varrho} \rIm ),
		\end{gather*}
		and $H$ as in
		\ref{lemma:lipschitz_approximation}\,\eqref{item:lipschitz_approximation:poincare},
		i.e.~$H$ denotes the set of all $z \in
		\cylinder{T}{0}{r}{3r}$ such that
		\begin{gather*}
			\measureball{\| \delta V \|}{\oball{z}{2r}} \leq
			\varepsilon \, \| V \| ( \oball{z}{2r} )^{1-1/\vdim},
			\\
			{\textstyle\int_{\oball{z}{2r} \times
			\grass{\adim}{\vdim}}} | \project{S} - \project{T} |
			\ud V (z,S) \leq \varepsilon \, \measureball{\| V
			\|}{\oball{z}{2r}}, \\
			\measureball{\| V \|}{\cball{z}{\varrho}} \geq
			\delta_5 \unitmeasure{\vdim} \varrho^\vdim \quad
			\text{for $0 < \varrho < 2r$}.
		\end{gather*}
		\intertextenum{Then the following six conclusions hold:}
		\item \label{item:iteration_prep:estimate_bad_set} There
		exists a positive finite number
		$\varepsilon_{\eqref{item:iteration_prep:estimate_bad_set}}$
		depending only on $\vdim$, $\delta_4$, and $\delta$ with the
		following property.
		
		If $R \in \grass{\adim}{\vdim}$, $| \project{R} - \project{T}
		| \leq \delta / 2$, $\varrho \in J_0 \cap J_4$, $\kappa \leq
		\varepsilon_{\eqref{item:iteration_prep:estimate_bad_set}}$,
		then
		\begin{gather*}
			\varrho^{-\vdim} \| V \|  ( B_{a,\varrho} ) \leq
			2^\vdim \besicovitch{\adim} \left ( 4 \delta^{-2}
			\phi_2 ( 2 \varrho, R )^2 + \phi_4 ( 2 \varrho )
			\right ).
		\end{gather*}
		Moreover, $4 \delta^{-2} \phi_2 ( 2\varrho, R )^2$ may be
		replaced by $\delta^{-1} \kappa$.
		\item \label{item:iteration_prep:H_is_H} There exists a
		positive, finite number
		$\varepsilon_{\eqref{item:iteration_prep:H_is_H}}$ depending
		only on $\vdim$, $\delta_4$, $\delta_5$, and $\varepsilon$
		with the following property.

		If $8r/\delta_4 \in J_2 \cap J_3$ and $\kappa \leq
		\varepsilon_{\eqref{item:iteration_prep:H_is_H}}$, then $H$ is
		the set of all $z \in \cylinder{T}{0}{r}{3r}$ such that
		\begin{gather*}
			\measureball{\| V \|}{\cball{z}{t}} \geq \delta_5
			\unitmeasure{\vdim} t^\vdim \quad \text{whenever $0 <
			t < 2r$}.
		\end{gather*}
		\item \label{item:iteration_prep:lower_mass_bound} If $0 <
		\alpha \leq 1$ and $0 < \delta_6 \leq 1$ then there exists a
		positive, finite number
		$\varepsilon_{\eqref{item:iteration_prep:lower_mass_bound}}$
		depending only on $\adim$, $Q$, $\delta_4$, $p$, $\alpha$, and
		$\delta_6$ with the following property.

		If $\density^{\ast \vdim} ( \| V \|, a ) \geq Q-1+\delta_6$,
		$\varrho \in J_0 \cap J_3$, $\kappa \leq
		\varepsilon_{\eqref{item:iteration_prep:lower_mass_bound}}$,
		and
		\begin{gather*}
			\phi_3 ( t ) \leq
			\varepsilon_{\eqref{item:iteration_prep:lower_mass_bound}}
			(t/\varrho)^\alpha \quad \text{for $0 < t < \varrho$},
		\end{gather*}
		then $\varrho \in J_5$.
		\item \label{item:iteration_prep:side_conditions} There exists
		a positive, finite number
		$\varepsilon_{\eqref{item:iteration_prep:side_conditions}}$
		depending only on $\adim$, $\delta_4$, and $\delta$ with the
		following three properties.

		\begin{enumerate}
			\item
			\label{item:item:iteration_prep:side_conditions:upper_bound}
			If $\varrho \in J_0 \cap J_4$, $\kappa \leq
			\varepsilon_{\eqref{item:iteration_prep:side_conditions}}$,
			and $\phi_4 ( 2 \varrho ) \leq 2^{-\vdim}
			\besicovitch{\adim}^{-1} \unitmeasure{\vdim} (1/8)$,
			then
			\begin{gather*}
				\| V \| (
				\cylinder{T}{a}{\varrho}{\delta_4\varrho} )
				\leq ( Q + 1/2 ) \unitmeasure{\vdim}
				\varrho^\vdim.
			\end{gather*}
			
			\item
			\label{item:item:iteration_prep:side_conditions:inclusion} If, additionally to the conditions of
			\eqref{item:item:iteration_prep:side_conditions:upper_bound},
			$\varrho \in J_5$, then
			\begin{gather*}
				\graph_Q f | \cball{c}{\varrho} \subset
				\cylinder{T}{a}{\varrho}{\delta_4 \varrho/ 2
				}.
			\end{gather*}

			\item
			\label{item:item:iteration_prep:side_conditions:estimate}
			If, additionally to the conditions of
			\eqref{item:item:iteration_prep:side_conditions:upper_bound}
			and
			\eqref{item:item:iteration_prep:side_conditions:inclusion},
			$0 < \lambda < \infty$,
			\begin{gather*}
				\kappa \leq 2^{-\vdim}
				\besicovitch{\adim}^{-1} \unitmeasure{\vdim}
				\lambda ( 2
				\Gamma_{\ref{lemma:lipschitz_approximation}\eqref{item:lipschitz_approximation:estimate_b}}
				( Q, \vdim ) )^{-1} \delta, \\
				\phi_4 ( 2 \varrho ) \leq 2^{-\vdim}
				\besicovitch{\adim}^{-1} \unitmeasure{\vdim}
				\lambda ( 2
				\Gamma_{\ref{lemma:lipschitz_approximation}\eqref{item:lipschitz_approximation:estimate_b}}
				( Q, \vdim ))^{-1},
			\end{gather*}
			then
			\begin{gather*}
				\mathscr{L}^\vdim ( C_{a,\varrho} ) \leq
				\lambda \unitmeasure{\vdim} \varrho^\vdim.
			\end{gather*}
		\end{enumerate}
		\item \label{item:iteration_prep:graph_tilt} If $\varrho \in
		J_4 \cap J_5$, $\kappa \leq \min \{
		\varepsilon_{\eqref{item:iteration_prep:estimate_bad_set}} (
		\vdim, \delta_4, \delta ),
		\varepsilon_{\eqref{item:iteration_prep:side_conditions}} (
		\adim, \delta_4, \delta ) \}$, and
		\begin{gather*}
			\sigma \in \Hom ( \rel^\vdim, \rel^\codim ), \quad \|
			\sigma \| \leq \adim^{-1/2} \delta/ 2, \quad
			\sigma = R \in \grass{\adim}{\vdim},
		\end{gather*}
		then
		\begin{gather*}
			\varrho^{-\vdim} {\textstyle\int_{\oball{c}{\varrho}}}
			| AF (x) \aplus ( - \sigma ) |^2 \ud \mathscr{L}^\vdim
			x \leq \Gamma_{\eqref{item:iteration_prep:graph_tilt}}
			\big ( \phi_2 ( 2 \varrho, R )^2 + \phi_4 ( 2 \varrho
			) \big )
		\end{gather*}
		where $\Gamma_{\eqref{item:iteration_prep:graph_tilt}}$ is a
		positive, finite number depending only on $\adim$, $Q$, and
		$\delta$.
		\item \label{item:iteration_prep:ch_planes} If $ \varrho \in
		J_0 \cap J_1$, $\varrho/8 \leq s \leq t \leq \varrho$, $0 <
		\lambda \leq 1$, and
		\begin{gather*}
			\| \sigma_\varrho \| \leq \adim^{-1/2} /4,
			\quad \phi_2 ( \varrho, T_\varrho ) \leq \lambda^{1/2}
			2^{-2\vdim-3} \unitmeasure{\vdim}^{1/2},  \\
			\| V \| ( \cylinder{T}{a}{s}{\delta_4 s} ) \geq
			\lambda \unitmeasure{\vdim} s^\vdim,
		\end{gather*}
		then $t \in J_1$ and
		\begin{gather*}
			\| \sigma_\varrho - \sigma_t \| \leq \lambda^{-1/2}
			2^{2\vdim+2} \unitmeasure{\vdim}^{-1/2} \phi_2 (
			\varrho, T_\varrho ).
		\end{gather*}
	\end{enumerate}
\end{lemma}
\begin{proof} [Proof of \eqref{item:iteration_prep:estimate_bad_set}]
	Let
	\begin{gather*}
		\varepsilon_{\eqref{item:iteration_prep:estimate_bad_set}} =
		\inf \big \{ (1/2) ( 4 \isoperimetric{\vdim} \vdim )^{1-\vdim}
		( \delta_4)^{\vdim-1} \delta, ( 4 \isoperimetric{\vdim} \vdim
		)^{-\vdim} ( \delta_4)^\vdim \delta \big \}.
	\end{gather*}

	Define the sets $B_{a,\varrho}'$ and $B_{a,\varrho}''$ by
	\begin{align*}
		B_{a,\varrho}' & =
		\classification{B_{a,\varrho}}{z}{\text{$\measureball{\|
		\delta V \|}{\cball{z}{t}} > \delta \, \| V \| ( \cball{z}{t}
		)^{1-1/\vdim}$ for some $0 < t < 2r$}}, \\
		B_{a,\varrho}'' & = B_{a,\varrho} \without B_{a,\varrho}'
	\end{align*}
	and $D$ to be the set of all $z \in \spt \| V \|$ such that
	\begin{gather*}
		\limsup_{t \pluslim{0}} \frac{\measureball{\| \delta V
		\|}{\cball{z}{t}}}{\| V \| ( \cball{z}{t} )^{1-1/\vdim}} > 0.
	\end{gather*}
	Note $\| V \| ( D ) = 0$ by \cite[2.9.5]{MR41:1976} or
	\cite[4.7]{MR87a:49001}.

	First, the following assertion will be shown. \emph{If $\vdim = 1$
	then $B_{a,\varrho}' \without D = \emptyset$ and if $\vdim > 1$ then
	for $z \in B_{a,\varrho}' \without D$ there exists $0 < t < \delta_4
	\varrho$ such that
	\begin{gather*}
		\measureball{\| V \|}{\cball{z}{t}} \leq \delta^{-\vdim
		p/(\vdim-p)} \psi ( \cball{z}{t} )^{\vdim/(\vdim-p)}.
	\end{gather*}}
	For this purpose assume $z \in B_{a,\varrho}' \without D$ and define
	\begin{gather*}
		t = \inf \big \{ s \with \measureball{\| \delta V
		\|}{\cball{z}{s}} > \delta \, \| V \| ( \cball{z}{s}
		)^{1-1/\vdim} \big \}.
	\end{gather*}
	One infers $0 < t < 2r$ and
	\begin{gather*}
		\measureball{\| \delta V \|}{\cball{z}{t}} \geq \delta \, \| V
		\| ( \cball{z}{t} )^{1-1/\vdim} \geq ( \delta / \Delta_1 )
		t^{\vdim - 1}
	\end{gather*}
	by \ref{app:lemma:good_point} where $\Delta_1 = ( 2
	\isoperimetric{\vdim} \vdim )^{\vdim-1}$ since $\delta \leq ( 2
	\isoperimetric{\vdim} )^{-1}$. Noting
	\begin{gather*}
		\varrho + t / \delta_4 \in J_2, \quad \cball{z}{t} \subset U
		\cap \cylinder{T}{a}{\varrho+t/\delta_4}{\delta_4 ( \varrho +
		t / \delta_ 4 )},
	\end{gather*}
	one obtains
	\begin{gather*}
		( \delta/ \Delta_1 ) t^{\vdim-1} \leq \kappa ( \varrho + t /
		\delta_4 )^{\vdim-1}, \quad \vdim > 1, \\
		t \leq ( \varrho + t / \delta_4 ) ( \kappa \Delta_1/ \delta
		)^{1/(\vdim-1)} < ( \varrho + t / \delta_4 ) \delta_4 / 2,
		\quad t < \delta_4 \varrho.
	\end{gather*}
	The assertion now follows from the definition of $t$ in conjunction
	with H\"older's inequality.
	
	The preceding assertion yields
	\begin{align*}
		\| V \| ( B_{a,\varrho}' ) & = 0 \quad \text{if $\vdim = 1$},
		\\
		\| V \| ( B_{a,\varrho}' ) & \leq \delta^{-\vdim p/(\vdim-p)}
		\besicovitch{\adim} \psi ( U \cap
		\cylinder{T}{a}{2\varrho}{2\delta_4\varrho})^{\vdim/(\vdim-p)}
		\quad \text{if $\vdim > 1$};
	\end{align*}
	in fact if $\vdim > 1$ there exist countable disjointed families
	$F_1, \ldots, F_{\besicovitch{\adim}}$ of closed balls such that
	\begin{gather*}
		B_{a,\varrho}' \without D \subset {\textstyle\bigcup\bigcup}
		\{F_i \with i = 1, \ldots, \besicovitch{\adim} \}, \\
		\| V \| (S) \leq \Delta_2 \psi (S)^{\vdim/ (\vdim-p)}, \quad S
		\subset U \cap \cylinder{T}{a}{2\varrho}{2\delta_4\varrho}
	\end{gather*}
	whenever $S \in \bigcup \{ F_i \with i = 1, \ldots,
	\besicovitch{\adim} \}$ where $\Delta_2 = \delta^{-\vdim p/ ( \vdim-p
	)}$, hence
	\begin{gather*}
		\| V \| ( B_{a,\varrho}' ) = \| V \| ( B_{a,\varrho}' \without
		D ) \leq \Delta_2 {\textstyle\sum_{i=1}^{\besicovitch{\adim}}}
		{\textstyle\sum_{S \in F_i}} \psi (S)^{\vdim/ (\vdim-p)} \\
		\leq \Delta_2 {\textstyle\sum_{i=1}^{\besicovitch{\adim}}}
		\big ( {\textstyle\sum_{S \in F_i}} \psi (S) \big )^{\vdim /
		(\vdim-p)} \leq \Delta_2 \besicovitch{\adim} \psi ( U \cap
		\cylinder{T}{a}{2\varrho}{2\delta_4\varrho} )^{\vdim / (
		\vdim-p )}.
	\end{gather*}

	Next, it will be shown that \emph{for $z \in B_{a,\varrho}''$
	there exists $0 < t \leq \delta_4 \varrho$ such that
	\begin{align*}
		\measureball{\| V \|}{\cball{z}{t}} & \leq 4 \delta^{-2}
		{\textstyle\int_{\cball{z}{t} \times \grass{\adim}{\vdim}}} |
		\project{S} - \project{R} |^2 \ud V (z,S), \\
		\measureball{\| V \|}{\cball{z}{t}} & < \delta^{-1}
		{\textstyle\int_{\cball{z}{t} \times \grass{\adim}{\vdim}}} |
		\project{S} - \project{T} | \ud V (z,S).
	\end{align*}}
	In fact, one can take any $0 < t < 2r$ satisfying the last inequality
	since this firstly implies, using \ref{app:lemma:good_point},
	$\delta \leq ( 2 \isoperimetric{\vdim} )^{-1}$ and $\varrho + t/
	\delta_4 \in J_3$,
	\begin{gather*}
		( 2 \isoperimetric{\vdim}\vdim )^{-\vdim} t^\vdim \leq
		\measureball{\| V \|}{\cball{z}{t}} < \delta^{-1}
		{\textstyle\int_{\cball{z}{t} \times \grass{\adim}{\vdim}}} |
		\project{S} - \project{T} | \ud V (z,S) \\
		\leq \delta^{-1} {\textstyle\int_{(U \cap
		\cylinder{T}{a}{\varrho+t/\delta_4}{\delta_4
		(\varrho+t/\delta_4)}) \times \grass{\adim}{\vdim}}} |
		\project{S} - \project{T} | \ud V (z,S) \leq ( \kappa / \delta
		) ( \varrho + t / \delta_4)^\vdim, \\
		t \leq ( 2 \isoperimetric{\vdim} \vdim ) ( \kappa / \delta
		)^{1/\vdim} ( \varrho + t / \delta_4 ) \leq ( \varrho + t /
		\delta_4 ) \delta_4 / 2, \quad t \leq \delta_4 \varrho,
	\end{gather*}
	and secondly, using $| \project{R} - \project{T} | \leq \delta/2$ and
	H\"older's inequality,
	\begin{align*}
		\measureball{\| V \|}{\cball{z}{t}} & \leq 2 \delta^{-1}
		{\textstyle\int_{\cball{z}{t} \times \grass{\adim}{\vdim}}} |
		\project{S} - \project{R} | \ud V (z,S), \\
		\measureball{\| V \|}{\cball{z}{t}} & \leq 4 \delta^{-2}
		{\textstyle\int_{\cball{z}{t} \times \grass{\adim}{\vdim}}} |
		\project{S} - \project{R} |^2 \ud V (z,S).
	\end{align*}

	Since $2 \varrho \in J_3$ and
	\begin{gather*}
		\cball{z}{t} \subset U \cap
		\cylinder{T}{a}{2\varrho}{2\delta_4\varrho} \quad
		\text{whenever $z \in B_{a,\varrho}''$, $0 < t \leq \delta_4
		\varrho$},
	\end{gather*}
	the assertion implies
	\begin{align*}
		\| V \| ( B_{a,\varrho}'' ) & \leq 4 \delta^{-2}
		\besicovitch{\adim} {\textstyle\int_{( U \cap
		\cylinder{T}{a}{2\varrho}{2\delta_4\varrho} ) \times
		\grass{\adim}{\vdim}}} | \project{S} - \project{R} |^2 \ud V
		(z,S), \\
		\| V \| ( B_{a,\varrho}'' ) & \leq \besicovitch{\adim}
		\delta^{-1} \kappa ( 2 \varrho )^\vdim.
	\end{align*}
	and the conclusion follows.
\end{proof}
\begin{proof} [Proof of \eqref{item:iteration_prep:H_is_H}]
	Defining
	\begin{gather*}
		\varepsilon_{\eqref{item:iteration_prep:H_is_H}} = \varepsilon
		\inf \{ 4^{1-\vdim} ( \delta_4 )^{\vdim-1} ( \delta_5
		\unitmeasure{\vdim} )^{1-1/\vdim}, 4^{-\vdim} ( \delta_4
		)^\vdim \delta_5 \unitmeasure{\vdim} \},
	\end{gather*}
	one estimates for $z \in \cylinder{T}{0}{r}{3r}$
	\begin{gather*}
		\begin{aligned}
			\measureball{\| \delta V \|}{\oball{z}{2r}} & \leq \|
			\delta V \| ( U \cap \cylinder{T}{a}{4r}{8r} ) \\
			& \leq \kappa ( 8r/ \delta_4)^{\vdim -1} \leq
			\varepsilon \big ( \delta_5 \unitmeasure{\vdim} ( 2r
			)^\vdim \big )^{1-1/\vdim},
		\end{aligned} \\
		\begin{aligned}
			{\textstyle\int_{\oball{z}{2r} \times
			\grass{\adim}{\vdim}}} | \project{S} - \project{T} |
			\ud V (z,S) & \leq {\textstyle\int_{( U \cap
			\cylinder{T}{a}{4r}{8r} ) \times
			\grass{\adim}{\vdim}}} | \project{S} - \project{T} |
			\ud V ( z,S ) \\
			& \leq \kappa ( 8r/\delta_4 )^\vdim \leq \varepsilon
			\delta_5 \unitmeasure{\vdim} ( 2r )^\vdim
		\end{aligned}
	\end{gather*}
	and the conclusion follows.
\end{proof}
\begin{proof} [Proof of \eqref{item:iteration_prep:lower_mass_bound}]
	Defining
	$\varepsilon_{\eqref{item:iteration_prep:lower_mass_bound}} = (
	\delta_4 )^\adim \varepsilon_{\ref{lemma:lower_mass_bound}} ( \adim,
	Q, \alpha, p, \inf \{ \delta_6, \delta_4/4 \} )$ and noting
	\begin{gather*}
		\begin{aligned}
			& \psi ( \classification{\cball{a}{t}}{z}{\dist (z-a,
			T ) < \delta_4 \varrho / 4} )^{1/p} \leq \psi (
			\cylinder{T}{a}{t}{\delta_4 \inf \{t/\delta_4,
			\varrho/4\}})^{1/p} \\
			& \qquad \leq
			\varepsilon_{\eqref{item:iteration_prep:lower_mass_bound}}
			( t / \delta_4 )^{\vdim/p+\alpha - 1}
			\varrho^{-\alpha} \leq
			\varepsilon_{\eqref{item:iteration_prep:lower_mass_bound}}
			( \delta_4 )^{-\vdim / p} t^{\vdim/p+\alpha-1}
			\varrho^{-\alpha}
		\end{aligned}
	\end{gather*}
	for $0 < t < \varrho$, the assertion follows from
	\ref{lemma:lower_mass_bound} with $\delta$, $r$ replaced by $\inf \{
	\delta_6, \delta_4 / 4 \}$, $\varrho$.
\end{proof}
\begin{proof} [Proof of \eqref{item:iteration_prep:side_conditions}]
	Define $\varepsilon_{\eqref{item:iteration_prep:side_conditions}}$ to
	be the infimum of all numbers
	\begin{gather*}
		\inf \big \{ 2^{-\adim} \besicovitch{\adim}^{-1}
		\unitmeasure{i} (1/8) \delta, 2^{-3} \adim^{-1}
		\unitmeasure{i},
		\varepsilon_{\eqref{item:iteration_prep:estimate_bad_set}} (
		i, \delta_4, \delta ) \big \}
	\end{gather*}
	corresponding to $\adim > i \in \nat$.
	
	If the conclusion of
	\eqref{item:item:iteration_prep:side_conditions:inclusion} were not
	true, one would infer
	\begin{gather*}
		\spt f ( x ) \without \cball{\qq(a)}{\delta_4\varrho/4}
		\neq \emptyset, \\
		{\textstyle\sum_{y \in
		\cball{\qq(a)}{\delta_4 \varrho/ 4} \cap \spt f ( x )}}
		\density^0 ( \| f ( x ) \|, y ) \leq Q-1
	\end{gather*}
	whenever $x \in \dmn f | \cball{c}{\varrho}$ by
	\eqref{item:iteration_prep:initial} and
	\ref{lemma:lipschitz_approximation}\,\eqref{item:lipschitz_approximation:lip}
	and therefore by
	\ref{lemma:lipschitz_approximation}\,\eqref{item:lipschitz_approximation:yz}\,\eqref{item:lipschitz_approximation:ab}
	and the coarea formula, see e.g.~\cite[3.2.22\,(3)]{MR41:1976} or
	\cite[12.7]{MR87a:49001}, one would obtain
	\begin{gather*}
		{\textstyle\int_{\cylinder{T}{a}{\varrho}{\delta_4 \varrho/4}
		\cap A} \| \bigwedge_\vdim ( \pp | S ) \| \ud V (z,S) \leq (
		Q-1 ) \unitmeasure{\vdim} \varrho^\vdim},
	\end{gather*}
	hence by \ref{app:miniremark:tilt} and
	\eqref{item:iteration_prep:estimate_bad_set} with $R$ replaced by $T$,
	noting $\varrho \in J_4 \subset J_3$,
	\begin{gather*}
		\begin{aligned}
			& \| V \| (
			\cylinder{T}{a}{\varrho}{\delta_4\varrho/4} ) - ( Q-1
			) \unitmeasure{\vdim} \varrho^\vdim \\
			& \qquad \leq \| V \| ( B_{a,\varrho} ) + 2\vdim
			{\textstyle\int_{\cylinder{T}{a}{\varrho}{\delta_4\varrho/4}}}
			| \project{S} - \project{T} | \ud V (z,S) \leq ( 1/2 )
			\unitmeasure{\vdim} \varrho^\vdim
		\end{aligned}
	\end{gather*}
	in contradiction to $\varrho \in J_5$.

	Using similarly
	\begin{gather*}
		\tsum{y \in A(x)}{} \density^0 ( \| V \| , (x,y) ) \leq Q \quad
		\text{for $x \in X_1 \cup X_2$},
	\end{gather*}
	one obtains
	\eqref{item:item:iteration_prep:side_conditions:upper_bound}.

	To prove \eqref{item:item:iteration_prep:side_conditions:estimate},
	one estimates with
	\ref{lemma:lipschitz_approximation}\,\eqref{item:lipschitz_approximation:estimate_b}
	and \eqref{item:iteration_prep:estimate_bad_set} with $R$ replaced by
	$T$
	\begin{gather*}
		\mathscr{L}^\vdim ( C_{a,\varrho} ) \leq
		\Gamma_{\ref{lemma:lipschitz_approximation}\eqref{item:lipschitz_approximation:estimate_b}}
		( Q, m ) \| V \| ( B_{a,\varrho} ) \leq \lambda
		\unitmeasure{\vdim} \varrho^\vdim. \qedhere
	\end{gather*}
\end{proof}
\begin{proof} [Proof of \eqref{item:iteration_prep:graph_tilt}]
	Denote by $X_1'$ the set of all $x \in X_1$ such that
	\ref{lemma:lipschitz_approximation}\,\eqref{item:lipschitz_approximation:misc}
	is true for $x$ and note $\mathscr{L}^\vdim ( X_1 \without X_1' ) =
	0$. Since
	\begin{gather*}
		| \ap AF (x) \aplus (-\sigma) | \leq ( 1 + \Lip F ) ( Q \vdim
		)^{1/2} \leq ( 1 +
		\Gamma_{\eqref{item:iteration_prep:extension}} ( \codim, Q ) )
		( Q \vdim )^{1/2}
	\end{gather*}
	for $x \in \dmn \ap AF$, one may assume
	\begin{gather*}
		\phi_4 ( 2 \varrho ) \leq 2^{-\vdim} \besicovitch{\adim}^{-1}
		\unitmeasure{\vdim} (1/8).
	\end{gather*}

	Next, it will shown with $G = \graph_Q f$
	\begin{gather*}
		\cball{c}{\varrho} \cap \classification{X_1'}{x}{|\ap Af(x)
		\aplus (-\sigma) | > \gamma} \\
		\subset \pp \biglIm \cylinder{T}{a}{\varrho}{\delta_4\varrho}
		\cap \classification{G}{z}{| \project{\Tan^\vdim (
		\| V \|, z)} - \project{R} | > 2^{-1} (Q\vdim)^{-1/2} \gamma}
		\bigrIm
	\end{gather*}
	whenever $0 < \gamma < \infty$. In fact, if $x$ is a member of the
	first set there exist $y \in \spt f (x)$ and $\tau \in \Hom (
	\rel^\vdim, \rel^\codim )$ such that
	\begin{gather*}
		\tau = \Tan^\vdim ( \| V \|, (x,y) ), \quad |\tau-\sigma| >
		Q^{-1/2} \gamma,
	\end{gather*}
	hence, noting $\| \sigma \| \leq 1$ and $\big \| \project{\Tan^\vdim
	( \| V \|, (x,y) )} - \project{T} \big \| \leq \| \tau \| \leq L \leq
	1/2$ by \ref{miniremark:projections},
	\begin{gather*}
		\| \sigma - \tau \| \leq 2 \, \big \| \project{\Tan^\vdim ( \|
		V \|, (x,y) )} - \project{R} \big \|
	\end{gather*}
	by \ref{miniremark:projections}, and the inclusion follows, since
	$(x,y) \in \cylinder{T}{a}{\varrho}{\delta_4\varrho}$ by
	\eqref{item:item:iteration_prep:side_conditions:inclusion}. Therefore,
	since $\density^\vdim ( \| V \|, z ) \geq 1$ for $z \in G$,
	\begin{align*}
		& \| V \| (
		\classification{\cylinder{T}{a}{\varrho}{\delta_4\varrho}}{z}{|
		\project{\Tan^\vdim ( \| V \|, z )} - \project{R}| > 2^{-1} (Q
		\vdim)^{-1/2} \gamma} ) \\
		& \quad \geq \mathscr{H}^\vdim (
		\cylinder{T}{a}{\varrho}{\delta_4 \varrho} \cap
		\classification{G}{z}{| \project{\Tan^\vdim ( \| V
		\|, z)} - \project{R} | > 2^{-1} (Q\vdim)^{-1/2} \gamma} ) \\
		& \quad \geq \mathscr{L}^\vdim ( \cball{c}{\varrho} \cap
		\classification{X_1}{x}{| \ap A f(x) \aplus ( - \sigma ) | >
		\gamma} )
	\end{align*}
	and one obtains
	\begin{gather*}
		\varrho^{-\vdim} {\textstyle\int_{\oball{c}{\varrho} \cap
		X_1}} | \ap Af (x) \aplus (-\sigma) |^2 \ud \mathscr{L}^\vdim
		\leq 2^{\vdim+2} Q \vdim \, \phi_2 ( 2\varrho, R )^2.
	\end{gather*}

	Recalling the first paragraph of the proof, and noting
	\begin{gather*}
		| \project{R} - \project{T} | \leq \adim^{1/2} \| \project{R}
		- \project{T} \| \leq \adim^{1/2} \| \sigma \| \leq \delta/2
	\end{gather*}
	by \ref{miniremark:projections} and $ \oball{c}{\varrho} \without
	X_1 \subset C_{a,\varrho}$, the conclusion follows combining
	\eqref{item:item:iteration_prep:side_conditions:inclusion},
	\eqref{item:iteration_prep:estimate_bad_set} and
	\ref{lemma:lipschitz_approximation}\,\eqref{item:lipschitz_approximation:estimate_b}.
\end{proof}
\begin{proof} [Proof of \eqref{item:iteration_prep:ch_planes}]
	Using H\"older's inequality, one obtains
	\begin{gather*}
		\begin{aligned}
			| \eqproject{T_t} -\eqproject{T_\varrho} | & \leq \| V
			\| ( \cylinder{T}{a}{s}{\delta_4s})^{-1/2} \big (
			t^{\vdim/2} \phi_2 ( t, T_t) + \varrho^{\vdim/2}
			\phi_2 ( \varrho, T_\varrho ) \big ) \\
			& \leq \lambda^{-1/2} 2^{2\vdim+1}
			\unitmeasure{\vdim}^{-1/2} \phi_2 ( \varrho,
			T_\varrho),
		\end{aligned}
	\end{gather*}
	since $t^{\vdim/2} \phi_2 ( t , T_t ) \leq \varrho^{\vdim/2} \phi_2 (
	\varrho, T_\varrho)$. Noting by \ref{miniremark:projections}
	\begin{gather*}
		\begin{aligned}
			| \eqproject{T_t} - \project{T} | & \leq |
			\eqproject{T_t} - \eqproject{T_\varrho}  | +
			| \eqproject{T_\varrho} - \project{T}| \\
			& \leq \lambda^{-1/2} 2^{2\vdim+1}
			\unitmeasure{\vdim}^{-1/2} \phi_2 ( \varrho, T_\varrho
			) + \adim^{1/2} \| \sigma_\varrho \| \leq 1/2,
		\end{aligned} \\
		\| \eqproject{T_t} - \project{T} \| \leq 1/2, \quad
		T_t \cap \ker \pp = \{ 0 \}, \quad t \in J_1,
	\end{gather*}
	one applies \ref{miniremark:projections} with $S$, $S_1$, $S_2$
	replaced by $T$, $T$, $T_t$ to infer
	\begin{gather*}
		\| \sigma_{t} \|^2 \leq ( 1 + \| \sigma_{t}
		\|^2 ) \| \eqproject{T_{t}} - \project{T} \|^2, \\
		\| \sigma_{t} \|^2 \leq \| \eqproject{T_{t}} -
		\project{T}
		\|^2 / ( 1 - \| \eqproject{T_{t}} - \project{T} \|^2 )
		\leq 2 \| \eqproject{T_{t}} - \project{T} \|^2 \leq
		1/2,
	\end{gather*}
	Now, \ref{miniremark:projections} with $S$, $S_1$, $S_2$ replaced
	by $T$, $T_t$, $T_\varrho$ implies
	\begin{gather*}
		\| \sigma_t - \sigma_\varrho \| \leq 2 | \eqproject{T_t} -
		\eqproject{T_\varrho} |. \qedhere
	\end{gather*}
\end{proof}
\begin{lemma} \label{lemma:iteration}
	Suppose $\vdim$, $\adim$, $Q$, $L$, $M$, $\delta_1$, $\delta_2$,
	$\delta_3$, $\delta_4$, $\delta_5$, $\varepsilon$, $r$, $T$, $U$, $V$,
	$\delta$, $X_1$, $f$, $a$, $c$, $\kappa$, $F$, $p$, $\psi$, $J$,
	$\phi_2$, $\phi_3$, $\phi_4$, $T_\varrho$, $J_0$, $J_1$, $J_2$, $J_3$,
	$J_4$, $J_5$, and $\sigma_\varrho$ are as in
	\ref{lemma:iteration_prep}.  Suppose additionally:
	\begin{enumerate}
		\item Suppose $\Psi$ and $C$ are as in
		\ref{miniremark:function_C}.
		\item Whenever $\varrho \in J_1$ suppose $u_\varrho$ denotes
		the unique analytic function in $\Sob{}{1}{2} (
		\oball{c}{\varrho}, \rel^\codim )$ such that
		\begin{gather*}
			\left < D^2 u_\varrho (x), C ( \sigma_\varrho ) \right
			> = 0 \quad \text{for $x \in \oball{c}{\varrho}$}, \\
			u_\varrho - g \in \Sob{0}{1}{2} ( \oball{c}{\varrho},
			\rel^\codim ),
		\end{gather*}
		see \ref{miniremark:function_C}--\ref{remark:standard_w012}
		and \cite[5.1.2,\,10]{MR41:1976}.
		\item \label{item:iteration:leading_quantity} Define the
		function $\phi_1 : J_1 \to \rel$ by $\phi_1 ( \varrho ) =
		\norm{D^2u_\varrho}{\infty}{c,\varrho/2}$ for $\varrho \in
		J_1$.
		\item Suppose $0 < \tau \leq 1$ and $\tau = 1$ if $\vdim = 1$,
		$p/2 \leq \tau < \frac{\vdim p}{2 (\vdim-p)}$ if $\vdim =
		2$ and $\tau = \frac{\vdim p}{2 (\vdim-p)}$ if $\vdim > 2$.
		\intertextenum{Then the following seven conclusions hold:}
		\item \label{item:iteration:coerciveness}
		There exists a positive, finite number
		$\Gamma_{\eqref{item:iteration:coerciveness}}$ depending
		only on $\adim$ such that
		\begin{gather*}
			\text{$D^2 \Psi_0^\S ( \sigma )$ is strongly elliptic
			with ellipticity bound
			$(\Gamma_{\eqref{item:iteration:coerciveness}})^{-1}$},
			\\
			\| D^2 \Psi_0^\S ( \sigma ) \| \leq
			\Gamma_{\eqref{item:iteration:coerciveness}}
		\end{gather*}
		whenever $\sigma \in \Hom ( \rel^\vdim, \rel^\codim )$ with
		$\| \sigma \| \leq 1$.
		\item \label{item:iteration:starting_control} If $\varrho
		\in J_4 \cap J_5$, $2 \varrho \in J_0 \cap J_1$, $\|
		\sigma_{2\varrho} \| \leq \adim^{-1/2} \inf \{ \delta /2 , 1/4
		\}$, and
		\begin{gather*}
			\phi_2 ( 2 \varrho, T_{2\varrho} ) \leq 2^{-2\vdim-4}
			\unitmeasure{\vdim}^{1/2}, \\
			\kappa \leq \inf \{
			\varepsilon_{\ref{lemma:iteration_prep}\eqref{item:iteration_prep:estimate_bad_set}}
			( \vdim, \delta_4, \delta ),
			\varepsilon_{\ref{lemma:iteration_prep}\eqref{item:iteration_prep:side_conditions}}
			( \adim, \delta_4, \delta ) \},
		\end{gather*}
		then
		\begin{gather*}
			\phi_1 ( \varrho ) \leq
			\Gamma_{\eqref{item:iteration:starting_control}}
			\varrho^{-1} \big ( \phi_2 ( 2\varrho, T_{2\varrho} ) +
			\phi_4 ( 2 \varrho )^{1/2} \big )
		\end{gather*}
		where $\Gamma_{\eqref{item:iteration:starting_control}}$
		is a positive, finite number depending only on $\adim$, $Q$,
		and $\delta$.
		\item \label{item:iteration:l1_estimate}
		If $\varrho \in J_1 \cap J_4 \cap J_5$, $\| \sigma_\varrho
		\|\leq 1$, $2 \varrho \in J_1$, $\|
		\sigma_{2\varrho} \| \leq \adim^{-1/2} \delta/2$,
		\begin{gather*}
			\kappa \leq \inf \{
			\varepsilon_{\ref{lemma:iteration_prep}\eqref{item:iteration_prep:estimate_bad_set}}
			( \vdim, \delta_4, \delta ),
			\varepsilon_{\ref{lemma:iteration_prep}\eqref{item:iteration_prep:side_conditions}}
			( \adim, \delta_4, \delta ) \}, \\
			\phi_4 ( 2 \varrho ) \leq 2^{-\vdim}
			\besicovitch{\adim}^{-1} \unitmeasure{\vdim} ( 1/8 ),
		\end{gather*}
		then
		\begin{gather*}
			\varrho^{-\vdim-1} \norm{u_\varrho- g}{1}{c,\varrho}
			\leq \Gamma_{\eqref{item:iteration:l1_estimate}}
			\big ( \phi_2 ( 2 \varrho, T_{2\varrho})^2 + \phi_3 (
			2 \varrho ) \big )
		\end{gather*}
		where $\Gamma_{\eqref{item:iteration:l1_estimate}}$ is a
		positive, finite number depending only on $\vdim$, $\adim$, $Q$,
		$\delta_4$, $\delta$, and $p$.
		\item \label{item:iteration:key_estimate} There exists a
		positive, finite number
		$\varepsilon_{\eqref{item:iteration:key_estimate}}$
		depending only on $\adim$, $\delta_4$, and $\delta$ with the
		following property.
		
		If $\varrho \in J$, $2 \varrho \in J_0 \cap J_1$, $\|
		\sigma_{2\varrho} \| \leq \adim^{-1/2} \delta/ 4$, $\kappa
		\leq \varepsilon_{\eqref{item:iteration:key_estimate}}$, and
		for $s \in \{ \varrho/4, \varrho \}$
		\begin{gather*}
			s \in J_4 \cap J_5, \quad \phi_4 ( 2 s ) \leq
			2^{-\vdim} \besicovitch{\adim}^{-1}
			\unitmeasure{\vdim} ( 1/8 ),
		\end{gather*}
		then
		\begin{gather*}
			\phi_1 ( \varrho / 4 ) \leq \phi_1 ( \varrho ) +
			\Gamma_{\eqref{item:iteration:key_estimate}} \big
			( \phi_1 ( \varrho ) \phi_2 ( \varrho, T_\varrho ) +
			\varrho^{-1} ( \phi_2 ( 2 \varrho, T_{2\varrho})^2 +
			\phi_3 ( 2\varrho ) ) \big )
		\end{gather*}
		where $\Gamma_{\eqref{item:iteration:key_estimate}}$ is
		a positive, finite number depending only on $\vdim$, $\adim$,
		$Q$, $\delta_4$, $\delta$ and $p$.
		\item \label{item:iteration:prep_tilt} There exists a
		positive, finite number
		$\varepsilon_{\eqref{item:iteration:prep_tilt}}$ depending
		only on $\vdim$, $\adim$, $Q$, $\delta_2$, $\varepsilon$, 
		$\delta$, and $p$ with the following property.

		If $\delta_4 = 1$, $\delta_5 = ( 40 )^{-\vdim} (
		\isoperimetric{\vdim} \vdim )^{-\vdim} / \unitmeasure{\vdim}$,
		$0 < \eta < 2^{-\vdim}$, $P : \rel^\vdim \to \rel^\codim$ is
		affine, $\Lip P \leq \adim^{-1/2} \delta/2$, $R = \im D (
		\pp^\ast  + \qq^\ast \circ P ) (0)$, $\varrho \in J$, $X$ is
		an $\mathscr{L}^\vdim$ measurable subset of
		$\oball{c}{\varrho/2} \cap X_1$,
		\begin{gather*}
			\mu = 1/2 \quad \text{if $\vdim = 1$}, \quad \mu
			=1/\vdim \quad \text{if $\vdim > 1$}, \\
			\varrho/2 \in J_4 \cap J_5, \quad 8r \in J_2 \cap J_3,
			\quad \varrho \in J_1, \quad \| \sigma_\varrho \| \leq
			\adim^{-1/2} \delta/2, \\
			\kappa \leq
			\varepsilon_{\eqref{item:iteration:prep_tilt}}, \quad
			\phi_3 ( \varrho ) \leq
			\varepsilon_{\eqref{item:iteration:prep_tilt}}, \quad
			\mathscr{L}^\vdim ( \oball{c}{\varrho/2} \without X  )
			\leq \eta \unitmeasure{\vdim} ( \varrho/2 )^\vdim,
		\end{gather*}
		then for $0 < \lambda \leq 1$
		\begin{multline*}
			\phi_2 ( \varrho/4, R ) \leq
			\Gamma_{\eqref{item:iteration:prep_tilt}} \Big ( \big
			( \lambda + \phi_2 ( \varrho, T_\varrho)^{2/\vdim}
			\big) \phi_2 ( \varrho, T_\varrho ) + ( \lambda +
			\eta^\mu ) \phi_2 ( \varrho, R ) \\
			+ \eta^{-1} \varrho^{-\vdim-1} \norm{f \aplus
			(-P)}{1}{X} + \lambda^{-\tau} \phi_3 ( \varrho )^\tau
			\Big)
		\end{multline*}
		where $\Gamma_{\eqref{item:iteration:prep_tilt}}$ is a
		positive, finite number depending only on $\vdim$, $\adim$,
		$Q$, $\delta$, $p$, and $\tau$.
		\item \label{item:iteration:tilt_estimate} There exists a
		positive, finite number
		$\varepsilon_{\eqref{item:iteration:tilt_estimate}}$ depending
		only on $\vdim$, $\adim$, $Q$, $\delta_2$, $\varepsilon$,
		$\delta$, and $p$ with the following property.

		If $\delta_4 = 1$, $\delta_5 = ( 40 )^{-\vdim} (
		\isoperimetric{\vdim} \vdim )^{-\vdim} / \unitmeasure{\vdim}$,
		$0 < \eta < 2^{-\vdim}$, $\varrho \in J$,
		\begin{gather*}
			\mu = 1/2 \quad \text{if $\vdim = 1$}, \quad \mu
			=1/\vdim \quad \text{if $\vdim > 1$}, \\
			\{ \varrho/2, \varrho \} \subset J_4 \cap J_5, \quad
			2\varrho \in J_0 \cap J_1, \quad \| \sigma_{2\varrho}
			\| \leq \adim^{-1/2} \delta/4, \\
			8r \in J_2 \cap J_3, \quad \kappa \leq
			\varepsilon_{\eqref{item:iteration:tilt_estimate}}
			, \quad \phi_3 ( 2 \varrho ) \leq
			\varepsilon_{\eqref{item:iteration:tilt_estimate}},
			\\
			\mathscr{L}^\vdim ( \oball{c}{\varrho/2} \without \{ x
			\with \density^0 ( \| f (x) \|, g (x) ) = Q \} ) \leq
			\eta \unitmeasure{\vdim} ( \varrho/2 )^\vdim,
		\end{gather*}
		then for $0 < \lambda \leq 1$
		\begin{multline*}
			\phi_2 ( \varrho/4, T_{\varrho/4} ) \leq
			\Gamma_{\eqref{item:iteration:tilt_estimate}}
			\Big ( \big ( \lambda + \eta^\mu + \eta^{-1} \phi_2 (
			2\varrho, T_{2\varrho} )^{\inf \{ 1, 2/ \vdim \}} \big
			) \phi_2 ( 2\varrho, T_{2\varrho} ) \\
			+ \eta^{-1} \varrho \phi_1 ( \varrho ) + ( \eta^{-1} +
			\lambda^{-\tau} ) \phi_3 ( 2\varrho )^\tau \Big)
		\end{multline*}
		where $\Gamma_{\eqref{item:iteration:tilt_estimate}}$ is a
		positive, finite number depending only on $\vdim$, $\adim$,
		$Q$, $\delta$, $p$, and $\tau$.
		\item \label{item:iteration:iteration} Let $\delta_4 = 1$,
		$\delta_5 = ( 40 )^{-\vdim} ( \isoperimetric{\vdim} \vdim
		)^{-\vdim} / \unitmeasure{\vdim}$, $\delta = \inf \{ 1,
		\varepsilon, ( 2 \isoperimetric{\vdim} )^{-1} \}$, $0 < \alpha
		\leq 1$, and $0 < \delta_6 \leq 1$.
		
		Then there positive, finite numbers $\gamma_i$ for $i \in \{
		1,2,3 \}$ and a positive, finite number
		$\varepsilon_{\eqref{item:iteration:iteration}}$ both
		depending only on $\vdim$, $\adim$, $Q$, $L$, $M$, $\delta_1$,
		$\delta_2$, $\delta_3$, $p$, $\tau$, $\alpha$, and $\delta_6$
		with the following property.

		If $a \in \cylinder{T}{0}{r/2}{2r}$, $\density^{\ast\vdim} (
		\| V \|, a ) \geq Q-1+\delta_6$, $0 < t \leq \frac{r}{64}$, $0
		< \gamma \leq 1$,
		\begin{gather*}
			\phi_2 ( 8r, T ) \leq
			\varepsilon_{\eqref{item:iteration:iteration}},
			\quad \phi_2 ( 8r, T_{8r} ) \leq
			\varepsilon_{\eqref{item:iteration:iteration}}
			\gamma, \\
			\| V \| (
			\classification{\cylinder{T}{a}{\varrho}{\varrho}} {z}
			{\density^\vdim ( \| V \|, z ) \leq Q-1} ) \leq
			\varepsilon_{\eqref{item:iteration:iteration}}
			\unitmeasure{\vdim} \varrho^\vdim
		\end{gather*}
		whenever $t \leq \varrho \leq r/8$, and
		\begin{gather*}
			\phi_3 ( \varrho )^\tau \leq \gamma \gamma_3 ( \varrho
			/ r )^{\alpha \tau} \quad \text{whenever $0 < \varrho
			\leq 8r$},
		\end{gather*}
		then, in case $\alpha \tau < 1$,
		\begin{gather*}
			\varrho \in J_1 \quad \text{and} \quad \varrho \phi_1
			( \varrho ) \leq \gamma \gamma_1 ( \varrho / r
			)^{\alpha \tau} \quad \text{for $t \leq \varrho \leq
			r/4$}, \\
			\phi_2 ( \varrho, T_\varrho ) \leq \gamma \gamma_2 (
			\varrho  / r )^{\alpha \tau} \quad \text{for $t \leq
			\varrho \leq r$}
		\end{gather*}
		and, in case $\alpha \tau = 1$,
		\begin{gather*}
			\varrho \in J_1 \quad \text{and} \quad \varrho \phi_1
			( \varrho ) \leq \gamma \gamma_1 ( \varrho / r
			) ( 1 + \log ( r/\varrho ) ) \quad \text{for $t \leq
			\varrho \leq r/4$}, \\
			\phi_2 ( \varrho, T_\varrho ) \leq \gamma \gamma_2 (
			\varrho  / r ) ( 1 + \log ( r/\varrho )) \quad
			\text{for $t \leq \varrho \leq r$}.
		\end{gather*}
	\end{enumerate}
\end{lemma}
\begin{proof} [Proof of \eqref{item:iteration:coerciveness}]
	This follows from \cite[5.1.2,\,10]{MR41:1976}.
\end{proof}
\begin{proof} [Proof of \eqref{item:iteration:starting_control}]
	Note by \ref{lemma:iteration_prep}\,\eqref{item:iteration_prep:ch_planes} applied with $\varrho$,
	$s$, $t$, $\lambda$ replaced by $2\varrho$, $\varrho$, $\varrho$,
	$1/2$
	\begin{gather*}
		\varrho \in J_1, \quad \| \sigma_\varrho \| \leq \|
		\sigma_{2\varrho} \| + 2^{2\vdim+3} \unitmeasure{\vdim}^{-1/2}
		\phi_2 ( 2\varrho, T_{2\varrho} ) \leq 1.
	\end{gather*}
	Since $u_\varrho - \sigma_{2\varrho}$ is $D^2 \Psi_0^\S (
	\sigma_\varrho )$ harmonic, applying \cite[5.2.5]{MR41:1976} yields,
	noting \eqref{item:iteration:coerciveness},
	\begin{gather*}
		\norm{D^2 u_\varrho}{\infty}{c,\varrho/2} \leq \Delta_1
		\varrho^{-1-\vdim/2} \norm{D ( u_\varrho - \sigma_{2\varrho}
		)}{2}{c,\varrho}
	\end{gather*}
	where $\Delta_1 = 2^{\adim+5} \adim^{\adim+1}
	\Gamma_{\eqref{item:iteration:coerciveness}}  ( \adim )^\adim
	\sup \{ \unitmeasure{i}^{-1/2} \with \adim > i \in \nat \}$. Using
	\ref{lemma:standard_w012}, one obtains
	\begin{gather*}
		\norm{D ( u_\varrho - \sigma_{2\varrho} )}{2}{c,\varrho} \leq
		\norm{D ( u_\varrho - g )}{2}{c,\varrho} + \norm{D (
		g-\sigma_{2\varrho} )}{2}{c,\varrho} \leq \Delta_2 \norm{D (
		g-\sigma_{2\varrho} )}{2}{c,\varrho}
	\end{gather*}
	where $\Delta_2 = 1 +
	\Gamma_{\eqref{item:iteration:coerciveness}} ( \adim )^2$. Taking
	$\Gamma_{\eqref{item:iteration:starting_control}} =
	\Delta_1 \Delta_2
	\Gamma_{\ref{lemma:iteration_prep}\eqref{item:iteration_prep:graph_tilt}} ( \adim, Q,
	\delta)^{1/2}$, the conclusion now follows from
	\ref{lemma:iteration_prep}\,\eqref{item:iteration_prep:graph_tilt}
	with $\sigma$ replaced by $\sigma_{2\varrho}$.
\end{proof}
\begin{proof} [Proof of \eqref{item:iteration:l1_estimate}]
	Suppose $B$, and $B_{a,t}$, $C_{a,t}$ for $t \in J_0$ are as in
	\ref{lemma:iteration_prep}. Define $S, R \in \mathscr{D}' (
	\oball{c}{\varrho}, \rel^\codim )$ by
	\begin{gather*}
		\begin{aligned}
			S ( \theta ) & = - \tint{\oball{c}{\varrho}}{}
			\leftB D \theta (x), D \Psi_0^\S ( D g (x) ) \rightB
			\ud \mathscr{L}^\vdim x, \\
			R ( \theta ) & = - \tint{\oball{c}{\varrho}}{}
			\leftB D \theta (x) \odot D g (x), D^2 \Psi_0^\S (
			\sigma_\varrho ) \rightB \ud \mathscr{L}^\vdim x
		\end{aligned}
	\end{gather*}
	whenever $\theta \in \mathscr{D} ( \oball{c}{\varrho}, \rel^\codim )$.
	Since $u_\varrho$ is $D^2 \Psi_0^\S ( \sigma_\varrho)$ harmonic,
	\begin{gather} \label{eqn:l1_estimate:4}
		\norm{u_\varrho - g}{1}{c,\varrho} \leq \Delta_1 \varrho
		\dnorm{R}{1}{c,\varrho}
	\end{gather}
	by \ref{lemma:l1_estimate} and \eqref{item:iteration:coerciveness}
	where $\Delta_1 = \Gamma_{\ref{lemma:l1_estimate}} ( \adim,
	\Gamma_{\eqref{item:iteration:coerciveness}} ( \adim)^{-1},
	\Gamma_{\eqref{item:iteration:coerciveness}} ( \adim ) )$.  One
	computes for $x \in \dmn D g$
	\begin{gather*}
		D \Psi_0^\S ( D g (x) ) - D \Psi_0^\S ( \sigma_\varrho ) - ( D
		g (x)
		- \sigma_\varrho ) \mathop{\lrcorner} D^2 \Psi_0^\S (
		  \sigma_\varrho ) \\
		= ( D g (x) - \sigma_\varrho ) \mathop{\lrcorner} \tint{0}{1}
		D^2 \Psi_0^\S ( t D g (x) + (1-t) \sigma_\varrho ) - D^2
		\Psi_0^\S ( \sigma_\varrho ) \ud \mathscr{L}^1 t, \\
		\begin{aligned}
			& \| D^2 \Psi_0^\S ( t D g (x) + (1-t) \sigma_\varrho
			) - D^2 \Psi_0^\S ( \sigma_\varrho ) \| \\
			& \qquad \leq \Lip ( D^2 \Psi_0^\S |
			\cball{0}{\gamma}) \, t | D g (x) - \sigma_\varrho |
			\qquad \text{for $0 \leq t \leq 1$}
		\end{aligned}
	\end{gather*}
	where $\gamma = \vdim^{1/2} \sup
	\{1,\Gamma_{\eqref{item:iteration_prep:extension}} ( \codim, Q ) \}$,
	hence, since
	\begin{gather*}
		\tint{\oball{c}{\varrho}}{} \left < D \theta (x), \beta \right >
		\ud \mathscr{L}^\vdim x = 0
	\end{gather*}
	for $\theta \in \mathscr{D} ( \oball{c}{\varrho}, \rel^\codim )$ and
	$\beta \in \{ D \Psi_0^\S ( \sigma_\varrho), \sigma_\varrho
	\mathop{\lrcorner} D^2 \Psi_0^\S ( \sigma_\varrho ) \}$,
	\begin{gather*}
		\varrho^{-\vdim} \dnorm{S - R}{1}{c,\varrho}
		\leq \Delta_2 \varrho^{-\vdim} \tint{\oball{c}{\varrho}}{} | D
		g (x) - \sigma_\varrho |^2 \ud \mathscr{L}^\vdim x
	\end{gather*}
	where $\Delta_2$ is a positive, finite number depending only on
	$\adim$ and $Q$. Therefore by
	\ref{lemma:iteration_prep}\,\eqref{item:iteration_prep:graph_tilt}
	with $\sigma$ replaced by $\sigma_{2\varrho}$
	\begin{gather} \label{eqn:l1_estimate:5}
		\varrho^{-\vdim}\dnorm{S - R}{1}{c,\varrho}
		\leq \Delta_3 \big ( \phi_2 ( 2\varrho, T_{2\varrho} )^2 +
		\phi_4 ( 2\varrho ) \big )
	\end{gather}
	where $\Delta_3 = \Delta_2
	\Gamma_{\ref{lemma:iteration_prep}\eqref{item:iteration_prep:graph_tilt}} ( \adim, Q, \delta )$.

	Let $\theta \in \mathscr{D} ( \oball{c}{\varrho}, \rel^\codim )$ with
	$\norm{D\theta}{\infty}{c,\varrho} \leq 1$ and $\eta \in \mathscr{D}^0
	( \rel^\codim )$ with
	\begin{gather*}
		\spt \eta \subset \oball{\qq (a)}{\delta_4 \varrho}, \quad
		\cball{\qq(a)}{\delta_4 \varrho/2} \subset \Int (
		\classification{\rel^\codim}{y}{\eta (y) = 1} ), \\
		0 \leq \eta (y) \leq 1, \quad | D \eta (y) | \leq 4 (
		\delta_4)^{-1} \varrho^{-1} \qquad \text{for $y \in
		\rel^\codim$}.
	\end{gather*}
	From
	\ref{lemma:lipschitz_approximation}\,\eqref{item:lipschitz_approximation:pde}
	with $\tau$ replaced by $\sigma_{2\varrho}$
	one infers with $D_{a,\varrho} =
	\cylinder{T}{a}{\varrho}{\delta_4\varrho} \cap \pp^{-1} \lIm
	C_{a,\varrho} \rIm$, noting
	\ref{lemma:iteration_prep}\,\eqref{item:item:iteration_prep:side_conditions:inclusion}
	and $\norm{\theta}{\infty}{c,\varrho} \leq \varrho$,
	\begin{gather*}
		\begin{aligned}
			& | Q S ( \theta ) + ( \delta V ) ( ( \eta
			\circ \qq) \cdot ( \qq^\ast \circ \theta \circ \pp ) )
			| \\
			& \qquad \leq \Delta_4 \big ( \mathscr{L}^\vdim (
			C_{a,\varrho} ) + \tint{\oball{c}{\varrho}}{} | AF (x)
			\aplus ( - \sigma_{2\varrho} ) |^2 \ud
			\mathscr{L}^\vdim x + \| V \| (
			D_{a,\varrho} ) \big )
		\end{aligned}
	\end{gather*}
	where $\Delta_4$ is a positive, finite number depending only on
	$\adim$, $Q$, and $\delta_4$. By
	\ref{lemma:lipschitz_approximation}\,\eqref{item:lipschitz_approximation:estimate_b},
	noting \ref{lemma:iteration_prep}\,\eqref{item:item:iteration_prep:side_conditions:inclusion}, and
	\ref{lemma:iteration_prep}\,\eqref{item:iteration_prep:graph_tilt}
	with $\sigma$ replaced by $\sigma_{2\varrho}$
	\begin{gather*}
		\begin{aligned}
			& \quad \varrho^{-\vdim} | Q S ( \theta ) + ( \delta
			V ) ( ( \eta \circ \qq ) \cdot ( \qq^\ast \circ \theta
			\circ \pp ) ) | \\
			& \leq \Delta_4
			\Gamma_{\ref{lemma:lipschitz_approximation}\eqref{item:lipschitz_approximation:estimate_b}}
			( Q, \vdim ) \varrho^{-\vdim} \| V \| ( B_{a,\varrho}
			) + \Delta_4
			\Gamma_{\ref{lemma:iteration_prep}\eqref{item:iteration_prep:graph_tilt}} (
			\adim, Q, \delta ) \big ( \phi_2 ( 2\varrho,
			T_{2\varrho} )^2  + \phi_4 (2\varrho) \big )
		\end{aligned}
	\end{gather*}
	Therefore one obtains in view of
	\ref{lemma:iteration_prep}\,\eqref{item:iteration_prep:estimate_bad_set}, since $|
	\eqproject{T_\varrho} - \project{T} | \leq \adim^{1/2} \|
	\eqproject{T_{2\varrho}} - \project{T} \| \leq \adim^{1/2} \|
	\sigma_{2\varrho} \| \leq \delta/2$ by \ref{miniremark:projections},
	\begin{gather} \label{eqn:l1_estimate:6}
		\varrho^{-\vdim} | Q S ( \theta ) + ( \delta V ) ( (
		\eta \circ \qq ) \cdot ( \qq^\ast \circ \theta \circ \pp ) ) |
		\leq \Delta_5 \big ( \phi_2 ( 2\varrho, T_{2\varrho} )^2 +
		\phi_4 (2\varrho) \big )
	\end{gather}
	where $\Delta_5$ is a positive, finite number depending only on
	$\adim$, $Q$, $\delta_4$, and $\delta$. Also, using
	\ref{lemma:iteration_prep}\,\eqref{item:item:iteration_prep:side_conditions:upper_bound} and
	H\"older's inequality, recalling $\norm{\theta}{\infty}{c,\varrho}
	\leq \varrho$,
	\begin{gather} \label{eqn:l1_estimate:7}
		\varrho^{-\vdim} | ( \delta V ) ( ( \eta \circ \qq ) \cdot (
		\qq^\ast \circ \theta \circ \pp ) ) | \leq (
		\unitmeasure{\vdim} (Q+1/2))^{1-1/p} \phi_3 (\varrho).
	\end{gather}
	Finally, noting
	\begin{gather*}
		\phi_3 ( 2 \varrho ) = \delta \phi_4 ( 2 \varrho
		)^{\frac{\vdim-p}{\vdim p}} \leq \delta \big ( 2^{-\vdim}
		\besicovitch{\adim}^{-1} \unitmeasure{\vdim} (1/8) \big
		)^{\frac{\vdim-p}{\vdim p}} \quad \text{if $\vdim > 1$}, \\
		\phi_4 (2\varrho) \leq \Delta_6 \phi_3 ( 2 \varrho )
	\end{gather*}
	where $\Delta_6 = \delta^{-1} \big ( 2^{-\vdim}
	\besicovitch{\adim}^{-1} \unitmeasure{\vdim} ( 1/8 ) \big
	)^{1-\frac{\vdim-p}{\vdim p}}$, the conclusion may be obtained by
	combining \eqref{eqn:l1_estimate:4}, \eqref{eqn:l1_estimate:5},
	\eqref{eqn:l1_estimate:6} and \eqref{eqn:l1_estimate:7}.
\end{proof}
\begin{proof} [Proof of \eqref{item:iteration:key_estimate}]
	Define $\varepsilon_{\eqref{item:iteration:key_estimate}}$ to be
	the infimum of all numbers
	\begin{gather*}
		\inf \big \{
		\varepsilon_{\ref{lemma:iteration_prep}\eqref{item:iteration_prep:estimate_bad_set}} (
		i, \delta_4, \delta ),
		\varepsilon_{\ref{lemma:iteration_prep}\eqref{item:iteration_prep:side_conditions}} (
		\adim, \delta_4, \delta ), 2^{-4\adim-5} \adim^{-2}
		\unitmeasure{i} \delta^2 \big \}
	\end{gather*}
	corresponding to $\adim > i \in \nat$.

	Noting
	\begin{gather*}
		\phi_1 ( \varrho/4 ) \leq \phi_1 ( \varrho ) + \norm{D^2 (
		u_{\varrho/4} - u_\varrho )}{\infty}{c,\varrho/8},
	\end{gather*}
	only $\norm{D^2 (u_{\varrho/4}-u_\varrho)}{\infty}{c,\varrho/8}$ needs
	to be estimated. Since $\varrho < 2r$ as $2\varrho \in J_0$ and
	$\varrho \in J_4$, one notes
	\begin{gather*}
		2 \varrho \in J_3, \quad \phi_2 ( 2 \varrho, T_{2\varrho} )
		\leq \phi_2 ( 2 \varrho, T) \leq ( 2 \vdim^{1/2} \kappa
		)^{1/2}.
	\end{gather*}
	Therefore one may apply \ref{lemma:iteration_prep}\,\eqref{item:iteration_prep:ch_planes} for each
	$t \in \{ \varrho/4, \varrho/2, \varrho \}$ with $\varrho$, $s$,
	$\lambda$ replaced by $2\varrho$, $\varrho/4$, $1/2$ to obtain $\{
	\varrho/4,\varrho/2,\varrho\} \subset J_1$ and
	\begin{gather*}
		\sup \{ \| \sigma_{\varrho/4} \|, \| \sigma_{\varrho/2} \|, \|
		\sigma_\varrho \| \} \leq \| \sigma_{2\varrho} \| + 2^{2\vdim
		+ 3} \unitmeasure{\vdim}^{-1/2} \phi_2 ( 2\varrho,
		T_{2\varrho} ) \leq \adim^{-1/2} \delta /2.
	\end{gather*}
	Computing for $x \in \oball{c}{\varrho/4}$
	\begin{gather*}
		\left < D^2 ( u_\varrho - u_{\varrho/4} )(x), C (
		\sigma_{\varrho/4} ) \right > = \left < D^2 u_\varrho (x), C (
		\sigma_{\varrho/4} ) - C ( \sigma_\varrho ) \right >,
	\end{gather*}
	one infers from \ref{lemma:interior_c2a_estimate} with $c$, $M$,
	$\Upsilon$, $\alpha$, $a$, $r$, and $u$ replaced by
	$\Gamma_{\eqref{item:iteration:coerciveness}} ( \adim )^{-1}$,
	$\Gamma_{\eqref{item:iteration:coerciveness}} ( \adim )$ $D^2
	\Psi_0^\S ( \sigma_{\varrho/4} )$, $1/2$, $c$, $\varrho/4$, and
	$u_\varrho-u_{\varrho/4}$ that
	\begin{align*}
		& \norm{D^2 (u_\varrho-u_{\varrho/4})}{\infty}{c,\varrho/8} \\
		& \qquad \leq \Delta_1 \big ( \varrho^{-2-\vdim}
		\norm{u_\varrho - u_{\varrho/4}}{1}{c,\varrho/4} +
		\varrho^{1/2} \hoelder{1/2}{D^2 u_\varrho |
		\cball{c}{\varrho/4}} \| \sigma_{\varrho/4} - \sigma_\varrho
		\| \big )
	\end{align*}
	where $\Delta_1$ is a positive, finite number depending only on
	$\adim$. Since
	\begin{gather*}
		\varrho^{1/2} \hoelder{1/2}{D^2
		u_\varrho|\cball{c}{\varrho/4}} \leq \Delta_2 \phi_1 ( \varrho
		)
	\end{gather*}
	by \cite[5.2.5]{MR41:1976} and
	\eqref{item:iteration:coerciveness} for some positive, finite
	number $\Delta_2$ depending only on $\adim$, the conclusion now
	follows, noting
	\ref{lemma:iteration_prep}\,\eqref{item:iteration_prep:ch_planes}, by
	applying \eqref{item:iteration:l1_estimate} twice, once with
	$\varrho$ as given and once with $\varrho$ replaced by $\varrho/4$.
\end{proof} \setcounter{equation}{0}
\begin{proof} [Proof of \eqref{item:iteration:prep_tilt}]
	Define $q= \infty$ if $\vdim = 1$ and $q = ( \frac{1}{2\tau} +
	\frac{1}{2} - \frac{1}{p} )^{-1}$ if $\vdim > 1$ and note $2 \leq q <
	\infty$ if $\vdim = 2$ and $q = 2 \vdim / ( \vdim-2 )$ if $\vdim > 2$
	and
	\begin{gather*}
		1/p + 1/q \geq 1, \quad \tau = ( 2 ( 1/p + 1/q) - 1 )^{-1}.
	\end{gather*}
	With $\delta_4 = 1$ and $\delta_5 = ( 40 )^{-\vdim} (
	\isoperimetric{\vdim} \vdim )^{-\vdim} / \unitmeasure{\vdim}$ define
	\begin{gather*}
		\Delta_1 = \inf \big \{
		\begin{aligned}[t]
			&
			\varepsilon_{\ref{lemma:iteration_prep}\eqref{item:iteration_prep:estimate_bad_set}}
			( \vdim, \delta_4, \delta ),
			\varepsilon_{\ref{lemma:iteration_prep}\eqref{item:iteration_prep:H_is_H}}
			( \vdim, \delta_4, \delta_5, \varepsilon ),
			\varepsilon_{\ref{lemma:iteration_prep}\eqref{item:iteration_prep:side_conditions}}
			( \adim, \delta_4, \delta ), \\
			& 2^{-\vdim} \besicovitch{\adim}^{-1}
			\unitmeasure{\vdim}
			\varepsilon_{\ref{lemma:lipschitz_approximation}\eqref{item:lipschitz_approximation:poincare}}
			( \vdim, \delta_2, \delta_4 ) ( 2
			\Gamma_{\ref{lemma:lipschitz_approximation}\eqref{item:lipschitz_approximation:estimate_b}}
			( Q, \vdim ) )^{-1} \delta \big \},
		\end{aligned} \\
		\Delta_2 = \inf \big \{ 1, ( 2 \isoperimetric{1} )^{-1} \big
		\}, \\
		\Delta_3 = \inf \big \{ 1, 2^{-\vdim} \besicovitch{\adim}^{-1}
		\unitmeasure{\vdim} \inf \{
		\varepsilon_{\ref{lemma:lipschitz_approximation}\eqref{item:lipschitz_approximation:poincare}}
		( \vdim, \delta_2, \delta_4 ) ( 2
		\Gamma_{\ref{lemma:lipschitz_approximation}\eqref{item:lipschitz_approximation:estimate_b}}
		( Q, \vdim ))^{-1}, 1/ 8 \} \big \}, \\
		\varepsilon_{\eqref{item:iteration:prep_tilt}} = \inf \big \{
		\Delta_1, 2^{-1} \vdim^{-1/2}, \Delta_2, \delta ( \Delta_3
		)^{1/p-1/\vdim} \big \}, \quad \Delta_4 = \sup \{ 2^\vdim
		\Gamma_{\ref{lemma:Q_function_with_holes}} ( \adim, Q, q ), 1 \},
		\\
		\Delta_5 = \sup \big \{ 2
		\Gamma_{\ref{lemma:Q_function_with_holes}} ( \adim, Q, \infty ),
		2^\vdim \Gamma_{\ref{lemma:Q_function_with_holes}} ( \adim, Q, 2),
		1 \big \}, \quad \Delta_6  =
		\Gamma_{\ref{lemma:iteration_prep}\eqref{item:iteration_prep:graph_tilt}}
		( \adim, Q, \delta )^{1/2} \delta^{-\tau}, \\
		\Delta_7 = \sup \{ Q
		\Gamma_{\ref{lemma:lipschitz_approximation}\eqref{item:lipschitz_approximation:poincare}}
		( \vdim ) , 1 \}, \quad \Delta_8 = 2^{\vdim+2} \delta^{-2}
		\besicovitch{\adim}, \\
		\Delta_9 = 19/(2^{1/2} \cdot 40 + 19), \quad \Delta_{10} =
		\Gamma_{\ref{lemma:coercive_estimate}} ( \vdim, p, q ) \quad
		\text{if $\vdim = 1$}, \\
		\Delta_{10} = \Gamma_{\ref{lemma:coercive_estimate_rect}} (
		\vdim, p, q ) \quad \text{if $\vdim > 1$}, \\
		\Delta_{11} = 2^\vdim \sup \big \{ 2 ( \Delta_{10})^{1/2}, 2 (
		16 + 4 \vdim )^{1/2} | \Delta_9 - 1/4 |^{-1} \big \}, \\
		\Delta_{12}  = \big ( 4 ( \Delta_4 + \Delta_5 )
		\Delta_7 ( \Delta_8 )^2 \delta^{-\tau} + 1 \big )
		\Delta_{11}, \quad \Gamma_{\eqref{item:iteration:prep_tilt}} =
		\Delta_{12} ( 2 + \Delta_6 ).
	\end{gather*}
	It will be shown that
	$\varepsilon_{\eqref{item:iteration:prep_tilt}}$ and
	$\Gamma_{\eqref{item:iteration:prep_tilt}}$ have the asserted
	property.

	Suppose $B$, $A$, $B_{a,t}$, $C_{a,t}$, and $H$ for $t \in J_0$ are as
	in \ref{lemma:iteration_prep}.

	Since $\varrho/2 \in J_0 \cap J_4$, it follows
	\begin{gather*}
		\varrho/2 < 2r, \quad \varrho \in J_3, \quad \phi_2 ( \varrho,
		T_\varrho ) \leq \phi_2 ( \varrho, T ) \leq ( 2 \vdim^{1/2}
		\kappa )^{1/2}.
	\end{gather*}
	One therefore obtains
	\begin{gather} \label{eqn:prep_tilt:1}
		\kappa \leq \Delta_1, \quad \phi_2 ( \varrho, T_\varrho ) \leq
		1, \quad \phi_3 ( \varrho) \leq \Delta_2, \quad \phi_4 (
		\varrho ) \leq \Delta_3.
	\end{gather}

	Applying \ref{lemma:Q_function_with_holes} with $a$, $r$, $f$, and $A$
	replaced by $c$, $\varrho/2$, $F \aplus (-P)| \oball{c}{\varrho/2}$,
	and $X$, noting
	\ref{lemma:lipschitz_approximation}\,\eqref{item:lipschitz_approximation:lip},
	one obtains
	\begin{gather} \label{eqn:prep_tilt:7}
		\begin{split}
			& \varrho^{-1-\vdim/q} \norm{F \aplus
			(-P)}{q}{c,\varrho/2} \\
			& \qquad \leq \Delta_4 \big ( \varrho^{-\vdim/2}
			\norm{A ( F \aplus (-P))}{2}{c,\varrho/2} +
			\eta^{1/q-1} \varrho^{-\vdim-1} \norm{f \aplus
			(-P)}{1}{X} \big ).
		\end{split}
	\end{gather}
	Similarly, one obtains
	\begin{gather} \label{eqn:prep_tilt:8}
		\begin{split}
			& \varrho^{-1-\vdim/2} \norm{F \aplus
			(-P)}{2}{c,\varrho/2} \\
			& \qquad \leq \Delta_5 \big ( \eta^\mu
			\varrho^{-\vdim/2} \norm{A ( F \aplus (-P)
			)}{2}{c,\varrho/2} + \eta^{-1} \varrho^{-\vdim-1}
			\norm{f \aplus (-P)}{1}{X} \big ).
		\end{split}
	\end{gather}
	Applying
	\ref{lemma:iteration_prep}\,\eqref{item:iteration_prep:graph_tilt}
	applied with $\varrho$, $\sigma$ replaced by $\varrho/2$, $DP(0)$
	yields, noting $\phi_4 ( \varrho ) \leq 1$ by \eqref{eqn:prep_tilt:1}
	and $1/2 \geq \tau ( 1/p-1/\vdim )$,
	\begin{gather} \label{eqn:prep_tilt:9}
		\varrho^{-\vdim/2} \norm{A ( F \aplus (-P))}{2}{c,\varrho/2}
		\leq \Delta_6 \big ( \phi_2 ( \varrho, R ) + \phi_3 (
		\varrho )^\tau \big ).
	\end{gather}

	Define $d : \rel^\adim \to \rel$ by
	\begin{gather*}
		d (z) =  \inf \{ ( | \pp ( z - \xi ) |^2 + | \qq ( z - \xi )
		|^2)^{1/2} \with \text{$\xi \in \rel^\adim$, $P ( \pp ( \xi )
		) = \qq ( \xi )$} \}
	\end{gather*}
	whenever $z \in \rel^\adim$ and note, taking $\xi = ( \pp^\ast +
	\qq^\ast \circ P ) ( \pp ( z ) )$,
	\begin{gather*}
		d (z) \leq | P ( \pp (z) ) - \qq (z) | \quad \text{for $z \in
		\rel^\adim$}.
	\end{gather*}
	Hence, defining
	$d_{\ref{lemma:lipschitz_approximation}\eqref{item:lipschitz_approximation:poincare}}$
	and 
	$g_{\ref{lemma:lipschitz_approximation}\eqref{item:lipschitz_approximation:poincare}}$
	to be the functions defined in
	\ref{lemma:lipschitz_approximation}\,\eqref{item:lipschitz_approximation:poincare}
	under the names ``$d$'' and ``$g$'' with
	\begin{gather*}
		\text{$\varrho$, $P$  replaced by $\varrho/2$,
		$\classification{\cylind{T}{\pp^\ast (c)}{\varrho/2}}{z}{ P (
		\pp (z)) = \qq (z)}$},
	\end{gather*}
	one infers
	\begin{gather*}
		d | \cylinder{T}{\pp^\ast (c)}{\varrho/2}{3r} \leq
		d_{\ref{lemma:lipschitz_approximation}\eqref{item:lipschitz_approximation:poincare}}, \\
		g_{\ref{lemma:lipschitz_approximation}\eqref{item:lipschitz_approximation:poincare}}
		(x) \leq \mathscr{G} ( f (x), Q \Lbrack P (x) \Rbrack ) =
		\mathscr{G} ( ( f \aplus (-P) ) (x), Q \Lbrack 0 \Rbrack )
	\end{gather*}
	for $x \in X_1 \cap \cball{c}{\varrho/2}$. Therefore
	\ref{lemma:lipschitz_approximation}\,\eqref{item:lipschitz_approximation:poincare}
	with $\varrho$, $P$ replaced as in the definition of
	$d_{\ref{lemma:lipschitz_approximation}\eqref{item:lipschitz_approximation:poincare}}$
	and
	$g_{\ref{lemma:lipschitz_approximation}\eqref{item:lipschitz_approximation:poincare}}$
	yields, noting
	\begin{gather*}
		\mathscr{L}^\vdim ( \cball{c}{\varrho/2} \without X_1 ) \leq
		\mathscr{L}^\vdim ( C_{a,\varrho/2} ) \leq
		\varepsilon_{\ref{lemma:lipschitz_approximation}\eqref{item:lipschitz_approximation:poincare}}
		( \vdim, \delta_2, \delta_4 ) \unitmeasure{\vdim} ( \varrho/2
		)^\vdim
	\end{gather*}
	by \ref{lemma:iteration_prep}\,\eqref{item:item:iteration_prep:side_conditions:estimate} with
	$\varrho$ replaced by $\varrho/2$ and \eqref{eqn:prep_tilt:1},
	\begin{gather} \label{eqn:prep_tilt:10}
		\begin{split}
			& \eqLpnorm{\| V \| \restrict H \cap
			\cylinder{T}{\pp^\ast (c)}{\varrho/2}{3r}}{s}{d} \\
			& \qquad \leq \Delta_7 \big ( \norm{F \aplus
			(-P)}{s}{c,\varrho/2} + \mathscr{L}^\vdim (
			\cball{c}{\varrho/2} \without X_1 )^{1/s+1/\vdim} \big
			)
		\end{split}
	\end{gather}
	whenever $1 \leq s \leq \infty$. Using
	\ref{lemma:lipschitz_approximation}\,\eqref{item:lipschitz_approximation:estimate_b}
	with $\varrho$ replaced by $\varrho/2$
	in conjunction with
	\ref{lemma:iteration_prep}\,\eqref{item:item:iteration_prep:side_conditions:inclusion} with
	$\varrho$ replaced by $\varrho/2$, one
	estimates
	\begin{gather*}
		\mathscr{L}^\vdim ( \cball{c}{\varrho/2} \without X_1 ) \leq
		\mathscr{L}^\vdim ( C_{a,\varrho/2} ) \leq
		\Gamma_{\ref{lemma:lipschitz_approximation}\eqref{item:lipschitz_approximation:estimate_b}}
		( Q, \vdim ) \| V \| ( B_{a,\varrho/2} ),
	\end{gather*}
	hence by \ref{lemma:iteration_prep}\,\eqref{item:iteration_prep:estimate_bad_set} with $\varrho$
	and $R$ replaced by $\varrho/2$ and $T_\varrho$, noting
	\eqref{eqn:prep_tilt:1} and $| \eqproject{T_\varrho} - \project{T}
	| \leq \adim^{1/2} \| \eqproject{T_\varrho} - \project{T} \| \leq
	\adim^{1/2} \| \sigma_\varrho \| \leq \delta/2$ by
	\ref{miniremark:projections},
	\begin{gather} \label{eqn:prep_tilt:11}
		\varrho^{-\vdim} \mathscr{L}^\vdim ( \cball{c}{\varrho/2}
		\without X_1 ) \leq \Delta_8 \big ( \phi_2 ( \varrho,
		T_\varrho )^2 + \phi_4 ( \varrho ) \big ).
	\end{gather}
		
	In order to apply \ref{lemma:coercive_estimate_rect}, first define $K
	= \cylinder{T}{\pp^\ast (c)}{\varrho}{\varrho}$ and
	$H_{\ref{lemma:coercive_estimate_rect}}$ to be the set defined in
	\ref{lemma:coercive_estimate_rect} under the name ``$H$'', i.e.~the
	set of all $z \in \spt \|V\|$ such that
	\begin{gather*}
		\measureball{\| V \|}{\cball{z}{t}} \geq ( 40 )^{-\vdim} (
		\isoperimetric{\vdim}\vdim )^{-\vdim} t^\vdim \quad
		\text{whenever $0 < t < \infty$, $\cball{z}{t} \subset K$}.
	\end{gather*}
	One infers that
	\begin{gather*}
		\cylinder{T}{a}{\varrho}{\varrho} \cap \spt \| V \| \subset
		H_{\ref{lemma:coercive_estimate_rect}} \quad \text{if $\vdim =
		1$}, \\
		H_{\ref{lemma:coercive_estimate_rect}} \cap
		\cylinder{T}{a}{\Delta_9 \varrho}{\Delta_9 \varrho} \subset H;
	\end{gather*}
	in fact the first inclusion follows by
	\ref{app:lemma:good_point} and \eqref{eqn:prep_tilt:1} whereas
	concerning the second inclusion $\eta < 2^{-\vdim}$ implies by
	\ref{lemma:iteration_prep}\,\eqref{item:item:iteration_prep:side_conditions:inclusion}
	with $\varrho$ replaced by $\varrho/2$ the existence of $\xi \in A
	\cap \cylinder{T}{a}{\varrho/4}{\varrho/4}$ hence, verifying $1/4 <
	\Delta_9 < 1/2$ and $2^{3/2} \Delta_9/(1-\Delta_9) \leq
	\frac{19}{20}$, one obtains for $z \in \cylinder{T}{a}{\Delta_9
	\varrho}{\Delta_9 \varrho}$, $(1-\Delta_9) \varrho < t < 2r$
	\begin{gather*}
		| \xi - z | \leq 2^{3/2} \Delta_9 \varrho \leq 2^{3/2}
		\Delta_9 t / (1-\Delta_9 ) \leq {\textstyle\frac{19}{20}}
		t, \quad \cball{z}{t} \supset \cball{\xi}{t/(20)}, \\
		\measureball{\| V \|}{\cball{z}{t}} \geq \measureball{\| V
		\|}{\cball{\xi}{t/(20)}} \geq ( 40 )^{-\vdim} (
		\isoperimetric{\vdim} \vdim )^{-\vdim} t^\vdim = \delta_5
		\unitmeasure{\vdim} t^\vdim
	\end{gather*}
	by \ref{app:lemma:good_point} since $\delta \leq ( 2
	\isoperimetric{\vdim} )^{-1}$ and, noting \eqref{eqn:prep_tilt:1},
	the inclusion follows from \ref{lemma:iteration_prep}\,\eqref{item:iteration_prep:H_is_H} as
	$\cball{z}{(1-\Delta_9) \varrho} \subset K$. Choose $\phi \in
	\mathscr{D}^0 ( U )$ such that
	\begin{gather*}
		0 \leq \phi (x) \leq 1 \quad \text{and} \quad | D \phi (x) |
		\leq 2 \cdot ( \Delta_9 - 1/4 )^{-1} \varrho^{-1} \qquad
		\text{for $x \in U$}, \\
		\phi (x) = 1 \quad \text{for $x \in \cylinder{T}{a}
		{\varrho/4}{\varrho/4}$}, \\
		\spt \phi \subset \cylinder{T}{a}{\Delta_9 \varrho}{\Delta_9
		\varrho} \subset K \cap \Int
		\cylinder{T}{a}{\varrho/2}{\varrho/2}.
	\end{gather*}
	One now applies \ref{lemma:coercive_estimate} if $\vdim = 1$ and
	\ref{lemma:coercive_estimate_rect} if $\vdim > 1$ both with $a$ and
	$T$ replaced by $( \pp^\ast + \qq^\ast \circ P ) (0)$ and $\im D (
	\pp^\ast + \qq^\ast \circ P ) (0)$ to obtain with $\alpha_\vdim = 0$
	if $\vdim = 1$ and $\alpha_\vdim = ( \varrho^{1-\vdim/p}
	\alpha)^{\frac{\vdim p}{\vdim-p}}$ if $\vdim > 1$
	\begin{gather*}
		\varrho^{-\vdim} \beta^2 \leq \Delta_{10} \big ( \alpha_\vdim
		+ ( \varrho^{1-\vdim/p} \alpha \varrho^{-1-\vdim/q} \gamma
		)^{1/(1/p+1/q)} \big ) + (16+4\vdim) \varrho^{-\vdim} \xi^2;
	\end{gather*}
	here $\alpha$, $\beta$, $\gamma$, and $\xi$ are as in
	\ref{lemma:coercive_estimate_rect} and \ref{lemma:coercive_estimate}
	respectively. Noting $(\alpha_\vdim)^{1/2} \leq \phi_3 ( \varrho
	)^\tau$, since $\phi_3 ( \varrho ) \leq 1$ by
	\eqref{eqn:prep_tilt:1}, and using the inequality relating
	arithmetic and geometric means as in
	\ref{remark:coercive_estimate_rect}, one infers
	\begin{gather} \label{eqn:prep_tilt:12}
		\begin{split}
			& \phi_2 ( \varrho/4, R ) \leq \Delta_{11} \big (
			\lambda \varrho^{-1-\vdim/q} \eqLpnorm{\| V \|
			\restrict H \cap \cylinder{T}{\pp^\ast
			(c)}{\varrho/2}{3r}}{q}{d} \\
			& \qquad \quad + \lambda^{-\tau} \phi_3 (
			\varrho )^\tau + \varrho^{-1-\vdim/2} \eqLpnorm{\| V
			\| \restrict H \cap
			\cylinder{T}{\pp^\ast(c)}{\varrho/2}{3r}}{2}{d} \big
			).
		\end{split}
	\end{gather}

	Finally, the estimates
	\eqref{eqn:prep_tilt:7}--\eqref{eqn:prep_tilt:12} are combined as
	follows: Firstly,
	\begin{gather*}
		\begin{aligned}
			& \phi_2 ( \varrho/ 4, R ) \leq \Delta_{11}
			\lambda^{-\tau} \phi_3 ( \varrho )^\tau \\
			& \qquad + \Delta_7 \Delta_{11} \lambda
			\varrho^{-1-\vdim/q} \big ( \norm{F \aplus
			(-P)}{q}{c,\varrho/2} + \mathscr{L}^\vdim (
			\cball{c}{\varrho/2} \without X_1 )^{1/q+1/\vdim} \big
			) \\
			& \qquad + \Delta_7 \Delta_{11} \varrho^{-1-\vdim/2}
			\big ( \norm{F \aplus (-P)}{2}{c,\varrho/2} +
			\mathscr{L}^\vdim ( \cball{c}{\varrho/2} \without X_1
			)^{1/2+1/\vdim} \big )
		\end{aligned}
	\end{gather*}
	by \eqref{eqn:prep_tilt:12} and \eqref{eqn:prep_tilt:10}. Then,
	by \eqref{eqn:prep_tilt:7}, \eqref{eqn:prep_tilt:8}, and
	\eqref{eqn:prep_tilt:11}
	\begin{gather*}
		\begin{aligned}
			& \phi_2 ( \varrho / 4, R ) \leq
			\Delta_{11} \lambda^{-\tau} \phi_3 ( \varrho )^\tau \\
			& \qquad + \Delta_7 \Delta_{11} ( \Delta_4 +
			\Delta_5) ( \lambda + \eta^\mu ) \varrho^{-\vdim/2}
			\norm{A ( F \aplus (-P))}{2}{c,\varrho/2} \\
			& \qquad + \Delta_7 \Delta_{11} ( \Delta_4 +
			\Delta_5) ( \eta^{1/q-1} + \eta^{-1} )
			\varrho^{-1-\vdim} \norm{f \aplus
			(-P)}{1}{X} \\
			& \qquad + 2 \Delta_7 ( \Delta_8 )^{1/q+1/\vdim}
			\Delta_{11} \lambda \big ( \phi_2 ( \varrho, T_\varrho
			)^{2/q + 2/\vdim} + \phi_4 ( \varrho )^{1/q + 1/\vdim}
			\big ) \\
			& \qquad + 2 \Delta_7 ( \Delta_8 )^{1/2+1/\vdim}
			\Delta_{11} \big ( \phi_2 ( \varrho, T_\varrho
			)^{1+2/\vdim} + \phi_4 ( \varrho)^{1/2+1/\vdim} \big
			).
		\end{aligned}
	\end{gather*}
	Finally, using $\phi_2 ( \varrho, T_\varrho ) \leq 1$ and $\phi_4 (
	\varrho) \leq 1$ by \eqref{eqn:prep_tilt:1}, $q \geq 2$, and
	$\tau \leq \frac{\vdim p}{2 ( \vdim-p )} \leq \big ( \frac{1}{q} +
	\frac{1}{\vdim} \big ) \frac{\vdim p }{\vdim-p}$ if $\vdim > 1$ this
	simplifies to
	\begin{multline*}
		\phi_2 ( \varrho/ 4, R ) \leq \Delta_{12} \Big (
		\lambda^{-\tau} \phi_3 ( \varrho)^\tau + \big ( \lambda +
		\phi_2 ( \varrho, T_\varrho )^{2/\vdim} \big ) \phi_2 (
		\varrho, T_\varrho ) \\
		+ ( \lambda + \eta^\mu ) \varrho^{-\vdim/2} \norm{A ( F \aplus
		(-P))}{2}{c,\varrho/2} + \eta^{-1} \varrho^{-\vdim-1}
		\norm{f \aplus (-P)}{1}{X} \Big )
	\end{multline*}
	and the conclusion is a consequence of \eqref{eqn:prep_tilt:9}.
\end{proof} \setcounter{equation}{0}
\begin{proof} [Proof of \eqref{item:iteration:tilt_estimate}]
	With $\delta_4 = 1$ and $\delta_5 = ( 40 )^{-\vdim} (
	\isoperimetric{\vdim} \vdim )^{-\vdim} / \unitmeasure{\vdim}$ define
	\begin{gather*}
		\Delta_1  = \inf \{
		\varepsilon_{\ref{lemma:iteration_prep}\eqref{item:iteration_prep:estimate_bad_set}}
		( \vdim, \delta_4, \delta ),
		\varepsilon_{\ref{lemma:iteration_prep}\eqref{item:iteration_prep:side_conditions}}
		( \adim, \delta_4,  \delta ),
		\varepsilon_{\eqref{item:iteration:prep_tilt}} ( \vdim, \adim,
		Q, \delta_2, \varepsilon, \delta, p ) \}, \\ \displaybreak[0]
		\Delta_2  = 6 ( 2 \vdim
		\Gamma_{\eqref{item:iteration:coerciveness}} (
		\adim))^{\vdim+1} \unitmeasure{\vdim}^{-1/2}, \quad \Delta_3 =
		\Delta_2 \big ( \Gamma_{\eqref{item:iteration:coerciveness}}
		(\adim)^2 + 1 \big ), \\ \displaybreak[0] \Delta_4  = \Delta_3
		\Gamma_{\ref{lemma:iteration_prep}\eqref{item:iteration_prep:graph_tilt}}
		( \adim, Q, \delta )^{1/2}, \\ \displaybreak[0] \Delta_5  =
		\inf \big \{ 2^{-2\vdim-5} \unitmeasure{\vdim} \adim^{-1/2}
		\delta, ( \Delta_4 )^{-1} \adim^{-1/2} \delta/4, 1 \big \}, \\
		\displaybreak[0] \Delta_6  = \inf \big \{ 1, 2^{-\vdim}
		\varepsilon_{\eqref{item:iteration:prep_tilt}} ( \vdim, \adim,
		Q, \delta_2, \varepsilon, \delta, p ) \big \}, \\
		\displaybreak[0] \Delta_7  = \inf \big \{ ( \Delta_4 )^{-2}
		\adim^{-1} \delta^2 2^{-4} , 2^{-\vdim}
		\besicovitch{\adim}^{-1} \unitmeasure{\vdim} (1/8), 2^{-\vdim}
		\big \}, \\ \displaybreak[0]
		\varepsilon_{\eqref{item:iteration:tilt_estimate}}  = \inf
		\big \{ \Delta_1, 2^{-1} \vdim^{-1/2} ( \Delta_5 )^2,
		\Delta_6, \delta ( \Delta_7 )^{1/p-1/\vdim} \big \}.
	\end{gather*}
	Moreover, define
	\begin{gather*}
		\Delta_8  = \Gamma_{\eqref{item:iteration:l1_estimate}} (
		\vdim, \adim, Q, \delta_4, \delta, p ), \quad
		\Delta_9  = \Gamma_{\ref{lemma:poincare}} (
		\adim ) \unitmeasure{\vdim}^{1/2}, \\
		\Delta_{10}  = \Delta_9
		\Gamma_{\ref{lemma:iteration_prep}\eqref{item:iteration_prep:graph_tilt}}
		( n, Q, \delta)^{1/2}, \quad
		\Delta_{11}  =
		2^{\vdim+1} \Gamma_{\ref{lemma:simple_interpolation}} ( 2, \adim
		), \\
		\Delta_{12}  = \Delta_{11} \sup \{
		\unitmeasure{\vdim}, 
		\Delta_8 + 2^\vdim \Delta_{10} \delta^{-\tau} \}, \\
		\Delta_{13}  = ( Q + 1 )^{1/2} \unitmeasure{\vdim}^{1/2}
		\Delta_{12} \adim^{1/2} + 2^\vdim, \quad
		\Delta_{14}  = Q^{1/2} \sup \{ \unitmeasure{\vdim}, \Delta_8
		\}, \\
		\Gamma_{\eqref{item:iteration:tilt_estimate}}  =
		\Gamma_{\eqref{item:iteration:prep_tilt}} ( \vdim, \adim, Q,
		\delta, p, \tau ) ( 2^{\vdim+1} + 2 \Delta_{13} + \Delta_{14}
		).
	\end{gather*}
	It will be shown that
	$\varepsilon_{\eqref{item:iteration:tilt_estimate}}$ and
	$\Gamma_{\eqref{item:iteration:tilt_estimate}}$ have the asserted
	property.

	Since $\varrho \in J_4$ and $2\varrho \in J_0$, it follows
	\begin{gather*}
		\varrho < 2r, \quad 2 \varrho \in J_3, \quad \phi_2 (
		2\varrho, T ) \leq ( 2 \vdim^{1/2} \kappa )^{1/2}.
	\end{gather*}
	One therefore obtains
	\begin{gather} \label{eqn:tilt_estimate:1}
		\begin{split}
			& \kappa \leq \Delta_1, \quad \phi_2 ( 2\varrho, T )
			\leq \Delta_5, \quad \phi_3 ( 2 \varrho) \leq
			\Delta_6, \quad \phi_4 ( 2 \varrho ) \leq \Delta_7, \\
			& \qquad \qquad \qquad \varrho \in J_1, \quad \|
			\sigma_\varrho \| \leq \adim^{-1/2} \delta/2;
		\end{split}
	\end{gather}
	in fact the first four inequalities are directly implied by the
	definition of $\varepsilon_{\eqref{item:iteration:tilt_estimate}}$ and
	the last two statements follow from
	\ref{lemma:iteration_prep}\,\eqref{item:iteration_prep:ch_planes}
	applied with $\varrho$, $s$, $t$, $\lambda$ replaced by $2 \varrho$,
	$\varrho$, $\varrho$, $1/2$ since $\phi_2 ( 2\varrho, T_{2\varrho} )
	\leq 2^{-2\vdim-5} \unitmeasure{\vdim} \adim^{-1/2} \delta$ by the
	second inequality.

	Define $P : \rel^\vdim \to \rel^\codim$ by $P (x) = u_\varrho (c) +
	\left < x-c, Du_\varrho (c) \right >$ for $x \in \rel^\vdim$. One
	verifies
	\begin{gather} \label{eqn:tilt_estimate:2}
		\Lip P = \| D P (0) \| \leq \adim^{-1/2} \delta/2;
	\end{gather}
	in fact using \cite[5.2.5]{MR41:1976}, \ref{lemma:standard_w012},
	\ref{lemma:iteration_prep}\,\eqref{item:iteration_prep:graph_tilt}
	with $\sigma$ replaced by $0$, and \eqref{eqn:tilt_estimate:1}
	\begin{gather*}
		\begin{aligned}
			\| D P (0) \| & = \| Du_\varrho (c) \| \leq \Delta_2
			\varrho^{-\vdim/2} \norm{Du_\varrho}{2}{c,\varrho} \\
			& \leq \Delta_2 \varrho^{-\vdim/2} \big ( \norm{D
			(u_\varrho-g)}{2}{c,\varrho} + \norm{Dg}{2}{c,\varrho}
			\big ) \\
			& \leq \Delta_3 \varrho^{-\vdim/2}
			\norm{Dg}{2}{c,\varrho} \leq \Delta_4 \big ( \phi_2 (
			2\varrho, T ) + \phi_4 ( 2 \varrho )^{1/2} \big ) \leq
			\adim^{-1/2} \delta/2.
		\end{aligned}
	\end{gather*}

	Taylor's expansion yields
	\begin{gather} \label{eqn:tilt_estimate:3}
		\varrho^{-\vdim-1} \norm{u_\varrho-P}{1}{c,\varrho/2} \leq
		\unitmeasure{\vdim} \varrho \norm{D^2
		u_\varrho}{\infty}{c,\varrho/2}.
	\end{gather}
	Noting \eqref{eqn:tilt_estimate:1}, one obtains from
	\eqref{item:iteration:l1_estimate} that
	\begin{gather} \label{eqn:tilt_estimate:4}
		\varrho^{-\vdim-1} \norm{u_\varrho - g}{1}{c,\varrho} \leq
		\Delta_8 \big ( \phi_2 ( 2\varrho, T_{2\varrho} )^2 + \phi_3 (
		2\varrho)^\tau \big ).
	\end{gather}
	By \ref{lemma:poincare} with $a$, $r$, $u$ replaced by $c$,
	$\varrho/2$, $(g-\sigma_\varrho) | \oball{c}{\varrho/2}$ there exists
	an affine function $R : \rel^\vdim \to \rel^\codim$ with $DR (0) =
	\sigma_\varrho$ such that
	\begin{gather*}
		\varrho^{-\vdim-1} \norm{g-R}{1}{c,\varrho/2} \leq \Delta_9
		\varrho^{-\vdim/2} \norm{D(g-R)}{2}{c,\varrho/2},
	\end{gather*}
	hence by \ref{lemma:iteration_prep}\,\eqref{item:iteration_prep:graph_tilt} with $\varrho$,
	$\sigma$ replaced by $\varrho/2$, $\sigma_\varrho$, noting
	\eqref{eqn:tilt_estimate:1},
	\begin{gather} \label{eqn:tilt_estimate:5}
		\varrho^{-\vdim-1} \norm{g-R}{1}{c,\varrho/2} \leq \Delta_{10}
		\big ( \phi_2 ( \varrho, T_\varrho ) + \phi_4 ( \varrho
		)^{1/2} \big ).
	\end{gather}
	Since by \ref{lemma:simple_interpolation} with $k$, $a$, $r$, $u$
	replaced by $2$, $c$, $\varrho/2$, $P-R$
	\begin{gather*}
		\begin{aligned}
			| D P (0) - \sigma_\varrho | & = | D ( P-R ) (0) |
			\leq \Delta_{11} \varrho^{-1-\vdim}
			\norm{P-R}{1}{c,\varrho/2} \\
			& \leq \Delta_{11} \varrho^{-1-\vdim} \big (
			\norm{P-u_\varrho}{1}{c,\varrho/2} +
			\norm{u_\varrho-g}{1}{c,\varrho/2} +
			\norm{g-R}{1}{c,\varrho/2} \big ),
		\end{aligned}
	\end{gather*}
	one obtains from
	\eqref{eqn:tilt_estimate:3}--\eqref{eqn:tilt_estimate:5}, noting $\sup
	\{ \phi_2 ( 2\varrho, T_{2\varrho} ), \phi_3 ( 2 \varrho ), \phi_4 (
	\varrho ) \} \leq 1$ by \eqref{eqn:tilt_estimate:1} and $1/2 \geq \tau
	(1/p - 1/\vdim)$,
	\begin{gather*}
		| D P (0) - \sigma_\varrho | \leq \Delta_{12} \big ( \varrho
		\phi_1 ( \varrho ) + \phi_2 ( 2\varrho, T_{2\varrho} ) +
		\phi_3 ( 2\varrho )^\tau \big ),
	\end{gather*}
	hence using
	\ref{lemma:iteration_prep}\,\eqref{item:item:iteration_prep:side_conditions:upper_bound}
	and \ref{miniremark:projections}
	\begin{gather} \label{eqn:tilt_estimate:6}
		\phi_2 ( \varrho, S ) \leq \Delta_{13} \big ( \varrho \phi_1 (
		\varrho ) + \phi_2 ( 2 \varrho, T_{2\varrho} ) + \phi_3 ( 2
		\varrho )^\tau \big )
	\end{gather}
	where $S = \im D ( \pp^\ast + \qq^\ast \circ P ) (0)$.

	Define $X = \classification{\oball{c}{\varrho/2} \cap X_1}{x}{
	\density^0 ( \| f (x) \|, g (x) ) = Q }$ and note
	\begin{gather*}
		\norm{f \aplus (-P)}{1}{X} \leq Q^{1/2} (
		\norm{g-u_\varrho}{1}{c,\varrho} +
		\norm{u_\varrho-P}{1}{c,\varrho/2} ).
	\end{gather*}
	Combining this with \eqref{eqn:tilt_estimate:3} and
	\eqref{eqn:tilt_estimate:4} yields
	\begin{gather*}
		\varrho^{-1-\vdim} \norm{f \aplus (-P)}{1}{X} \leq \Delta_{14}
		\big ( \varrho \phi_1 ( \varrho ) + \phi_2 ( 2 \varrho,
		T_{2\varrho} )^2 + \phi_3 ( 2 \varrho )^\tau \big ).
	\end{gather*}
	Therefore noting \eqref{eqn:tilt_estimate:1},
	\eqref{eqn:tilt_estimate:2} and
	\ref{lemma:lipschitz_approximation}\,\eqref{item:lipschitz_approximation:yz}
	and applying \eqref{item:iteration:prep_tilt} with $R$ replaced by
	$S$, one obtains in conjunction with \eqref{eqn:tilt_estimate:6} the
	conclusion.
\end{proof} \setcounter{equation}{0}
\begin{proof} [Proof of \eqref{item:iteration:iteration}]
	As the assertion does not involve $\kappa$ it may be restricted to a
	specific value. One defines
	\begin{align*}
		\Delta_1 & = \sup \{
		\Gamma_{\eqref{item:iteration:key_estimate}} ( \vdim, \adim, Q,
		\delta_4, \delta, p ),
		\Gamma_{\eqref{item:iteration:tilt_estimate}} ( \vdim, \adim, Q,
		\delta, p, \tau ) , 1 \} , \\
		\displaybreak[0] \eta & = \inf \big \{ ( 48 \Delta_1
		)^{-\adim}, 2^{-\adim} \big \}, \\
		\displaybreak[0] \kappa & =
		\begin{aligned}[t]
			\inf \big \{ &
			\varepsilon_{\ref{lemma:iteration_prep}\eqref{item:iteration_prep:estimate_bad_set}}
			( \vdim, \delta_4, \delta ),
			\varepsilon_{\ref{lemma:iteration_prep}\eqref{item:iteration_prep:lower_mass_bound}}
			( \adim, Q, \delta_4, p, \alpha, \delta_6 ),
			\varepsilon_{\ref{lemma:iteration_prep}\eqref{item:iteration_prep:side_conditions}}
			( \adim, \delta_4, \delta ), \\
			& \varepsilon_{\eqref{item:iteration:key_estimate}}
			( \adim, \delta_4, \delta ),
			2^{-\vdim-2} \besicovitch{\adim}^{-1}
			\unitmeasure{\vdim} \eta
			\Gamma_{\ref{lemma:lipschitz_approximation}\eqref{item:lipschitz_approximation:estimate_b}}
			( Q, \vdim )^{-1}, \\
			& \varepsilon_{\eqref{item:iteration:tilt_estimate}}
			( \vdim, \adim, Q, \delta_2, \varepsilon, \delta, p )
			\big \},
		\end{aligned} \\
		\displaybreak[0] \Delta_2 & =  \inf \big \{ 1, 2^{-\vdim}
		\besicovitch{\adim}^{-1} \unitmeasure{\vdim} \inf \{ \eta ( 4
		\Gamma_{\ref{lemma:lipschitz_approximation}\eqref{item:lipschitz_approximation:estimate_b}}
		( Q, \vdim ))^{-1}, 1/8 \} \big \}, \\
		\displaybreak[0] \Delta_3 & =
		\begin{aligned}[t]
			\inf \big \{ & 2^{-2\vdim} \sup \{ ( Q+1)
			\unitmeasure{\vdim}, 1 \}^{-1} \kappa, 1,
			\varepsilon_{\eqref{item:iteration:tilt_estimate}}
			( \vdim, \adim, Q, \delta_2, \varepsilon, \delta, p ),
			\\
			& ( \Delta_2 )^{1/p-1/\vdim} \delta, 2^{-9\vdim} \sup
			\{ M \unitmeasure{\vdim}, 1 \}^{-1} \kappa \big \},
		\end{aligned} \\
		\displaybreak[0] \Delta_4 & =
		\begin{aligned}[t]
			\inf \big \{ & ( \Delta_3 / 8 )^\tau,
			\varepsilon_{\ref{lemma:iteration_prep}\eqref{item:iteration_prep:lower_mass_bound}}
			( \adim, Q, \delta_4, p, \alpha , \delta_6 )^\tau , \\
			& \big ( \alpha p \unitmeasure{\vdim}^{1/p} ( (
			Q-1+\delta_6 )^{1/p} - ( Q-1+\delta_6/2)^{1/p} ) \big
			)^\tau \big \},
		\end{aligned} \\
		\displaybreak[0] \Delta_5 & = \inf \big \{ 2^{-2\vdim} ( Q + 1
		)^{-1/2} \unitmeasure{\vdim}^{-1/2} \kappa, 2^{-\vdim-2}
		\unitmeasure{\vdim}^{1/2} \big \}, \\
		\displaybreak[0] \Delta_6 & = n^{-1/2} \inf \big \{ \delta/4,
		2^{-\vdim-1} \sup \{ (Q+1) \unitmeasure{\vdim}, 1 \}^{-1}
		\Delta_5 \big \}, \\
		\displaybreak[0] \Delta_7 & = \inf \big \{ \adim^{-1/2} \inf
		\{ \delta/2, 1/4 \}, \Delta_6/2 \big \}, \\
		\displaybreak[0] \Delta_8 & = 1 - 4^{\alpha \tau-1} \quad
		\text{if $\alpha \tau < 1$}, \\
		\Delta_8 & = \log 4 \quad \text{if $\alpha \tau = 1$}, \\
		\displaybreak[0] \Delta_9 & =
		\begin{aligned}[t]
			\inf \big \{ & 2^{-2\vdim-4}
			\unitmeasure{\vdim}^{1/2}, 2^{-2\vdim-4}
			\unitmeasure{\vdim}^{1/2} ( 1- 2^{-\alpha \tau} )
			\Delta_6, 2^{-\vdim-1} \Delta_5, 1, \\
			& ( 3 \Delta_1 )^{-1} \Delta_8,
			{\textstyle\frac{1}{576}}  ( \Delta_1 )^{-2} \eta
			\Delta_8, ( 48 \Delta_1 )^{-\adim} \eta^\adim
			\big \},
		\end{aligned} \\
		\displaybreak[0] \Delta_{10} & =
		\Gamma_{\eqref{item:iteration:starting_control}} ( \adim,
		Q, \delta ), \\ \displaybreak[0] \Delta_{11} & = \inf \big \{
		\delta^\tau 2^{-7} ( \Delta_{10})^{-1}, {\textstyle\frac{1}{24}}
		\Delta_8 ( \Delta_1 )^{-1} \big \}, \\
		\displaybreak[0] \lambda & = ( 48 \Delta_1 )^{-1}, \\
		\displaybreak[0]
		\Delta_{12} & = \big ( 24 \Delta_1 ( \eta^{-1} +
		\lambda^{-\tau} ) \big )^{-1}, \\ \displaybreak[0]
		\gamma_2 & = (e/4) \Delta_9, \\ \displaybreak[0]
		\gamma_1 & = \eta ( 24 \Delta_1 )^{-1} \gamma_2, \\
		\displaybreak[0] \gamma_3 & = \inf \{ \Delta_4 , \Delta_{11}
		\gamma_1, \Delta_{12} \gamma_2 \}, \\ \displaybreak[0]
		\varepsilon_{\eqref{item:iteration:iteration}} & =
		\begin{aligned}[t]
			\inf \big \{ & 2^{-8\vdim} \sup \{ M
			\unitmeasure{\vdim}, 1 \}^{-1} \kappa, 2^{-6\vdim-4}
			\unitmeasure{\vdim}^{1/2}, \\
			& 2^{-5\vdim-3} \unitmeasure{\vdim}^{1/2} \Delta_7,
			2^{-5 \vdim} \Delta_5,
			2^{-4\vdim-7} (\Delta_{10} )^{-1} \gamma_1, 2^{-5\vdim-6}
			\gamma_2 , \eta/2 \big \};
		\end{aligned}
	\end{align*}
	here $e$ denotes Euler's number. It will be shown that $\gamma_i$ and
	$\varepsilon_{\eqref{item:iteration:iteration}}$ have the
	asserted property.

	Suppose $C_{a,t}$ for $t \in J_0$ is as in \ref{lemma:iteration_prep}.

	First, note that
	\begin{gather} \label{eqn:iteration:1}
		\phi_3 ( \varrho )^\tau \leq \gamma \gamma_3 ( \varrho/r
		)^{\alpha \tau} \quad \text{for $0 < t \leq 8r$}
	\end{gather}
	implies, noting $\gamma_3 \leq \Delta_4$,
	\begin{gather} \label{eqn:iteration:1'}
		\phi_3 ( \varrho ) \leq \Delta_3 \quad \text{and} \quad \phi_4
		( \varrho ) \leq \Delta_2 \quad \text{whenever $0 < \varrho
		\leq 8r$}. \tag{\ref{eqn:iteration:1}'}
	\end{gather}

	Next, some auxiliary assertions will be shown:
	\begin{gather}
		\classification{\rel}{\varrho}{0 < \varrho \leq r/2} \subset
		J_0, \label{eqn:iteration:2} \\
		\displaybreak[0]
		\classification{\rel}{\varrho}{{\textstyle\frac{r}{64}} \leq
		\varrho \leq r} \subset J_1, \label{eqn:iteration:3} \\
		\displaybreak[0]
		\classification{\rel}{\varrho}{{\textstyle\frac{r}{64}} \leq
		\varrho \leq 8r } \subset J_2 \cap J_3,
		\label{eqn:iteration:4} \\
		\displaybreak[0]
		\classification{\rel}{\varrho}{{\textstyle\frac{r}{64}} \leq
		\varrho \leq 4r} \subset J_4, \label{eqn:iteration:5} \\
		\displaybreak[0]
		\classification{\rel}{\varrho}{{\textstyle\frac{r}{64}} \leq
		\varrho \leq r/2} \subset J_5, \label{eqn:iteration:6} \\
		\displaybreak[0] \| V \| ( \cylinder{T}{a}{\varrho}{\varrho} )
		\geq ( Q-1+\delta_4/2) \unitmeasure{\vdim} \varrho^\vdim \quad
		\text{whenever $0 < \varrho \leq r/2$},
		\label{eqn:iteration:aux_j6} \\
		\displaybreak[0] \| \sigma_\varrho \| \leq \Delta_7 \quad
		\text{whenever ${\textstyle\frac{r}{64}} \leq \varrho \leq
		r$}, \label{eqn:iteration:7}
	\end{gather}

	\paragraph{Proof of \eqref{eqn:iteration:2}.} This follows from $a \in
	\cylinder{T}{0}{r/2}{2r}$.

	\paragraph{Proof of \eqref{eqn:iteration:4}.} For $\frac{r}{64} \leq
	\varrho \leq 8r$ one computes, using H\"older's inequality and
	\eqref{eqn:iteration:1'},
	\begin{gather*}
		\begin{aligned}
			& \| \delta V \| ( U \cap
			\cylinder{T}{a}{\varrho}{\varrho} ) \leq \| V \| (
			U)^{1-1/p} \psi ( U \cap
			\cylinder{T}{a}{8r}{8r})^{1/p} \\
			& \qquad \leq \sup \{ M \unitmeasure{\vdim}, 1 \}
			r^{\vdim-\vdim/p} (8r)^{\vdim/p-1} \phi_3 ( 8r ) \\
			& \qquad \leq \Delta_3 \sup \{ M \unitmeasure{\vdim} ,
			1 \} 2^{9\vdim} ( {\textstyle\frac{r}{64}})^{\vdim-1}
			\leq \kappa \varrho^{\vdim-1},
		\end{aligned} \\
		\begin{aligned}
			& \tint{(U \cap \cylinder{T}{a}{\varrho}{\varrho})
			\times \grass{\adim}{\vdim}}{} | \project{S} -
			\project{T} | \ud V (z,S) \leq \| V \| ( U )^{1/2}
			(8r)^{\vdim/2} \phi_2 ( 8r, T) \\
			& \qquad \leq \sup \{ M \unitmeasure{\vdim}, 1 \}
			2^{8\vdim}
			\varepsilon_{\eqref{item:iteration:iteration}} (
			{\textstyle\frac{r}{64}} )^\vdim \leq \kappa
			\varrho^\vdim.
		\end{aligned}
	\end{gather*}
	
	\paragraph{Proof of \eqref{eqn:iteration:5}.} This follows from
	\eqref{eqn:iteration:4}.

	\paragraph{Proof of \eqref{eqn:iteration:6}.} Let $\frac{r}{64} \leq
	\varrho \leq r/2$. One computes for $0 < t < \varrho$,
	\eqref{eqn:iteration:1} and $\gamma_3 \leq \Delta_4$,
	\begin{gather*}
		\phi_3 ( t ) \leq ( \Delta_4 )^{1/\tau} (t/r)^\alpha \leq
		\varepsilon_{\ref{lemma:iteration_prep}\eqref{item:iteration_prep:lower_mass_bound}}
		( \adim, Q, \delta_4, p, \alpha , \delta_6 )
		(t/\varrho)^\alpha.
	\end{gather*}
	Therefore, noting \eqref{eqn:iteration:2} and \eqref{eqn:iteration:4},
	\eqref{eqn:iteration:6} is implied by
	\ref{lemma:iteration_prep}\,\eqref{item:iteration_prep:lower_mass_bound}.
	
	\paragraph{Proof of \eqref{eqn:iteration:aux_j6}.} Applying
	\ref{lemma:aux_monotonicity} with $r$, $M$, $\varrho$ replaced by
	$\varrho$, $(\Delta_4)^{1/\tau}$, $\varrho$ in conjunction with
	H\"older's inequality, noting
	\eqref{eqn:iteration:1} and $\gamma_3 \leq \Delta_4$, yields
	\begin{gather*}
		\begin{aligned}
			\big ( \varrho^{-\vdim} \| V \| (
			\cylinder{T}{a}{\varrho}{\varrho} ) \big )^{1/p} &
			\geq ( ( Q-1+\delta_6 ) \unitmeasure{\vdim} )^{1/p} -
			( \Delta_4)^{1/\tau} \alpha^{-1} p^{-1} \\
			& \geq ( ( Q - 1 +\delta_6/2 ) \unitmeasure{\vdim}
			)^{1/p}.
		\end{aligned}
	\end{gather*}

	\paragraph{Proof of \eqref{eqn:iteration:3} and
	\eqref{eqn:iteration:7}.} Let $\frac{r}{64} \leq \varrho \leq r$.
	Using H\"older's inequality and $\varrho/2 \leq \inf \{ \varrho, r/2
	\} \in J_5$ by \eqref{eqn:iteration:6}, one estimates
	\begin{gather*}
		\begin{aligned}
			& \| \eqproject{T_\varrho} - \project{T} \| \leq \|
			V \| ( U \cap \cylinder{T}{a}{\varrho}{\varrho}
			)^{-1/2} \varrho^{\vdim/2} \big ( \phi_2 ( \varrho,
			T_\varrho ) + \phi_2 ( \varrho, T ) \big ) \\
			& \qquad \leq \unitmeasure{\vdim}^{-1/2}
			2^{\vdim/2+3/2} \phi_2 ( \varrho, T ) \leq
			\unitmeasure{\vdim}^{-1/2} 2^{5\vdim+2} \phi_2 ( 8r, T
			) \\
			& \qquad \leq \unitmeasure{\vdim}^{-1/2} 2^{5\vdim+2}
			\varepsilon_{\eqref{item:iteration:iteration}}
			\leq 1/2,
		\end{aligned}
	\end{gather*}
	hence $T_\varrho \cap \ker \pp = \{ 0 \}$ and $\varrho \in J_1$,
	i.e.~\eqref{eqn:iteration:3}. Now, \ref{miniremark:projections}
	applied with $S$, $S_1$, $S_2$ replaced by $T$, $T$, $T_\varrho$
	yields
	\begin{gather*}
		\| \sigma_\varrho \|^2 \leq ( 1 + \| \sigma_\varrho \|^2 ) \|
		\eqproject{T_\varrho} - \project{T} \|^2, \\
		\| \sigma_\varrho \|^2 \leq \| \eqproject{T_\varrho} -
		\project{T} \|^2 / ( 1 - \| \eqproject{T_\varrho} -
		\project{T} \|^2 ) \leq 2 \| \eqproject{T_\varrho} -
		\project{T} \|^2, \\
		\| \sigma_\varrho \| \leq 2 \| \eqproject{T_\varrho} -
		\project{T} \| \leq \unitmeasure{\vdim}^{-1/2} 2^{5\vdim+3}
		\varepsilon_{\eqref{item:iteration:iteration}} \leq
		\Delta_7.
	\end{gather*}

	Having shown the auxiliary assertions
	\eqref{eqn:iteration:2}--\eqref{eqn:iteration:7}, one chooses $j \in
	\nat$ such that $\frac{r}{64} < 4^j t \leq \frac{r}{16}$ and defines
	$t_i = 4^{j+1-i} t$ whenever $i \in \nat$, $i \leq j+1$ in order to
	show inductively the following assertions whenever $i \in \nat$, $i
	\leq j+1$:
	\begin{align}
		& \classification{\rel}{\varrho}{t_i \leq \varrho \leq r}
		\subset J_4 \label{eqn:iteration:ind_j4} \\
		& \classification{\rel}{\varrho}{t_i \leq \varrho \leq r/2}
		\subset J_5, \label{eqn:iteration:ind_j5} \\
		& \classification{\rel}{\varrho}{t_i \leq \varrho \leq r}
		\subset J_1, \label{eqn:iteration:ind_j1} \\
		& \| \sigma_\varrho \| \leq \Delta_6 \quad \text{for $t_i \leq
		\varrho \leq r$}, \label{eqn:iteration:9} \\
		& \phi_2 ( \varrho, T ) \leq \Delta_5 \quad \text{for $t_i
		\leq \varrho \leq r$}, \label{eqn:iteration:10} \\
		\begin{split}
			& \phi_1 ( \varrho ) \leq \gamma \gamma_1
			\varrho^{-1+\alpha \tau} r^{-\alpha \tau} \quad
			\text{whenever $t_i \leq \varrho \leq r/4$, $\alpha
			\tau < 1$}, \\
			& \phi_1 ( \varrho ) \leq \gamma \gamma_1 r^{-1} ( 1 +
			\log ( r/ \varrho ) ) \quad \text{whenever $t_i
			\leq \varrho \leq r/4$, $\alpha \tau = 1$},
		\end{split} \label{eqn:iteration:11} \\
		\begin{split}
			& \phi_2 ( \varrho, T_\varrho ) \leq \gamma \gamma_2 (
			\varrho / r )^{\alpha \tau} \quad \text{whenever $t_i
			\leq \varrho \leq r$, $\alpha \tau < 1$}, \\
			& \phi_2 ( \varrho, T_\varrho ) \leq \gamma \gamma_2
			( \varrho/r ) ( 1 + \log ( r/\varrho ) ) \quad
			\text{whenever $t_i \leq \varrho \leq r$, $\alpha \tau
			= 1$}.
		\end{split} \label{eqn:iteration:12}
	\end{align}
	One verifies that $\eqref{eqn:iteration:12}_i$ implies
	\begin{gather}
		\phi_2 ( \varrho, T_\varrho ) \leq \Delta_9 ( \varrho/r
		)^{\alpha \tau/2} \quad \text{whenever $t_i \leq \varrho \leq
		r$} \tag{\ref{eqn:iteration:12}'} \label{eqn:iteration:12'},
		\\
		\phi_2 ( \varrho, T_\varrho )  \leq \Delta_9 ( 1 + \log (
		r/\varrho ) )^{-1} \quad \text{whenever $t_i \leq \varrho \leq
		r$, $\alpha \tau = 1$} \tag{\ref{eqn:iteration:12}''}
		\label{eqn:iteration:12''};
	\end{gather}
	here and in the remaining proof references to equations involving the
	inductive parameter will be supplemented by the value of this
	parameter as index.

	\paragraph{Proof of
	$\eqref{eqn:iteration:ind_j4}_1$, $\eqref{eqn:iteration:ind_j5}_1$
	and $\eqref{eqn:iteration:ind_j1}_1$.}
	Since $t_1 = 4^j t \geq \frac{r}{64}$ the assertions follow from
	\eqref{eqn:iteration:5}, \eqref{eqn:iteration:3} and
	\eqref{eqn:iteration:6}.

	\paragraph{Proof of $\eqref{eqn:iteration:9}_1$.} Since $t_1 \geq
	\frac{r}{64}$ and $\Delta_7 \leq \Delta_6$, this follows from
	\eqref{eqn:iteration:7}.
	
	\paragraph{Proof of $\eqref{eqn:iteration:10}_1$.} For $t_1 \leq
	\varrho \leq r$
	\begin{gather*}
		\phi_2 ( \varrho, T ) \leq 2^{5\vdim} \phi_2 ( 8r, T ) \leq
		2^{5 \vdim}
		\varepsilon_{\eqref{item:iteration:iteration}} \leq
		\Delta_5.
	\end{gather*}

	\paragraph{Proof of $\eqref{eqn:iteration:11}_1$.} Let $\frac{r}{64}
	\leq \varrho \leq r/4$ and note
	\begin{gather*}
		\varrho \in J_4 \cap J_5 \quad \text{by
		\eqref{eqn:iteration:5} and \eqref{eqn:iteration:6}}, \quad
		2 \varrho \in J_0 \cap J_1 \quad \text{by
		\eqref{eqn:iteration:2} and \eqref{eqn:iteration:3}}, \\
		\| \sigma_{2\varrho} \| \leq \adim^{-1/2} \inf \{ \delta/2,
		1/4 \} \quad \text{by \eqref{eqn:iteration:7}}, \\
		\phi_2 ( 2 \varrho, T_{2\varrho} ) \leq 2^{4\vdim} \phi_2 (
		8r, T) \leq 2^{4 \vdim}
		\varepsilon_{\eqref{item:iteration:iteration}} \leq
		2^{-2\vdim-4} \unitmeasure{\vdim}^{1/2}.
	\end{gather*}
	Therefore by \eqref{item:iteration:starting_control}, using
	$\phi_4 ( 2 \varrho ) \leq 1$ by \eqref{eqn:iteration:1'}, $1/2 \geq
	\tau ( 1/p-1/\vdim )$, \eqref{eqn:iteration:1} and $\gamma_3 \leq
	\Delta_{11} \gamma_1$,
	\begin{gather*}
		\begin{aligned}
			\varrho \phi_1 ( \varrho ) & \leq \Delta_{10} \big (
			\phi_2 ( 2 \varrho, T_{2\varrho} ) + \phi_4 ( 2
			\varrho )^{1/2} \big ) \leq \Delta_{10} \big ( 2^{4\vdim}
			\phi_2 ( 8r , T_{8r}) + \delta^{-\tau} \phi_3 ( 2
			\varrho )^\tau \big ) \\
			& \leq \gamma \Delta_{10} \big ( 2^{4\vdim}
			\varepsilon_{\eqref{item:iteration:iteration}} +
			\delta^{-\tau} \Delta_{11} \gamma_1 \big ) \leq
			\gamma \gamma_1 {\textstyle\frac{1}{64}} \leq \gamma
			\gamma_1 ( \varrho / r )^{\alpha \tau}.
		\end{aligned}
	\end{gather*}

	\paragraph{Proof of $\eqref{eqn:iteration:12}_1$.} For $\frac{r}{64}
	\leq \varrho \leq r$ one estimates
	\begin{gather*}
		\phi_2 ( \varrho, T_{\varrho} ) \leq 2^{5\vdim} \phi_2 ( 8r,
		T_{8r}) \leq 2^{5\vdim}
		\varepsilon_{\eqref{item:iteration:iteration}} \gamma
		\leq \gamma \gamma_2 {\textstyle\frac{1}{64}} \leq \gamma
		\gamma_2 ( \varrho / r )^{\alpha \tau}.
	\end{gather*}
	
	Therefore the assertions
	$\eqref{eqn:iteration:ind_j4}_1$--$\eqref{eqn:iteration:12}_1$ are
	proven in the case $i=1$. Suppose now that the assertions
	$\eqref{eqn:iteration:ind_j4}_i$--$\eqref{eqn:iteration:12}_i$ hold
	for some $i \in \nat$ with $i \leq j$. Note $t_i \leq t_1 = 4^j t \leq
	\frac{r}{16}$. Since $t_i \in J_0 \cap J_4$ by \eqref{eqn:iteration:2}
	and $\eqref{eqn:iteration:ind_j4}_i$ and
	\begin{gather*}
		\phi_4 ( 2 t_i ) \leq \Delta_2 \leq 2^{-\vdim}
		\besicovitch{\adim}^{-1} \unitmeasure{\vdim} (1/8)
	\end{gather*}
	by \eqref{eqn:iteration:1'},
	\ref{lemma:iteration_prep}\,\eqref{item:item:iteration_prep:side_conditions:upper_bound} with
	$\varrho$ replaced by $t_i$ implies
	\begin{gather} \label{eqn:iteration:13}
		\| V \| ( \cylinder{T}{a}{\varrho}{\varrho} ) \leq ( Q + 1 )
		\unitmeasure{\vdim} 4^\vdim \varrho^\vdim \quad \text{for
		$t_{i+1} \leq \varrho \leq t_i$}.
	\end{gather}

	\paragraph{Proof of $\eqref{eqn:iteration:ind_j4}_{i+1}$,
	$\eqref{eqn:iteration:ind_j5}_{i+1}$ and
	$\eqref{eqn:iteration:ind_j1}_{i+1}$.} Let $t_{i+1} \leq
	\varrho \leq t_i$. Note $\varrho \in J_0$ by \eqref{eqn:iteration:2}.
	One estimates, using H\"older's inequality,
	\eqref{eqn:iteration:13} and \eqref{eqn:iteration:1'},
	\begin{gather*}
		\begin{aligned}
			\| \delta V \| ( \cylinder{T}{a}{\varrho}{\varrho} ) &
			\leq \| V \| ( \cylinder{T}{a}{\varrho}{\varrho}
			)^{1-1/p} \psi ( \cylinder{T}{a}{t_i}{t_i} )^{1/p} \\
			& \leq \sup \{ ( Q + 1 ) \unitmeasure{\vdim}, 1 \}
			4^\vdim \varrho^{\vdim-1} \Delta_3 \leq \kappa
			\varrho^{\vdim-1},
		\end{aligned}
	\end{gather*}
	hence $\varrho \in J_2$. Similarly, using $\eqref{eqn:iteration:10}_i$,
	\begin{gather*}
		\begin{aligned}
			& \tint{\cylinder{T}{a}{\varrho}{\varrho} \times
			\grass{\adim}{\vdim}}{} | \project{S} - \project{T} |
			\ud V ( z, S ) \\
			& \qquad \leq \| V \| (
			\cylinder{T}{a}{\varrho}{\varrho})^{1/2} \big (
			\tint{\cylinder{T}{a}{t_i}{t_i} \times
			\grass{\adim}{\vdim}}{} | \project{S} - \project{T}
			|^2 \ud V ( z,S ) \big )^{1/2} \\
			& \qquad \leq ( Q + 1 )^{1/2}
			\unitmeasure{\vdim}^{1/2} 4^\vdim \varrho^\vdim
			\Delta_5 \leq \kappa \varrho^\vdim
		\end{aligned}
	\end{gather*}
	and $\varrho \in J_3$. Together with $\eqref{eqn:iteration:ind_j4}_i$
	this implies
	\begin{gather*}
		\classification{\rel}{s}{t_{i+1} \leq s < 2r }
		\subset J_2 \cap J_3, \quad
		\classification{\rel}{s}{t_{i+1} \leq s \leq r }
		\subset J_4,
	\end{gather*}
	hence $\eqref{eqn:iteration:ind_j4}_{i+1}$. One computes for $0 < t <
	\varrho$, using \eqref{eqn:iteration:2}, \eqref{eqn:iteration:1} and
	$\gamma_3 \leq \Delta_4$,
	\begin{gather*}
		\phi_3 ( t ) \leq ( \Delta_4 )^{1/\tau} ( t/r )^\alpha \leq
			\varepsilon_{\ref{lemma:iteration_prep}\eqref{item:iteration_prep:lower_mass_bound}}
			( \adim, Q, \delta_4, p, \alpha , \delta_6)
			(t/\varrho)^\alpha.
	\end{gather*}
	Therefore, noting \eqref{eqn:iteration:2} and
	$\eqref{eqn:iteration:ind_j4}_{i+1}$,
	\ref{lemma:iteration_prep}\,\eqref{item:iteration_prep:lower_mass_bound}
	implies
	$\eqref{eqn:iteration:ind_j5}_{i+1}$. To prove $\varrho \in J_1$, one
	estimates
	\begin{gather*}
		\begin{aligned}
			& \| \eqproject{T_\varrho} - \project{T} \| \leq \| V
			\| ( \cylinder{T}{a}{\varrho}{\varrho} )^{-1/2}
			\varrho^{\vdim/2} \big ( \phi_2 ( \varrho, T_\varrho )
			+ \phi_2 ( \varrho, T ) \big ) \\
			& \qquad \leq \| V \| (
			\cylinder{T}{a}{t_{i+1}}{t_{i+1}} )^{-1/2} (
			t_i )^{\vdim/2} \big ( \phi_2 ( t_i, T_{t_i} ) +
			\phi_2 ( t_i, T ) \big ) \\
			& \qquad \leq \unitmeasure{\vdim}^{-1/2} 2^\vdim (
			\Delta_9 + \Delta_5 ) \leq 1/2
		\end{aligned}
	\end{gather*}
	by $\eqref{eqn:iteration:ind_j5}_{i+1}$ and
	$\eqref{eqn:iteration:12'}_i$, $\eqref{eqn:iteration:10}_i$, hence
	\begin{gather*}
		T_\varrho \cap \ker \pp = \{ 0 \}, \quad \varrho \in J_1.
	\end{gather*}
	
	\paragraph{Proof of $\eqref{eqn:iteration:9}_{i+1}$.} Let $t_{i+1} \leq
	\varrho \leq t_i$ and define $\varrho_k = 4^{k-1} \varrho$ for $k \in
	\nat$. Since $\varrho \leq t_i \leq r/4$, there exists $l \in \nat$
	such that $\frac{r}{16} < \varrho_l \leq r/4$. Note
	\begin{gather*}
		\varrho_k \in J_1 \cap J_5 \quad \text{for $k = 1,
		\ldots, l$}
	\end{gather*}
	by $\eqref{eqn:iteration:ind_j1}_{i+1}$ and
	$\eqref{eqn:iteration:ind_j5}_{i+1}$. Also, by
	$\eqref{eqn:iteration:9}_i$,
	\begin{gather*}
		\| \sigma_{\varrho_k} \| \leq \adim^{-1/2} / 4
		\quad \text{whenever $k \in \nat$, $2 \leq k \leq l$}
	\end{gather*}
	and, by $\eqref{eqn:iteration:12'}_i$,
	\begin{gather*}
		\phi_2 ( \varrho_k, T_{\varrho_k} ) \leq \Delta_9 \leq
		2^{-2\vdim-4} \unitmeasure{\vdim}^{1/2} \quad \text{whenever
		$k \in \nat$, $2 \leq k \leq l$}.
	\end{gather*}
	Now, applying
	\ref{lemma:iteration_prep}\,\eqref{item:iteration_prep:ch_planes} with
	$\varrho$, $s$, $t$, $\lambda$ replaced by $\varrho_k$,
	$\varrho_{k-1}$, $\varrho_{k-1}$, $1/2$ and using
	$\eqref{eqn:iteration:12'}_i$, one obtains
	\begin{gather*}
		\| \sigma_{\varrho_{k-1}} - \sigma_{\varrho_k} \| \leq
		2^{2\vdim+3} \unitmeasure{\vdim}^{-1/2} \phi_2 ( \varrho_k,
		T_{\varrho_k} ) \leq 2^{2\vdim+3} \unitmeasure{\vdim}^{-1/2}
		\Delta_9 ( \varrho_k / r)^{\alpha \tau/2}
	\end{gather*}
	whenever $k \in \nat$, $2 \leq k \leq l$. Therefore by
	\eqref{eqn:iteration:7}
	\begin{gather*}
		\begin{aligned}
			\| \sigma_\varrho \| & \leq \| \sigma_{\varrho_l} \| +
			{\textstyle\sum_{k=2}^l} \| \sigma_{\varrho_{k-1}} -
			\sigma_{\varrho_k} \| \\
			& \leq \Delta_7 + 2^{2\vdim+3}
			\unitmeasure{\vdim}^{-1/2} \Delta_9 r^{-\alpha \tau/2}
			{\textstyle\sum_{k=2}^l} ( 4^{k-1} \varrho )^{\alpha
			\tau/2} \\
			& \leq \Delta_7 + 2^{2\vdim+3}
			\unitmeasure{\vdim}^{-1/2} \Delta_9 ( 4^{l-1}
			\varrho/r )^{\alpha\tau/2}
			{\textstyle\sum_{k=0}^\infty} 2^{-k\alpha \tau} \\
			& \leq \Delta_7 + 2^{2\vdim+3}
			\unitmeasure{\vdim}^{-1/2} (
			1 - 2^{-\alpha \tau} )^{-1} \Delta_9 \leq \Delta_6.
		\end{aligned}
	\end{gather*}

	\paragraph{Proof of $\eqref{eqn:iteration:10}_{i+1}$.} For $t_{i+1} \leq
	\varrho \leq t_i$, $\varrho \in J_0$ by \eqref{eqn:iteration:2} and
	\begin{gather*}
		\phi_2 ( \varrho, T ) \leq \phi_2 ( \varrho, T_\varrho ) +
		\varrho^{-\vdim/2} \| V \| ( \cylinder{T}{a}{\varrho}{\varrho}
		)^{1/2} | \project{T} - \eqproject{T_\varrho} |
	\end{gather*}
	by H\"older's inequality. By $\eqref{eqn:iteration:12'}_i$ and
	\eqref{eqn:iteration:13}
	\begin{gather*}
		\phi_2 ( \varrho, T ) \leq 2^\vdim \Delta_9 + 2^\vdim \sup
		\{ ( Q + 1 ) \unitmeasure{\vdim}, 1 \} | \project{T} -
		\eqproject{T_\varrho} |.
	\end{gather*}
	Also by \ref{miniremark:projections}, noting $\varrho \in J_1$ by
	$\eqref{eqn:iteration:ind_j1}_{i+1}$ and
	$\eqref{eqn:iteration:9}_{i+1}$,
	\begin{gather*}
		| \project{T} - \eqproject{T_\varrho} | \leq \adim^{1/2} \|
		\project{T} - \eqproject{T_\varrho} \| \leq \adim^{1/2} \|
		\sigma_\varrho \| \leq \adim^{1/2} \Delta_6,
	\end{gather*}
	hence
	\begin{gather*}
		\phi_2 ( \varrho, T  ) \leq 2^{\vdim} \Delta_9 + 2^\vdim
		\sup \{ ( Q + 1 ) \unitmeasure{\vdim}, 1 \} \adim^{1/2}
		\Delta_6 \leq \Delta_5.
	\end{gather*}

	\paragraph{Proof of $\eqref{eqn:iteration:11}_{i+1}$.} Let $t_{i+1} \leq
	\varrho \leq t_i$. It will be shown that the hypotheses of
	\eqref{item:iteration:key_estimate} are satisfied with $\varrho$
	replaced by $4\varrho$; in fact $\varrho \leq t_1 \leq \frac{r}{16}$,
	\begin{gather*}
		8 \varrho \in J_0 \cap J_1 \quad \text{by
		\eqref{eqn:iteration:2} and $\eqref{eqn:iteration:ind_j1}_i$},
		\qquad \| \sigma_{8\varrho} \| \leq \adim^{-1/2} \delta/4
		\quad \text{by $\eqref{eqn:iteration:9}_i$},
	\end{gather*}
	and for $s \in \{ \varrho, 4 \varrho \}$
	\begin{gather*}
		s \in J_4 \cap J_5 \quad \text{by
		$\eqref{eqn:iteration:ind_j4}_{i+1}$ and
		$\eqref{eqn:iteration:ind_j5}_{i+1}$}, \\
		\phi_4 ( 2s ) \leq 2^{-\vdim} \besicovitch{\adim}^{-1}
		\unitmeasure{\vdim} (1/8) \quad \text{by
		\eqref{eqn:iteration:1'}}.
	\end{gather*}
	Therefore, in case $\alpha \tau < 1$,
	\eqref{item:iteration:key_estimate} implies, using
	$\eqref{eqn:iteration:11}_i$, $\eqref{eqn:iteration:12'}_i$,
	$\eqref{eqn:iteration:12}_i$, $\phi_3 (8\varrho) \leq 1$ by
	\eqref{eqn:iteration:1'}, \eqref{eqn:iteration:1} and $\gamma_2 = (24
	\Delta_1 )\eta^{-1} \gamma_1$, $\gamma_3 \leq \Delta_{11} \gamma_1$,
	\begin{gather*}
		\begin{aligned}
			\phi_1 ( \varrho ) & \leq \phi_1 ( 4 \varrho ) +
			\Delta_1 \big ( \phi_1 ( 4 \varrho ) \phi_2 ( 4
			\varrho, T_{4\varrho} ) + \varrho^{-1} ( \phi_2 ( 8
			\varrho, T_{8\varrho} )^2 + \phi_3 ( 8 \varrho ) )
			\big) \\
			& \leq \gamma \varrho^{-1+\alpha \tau} r^{-\alpha
			\tau} \big ( 4^{\alpha \tau-1} \gamma_1 + \Delta_1
			\Delta_9 \gamma_1 + 8 \Delta_1 \Delta_9 \gamma_2 +
			8 \Delta_1 \gamma_3 \big) \\
			& \leq \gamma \gamma_1 \varrho^{-1+\alpha \tau}
			r^{-\alpha \tau} \big ( \Delta_8 + \Delta_1 \Delta_9 +
			192 ( \Delta_1 )^2 \eta^{-1} \Delta_9 + 8 \Delta_1
			\Delta_{11} \big ) \\
			& \leq \gamma \gamma_1 \varrho^{-1+\alpha\tau}
			r^{-\alpha \tau}.
		\end{aligned}
	\end{gather*}
	Similarly, in case $\alpha \tau = 1$,
	\eqref{item:iteration:key_estimate} implies, using
	$\eqref{eqn:iteration:11}_i$, $\eqref{eqn:iteration:12''}_i$,
	$\eqref{eqn:iteration:12}_i$, \eqref{eqn:iteration:1} and $\gamma_2 =
	(24 \Delta_1 )\eta^{-1} \gamma_1$, $\gamma_3 \leq \Delta_{11}
	\gamma_1$,
	\begin{gather*}
		\begin{aligned}
			\phi_1 ( \varrho ) & \leq \phi_1 ( 4 \varrho ) +
			\Delta_1 \big ( \phi_1 ( 4 \varrho ) \phi_2 ( 4
			\varrho, T_{4\varrho} ) + \varrho^{-1} ( \phi_2 ( 8
			\varrho, T_{8\varrho} )^2 + \phi_3 ( 8 \varrho ) )
			\big) \\
			& \leq \gamma r^{-1} \big ( ( 1 + \log ( r/\varrho ) -
			\log 4 ) \gamma_1 + \Delta_1
			\Delta_9 \gamma_1 + 8 \Delta_1 \Delta_9 \gamma_2 +
			8 \Delta_1 \gamma_3 \big) \\
			& \leq \gamma \gamma_1 r^{-1} \big ( ( 1 + \log (
			r/\varrho ) - \Delta_8 ) + 
			\Delta_1 \Delta_9 + 192 ( \Delta_1 )^2 \eta^{-1}
			\Delta_9 + 8 \Delta_1 \Delta_{11} \big ) \\
			& \leq \gamma \gamma_1 r^{-1} ( 1 + \log ( r/\varrho )
			).
		\end{aligned}
	\end{gather*}

	\paragraph{Proof of $\eqref{eqn:iteration:12}_{i+1}$.} Let $t_{i+1}
	\leq \varrho \leq t_i$. First, it will be shown that the hypotheses of
	\ref{lemma:iteration_prep}\,\eqref{item:item:iteration_prep:side_conditions:inclusion} and
	\ref{lemma:iteration_prep}\,\eqref{item:item:iteration_prep:side_conditions:estimate} are
	satisfied with $\varrho$, $\lambda$ replaced by $2 \varrho$, $\eta/2$;
	in fact
	\begin{gather*}
		2 \varrho \in J_4 \cap J_5 \quad \text{by
		$\eqref{eqn:iteration:ind_j4}_{i+1}$ and
		$\eqref{eqn:iteration:ind_j5}_{i+1}$}, \\
		\phi_4 ( 4 \varrho ) \leq 2^{-\vdim} \besicovitch{\adim}^{-1}
		\unitmeasure{\vdim} \inf \big \{ \eta ( 4
		\Gamma_{\ref{lemma:lipschitz_approximation}\eqref{item:lipschitz_approximation:estimate_b}}
		( Q,\vdim ) )^{-1}, 1/8 \big \} \quad \text{by
		\eqref{eqn:iteration:1'}}.
	\end{gather*}
	Next, it will be shown that the hypotheses of
	\eqref{item:iteration:tilt_estimate} are satisfied with $\varrho$
	replaced by $4 \varrho$; in fact, noting $t \leq \varrho \leq
	\frac{r}{16}$,
	\begin{gather*}
		\{ 2 \varrho, 4 \varrho \} \subset J_4 \cap J_5
		\quad \text{by $\eqref{eqn:iteration:ind_j4}_{i+1}$, and
		$\eqref{eqn:iteration:ind_j5}_{i+1}$}, \\
		8\varrho \in J_0 \cap J_1 \quad \text{by
		\eqref{eqn:iteration:2} and $\eqref{eqn:iteration:ind_j1}_i$},
		\qquad \| \sigma_{8\varrho} \| \leq \adim^{-1/2} \delta/4
		\quad \text{by $\eqref{eqn:iteration:9}_i$}, \\
		8r \in J_2 \cap J_3 \quad \text{by \eqref{eqn:iteration:4}},
		\qquad \phi_3 ( 8 \varrho ) \leq
		\varepsilon_{\eqref{item:iteration:tilt_estimate}} (
		\vdim, \adim, Q, \delta_2, \varepsilon, \delta, p ) \quad
		\text{by \eqref{eqn:iteration:1'}}, \\
		\begin{aligned}
			& \oball{c}{2\varrho} \without \{ x \with \density^0 (
			\| f (x) \|, g (x) ) = Q \} \\
			& \qquad \subset C_{a,2\varrho} \cup \pp \biglIm
			\classification{\cylinder{T}{a}{2\varrho}{2\varrho}}{z}
			{Q > \density^\vdim ( \| V \| , z ) \in \nat }
			\bigrIm,
		\end{aligned}
	\end{gather*}
	by \ref{lemma:iteration_prep}\,\eqref{item:item:iteration_prep:side_conditions:inclusion} with
	$\varrho$ replaced by $2\varrho$, hence
	\begin{gather*}
		\begin{aligned}
			& \mathscr{L}^\vdim ( \oball{c}{2\varrho} \without \{
			x \with \density^0 ( \| f(x) \|, g (x) ) = Q \} ) \\
			& \qquad \leq ( \eta/2 ) \unitmeasure{\vdim} ( 2
			\varrho )^\vdim +
			\varepsilon_{\eqref{item:iteration:iteration}}
			\unitmeasure{\vdim} ( 2 \varrho )^\vdim \leq \eta
			\unitmeasure{\vdim} ( 2 \varrho )^\vdim
		\end{aligned}
	\end{gather*}
	by \ref{lemma:iteration_prep}\,\eqref{item:item:iteration_prep:side_conditions:estimate} with
	$\varrho$, $\lambda$ replaced by $2 \varrho$, $\eta/2$. Therefore, in
	case $\alpha \tau < 1$, \eqref{item:iteration:tilt_estimate}
	implies, using $\eqref{eqn:iteration:12'}_i$,
	$\eqref{eqn:iteration:12}_i$, $\eqref{eqn:iteration:11}_i$,
	\eqref{eqn:iteration:1}, and $\gamma_1 = \eta ( 24 \Delta_1 )^{-1}
	\gamma_2$, $\gamma_3 \leq \Delta_{12} \gamma_2$,
	\begin{align*}
		\phi_2 ( \varrho, T_\varrho ) & \leq
		\begin{aligned}[t]
			\Delta_1 \Big ( & \big ( \lambda + \eta^{1/\adim} +
			\eta^{-1} \phi_2 ( 8 \varrho, T_{8\varrho} )^{\inf \{
			1, 2/\vdim \}} \big ) \phi_2 ( 8 \varrho, T_{8\varrho}
			) \\
			& \quad + \eta^{-1} 4 \varrho \phi_1 ( 4 \varrho ) + (
			\eta^{-1} + \lambda^{-\tau} ) \phi_3 ( 8 \varrho)^\tau
			\Big )
		\end{aligned} \\
		& \leq
		\begin{aligned}[t]
			\gamma ( \varrho/ r)^{\alpha \tau} \Big ( & 8 \Delta_1
			\big ( \lambda + \eta^{1/\adim} + \eta^{-1} ( \Delta_9
			)^{1/\adim} \big ) \gamma_2 \\
			& \quad + 4 \Delta_1 \eta^{-1} \gamma_1 + 8 \Delta_1 (
			\eta^{-1} + \lambda^{-\tau}) \gamma_3 \Big )
		\end{aligned} \\
		 & \leq \gamma ( \varrho/r )^{\alpha \tau} \big (
		 {\textstyle\frac{1}{6}} \gamma_2 + 
		 {\textstyle\frac{1}{6}} \gamma_2 +
		 {\textstyle\frac{1}{6}} \gamma_2 +
		 {\textstyle\frac{1}{6}} \gamma_2 +
		 {\textstyle\frac{1}{3}} \gamma_2 \big ) = \gamma \gamma_2 (
		 \varrho / r )^{\alpha \tau}.
	\end{align*}
	Similarly, in case $\alpha \tau =1$,
	\eqref{item:iteration:tilt_estimate} implies, using
	$\eqref{eqn:iteration:12'}_i$, $\eqref{eqn:iteration:12}_i$,
	$\eqref{eqn:iteration:11}_i$, \eqref{eqn:iteration:1}, and $\gamma_1 =
	\eta ( 24 \Delta_1 )^{-1} \gamma_2$, $\gamma_3 \leq \Delta_{12}
	\gamma_2$,
	\begin{align*}
		\phi_2 ( \varrho, T_\varrho ) & \leq
		\begin{aligned}[t]
			\gamma ( \varrho/ r) ( 1 + \log (  r/\varrho ) )  \Big
			( & 8 \Delta_1 \big ( \lambda + \eta^{1/\adim} +
			\eta^{-1} ( \Delta_9)^{1/\adim} \big ) \gamma_2 \\
			& \quad + 4 \Delta_1 \eta^{-1} \gamma_1 + 8
			\Delta_1 ( \eta^{-1} + \lambda^{-\tau}) \gamma_3
			\Big )
		\end{aligned} \\
		& \leq \gamma \gamma_2 ( \varrho / r ) ( 1 + \log ( r/\varrho
		)).
	\end{align*}

	Therefore the assertions
	$\eqref{eqn:iteration:ind_j4}_i$--$\eqref{eqn:iteration:12}_i$ are
	verified whenever $i \in \nat$, $i \leq j +1$. The conclusion now
	follows from $\eqref{eqn:iteration:ind_j1}_{j+1}$,
	$\eqref{eqn:iteration:11}_{j+1}$ and $\eqref{eqn:iteration:12}_{j+1}$.
\end{proof} \setcounter{equation}{0}
\begin{lemma} \label{lemma:aux_c2rekt}
	Suppose $\vdim, \adim, Q \in \nat$, $\vdim < \adim$, either $p = \vdim
	= 1$ or $1 < p < \vdim = 2$ or $1 \leq p < \vdim > 2$ and $\frac{\vdim
	p}{\vdim-p} = 2$, $0 < \delta \leq 1$, and $1 \leq M < \infty$.

	Then there exist positive, finite numbers $\varepsilon$ and $\Gamma$
	with the following property.

	If $a \in \rel^\adim$, $0 < r < \infty$, $V \in \IVar_\vdim (
	\oball{a}{6r} )$, $\psi$ and $p$ are related to $V$ as in
	\ref{miniremark:situation}, $T \in \grass{\adim}{\vdim}$, $Z$ is a $\|
	V \|$ measurable subset of $\cylinder{T}{a}{r}{3r}$,
	\begin{gather*}
		( Q - 1/2 ) \unitmeasure{\vdim} r^\vdim \leq \| V \| (
		\cylinder{T}{a}{r}{3r} ) \leq ( Q + 1/2 ) \unitmeasure{\vdim}
		r^\vdim, \\
		\| V \| ( \cylinder{T}{a}{r}{4r} \without
		\cylinder{T}{a}{r}{r} ) \leq ( 1/2 ) \unitmeasure{\vdim}
		r^\vdim, \\
		\measureball{\| V \|}{\oball{a}{6r}} \leq M
		\unitmeasure{\vdim} r^\vdim, \quad
		\| V \| ( \cylinder{T}{a}{r/2}{r/2} ) \geq ( Q - 1/4 )
		\unitmeasure{\vdim} (r/2)^\vdim, \\
		\| V \| ( \cylinder{T}{a}{r}{3r} \without Z ) \leq \varepsilon
		\unitmeasure{\vdim} r^\vdim, \quad \big ( \tint{}{} |
		\project{S} - \project{T} |^2 \ud V (z,S) \big )^{1/2} \leq
		\varepsilon r^{\vdim/2},
	\end{gather*}
	then
	\begin{align*}
			& \big ( r^{-\vdim} \tint{\cylinder{T}{a}{r/4}{r/4}
			\times \grass{\adim}{\vdim}}{} | \project{S} -
			\project{T} |^2 \ud V (z,S) \big )^{1/2} \\
			& \qquad \leq
			\begin{aligned}[t]
				& \delta \big ( r^{-\vdim}
				\tint{\cylinder{T}{a}{r}{r} \times
				\grass{\adim}{\vdim}}{} | \project{S} -
				\project{T} |^2 \ud V (z,S) \big )^{1/2} \\
				& + \Gamma \big ( r^{-\vdim-1} \tint{Z}{}
				\dist (z-a,T) \ud \| V \| z + r^{1-\vdim/p}
				\psi ( \oball{a}{6r} )^{1/p} \big ).
			\end{aligned}
	\end{align*}
\end{lemma}
\begin{proof}
	Define
	\begin{gather*}
		L = 1/8, \quad \delta_1 = \delta_2 = \delta_3 = 1/2, \quad
		\delta_4 = 1, \quad \delta_5 = ( 40 )^{-\vdim} (
		\isoperimetric{\vdim} \vdim )^{-\vdim} / \unitmeasure{\vdim},
		\\ \displaybreak[0] \Delta_1 =
		\varepsilon_{\ref{lemma:lipschitz_approximation}} ( \adim, Q,
		L, M, \delta_1, \delta_2, \delta_3, \delta_4, \delta_5 ),
		\quad \Delta_2 = \inf \big \{ 1, ( 2 \isoperimetric{\vdim}
		)^{-1}, \Delta_1 \big \}, \\ \displaybreak[0] \mu = 1/2 \quad
		\text{if $\vdim = 1$}, \quad \mu = 1/\vdim \quad \text{if
		$\vdim > 1$}, \quad \Delta_3 =
		\Gamma_{\ref{lemma:iteration}\eqref{item:iteration:prep_tilt}}
		( \vdim, \adim, Q, \Delta_2, p, 1 ), \\ \displaybreak[0] \eta
		= \inf \big \{ \delta^{1/\mu} ( 4 \Delta_3 )^{-1/\mu},
		2^{-\vdim-1} \big \}, \quad \lambda = \inf \big \{ \delta ( 4
		\Delta_3 )^{-1}, 1 \big \}, \\ \displaybreak[0]
		\begin{aligned}
		\kappa = \inf \big \{ &
			\varepsilon_{\ref{lemma:iteration}\eqref{item:iteration:prep_tilt}}
			( \vdim, \adim, Q, \delta_2, \Delta_1, \Delta_2 , p),
			\varepsilon_{\ref{lemma:iteration_prep}\eqref{item:iteration_prep:side_conditions}}
			( \adim, \delta_4, \Delta_2 ), \\
			& 2^{-\vdim-2} \besicovitch{\adim}^{-1}
			\unitmeasure{\vdim} \eta
			\Gamma_{\ref{lemma:lipschitz_approximation}\eqref{item:lipschitz_approximation:estimate_b}}
			( Q, \vdim )^{-1} \Delta_2 \big \},
		\end{aligned} \\ \displaybreak[0]
		\begin{aligned}
			\Delta_4 = \inf \big \{ & ( M \unitmeasure{\vdim}
			)^{-1/2} 2^{-\vdim} \kappa, \unitmeasure{\vdim}^{1/2}
			2^{-\vdim-4} \adim^{-1/2} \Delta_2, \\
			& ( M \unitmeasure{\vdim} )^{-1/2} \delta^{\vdim/2} (
			4 \Delta_3 )^{-\vdim/2} \big \},
		\end{aligned} \\ \displaybreak[0]
		\varepsilon = \inf \big \{ \Delta_4, 2^{-\vdim-1} \eta \big
		\}, \\ \displaybreak[0]
		\Delta_5 = 2^{-\vdim} \besicovitch{\adim}^{-1}
		\unitmeasure{\vdim} \inf \big \{ \eta
		\Gamma_{\ref{lemma:lipschitz_approximation}\eqref{item:lipschitz_approximation:estimate_b}}
		( Q, \vdim )^{-1}/4, 1/8 \big \}, \\ \displaybreak[0]
		\Delta_6 =
		\begin{aligned} [t]
			\inf \big \{ & ( M \unitmeasure{\vdim} )^{1/p-1}
			2^{1-\vdim} \kappa,
			\varepsilon_{\ref{lemma:iteration}\eqref{item:iteration:prep_tilt}}
			( \vdim, \adim, Q, \delta_2, \Delta_1, \Delta_2 , p),
			\\
			& \Delta_2 ( \Delta_5 )^{1/p-1/\vdim} \big \},
		\end{aligned} \\ \displaybreak[0]
		\Gamma = \sup \big \{ \Delta_3 Q^{1/2} \eta^{-1}, \Delta_3 
		\lambda^{-1}, (4( Q + 1 ) \unitmeasure{\vdim} \vdim )^{1/2}
		(\Delta_6)^{-1} \big \}.
	\end{gather*}
	It will be shown that $\varepsilon$ and $\Gamma$ have the asserted
	property.

	Suppose $a$, $r$, $V$, $\psi$, $p$, $T$, and $Z$ satisfy the
	hypotheses in the body of the lemma.

	By the definition of $\Gamma$ and
	\begin{gather*}
		r^{-\vdim} \tint{\cylinder{T}{a}{r/4}{r/4} \times
		\grass{\adim}{\vdim}}{} | \project{S} - \project{T} |^2 \ud V
		(z,S) \leq 4 ( Q + 1 ) \unitmeasure{\vdim} \vdim
	\end{gather*}
	one may assume that
	\begin{gather*}
		r^{1-\vdim/p} \psi ( \oball{a}{6r} )^{1/p} \leq \Delta_6.
	\end{gather*}
	Additionally, one may assume that $Z$ is a Borel set and that $a = 0$,
	$T = \im \pp^\ast$ using isometries and identifying $\rel^\adim \simeq
	\rel^\vdim \times \rel^\codim$.

	Defining $A$, $X_1$, $f$, $c$, $\phi_2$, $\phi_3$, $\phi_4$,
	$T_\varrho$, $J_1$, $J_2$, $J_3$, $J_4$, $J_5$, $\sigma_\varrho$, and
	$C_{a,\varrho}$ as in \ref{lemma:iteration_prep} and $X =
	\oball{c}{r/2} \cap X_1 \without \pp \lIm A \without Z \rIm$, next,
	the hypotheses of
	\ref{lemma:iteration}\,\eqref{item:iteration:prep_tilt} with $\delta$,
	$P$, $\varrho$ replaced by $\Delta_2$, $0$, $r$ will be verified. The
	$\mathscr{L}^\vdim$ measurability of $X$ is a consequence of
	\ref{lemma:lipschitz_approximation}\,\eqref{item:lipschitz_approximation:ab}
	and \cite[2.2.13]{MR41:1976}. One estimates
	\begin{gather*}
		\tint{}{} | \project{S} - \project{T} | \ud V (z,S) \leq ( M
		\unitmeasure{\vdim} )^{1/2} r^\vdim \Delta_4 \leq \kappa (
		r/2 )^\vdim, \\
		\measureball{\| \delta V \|}{\oball{a}{6r}} \leq ( M
		\unitmeasure{\vdim} )^{1-1/p} r^{\vdim-1} \Delta_6 \leq \kappa
		( r/2 )^{\vdim-1},
	\end{gather*}
	hence $r/2 \in J_4 \cap J_5$ and $8r \in J_2 \cap J_3$. Also
	\begin{gather*}
		\begin{aligned}
			\| \eqproject{T_r} - \project{T} \| & \leq \| V \| (
			\cylinder{T}{a}{r/2}{r/2} )^{-1/2} 2 \phi_2 ( 6r, T )
			( 6r)^{\vdim/2} \\
			& \leq 2^{\vdim+2} \unitmeasure{\vdim}^{-1/2} \Delta_4
			\leq 1/2,
		\end{aligned} \\
		T_r \cap \ker \pp = \{ 0 \}, \quad r \in J_1
	\end{gather*}
	and, using \ref{miniremark:projections} with $S$, $S_1$, $S_2$
	replaced by $T$, $T$, $T_r$,
	\begin{gather*}
		\| \sigma_r \|^2 \leq ( 1 + \| \sigma_r \|^2  ) \|
		\eqproject{T_r} - \project{T} \|^2, \\
		\| \sigma_r \|^2 \leq \| \eqproject{T_r} -
		\project{T} \|^2 / ( 1 - \| \eqproject{T_r} -
		\project{T} \|^2 ) \leq 2 \| \eqproject{T_r} -
		\project{T} \|^2, \\
		\| \sigma_r \| \leq 2 \| \eqproject{T_r} -
		\project{T} \| \leq 2^{\vdim+3} \unitmeasure{\vdim}^{-1/2}
		\Delta_4 \leq \adim^{-1/2} \Delta_2 / 2.
	\end{gather*}
	Noting $\phi_4 ( r ) \leq \Delta_5$, one infers from
	\ref{lemma:iteration_prep}\,\eqref{item:item:iteration_prep:side_conditions:estimate}
	with $\varrho$, $\lambda$ replaced by $r/2$, $\eta/2$ that
	\begin{gather*}
		\mathscr{L}^\vdim ( C_{a,r/2} ) \leq ( \eta/2 )
		\unitmeasure{\vdim} (r/2)^\vdim.
	\end{gather*}
	Combining this with
	\begin{gather*}
		\mathscr{L}^\vdim ( \pp \lIm A \without Z \rIm ) \leq
		\mathscr{H}^\vdim ( A \without Z ) \leq \| V \| (
		\cylinder{T}{a}{r}{3r} \without Z ) \leq ( \eta/2 )
		\unitmeasure{\vdim} ( r/2 )^\vdim, \\
		\oball{c}{r/2} \without X \subset C_{a,r/2} \cup \pp \lIm A
		\without Z \rIm,
	\end{gather*}
	one obtains
	\begin{gather*}
		\mathscr{L}^\vdim ( \oball{c}{r/2} \without X ) \leq \eta
		\unitmeasure{\vdim} ( r/2 )^\vdim.
	\end{gather*}

	Now, applying \ref{lemma:iteration}\,\eqref{item:iteration:prep_tilt}
	with $\delta$, $P$, $\varrho$, and $\tau$ replaced by $\Delta_2$, $0$,
	$r$, and $1$
	yields
	\begin{align*}
		\phi_2 (r/4, T ) & \leq
		\begin{aligned}[t]
			\Delta_3 \Big ( & \big ( \lambda + ( ( M
			\unitmeasure{\vdim} )^{1/2} \Delta_4 )^{2/\vdim} + (
			\lambda + \eta^\mu) \big ) \phi_2 ( r, T ) \\
			& + \eta^{-1} r^{-\vdim-1} \norm{f}{1}{X} +
			\lambda^{-1} \phi_3 (r ) \Big )
		\end{aligned} \\
		& \leq \delta \phi_2 ( r, T ) + \Gamma \big ( Q^{-1/2}
		r^{-\vdim-1} \norm{f}{1}{X} + \phi_3 (r) \big ).
	\end{align*}
	Finally, noting
	\begin{gather*}
		\bigclassification{X}{x}{\mathscr{G} (f(x), Q \Lbrack 0
		\Rbrack ) > Q^{1/2} \gamma } \subset \pp \biglIm
		\classification{A \cap Z}{z}{ \dist (z-a,T) > \gamma } \bigrIm
	\end{gather*}
	for $0 < \gamma < \infty$, one obtains
	\begin{gather*}
		Q^{-1/2} \norm{f}{1}{X} \leq \tint{Z}{} \dist (z-a,T) \ud \| V
		\| z
	\end{gather*}
	and the conclusion follows.
\end{proof}
\section{The pointwise regularity theorem} \label{sec:thm}
Here, after verifying the hypotheses of the approximation by a $\qspace_Q (
\rel^\codim )$ valued function in \ref{lemma:hypo_approx}, the pointwise
regularity theorem is deduced from
\ref{lemma:iteration}\,\eqref{item:iteration:iteration} in
\ref{thm:pointwise_decay}. An example demonstrating the sharpness of the
modulus of continuity obtained in case $\alpha \tau = 1$ and $\vdim > 1$ is
provided in \ref{remark:pointwise_decay}. Finally, a corollary concerning
almost everywhere decay rates is included in \ref{cor:some_decay_rates}.
\begin{lemma} \label{lemma:hypo_approx}
	Suppose $\vdim, \adim, Q \in \nat$, $\vdim < \adim$, either $p = \vdim
	= 1$ or $1 \leq p < \vdim$, $0 < \alpha \leq 1$, $1 \leq M <\infty$,
	$0 < \mu \leq 1$, and $0 < \delta_i \leq 1$ for $i \in \{1,2\}$.

	Then there exists a positive, finite number $\varepsilon$ with the
	following property.

	If $a \in \rel^\adim$, $0 < r < \infty$, $V \in \IVar_\vdim (
	\oball{a}{r} )$, $\psi$ is related to $p$ and $V$ as in
	\ref{miniremark:situation}, $T \in \grass{\adim}{\vdim}$,
	\begin{gather*}
		\Delta = \inf \big \{ \mu, ( 1 + M^2 )^{-1/2} \big ( 1 -
		(1-\delta_1/2)^{1/\vdim} ( 1- \delta_1/4)^{-1/\vdim} \big )
		\big \}, \\
		\density^{\ast \vdim} ( \| V \|, a ) \geq Q-1+\delta_2, \quad
		\measureball{\| V \|}{\oball{a}{r}} \leq ( Q + 1 - \delta_1 )
		\unitmeasure{\vdim} r^\vdim, \\
		\tint{}{} | \project{S} - \project{T} | \ud V (z,S) \leq
		\varepsilon r^\vdim, \\
		\varrho^{1-\vdim/p} \psi ( \cball{a}{\varrho} )^{1/p} \leq
		\varepsilon ( \varrho/r )^\alpha \quad \text{whenever $0 <
		\varrho < r$},
	\end{gather*}
	then with $s = \Delta r$
	\begin{gather*}
		\| V \| ( \cylinder{T}{a}{s}{Ms} \without
		\cylinder{T}{a}{s}{\delta_2 s} ) \leq \delta_2
		\unitmeasure{\vdim} s^\vdim.
	\end{gather*}
\end{lemma}
\begin{proof}
	Define $\Delta$ as in the hypotheses of the body of the lemma,
	$\lambda = \big ( 1 - ( \Delta \delta_2 /4 )^2 \big )^{1/2}$,
	\begin{gather*}
		\Delta_1 = \varepsilon_{\ref{app:lemma:little_helper}} ( \adim,
		\inf \{ ( 2 \isoperimetric{\vdim} \vdim )^{-\vdim} /
		\unitmeasure{\vdim}, \delta_1 / 4 \}, \lambda, 2 ( Q + 1 ) ),
	\end{gather*}
	let $\varepsilon$ be the infimum of the following five numbers
	\begin{gather*}
		\varepsilon_{\ref{lemma:lower_mass_bound}} ( \adim, Q, \alpha,
		p , \inf \{ \delta_1 / 3, \Delta \delta_2 /2 \} ), \quad ( (Q+1)
		\unitmeasure{\vdim})^{1/p-1} ( 4 \isoperimetric{\vdim} \vdim
		)^{1-\vdim} \Delta_1, \\
		( 4 \isoperimetric{\vdim} \vdim )^{-\vdim} \Delta_1, \quad ( 2
		\isoperimetric{\vdim} )^{-1}, \quad ( \delta_2 \Delta^\vdim
		\unitmeasure{\vdim} \besicovitch{\adim}^{-1})^{1/p-1/\vdim} (
		2 \isoperimetric{\vdim} )^{-1}
	\end{gather*}
	and suppose that $\vdim$, $a$, $r$, $V$, $\psi$, $T$ and $s$ satisfy
	the hypotheses in the body of the lemma.

	First, note by \ref{lemma:lower_mass_bound} with $\delta$ replaced by
	$\inf \{ \delta_1 / 3, \Delta \delta_2 / 2 \}$
	\begin{gather*}
		\| V \| ( \classification{\oball{a}{r}}{z}{ | \perpproject{T}
		(z-a) | < \delta_2 s / 2} ) \geq \unitmeasure{\vdim}
		(Q-\delta_1/3) r^\vdim.
	\end{gather*}

	Define $A$ to be set of all $z \in \spt \| V \|$ such that
	\begin{gather*}
		\measureball{\| \delta V \|}{\cball{z}{t}} \leq
		(2\isoperimetric{\vdim})^{-1} \| V \| ( \cball{z}{t}
		)^{1-1/\vdim}
	\end{gather*}
	whenever $0 < t < \infty$ and $\cball{z}{t} \subset \oball{a}{r}$.
	Next, the following assertion will be proven:
	\begin{gather*}
		A \cap \cylinder{T}{a}{s}{Ms} \subset
		\cylinder{T}{a}{s}{\delta_2s}.
	\end{gather*}
	For this purpose suppose $z \in A \cap \spt \| V \| \cap
	\cylinder{T}{a}{s}{Ms}$ and abbreviate $t = \dist (z,\rel^\adim
	\without \oball{a}{r})$. Since $\Delta < ( 1 + M^2 )^{-1/2}$, one
	notes $\cylinder{T}{a}{s}{Ms} \subset \oball{a}{r}$ and $t > 0$. From
	\ref{app:lemma:good_point} one obtains
	\begin{gather*}
		\measureball{\| V \|}{\cball{z}{\varrho}} \geq (
		2\isoperimetric{\vdim}\vdim )^{-\vdim} \varrho^\vdim \quad
		\text{for $0 < \varrho < t$}.
	\end{gather*}
	Therefore, noting
	\begin{gather*}
		t \geq r - ( 1 + M^2 )^{1/2} \Delta r, \quad (t/r)^\vdim \geq
		( 1- \delta_1/2 ) ( 1 - \delta_1/4)^{-1} \geq 2/3, \\
		\measureball{\| V \|}{\oball{z}{t}} \leq \measureball{\| V
		\|}{\oball{a}{r}} \leq ( Q + 1 ) \unitmeasure{\vdim} r^\vdim
		\leq 2 ( Q + 1 ) \unitmeasure{\vdim} t^\vdim, \\
		\begin{aligned}
			& \measureball{\| \delta V \|}{\oball{z}{t}} \leq
			\measureball{\| \delta V \|}{\oball{a}{r}}
			\leq \big ( ( Q + 1 ) \unitmeasure{\vdim}
			\big )^{1-1/p} \varepsilon r^{\vdim-1} \\
			& \qquad \leq \big ( ( Q + 1 ) \unitmeasure{\vdim}
			\big )^{1-1/p} ( 4 \isoperimetric{\vdim} \vdim
			)^{\vdim-1} \varepsilon \| V \| ( \oball{z}{t}
			)^{1-1/\vdim} \\
			& \qquad \leq \Delta_1 \| V \| ( \oball{z}{t}
			)^{1-1/\vdim},
		\end{aligned} \\
		\begin{aligned}
			& \tint{\oball{z}{t} \times \grass{\adim}{\vdim}}{} |
			\project{S} - \project{T} | \ud V (\xi,S) \leq
			\tint{}{} | \project{S} - \project{T} | \ud V (\xi,S)
			\\
			& \qquad \leq \varepsilon r^\vdim \leq \varepsilon ( 4
			\isoperimetric{\vdim} \vdim )^{\vdim} \measureball{\|
			V \|}{\oball{z}{t}} \leq \Delta_1 \measureball{\| V
			\|}{\oball{z}{t}},
		\end{aligned}
	\end{gather*}
	one uses \ref{app:lemma:little_helper} with $\delta$, $M$, $a$, and
	$r$ replaced by $\inf \{ ( 2 \isoperimetric{\vdim} \vdim)^{-\vdim} /
	\unitmeasure{\vdim}, \delta_1 / 4 \}$, $2 ( Q+1 )$, $z$, and $t$ to
	infer
	\begin{gather*}
		\begin{aligned}
			\| V \| ( \classification{\oball{z}{t}}{\xi}{ |
			\project{T} (\xi-z) | > \lambda | \xi-z |} ) & \geq (
			1 - \delta_1/4 ) \unitmeasure{\vdim} t^\vdim \\
			& \geq ( 1 - \delta_1/2 ) \unitmeasure{\vdim} r^\vdim.
		\end{aligned}
	\end{gather*}
	Since $\measureball{\| V \|}{\oball{a}{r}} \leq ( Q + 1 - \delta_1 )
	\unitmeasure{\vdim} r^\vdim$, this implies together with the second
	paragraph that the intersection of
	\begin{gather*}
		\perpproject{T} \lIm \classification{\oball{z}{t}}{\xi}{ |
		\project{T} ( \xi - z ) | > \lambda | \xi - z |} \rIm \quad
		\text{and} \quad \classification{\rel^\adim}{\xi}{ |
		\perpproject{T} ( \xi-a ) | < \delta_2 s / 2}
	\end{gather*}
	cannot be empty. Now, estimating for $\xi \in \oball{z}{t}$ with $|
	\project{T} ( \xi - z ) | > \lambda | \xi - z |$
	\begin{gather*}
		| \perpproject{T} ( \xi - z ) | \leq ( 1- \lambda^2 )^{1/2} |
		\xi - z | \leq 2 ( 1 - \lambda^2 )^{1/2} r = \delta_2 s / 2,
	\end{gather*}
	one obtains $| \perpproject{T} (z-a) | \leq \delta_2 s$ and the
	inclusion follows.

	If $\vdim = 1$ then $A = \spt \| V \|$ and the conclusion is
	evident. Hence suppose $\vdim > 1$. The assertion of the preceding
	paragraph implies with the help of Besicovitch's covering theorem and
	H\"older's inequality the existence of countable disjointed families
	of closed balls $F_1, \ldots, F_{\besicovitch{\adim}}$ such that
	\begin{gather*}
		\spt \| V \| \cap \cylinder{T}{a}{s}{Ms} \without
		\cylinder{T}{a}{s}{\delta_2s} \subset
		{\textstyle\bigcup\bigcup} \{ F_i \with i = 1, \ldots,
		\besicovitch{\adim} \}, \\
		S \subset \oball{a}{r}, \qquad \| V \| ( S ) \leq \Delta_2
		\psi ( S )^{\vdim/(\vdim-p)}
	\end{gather*}
	whenever $S \in \bigcup \{ F_i \with i = 1, \ldots,
	\besicovitch{\adim} \}$ where $\Delta_2 = ( 2 \isoperimetric{\vdim}
	)^{\vdim p/(\vdim-p)}$, hence
	\begin{gather*}
		\begin{aligned}[b]
			& \| V \| ( \cylinder{T}{a}{s}{Ms} \without
			\cylinder{T}{a}{s}{\delta_2s} ) \leq \Delta_2
			\tsum{i=1}{\besicovitch{\adim}} \tsum{S \in F_i}{}
			\psi ( S)^{\vdim/(\vdim-p)} \\
			& \qquad \leq \Delta_2 \tsum{i=1}{\besicovitch{\adim}}
			\big ( \tsum{S \in F_i}{} \psi (S) \big
			)^{\vdim/(\vdim-p)} \leq \Delta_2 \besicovitch{\adim}
			\psi ( \oball{a}{r} )^{\vdim/(\vdim-p)} \\
			& \qquad \leq ( 2 \isoperimetric{\vdim} \varepsilon
			)^{\vdim p/(\vdim-p)} \besicovitch{\adim} r^\vdim \leq
			\delta_2 \unitmeasure{\vdim} s^\vdim.
		\end{aligned} \qedhere
	\end{gather*}
\end{proof}
\begin{theorem} \label{thm:pointwise_decay}
	Suppose $\vdim, \adim, Q \in \nat$, $\vdim < \adim$, either $p = \vdim
	= 1$ or $1 \leq p < \vdim$, $0 < \delta \leq 1$, $0 < \alpha \leq 1$,
	$0 < \tau \leq 1$, and $\tau = 1$ if $\vdim = 1$, $p/2 \leq \tau <
	\frac{\vdim p}{2(\vdim-p)}$ if $\vdim = 2$ and $\tau = \frac{\vdim
	p}{2 ( \vdim-p )}$ if $\vdim > 2$.

	Then there exist positive, finite numbers $\varepsilon$ and $\Gamma$
	with the following property.

	If $a \in \rel^\adim$, $0 < r < \infty$, $V \in \IVar_\vdim (
	\oball{a}{r} )$, $p$ and $\psi$ are related to $V$ as in
	\ref{miniremark:situation}, $T \in \grass{\adim}{\vdim}$, $\omega :
	\classification{\rel}{t}{0 < t \leq 1} \to \rel$ with $\omega (t) =
	t^{\alpha \tau}$ if $\alpha \tau < 1$ and $\omega (t) = t ( 1 + \log
	(1/t))$ if $\alpha \tau = 1$ whenever $0 < t \leq 1$, and $0 < \gamma
	\leq \varepsilon$,
	\begin{gather*}
		\density^{\ast \vdim} ( \| V \|, a ) \geq Q-1+\delta, \quad
		\measureball{\| V \|}{\oball{a}{r}} \leq ( Q + 1 - \delta )
		\unitmeasure{\vdim} r^\vdim, \\
		\big ( r^{-\vdim} \tint{}{}| \project{S} - \project{T} |^2 \ud
		V (z,S) \big )^{1/2} \leq \gamma, \\
		\| V \| (
		\classification{\cball{a}{\varrho}}{z}{\density^\vdim ( \| V
		\|, z ) \leq Q-1} ) \leq \varepsilon \unitmeasure{\vdim}
		\varrho^\vdim \quad \text{for $0 < \varrho < r$}, \\
		\varrho^{1-\vdim/p} \psi ( \cball{a}{\varrho} )^{1/p} \leq
		\gamma^{1/\tau} ( \varrho/r )^\alpha \quad \text{for $0 <
		\varrho < r$},
	\end{gather*}
	then $\density^\vdim ( \| V \|, a ) = Q$, $R = \Tan^\vdim ( \| V \|, a
	) \in \grass{\adim}{\vdim}$ and
	\begin{gather*}
		\big ( \varrho^{-\vdim} \tint{\oball{a}{\varrho} \times
		\grass{\adim}{\vdim}}{} | \project{S} - \project{R} |^2 \ud V
		(z,S) \big )^{1/2} \leq \Gamma \gamma \omega ( \varrho/r )
		\quad \text{whenever $0 < \varrho \leq r$}.
	\end{gather*}
\end{theorem}
\begin{proof}
	Define, noting $( \isoperimetric{\vdim} \vdim )^{-\vdim} \leq
	\unitmeasure{\vdim}$,
	\begin{gather*}
		\Delta_1 = \inf \big \{ 1/6, ( 17 )^{-1/2} \big ( 1 - ( 1 -
		\delta/2)^{1/\vdim} ( 1 - \delta/4 )^{-1/\vdim} \big ) \big \},
		\\
		\delta_1 = \delta/2, \quad \delta_2 = \delta/4, \quad \delta_3 =
		1 - \delta/4, \quad \delta_4 = 1, \\
		\delta_5 = ( 40)^{-\vdim} ( \isoperimetric{\vdim} \vdim
		)^{-\vdim} / \unitmeasure{\vdim}, \quad \delta_6 = \delta,
		\quad L = \delta_4 / 8, \quad M = ( \Delta_1 )^{-\vdim} ( Q +
		1 ), \\
		\delta' = \inf \big \{ 1,
		\varepsilon_{\ref{lemma:lipschitz_approximation}} ( \adim, Q,
		L , M, \delta_1, \delta_2, \delta_3, \delta_4, \delta_5 ), ( 2
		\isoperimetric{\vdim})^{-1} \big \}, \\
		\eta = \inf \{ 1, ( Q + 1 - \delta/2 )^{1/\vdim} ( Q + 1 - 3
		\delta/4 )^{-1/\vdim} - 1 \}
	\end{gather*}
	and apply \ref{lemma:iteration}\,\eqref{item:iteration:iteration} with
	$\delta$ replaced by $\delta'$ to obtain $\gamma_i$ for $i \in \{ 2,3
	\}$. Define
	\begin{gather*}
		\Delta_2 =
		\begin{aligned}[t]
			\inf \big \{ & ( Q + 1 - 3 \delta/4 )^{1/p} - ( Q + 1
			- \delta )^{1/p}, \\
			& (Q-1+\delta)^{1/p} - ( Q-1+\delta/2)^{1/p} \big \},
		\end{aligned} \\
		\Delta_3 = \inf \big \{ ( \Delta_1)^{\vdim/2}
		\varepsilon_{\ref{lemma:iteration}\eqref{item:iteration:iteration}}
		( \vdim, \adim, Q, L , M, \delta_1, \delta_2, \delta_3, p,
		\tau, \alpha, \delta_6 ), \gamma_3 \big \}, \\
		\varepsilon =
		\begin{aligned}[t]
			\inf \{ & ( \alpha p \unitmeasure{\vdim}^{1/p}
			\Delta_2 )^\tau, \\
			& ( Q + 1 )^{-1/2} \unitmeasure{\vdim}^{-1/2}
			\varepsilon_{\ref{lemma:hypo_approx}} ( \vdim, \adim,
			Q, p, \alpha, 4, 1/6, \delta, \inf \{ \eta, \delta/4
			\}), \\
			& \varepsilon_{\ref{lemma:hypo_approx}} ( \vdim,
			\adim, Q, p, \alpha, 4, 1/6, \delta, \inf \{ \eta,
			\delta/4 \})^\tau, \Delta_3, 1 \}
		\end{aligned}
	\end{gather*}
	and also
	\begin{gather*}
		\Delta_4 = \sup \big \{ \gamma_2 ( \Delta_1 \Delta_3 )^{-1}, (
		\Delta_1)^{-\vdim/2-1} \big \}, \quad \Delta_5 = (
		1-2^{-\alpha\tau} )^{-1} \quad \text{if $\alpha \tau < 1$}, \\
		\Delta_5 = 2 + 2 \log 2 \quad \text{if $\alpha \tau = 1$},
		\quad \Delta_6 = 2^{\vdim+2} \delta^{-1}
		\unitmeasure{\vdim}^{-1/2} \Delta_4 \Delta_5, \\
		\Gamma = \Delta_4 + ( Q + 1 )^{1/2} \unitmeasure{\vdim}^{1/2}
		\Delta_6.
	\end{gather*}
	
	Suppose $a$, $r$, $V$, $\psi$, $T$, and $\omega$ satisfy the
	hypotheses of the body of the theorem.

	Let $s = \Delta_1 r$. Applying \ref{lemma:aux_monotonicity} twice with
	$M$ replaced by $\varepsilon^\tau$ in conjunction with H\"older's
	inequality, one deduces the \emph{mass bounds}:
	\begin{gather*}
		( Q - 1 + \delta/2 ) \unitmeasure{\vdim} \varrho^\vdim \leq
		\measureball{\| V \|}{\oball{a}{\varrho}} \leq ( Q + 1 -
		3\delta/4 ) \unitmeasure{\vdim} \varrho^\vdim
	\end{gather*}
	for $0 < \varrho \leq r$. From \ref{lemma:hypo_approx} applied with
	$M$, $\mu$, $\delta_1$, $\delta_2$ replaced by $4$, $1/6$, $\delta$,
	$\inf \{ \eta, \delta/4 \}$ one obtains, noting $\int | \project{S} -
	\project{T} | \ud V (z,S) \leq (Q+1)^{1/2} \unitmeasure{\vdim}^{1/2}
	\varepsilon r^\vdim$ by H\"older's inequality,
	\begin{gather*}
		\| V \| ( \cylinder{T}{a}{s}{4s} \without
		\cylinder{T}{a}{s}{\eta s} ) \leq ( \delta/4 )
		\unitmeasure{\vdim} s^\vdim.
	\end{gather*}
	Together this implies, noting $( 1 + \eta ) s \leq r$,
	\begin{gather*}
		\begin{aligned}
			\measureball{\| V \|}{\oball{a}{(1+\eta)s}} & \leq ( Q
			+ 1 - 3 \delta/4 ) \unitmeasure{\vdim} ( 1 + \eta
			)^\vdim s^\vdim \\
			& \leq ( Q + 1 - \delta/2 ) \unitmeasure{\vdim}
			s^\vdim,
		\end{aligned} \\
		\cylinder{T}{a}{s}{3s} \subset ( \cylinder{T}{a}{s}{4s}
		\without \cylinder{T}{a}{s}{\eta s} ) \cup \oball{a}{(1+\eta)
		s} \\
		\| V \| ( \cylinder{T}{a}{s}{3s} ) \leq ( Q + 1 - \delta/4 )
		\unitmeasure{\vdim} s^\vdim, \\
		\| V \| ( \cylinder{T}{a}{s}{3s} ) \geq \measureball{\| V \|}
		{\oball{a}{s}} \geq ( Q - 1 + \delta/2 ) \unitmeasure{\vdim}
		s^\vdim,
	\end{gather*}
	hence, using isometries and identifying $\rel^\adim \simeq \rel^\vdim
	\times \rel^\codim$, one may assume that $a = 0$, and the
	hypotheses of \ref{lemma:iteration_prep} and \ref{lemma:iteration} are
	satisfied with $r$, $\delta$ replaced by $s$, $\delta'$.

	Defining $\phi : ( \classification{\rel}{\varrho}{0 < \varrho \leq r})
	\times \grass{\adim}{\vdim} \to \rel$ by
	\begin{gather*}
		\phi ( \varrho, R ) = \big ( \varrho^{-\vdim}
		\tint{\oball{a}{\varrho} \times \grass{\adim}{\vdim}}{} |
		\project{S} - \project{R} |^2 \ud V (z,S) \big )^{1/2}
	\end{gather*}
	for $0 < \varrho \leq r$, $R \in \grass{\adim}{\vdim}$ and
	choosing $T_\varrho \in \grass{\adim}{\vdim}$ such that
	\begin{gather*}
		\phi ( \varrho, T_\varrho ) \leq \phi ( \varrho, R ) \quad
		\text{whenever $0 < \varrho \leq r$ and $R \in
		\grass{\adim}{\vdim}$}
	\end{gather*}
	and noting $\varepsilon \leq \Delta_3$ and $\Delta_1 \leq 1/4$, one
	obtains from \ref{lemma:iteration}\,\eqref{item:iteration:iteration}
	with $r$, $\delta$ and $\gamma$, replaced by $s$, $\delta'$ and $
	\gamma/\Delta_3$ that
	\begin{gather*}
		\phi ( \varrho, T_\varrho ) \leq ( \gamma / \Delta_3) \gamma_2
		\omega ( \varrho/s) \quad \text{for $0 < \varrho \leq s$}.
	\end{gather*}
	One infers the \emph{tilt estimate}
	\begin{gather*}
		\phi ( \varrho, T_\varrho ) \leq \Delta_4 \gamma \omega (
		\varrho/r ) \quad \text{for $0 < \varrho \leq r$}.
	\end{gather*}

	Next, it will be shown that a similar estimate holds with $T_\varrho$
	replaced by a suitable $R \in \grass{\adim}{\vdim}$. Using the lower
	mass bound, one notes for $0 < \varrho/2 \leq t \leq \varrho \leq r$
	\begin{gather*}
		\begin{aligned}
			| \eqproject{T_\varrho} - \eqproject{T_t} | & \leq
			2^{\vdim+1} \delta^{-1} \unitmeasure{\vdim}^{-1/2}
			\varrho^{-\vdim/2} \big ( \varrho^{\vdim/2} \phi (
			\varrho, T_\varrho ) + t^{\vdim/2} \phi ( t, T_t )
			\big ) \\
			& \leq 2^{\vdim+2} \delta^{-1}
			\unitmeasure{\vdim}^{-1/2} \phi ( \varrho, T_\varrho
			).
		\end{aligned}
	\end{gather*}
	This implies inductively for $0 < t \leq \varrho \leq r$
	\begin{gather*}
		| \eqproject{T_t} - \eqproject{T_\varrho} | \leq 2^{\vdim+2}
		\delta^{-1} \unitmeasure{\vdim}^{-1/2} \tsum{\nu=0}{\infty}
		\phi ( 2^{-\nu} \varrho, T_{2^{-\nu}\varrho} ),
	\end{gather*}
	hence, noting that the tilt estimate yields
	\begin{gather*}
		\tsum{\nu=0}{\infty} \phi ( 2^{-\nu} \varrho, T_{2^{-\nu}
		\varrho} ) \leq \Delta_4 \gamma \tsum{\nu=0}{\infty} ( 2^{-\nu}
		\varrho/r )^{\alpha \tau} = \Delta_4 \Delta_5 \gamma \omega (
		\varrho/r ) \quad \text{if $\alpha \tau < 1$}, \\
		\begin{aligned}
			& \tsum{\nu=0}{\infty} \phi ( 2^{-\nu} \varrho,
			T_{2^{-\nu} \varrho} ) \leq \Delta_4 \gamma
			\tsum{\nu=0}{\infty} ( 2^{-\nu} \varrho/r) ( 1 + \log
			(r/\varrho) + \nu \log 2 ) \\
			& \qquad \leq \Delta_4 \gamma ( \varrho/r) ( 1 + \log
			(r/\varrho) ) \big ( 2 + \log 2 \tsum{\nu=0}{\infty}
			2^{-\nu} \nu \big ) = \Delta_4 \Delta_5 \gamma \omega (
			\varrho/r )
		\end{aligned}
	\end{gather*}
	if $\alpha \tau = 1$, there exists $R \in \grass{\adim}{\vdim}$ with
	\begin{gather*}
		| \project{R} - \eqproject{T_\varrho} | \leq \Delta_6 \gamma
		\omega ( \varrho/r ) \quad \text{whenever $0 < \varrho \leq
		r$}.
	\end{gather*}
	Combining this with the tilt estimate, one obtains, using the upper
	mass bound,
	\begin{gather*}
		\phi ( \varrho, R ) \leq \phi ( \varrho, T_\varrho ) + ( Q +
		1)^{1/2} \unitmeasure{\vdim}^{1/2} \Delta_6 \gamma \omega (
		\varrho/r) \leq \Gamma \gamma \omega ( \varrho/r ) \quad
		\text{for $0 < \varrho \leq r$}.
	\end{gather*}

	Since $0 \leq \density^\vdim ( \| V \|, a ) < \infty$ by
	\ref{lemma:aux_monotonicity}, one now infers from Allard's compactness
	theorem for integral varifolds, see e.g.~\cite[6.4]{MR0307015} or
	\cite[42.8]{MR87a:49001}, in conjunction with, e.g.,
	\ref{app:lemma:planes} that
	\begin{gather*}
		\varrho^{-\vdim} \tint{}{} f ((z-a)/\varrho,S) \ud V (z,S) \to
		Q \tint{R}{} f(z,R) \ud \mathscr{H}^\vdim z \quad \text{as
		$\varrho \to 0+$}
	\end{gather*}
	for $f \in \ccspace{\rel^\adim \times \grass{\adim}{\vdim}}$, hence
	$\density^\vdim ( \| V \|, a ) = Q$ and $R = \Tan^\vdim ( \| V \|, a
	)$.
\end{proof}
\begin{remark} \label{remark:pointwise_decay2}
	If $\alpha \tau < 1$ and $\vdim > 2$, then $\tau$ cannot be replaced
	by any larger number.

	An example is provided as follows. Defining $\eta = \frac{\alpha
	p}{\vdim-p}$, choosing for each $i \in \nat$ an $\vdim$ dimensional
	sphere $M_i$ of radius $\varrho_i = 2^{-i-\eta i-2}$ with $M_i \subset
	\oball{a}{2^{-i}} \without \cball{a}{2^{-i-1}}$, one readily verifies
	that one may take $V \in \IVar_\vdim ( \rel^\adim )$ such that $\| V
	\| = Q \mathscr{H}^\vdim \restrict T + \mathscr{H}^\vdim \restrict M$
	where $M = \bigcup_{i=1}^\infty M_i$ and $r$ sufficiently small.
\end{remark}
\begin{remark} \label{remark:pointwise_decay}
	In case $\alpha \tau = 1$, $\vdim > 1$, it can happen that
	\begin{gather*}
		\liminf_{\varrho \to 0+} \big ( \varrho^{-\vdim}
		\tint{\oball{a}{\varrho} \times \grass{\adim}{\vdim}}{} |
		\project{S} - \project{R} |^2 \ud V (z,S) \big )^{1/2} \omega
		( \varrho/r )^{-1} > 0.
	\end{gather*}

	To construct an example, assume $\codim =1$, with $\complex = \rel^2$
	take $u : \complex \to \rel$ of class $1$ such that
	\begin{gather*}
		u \big ( r e^{\mathbf{i} \theta} \big ) = r^2 ( \log r ) \cos
		( 2 \theta) \quad \text{for $0 < r < \infty$, $\theta \in
		\rel$},
	\end{gather*}
	and verify, using the homogeneity of $u$,
	\begin{gather*}
		\Lap u \big ( r e^{\mathbf{i} \theta} \big ) = 4 \cos ( 2
		\theta ) \quad \text{for $0 < r < \infty$, $\theta \in \rel$},
		\\
		| D^i u (x) | \leq \Gamma | x |^{2-i} ( 1 + \log (1/|x|) )
		\quad \text{for $x \in \oball{0}{1} \without \{0\}$, $i \in
		\{1,2\}$}
	\end{gather*}
	where $\Gamma$ is a positive, finite number, hence computing with $C$
	as in \ref{miniremark:function_C}, noting \cite[5.1.9]{MR41:1976},
	\begin{gather*}
		\left < D^2 u (x) , C ( Du(x) ) \right > = \Lap u (x) + \left
		< D^2 u (x), C ( Du(x) ) - C (0) \right >
	\end{gather*}
	for $x \in \rel^2 \without \{ 0 \}$, one obtains, since $Du(0)= 0$,
	\begin{gather*}
		\left < D^2 u , C \circ Du \right > \in \Lp{\infty} (
		\mathscr{L}^2 \restrict \oball{0}{1} ), \\
		u | \oball{0}{1} \in \Sob{}{2}{q} ( \oball{0}{1} ) \quad
		\text{for $1 \leq q < \infty$}.
	\end{gather*}
	Choosing $g \in \mathbf{O}^\ast ( \vdim, 2 )$ and defining $f = u
	\circ g$, one may now take $V$ associated to $f$ as in
	\ref{miniremark:first_variation} with $Q=1$.
\end{remark}
\begin{remark} \label{remark:classical_examples}
	Considering $V_1 \in \IVar_7 ( \rel^4 \times \rel^4 )$ and $V_2 \in
	\IVar_2 ( \mathbf{C} \times \mathbf{C} )$ characterised by
	\begin{gather*}
		\| V_1 \| = \mathscr{H}^7 \restrict \eqclassification{\rel^4
		\times \rel^4}{(x,y)} { |x|^2 = |y|^2 }, \\
		\| V_2 \| = \mathscr{H}^2 \restrict
		\eqclassification{\mathbf{C} \times \mathbf{C}} {(w,z)}{ w^3 =
		z^2 }
	\end{gather*}
	one may verify the necessity of the hypotheses
	\begin{gather*}
		r^{-\vdim} \tint{}{} | \project{S} - \project{T} |^2 \ud V
		(z,S) \ud V (z,S) \leq \varepsilon, \\
		\| V \| (
		\classification{\cball{a}{\varrho}}{z}{\density^\vdim ( \| V
		\|, z ) \leq Q-1} ) \leq \varepsilon \unitmeasure{\vdim}
		\varrho^\vdim \quad \text{for $0 < \varrho < r$}
	\end{gather*}
	even if $V$ corresponds to an absolutely area minimising current, see
	Bombieri, de~Giorgi and Giusti \cite[Theorem A]{MR0250205},
	\cite[5.4.19]{MR41:1976}, and Allard \cite[4.8\,(4)]{MR0307015}.
\end{remark}
\begin{corollary} \label{cor:some_decay_rates}
	Suppose $\vdim$, $\adim$, $p$, $U$, and $V$ are as in
	\ref{miniremark:situation}, either $\vdim \in \{1,2\}$ and $0 < \tau <
	1$ or $\sup \{ 2, p \} < \vdim$ and $\tau = \frac{\vdim p}{2(
	\vdim-p)} < 1$, and $V \in \IVar_\vdim ( U )$.

	Then
	\begin{gather*}
		\limsup_{r \to 0+} r^{-\tau-\vdim/2} \big ( \tint{\oball{a}{r}
		\times \grass{\adim}{\vdim}}{} | \project{S} - \project{T} |^2
		\ud V (z,S) \big )^{1/2} < \infty
	\end{gather*}
	for $V$ almost all $(a,T)$.
\end{corollary}
\begin{proof}
	From \cite[2.9.13,\,5]{MR41:1976} one infers that for $\| V \|$ almost
	all $a \in U$ there exists $Q \in \nat$ and $T \in
	\grass{\adim}{\vdim}$ such that for $f \in \ccspace{\rel^\adim \times
	\grass{\adim}{\vdim}}$
	\begin{gather*}
		\lim_{r \to 0+} r^{-\vdim} \tint{}{} f (r^{-1}(z-a),S) \ud V
		(z,S) = Q \tint{T}{} f (z,T) \ud \mathscr{H}^\vdim z, \\
		\density^\vdim ( \| V \| \restrict \{ z \with \density^\vdim (
		\| V \|, z ) \leq Q-1 \}, a ) = 0, \quad \density^{\ast \vdim}
		( \psi, a ) < \infty,
	\end{gather*}
	hence for such $a$ one may apply \ref{thm:pointwise_decay} with $r$
	sufficiently small and $\alpha = 1$ to infer the conclusion.
\end{proof}
\begin{remark} \label{remark:discussion1}
	The examples in \cite[1.2]{snulmenn.isoperimetric} with $q_1 = q_2 =
	2$ and $\alpha_1 = \alpha_2$ slightly larger than $\frac{\vdim
	p}{\vdim-p}$ show that $\tau$ cannot be replaced by any larger number
	provided $\vdim > 2$. However, using the present result and
	\cite[3.7\,(i)]{snulmenn.isoperimetric}, \cite[3.6]{snulmenn:c2.v3},
	it is shown in \cite[4.2\,(1)]{snulmenn:c2.v3} that ``$< \infty$'' can
	be replaced by ``$=0$''.
\end{remark}
\begin{remark} \label{remark:discussion2}
	It is shown in \cite[4.2\,(2)]{snulmenn:c2.v3} that the conclusion
	holds with $\tau=1$ if $\vdim = 1$ or $\vdim=2$ and $p > 1$ or $\vdim
	> 2$ and $p \geq 2 \vdim/ (\vdim+2)$ by use of \ref{lemma:aux_c2rekt}
	and \cite[3.6]{snulmenn:c2.v3}.
\end{remark}
\medskip
\noindent
{\small
\newline
Max-Planck-Institute for Gravitational Physics (Albert-Einstein-Institute), OT
Golm, Am M\"uh\-len\-berg 1, DE-14476 Potsdam,
Germany \newline
{\tt Ulrich.Menne@aei.mpg.de}}
\addcontentsline{toc}{section}{\numberline{}References}
\bibliography{UlrichMenne3v3}
\bibliographystyle{alpha}
\end{document}